\documentclass[10pt]{amsart}
\usepackage{amsfonts, mathabx}
\usepackage{calrsfs}

\DeclareMathAlphabet{\pazocal}{OMS}{zplm}{m}{n}

\numberwithin{equation}{section}
\usepackage{pzccal}

\usepackage{etoolbox}

\usepackage[dvipsnames]{xcolor}
\usepackage{mathrsfs}
\usepackage[alphabetic,initials]{amsrefs}
\usepackage{enumitem}
\usepackage{graphicx}

\usepackage{pdflscape}
\usepackage{chngcntr}
\usepackage{stmaryrd}
\usepackage{tikz}
\usepackage{mathtools}

\usepackage{tikz-cd}

\input xy
\xyoption{all}
\CompileMatrices
\usepackage{xcolor,sseq,amssymb}

\usepackage{hyperref}
\hypersetup{
  colorlinks=true,
  linkcolor=blue,
  citecolor=green,
  urlcolor=blue,
  filecolor=black, 
  backref
}

\usepackage[capitalise,nameinlink,noabbrev]{cleveref}

\usepackage{gjm-cyc} 

\usepackage{wasysym}
\usepackage{pdflscape}
\usepackage{todonotes}

\usepackage{tikz}

\usepackage{mathrsfs} 
\newcommand{\mcTop}{{\rm Top}}
\newcommand{\mcSet}{{\rm Set}}

\renewcommand{\EE}{\mathbb{E}}
\newcommand{\SpGn}{\Sp^{G}_{{\rm naive}}}

\usepackage{xcolor}

\colorlet{RED}{red}
\usepackage{etoolbox}

\preto\appendix{%
  \renewcommand{\thesection}{\Alph{section}}%

  \renewcommand{\theequation}{\thesection.\arabic{equation}}%

  \renewcommand{\theHequation}{\thesection.\arabic{equation}}%

  \setcounter{equation}{0}%
}

\title
{What are cyclotomic spectra and why do we need them?}

\author{Douglas C. Ravenel}
\address{Department of Mathematics\\
University of Rochester\\
Rochester, NY 14627}

\email{dcravenel@gmail.com}
\date{\today}

\begin{document}

\begin{abstract}
  This paper is an expository account of cyclotomic spectra. They are
  spectra (in the sense of homotopy theory) with additional structure
  that includes an action of the circle group, which we will denote by
  $\mT$, for torus.  Such objects come up in algebraic $K$-theory and
  its close relatives topological Hochschild homology $THH$ and
  topological cyclic homology $TC$.  They figure prominently in the
  recent disproof of the telescope conjecture for chromatic heights
  greater than 1 by Robert Burklund, Jeremy Hahn, Ishan Levy and Tomer
  Schlank.  They show that for each $n\geq 1$ and each prime $p$,
  there is a $p$-local ring spectrum $X$ of chromatic height $n$ such
  that $L_{K (n+1)}TC (X)$ and $L_{T (n+1)}TC (X)$ are distinct.  The
  present work is part of my attempt to understand their proof.
\end{abstract}

\maketitle

\begin{center}
\includegraphics[width=11cm]{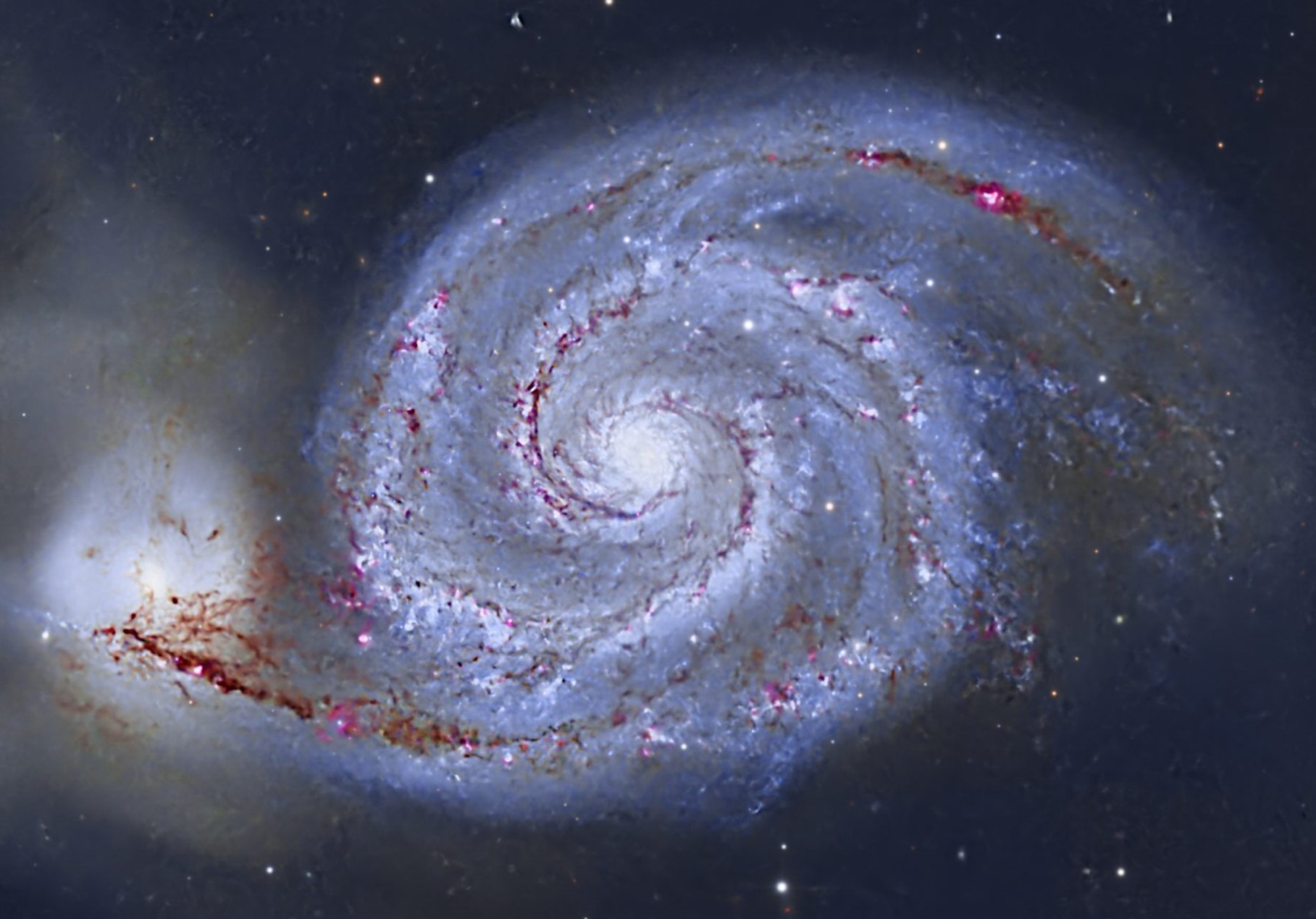}\\
M51, the Whirlpool Galaxy, photo by Dave Payne.
\end{center}

\tableofcontents

\section{Introduction}\label{sec-intro}

This paper is an expository account of cyclotomic spectra. They are
spectra (in the sense of homotopy theory) with additional structure
that includes an action of the circle group, which we will denote by
$\mT$, for torus.  Such objects come up in algebraic $K$-theory and
its close relatives topological Hochschild homology $\THH$ and
topological cyclic homology $\TopC$.  They figure prominently in the
recent disproof of the {\TC} for chromatic heights greater than 1 by
Robert Burklund, Jeremy Hahn, Ishan Levy and Tomer Schlank (hereafter
referred to collectively as BHLS) \cite{BHLS}.  Those authors show
that for each $n\geq 1$ and each prime $p$, there is a $p$-local ring
spectrum $X$ of chromatic height $n$ such that $L_{K (n+1)}\TopC (X)$
and $L_{T (n+1)}\TopC (X)$ (see \cref{def-KT-KK}) are distinct.  

The present work is part of my attempt to understand theirs.  It
includes some historical narrative based partly on this old
mathematician's personal recollections, such as the comments in
\cref{sec-Segal-conjecture}, and partly on things he has learned only
recently, such as most of \cref{sec-space-level,sec-BHM}, after
steering clear of algebraic $K$-theory for decades.

More precisely, for each prime $p$ and each integer $n>0$ they
consider a form of the truncated Brown-Peterson spectrum $\BPn$,
originally defined by Dave Johnson and Steve Wilson in \cite{JW2}, and
whose algebraic $K$-theory is the subject of a recent paper by 
Dylan Wilson and Hahn \cite{Hahn-Wilson}.  In \cite[\S5]{BHLS} the
authors define an action of the additive group of integers $\Z$, and
hence of its subgroups $p^{k}\Z$, via Adams operations.  They prove
that the ``$K$-theoretic coassembly map'' (see \cref{def-assembly})
\begin{numequation}\label{eq-loc-coassembly}
\begin{split}
L_{T (n+1)}\rK\left(\BPn^{hp^{k}\Z} \right) 
\to \left(L_{T (n+1)} \rK(\BPn) \right)^{hp^{k}\Z}
\end{split}
\end{numequation}%

\noindent is not an equivalence, but becomes one after $K(n +
1)$-localization. (\cite[Theorem A]{BHLS} says the spectrum on the
left is not $K (n+1)$-local for all $k\geq 0$.) In other words, the
algebraic $K$-theory functor on ring spectra with $\Z$-action does
does not commute with passage to homotopy fixed points, even $T
(n+1)$-locally, but its does so commute $K (n+1)$-locally.  Thus, the
telescope conjecture, which equates localizations with respect to
$T(n+1)$ and $K(n+1)$, fails.

\begin{remark}\label{rem-font}
We will follow the font convention of \cite{BHLS} and denote the
algebraic $K$-theory of a ring spectrum $R$ by $\rK(R)$ to avoid confusion
with its $nth$ Morava $K$-theory $K (n)_{*} (R)$.
\end{remark}

The reader of \cite{BHLS} should note, as its authors indicate in a
footnote on page 6, that when they speak of ``cyclotomic
hyperdescent,'' ``chromatic cyclotomic extensions'' and ``cyclotomic
redshift'' as in the title of \cite{BMCSY23}, they are using the word
``cyclotomic'' differently from its use in the title of this paper.
Their use has to do with adjoining roots of unity, or higher chromatic
analogs thereof.  In order to avoid confusion, we will sometimes refer
to that construction as ``discrete cyclotomy,'' and the subject of
the present work as ``smooth cyclotomy'' since it involves an action of the
circle group $\mT$.

Then one has the discrete cyclotomic extensions of the $T (n)$ and $K
(n)$-local sphere spectra studied in various papers of Schlank, Shay
Ben-Moshe, Shachar Carmeli and Lior Yanovski.  The {\CSY} use of the
word has deeper historical roots (by over a century) than that of
{\BHM}.  We will review discrete cyclotomy in the companion paper
\cite{Rav:gjmcyc-ext}.

Returning to smooth cyclotomy, the term ``cyclotomic spectrum'' was first
introduced by Ib Madsen in \cite[Definition 2.6]{Madsen94}, and
further studied by him and Lars Hesselholt in \cite[Definition
2.2]{HM1}.  The related notion of the cyclotomic trace was first
introduced in 1993 by Madsen, Marcel {\Bok} and Wu-Chung Hsiang in
\cite{BHM93}. The construction involves an action of $\mT$ where one
is interested in the fixed point sets of its finite subgroups.  It is
related to Alain Connes' notion of cyclic sets (see
\cref{sec-Connes}), which are simplicial sets with additional
structure.

BHLS take their discrete cyclotomic extensions and apply functors
such as algebraic $K$-theory, $\TopC$ and $\THH$, which are defined in
terms of smooth cyclotomy.  Their ``cyclotomic completion'' has to do
with discrete cyclotomy, while cyclotomic boundedness (in relation to
the Antieau-Nikolaus $t$-structure of \cite{AN21}, which we discuss in
\cref{sec-AN-t-structure}) has to do with smooth cyclotomy.  Most of
their proof takes place in the {\qcat} $\qCycSp$ of smoothly cyclotomic
spectra as in \cref{def-cyc-spectra}.

\subsection{Roadmap}\label{sec-roadmap} 
Here is a description of how the contents of this paper and
\cite{Rav:gjmcyc-ext} relate to that of \cite{BHLS}. In
\cite[\S2]{BHLS} they review cyclotomic spectra and the functors
$\TopC$, topological cyclic homology and $\TR$, topological
restriction homology.

We give two definitions of cyclotomic spectra: \cref{def-cyclo-spec}
describes them as orthogonal $\mT$-spectra with some additional
structure, and \cref{def-cyclo-spec2} describes them as objects in a
suitable {\qcat}.  The most important source of examples is
topological Hochschild homology, which we introduce at length in
\cref{sec-space-level,sec-BHM}.  It is defined from the two
perspectives in \cref{def-THHF,def-THH-qcat}.  $\TopC$ is
defined in \cref{def-TCF,def-TC}. $\TR$ is defined in
\cref{def-TR}.


The material of \cite[\S3]{BHLS} is partially treated here in the
similarly titled \cref{sec-BHLS}.  \cite[\S4]{BHLS} concerns locally
unipotent (see \cite[Definition 5.1]{Rav:gjmcyc-ext})
actions of integers on certain spectra and the {\LQ} property of
\cite[Definition 2.7]{Rav:gjmcyc-ext}.  The subject of
\cite[\S5]{BHLS} is Adams operations on $\BPn$, which we will treat in
\cite{Rav:gjmcyc-ext}.  

In the climactic \cite[\S6]{BHLS} they study two coassembly maps
 as in \cref{def-assembly}:
\begin{displaymath}
L_{T (n+1)}{\rm K} (\BPn^{hp^{k}\Z})
\to L_{T (n+1)}{\rm K} (\BPn)^{hp^{k}\Z}
\end{displaymath}

\noindent and
\begin{displaymath}
L_{K (n+1)}{\rm K} (\BPn^{hp^{k}\Z})
\to L_{K (n+1)}{\rm K} (\BPn)^{hp^{k}\Z}.
\end{displaymath}

\noindent They show that the first one is not an equivalence, but the
second one is, thereby disproving the {\TC}. In their words,

\begin{quote}
We do this by looking at the coassembly map from two highly divergent
perspectives, which are connected via trace theorems:
\begin{enumerate}
  \item From the perspective of locally unipotent $\mathbb{Z}$-actions
on ring spectra, the results of \cite[\S4]{BHLS} tell us that the coassembly
map cannot be an isomorphism.

  \item From the perspective of cyclotomic redshift of \cite{BMCSY23},
the map 

\[ L_{T(n)}\BPn ^{h p^k \mathbb{Z}}
\;\longrightarrow\; L_{T(n)}\BPn
\] 

\noindent splits after base change to the maximal abelian extension of
the $K(n)$-local sphere, and therefore the coassembly map is a
$K(n+1)$-local isomorphism.
\end{enumerate}
\end{quote}

We will discuss the maximal abelian extension of the $K(n)$-local
sphere in \cite{Rav:gjmcyc-ext}.  Like the maximal abelian extension
of the $p$-adic numbers, it is discretely cyclotomic.

\subsection{Actions of the circle group}\label{sec-circle-actions}

\begin{defin}\label{def-rho-n}
{\bf The $r$th root functor $\rho_{r}^{*}$.} 
\cite[Notation 4.1]{Blumberg-Mandell-12} 
For an integer $r\geq 1$, let 
\begin{displaymath}
\rho_{r}:\mT \to \mT /\rC _{r}
\end{displaymath}

\noindent denote the map which sends $\omega \in \mT$ to the image of
$\omega^{1/r}$ in $\mT /\rC _{r}$.  (This isomorphism is not to be
confused with the projection of $\mT$ to its quotient, which has a
kernel of order $r$, nor with the regular {\rep} $\varrho_{r}$ of
$\rC_{r}$ of \cref{def-action}.)  For a $\mT /\rC _{r}$-space $X$
(meaning a $\mT$-space on which $\rC_{r}\subseteq \mT$ acts
trivially), let $\rho_{r}^{*}X$ denote the $\mT$-space induced by the
isomorphism $\rho_{r}$.  Similarly for a {\rep} $V$ of $\mT /\rC
_{r}$, $\rho_{r}^{*}V$ denotes the {\rep} of $\mT$ induced by the same
isomorphism.
\end{defin}

\begin{defin}\label{def-smooth-cyclo}
A $\mT$-space $X$ is {\bf smoothly cyclotomic} if for each finite
subgroup $\rC_{r}\subseteq \mT$, there is a $\mT$-{\eqvr} equivalence
\begin{displaymath}
\varphi_{r}:\rho_{r}^{*}X^{\rC_{r}}\to X,
\end{displaymath}

\noindent the {\bf cyclotomic structure map}, where
$\rho_{r}^{*}X^{\rC_{r}}$ denotes the fixed point space $X^{\rC_{r}}$
with the residual action of $\mT/\rC_{r}$, which is isomorphic to
$\mT$.
\end{defin}


This definition should be compared with \cref{eq-BM-def} and
\cref{eq-NS-def} below.








\begin{ex}\label{ex-free-loop}
{\bf The free loop space.}  For any space $X$, the free loop space
$\mcL X$, the space of maps of $\mT$ into $X$, is smoothly cyclotomic.
Here we are regarding $\mT$ as the unit circle in the complex numbers
$\cxs $, and it acts on $\mcL X$ by rotation of loops.  For each
$r>1$, a loop is fixed by $C_{r}\subseteq \mT$ precisely when it
repeats itself $r$ times, meaning that it factors through the $r$-fold
covering of $\mT$.  Such a loop determines another (possibly
nonrepeating) loop by restriction to the subspace
\begin{displaymath}
\left\{e^{2\pi t\, \sqrt[]{-1}}:0\leq t\leq 1/r \right\}\subseteq \mT.
\end{displaymath}

\noindent (We do not denote $\sqrt[]{-1}$ by $i$ since we often
use that symbol as an index.) This defines a $\mT$-{\eqvr}
homeomorphism
\begin{displaymath}
\varphi_{r}:\rho_{r}^{*}(\mcL X)^{C_{r}}\to \mcL X.
\end{displaymath}
\end{ex}

\begin{prop}\label{prop-base-free}
{\bf Based loops and free loops.} Let 
\begin{displaymath}
\eval:\mcL X\to X
\end{displaymath}

\noindent be the map
given by evaluation at a given point on the circle. Then for each
point $x\in X$, $\eval^{-1} (x)$ is the space of based loops at $x$.
The map $\eval$ has a section sending each $x\in X$ to the constant
loop at that point.
\end{prop}

Under suitable hypotheses on $X$, including path connectivity, these
fibers are all equivalent and a base point can be chosen so that the
homotopy fiber of $\eval$ is $\Omega X$.

\subsection{Topological Hochschild homology and related
notions}\label{sec-THH-related-notions}

A more important example of a smoothly cyclotomic space for our
purposes is the topological Hochschild homology $\THH (R)$ of a
topological ring $R$, defined below in \cref{eq-THH} and shown to be
smoothly cyclotomic in \cref{prop-smooth-cyclo}.

Making similar definitions (of $\THH$ and cyclotomic objects) for
spectra is more delicate.  Like the definition of spectra themselves,
those of $\THH$ and cyclotomic spectra have undergone several
upheavals in the past 3 decades.  For the original definition we refer
the reader to {\Bok}'s remarkable preprint \cite{Bokstedt}, the work
of {\Bok}, Hsiang and Madsen \cite{BHM93}, and Madsen's subsequent
expositions of it \cite{Madsen, Madsen94}.  All were written in the
days before we knew how to define a smash product in the category of
spectra that is strictly associative.  This required them to tread
very carefully.  Their key idea is the use of a {\bf {\Bok} functor}
as in \cref{def-FSP}, which they call a {\bf functor with smash
product} or FSP.  Such a functor $\mcF$ determines a spectrum
$\mcF(\mS)$ with appropriate multiplicative structure. We will review
this material in \cref{sec-BHM}.

In order to do this in the category of spectra, we need to define a
symmetric monoidal structure in that category.  We know how to do this
now, but did not when $\THH$ was first considered.

The first category of spectra with a symmetric monoidal structure was
that of {\em $S$-modules} constructed by Tony Elmendorf, Igor Kriz,
Mike Mandell and Peter May in \cite{EKMM}, where $\THH$ is treated
briefly in Chapter IX.

This was followed shortly by the simpler definition of {\em symmetric
spectra} by Mark Hovey, Brooke Shipley and Jeff Smith
\cite{HoveySS}. $\THH$ in this setting is the subject of Shipley's
paper \cite{Shipley:THH}.  In it she explains how {\Bok} anticipated
the definition of symmetric spectra.  She also has to deal with
technical problems created by the unfortunate fact that an equivalence
of symmetric spectra need not induce an isomorphism of stable homotopy
groups.

These difficulties are not present in the category of {\em orthogonal
spectra}, the subject of the book \cite{MandellMay} by Mandell and May.
The study of cyclotomic orthogonal spectra is taken up by Andrew
Blumberg and Mandell in \cite[\S4]{Blumberg-Mandell-12}, \cite{BM15}
and \cite{BM24}.  We will discuss it in \cref{sec-orth-spec}.

An {\qcatal} treatment of cyclotomic spectra and related matters is
given by Thomas Nikolaus and Peter Scholze in their seminal paper
\cite[Chapter II]{NS18}, which we will discuss in \cref{sec-NS}.

In all three approaches (those of {\BHM}, Blumberg-Mandell and
Nikolaus-Scholze) there is both a global (in the sense of number
theory) definition involving the circle group $\mT$ and a $p$-adic
definition involving the Pr\"uffer group $\rC_{p^{\infty}}\subseteq
\mT$.  Only the latter is relevant to \cite{BHLS} since all spectra
there are assumed to be $p$-adically complete.

\subsection{Telescopic and chromatic localization}\label{sec-tel-chrom-loc} Recall
that the subject of the {\TC} is the relation between Bousfield
localizations with respect to $K (n)$, the $n$th Morava $K$-theory,
and ``the'' telescope $T (n)$.  The quotation marks will be explained
shortly.  First we need

\begin{defin}\label{def-MR99}
  {\bf Two notions of chromatic height.}

\begin{enumerate}[label={(\roman*)},itemindent=0em]
\item \label{def-MR99ii} \cite[Definition 1.5.3.]{Rav:NP} A $p$-local
finite spectrum $Y$ {\bf has type $n$} if\linebreak  $K (n)_{*}X\neq 0$ and $K
(m)_{*}X=0$ for all $m<n$.

\item \label{def-MR99i} 
\cite[\S3]{MahoRezk99} 
A $p$-complete
bounded below spectrum $Y$  {\bf has {\fp}-type $n$} if it has finite
type and for each finite spectrum $U$ of type $n+1$, $U\otimes Y$ is
$\pi$-finite, meaning that is has only finitely many nontrivial
homotopy groups, each of which is finite.
\end{enumerate}
\end{defin}

It is known \cite[Theorem 2.11]{Rav:Loc} that for finite spectra $X$, $K
(m)_{*}X=0$ implies $K (m-1)_{*}X=0$, but this is far from true for
infinite CW-spectra $Y$.  For example we have 
\begin{displaymath}
K (m)_{*}\BPn = \mycases{    
K (m)_{*}\otimes  H_{*} (\BPn;\Z/p)
       &\mbox{for }m\leq n\\
0      &\mbox{for }m>n,
}
\end{displaymath}

\noindent and this spectrum has  {\fp}-type $n$.  A spectrum need not
have an {\fp}-type even if it is connective and of finite type.

\begin{theorem}\label{thm-HS-per}
{\bf Hopkins-Smith periodicity.} \cite[Theorem 9]{HS}
Each type $n$ finite spectrum $V$ admits a map $v:V
\to \Sigma^{-d}V$ for some $d>0$ (when $n>0$) inducing an isomorphism
in $K (n)$ homology and a nilpotent map in ordinary mod $p$ homology.
The former condition implies that the cofiber of $v$ has type
$n+1$. We denote by $T (n)$ the filtered colimit obtained by iterating
$v$, the {\bf height $n$ telescope}.
\end{theorem}

For $n=0$, we need a map inducing an isomorphism in rational homology
and a nilpotent map in mod $p$ homology.  The degree $p$ map fits this
description.

For a given prime $p$ and height $n$, neither the finite spectrum $V$,
the map $v$, nor the telescope $T (n)$ is unique, hence the phrase
``the'' telescope above.  However for a given finite $V$ of type $n$,
the homotopy type of $T (n)$ is known to be independent of the choice
of $v$.  Better still, the Bousfield localization functor $L_{T (n)}$
associated with $T (n)$ is known to be independent of the choice of
$V$ as well, hence the notation.  It depends only on the height $n$
and the implicit prime $p$.

The original conjecture of \cite{Rav:Loc} was that the functors
$L_{T (n)}$ (telescopic localization, also known as $L_{n}^{f}$ or
$L_{n}^{\fin}$) and $L_{K (n)}$ (chromatic localization) are the same.
This was known at the time to be true for $n=0$ and $n=1$.  A few
years later it became apparent that the statement was likely to be
false for $n>1$.  {\em The authors of \cite{BHLS} use algebraic
  $K$-theory to construct counterexamples,} as their title indicates.


\begin{ex}\label{ex-BPn}
The Johnson-Wilson spectrum $BP\langle n\rangle$ has {\fp}-type $n$.
This includes $BP\langle 0\rangle:=H\Z_{(p)}$ and $BP\langle
-1\rangle:=H\Z/p$.
\end{ex}

\begin{ex}\label{ex-yn}
  For each prime $p$ and each height $n>0$, there are $p$-local Thom spectra
  $y(n)$ introduced by Mahowald in \cite{Mah:Ring} with
\begin{displaymath}
H_{*}(y(n) ;\Z/p) \cong \mycases{    
P(\xi_{i} :1\leq i\leq n)
       &\mbox{for }p=2\\
E(\tau_{i} :0\leq i\leq n-1) \otimes P(\xi_{i} :1\leq i\leq n)
       &\mbox{for }p>2
}
\end{displaymath}
 
\noindent as comodules over the dual Steenrod algebra.  Each is an
associative ({\ie} $\EE_{1}$) ring spectrum. They were studied
extensively in \cite{MRS4} by Paul Shick, Mahowald and the author in
hopes of disproving the {\TC}.

It is known that $K(m)_{*}(y(n))=0$ iff $m<n$. In other words $y(n)$
behaves as if it has type $n$ even though it is not a finite
complex. It does not have an $\fp$-type.  It is also known that there is
a self-map $\Sigma^{|v_{n}|}y(n) \to y(n)$ inducing multiplication by
$v_{n}$ in $K(n)_{*}(-)$, as \cref{thm-HS-per} would lead us to
expect.
\end{ex}

\begin{remark}\label{rem-fp-type}
The condition on $U\otimes X$ above implies that $K (m)_{*} (U\otimes
X)=0$ for all $m\geq 0$.  In the language of \cite[Definition
4.1]{Rav:Loc}, $U\otimes X$ is {\bf dissonant}.  This condition is
weaker than $\pi $-finiteness.  An infinite wedge of suspended mod $p$
{\SESM} spectra, such as $\THH (\Z/p)$ (see \cref{thm-Bokstedt}), is
dissonant but not $\pi $-finite.

If $X$ has {\fp}-type $n$, then $K (m)_{*}X=0$ for all $m>n$. On the
other hand, a type $n$ spectrum $Y$ has $K (m)_{*} (Y)\neq 0$ for all
$m\geq n$.
\end{remark}

In order to describe the counterexamples to the {\TC}, the following
notation is convenient.

\begin{defin}\label{def-KT-KK}
{\bf Telescopic and chromatic localizations of $K$-theory and 
$\TopC$.} For an $\EE_{1}$-ring spectrum $R$ and $n\geq 0$,
\begin{align*}
\rK_{T (n)} (R)
 & := L_{T (n)}\rK(R),  &
\rK_{K (n)} (R)
 & := L_{K (n)}\rK(R) , \\
\TopC_{T (n)} (R)
 & := L_{T (n)}\TopC (R),  &
\TopC_{K (n)} (R)
 & := L_{K (n)}\TopC (R).
\end{align*}
\end{defin}

The reader not familiar with the operads $\EE_{n}$ may find a brief
introduction to them (with references to other such works) in
\cite[Appendix B]{Rav:gjmcyc-ext}.

\subsection{Outline}\label{sec-outline}
We now describe the rest of the paper in more detail.

In \cref{sec-classical-algebra} we describe some classical algebra
starting with some definitions in Gerhard Hochschild's 1945 paper ``On
the cohomology groups of an associative algebra'' \cite{Hochschild}
and their generalizations due to Cartan-Eilenberg
\cite{CE}. Hochschild homology is the subject of \cref{def-HHA}.

The simplicial category $\bdelt$, simplicial sets, and related notions
are introduced in \cref{sec-simp-cat}.  Hochschild's chain complex is
reinterpreted as that of a simplicial abelian group in
\cref{eq-simp-chain}.

Connes' cyclic category $\blamb$ and cyclic sets are the subject of
\cref{sec-Connes}. The cyclic circle $\blamb^{0}$ is the subject of
\cref{cor-circle}.  $\blamb$ has the same objects as $\bdelt$, the
finite ordered sets $[n]$.  In both cases one has Yoneda functors
represented by $[n]$, the standard $n$-simplex
\cref{def-simp-sets} and the standard $n$-cyclex of
\cref{def-r-cyclic}.  In \cref{sec-edgewise-sub} we introduce the
paracyclic and $r$-cyclic categories in \cref{def-paracyclic} along
with edgewise subdivision in \cref{def-edgewise}.  We summarize these
indexing categories in \cref{sec-summary}.

In \cref{sec-double-complex} we describe two double complexes
associated with cyclic objects in an abelian category.  They fit into
two short exact sequences, \cref{eq-first-SES} and
\cref{eq-second-SES}, that we call {\bf Tate sequences}.
There is another related to cyclic spectra, \cref{eq-PPP},
which leads to a long exact sequence in homology.

In \cref{sec-BHM} we describe the groundbreaking work of {\Bok},
Hsiang and Madsen \cite{BHM93}.  Their main tool is the {\Bok} functor
of \cref{def-FSP}, which {\Bok} himself calls a ``functor with smash
product.''  We list some common examples in \cref{ex-BokF}.  To each
such functor $\mcF $ we associate spaces $\THH (\mcF )$ in
\cref{def-THHF} and $\rK (\mcF )$ in \cref{def-KF}.  They are related
by the Dennis trace of \cref{def-Dennis-trace}.

This brings us to the cyclotomic trace.  The space $\THH (\mcF )$,
which is defined to be the geometric realization of a certain
simplicial set $\THH_{\bullet} (\mcF )$, comes equipped with an action
of the circle group $\mT$ hence of each of its finite subgroups
$\rC_{r}$.  Replacing $\THH_{\bullet} (\mcF)$ with its $r$th
edgewise subdivision (which does not alter its topology) makes this
action of $\rC_{r}$ simplicial as explained in \cref{def-action}.
Topological cyclic homology $\TopC (\mcF )$ and the cyclotomic trace,
a map it receives from $\rK (\mcF )$, are the subject of \cref{def-TCF}.

We take up the ordinary (meaning without {\qcats}) theory of spectra
in \cref{sec-orth-spec}.  Let $\mcT$ denote the category of pointed
topological spaces. Initially, around 1960, a spectrum $X$ was defined
to be a sequence of pointed spaces $X_{n}$ for $n\geq 0$ with
structure maps $\epsilon^{X}_{n}:\Sigma X_{n}\to X_{n+1}$.  One could
require each of the spaces to have an action of a compact Lie group
$G$ so that the structure maps are {\eqvr} as in \cref{def-seq-spec}.

We call such spectra {\bf sequential}. They can be reinterpreted as
enriched $\mcT$-valued functors on a certain $\mcT$-enriched indexing
category $\msJ^{\bf N}$ having as objects the natural numbers.  It
turns out that $\msJ^{\bf N}$ lacks a symmetric monoidal structure,
and in hind sight this is the reason for the lack of a workable smash
product on the original category of spectra.  This was a major
headache for a generation.  In the late 90s it was found that
$\msJ^{\bf N}$ could be fattened up into an indexing category that
{\em is} symmetric monoidal.  Two instances of this are spelled out in
\cref{eq-symm-orth}, and they lead to the categories of symmetric and
orthogonal spectra of \cref{def-symm-orth}.

In \cref{sec-MM-cat} we study the Mandell-May category of
\cref{def-MM-cat}, the appropriate indexing category for orthogonal
$G$-spectra, the subject of \cref{sec-G-spectra}.  Such spectra come
equipped with two different kinds of fixed points, categorical and
geometric, spelled out in \cref{def-fixed}.

In \cref{sec-Loday-functor} we pay a brief visit to the functor of
Jean-Louis Loday, which is a method of tensoring a simplicial or
cyclic set $X$ with a commutative ring spectrum $A$.  When $X$ is the
simplicial circle it yields $\THH (A)$.

In \cref{sec-Tate-diagram} we describe the Greenlees-May diagram,
usually called the Tate diagram.  The latter name is used because the
construction imitates (in the category of spectra) the group
cohomology of \cref{eq-Tate-H*} originally defined by John Tate
(1925-2019) in \cite{Tate-coh}.

It is an essential tool in {\ESHT}. For a $G$-spectrum for
finite $G$, it leads to a cofiber sequence \cref{eq-Tate2} relating
the homotopy orbit spectrum $X_{hG}$, the homotopy fixed point
spectrum $X^{hG}$, and a third spectrum $X^{tG}$, the Tate
construction of $X$ \cref{eq-tG}.  When the group is compact but not
finite, $X_{hG}$ needs to be suspended by the adjoint {\rep}, whose
degree is the dimension of underlying manifold of $G$.

In \cref{sec-NS} we discuss {\qcats} and the work of Nikolaus
and Scholze.  Jacob Lurie has written thousands of pages on
{\qcats}, and we do not expect the reader to be familiar with all of
it.  We give specific references to this material when needed. As in
\cite{Rav:Whatis}, we will write {\qcats} which are not ordinary
categories in the color {\color{cyan}cyan}.

Elementary {\qcatal} notions are discussed in \cref{sec-elem-qcat},
limits and colimits in \cref{sec-lim-colim}, and some additional
structures in \cref{sec-additional}.  $\THH$ is defined in {\qcatal}
terms in \cref{sec-THH-qcat}. The {\qcat} of cyclotomic spectra is
defined in \cref{sec-qcat-cycsp}.  Polygonic spectra are recalled in
\cref{sec-polygonic}, and epicyclic spaces and spectra are the subject
of \cref{sec-epicyc}.

Topological cyclic homology $\TopC$ and the related functors $\TopC^{-}$
and $\TP$ are reviewed in \cref{sec-TC-et-al}.  Nikolaus-Scholze's
simplified way of computing $\TopC$ is the subject of \cref{prop-TC}.
It is applied to the mod $p$ {\SESM} spectrum in \cref{sec-TC(Zp)}, and
the integer version due to \cite{Bokstedt-Madsen94} is discussed
briefly in \cref{sec-Bok-Mad}.  The latter is  earliest instance of
{\bf chromatic redshift}.

In \cref{sec-t-structures} we introduce $t$-structures, which are
systems of full subcategories of stable {\qcats}.  The standard
example is the Postnikov $t$-structure on the {\qcat} of spectra, which
has to do with connectivity and coconnectivity.  In
\cref{sec-AN-t-structure} we describe the Antieau-Nikolaus
$t$-structure on the {\qcat} of cyclotomic spectra with its surprising
definition of coconnectivity.

Every $t$-structure determines a subcategory known as the heart, which
has an abelian homotopy category.  In the Postnikov case the latter is
the derived category of abelian groups.  In the Antieau-Nikolaus case
they call its objects $p$-typical Cartier modules
(\cref{def-Cart-mod}), which are abelian $p$-groups equipped with
natural endomorphisms $\ANF$ and $\ANV$, the Frobenius and
Verschiebung maps.  They are related to similar maps (see
\cref{def-WVrF}) in the theory of Witt vectors, which we review in
\cref{sec-Witt}.

Topological restriction homology $\TR$ (\cref{def-TR}) is the subject
of \cref{sec-TR}.  It is a functor that converts a cyclotomic spectrum
to one with an additional structure called a Frobenius lift as in
\cref{def-cyclo-spec2}\cref{def-cyclo-spec2i}.  It is known to be a
fully faithful right adjoint of the corresponding forgetful functor.

In \cref{sec-BHLS} we indicate how the machinery of the previous
sections can be brought to bear on the {\TC}.  Given a $p$-complete
$\EE_{1}$-ring spectrum $R$ with an action of the integers $\Z$, we get a diagram
of cyclotomic spectra
\begin{displaymath}
\THH (R^{h\Z})
   \to\THH (R^{h (p\Z)})
      \to\THH (R^{h (p^{2}\Z)})
          \to \dotsb \to \THH (R).
\end{displaymath}
 
\noindent and of ordinary spectra with $\THH$ replaced by $\TopC $,
$\rK$ or their localizations with respect to $K(n+1)$ or $T(n+1)$.
Whenever one applies a functor to a limit, such as homotopy fixed
points, one has a coassembly map $\epsilon$ of \cref{def-assembly}
from the value of the functor on the limit to the limiting value of
the functor.

Thus when $R$ is the sphere spectrum $\mS$ with trivial $\Z$-action, we have
\begin{displaymath}
  \mS^{h (p^{i}\Z)}\simeq \mS^{B (p^
    {i}\Z)_{+}}
\simeq \mS\vee \Sigma^{-1}\mS,
\end{displaymath}

\noindent also known as the dual circle $\DD S^{1}$.
The structure of $\THH(\DD S^{1})$ is the subject of
\cref{thm-dual-circle}, which is due to Cary Malkiewich.  Its underlying
spectrum has a single cell in dimension $-1$, but infinitely many in
dimension $0$.  On the other hand, $\THH(\mS)=\mS$, so
$\THH(\mS)^{B\Z}$ just has a single cell in dimensions $0$ and $-1$.
Hence the coassembly map
\begin{displaymath}
\epsilon :\THH (\mS^{B\Z})\to \THH (\mS)^{B\Z}
\end{displaymath}
 
\noindent
is very far from being an equivalence.

\cref{thm-failure} says that the same is true if we replace $\mS$ by a
$T(n)$-local ring spectrum $R$ with trivial $Z$-action on which
$\rK_{T (n+1)}$ is nontrivial, then the coassembly map
\begin{displaymath}
\epsilon :\rK_{T (n+1)} (R^{B\Z}) \to \rK_{T (n+1)} (R)^{B\Z},
\end{displaymath}

\noindent is not an equivalence.

If we knew that the $K (n+1)$-local analog of the coassembly map of
\cref{thm-failure} was an equivalence for $n\geq 1$, we would know
that the {\TC} is false.  What we do know is two steps removed from
this. There is a {\em particular $R$}, namely $L_{T (n)}BP\langle n
\rangle$, with a {\em nontrivial} action of $\Z$ for which the
coassembly map is a $K (n+1)$-local but not a $T (n+1)$-local
equivalence. This will be discussed in \cite{Rav:gjmcyc-ext}, where we
will see that a crucial ingredient is \cite[Theorem C]{BMCSY23}.

\bigskip

It is a pleasure to acknowledge helpful conversations with Tomer
Schlank, Ishan Levy, Jeremy Hahn, Robert Burklund, Hari Rau-Murthy,
Siddharth Gurumurthy, John Rognes, Mike Mandell, Inbar Klang, Liam
Keenan, and Mike Hopkins.

We also benefited from the notes of the 2024 Talbot Workshop \cite{Talbot24}. 

\section{Algebraic and space level constructions}\label{sec-space-level}

\subsection{Some classical algebra}\label{sec-classical-algebra}

We begin by recalling the relevant algebra.  For more background on
this material, we recommend Chuck Weibel's book \cite[Chapter
9]{Weibel-H}.

Let $A$ be an associative algebra over a field $k$ and let $M$ be a
two-sided $A$-module.  In Gerhard Hochschild's 1945 paper \cite{Hochschild},
he considered a cochain complex $C^{\bullet} (A;M)$ in which
\begin{numequation}\label{eq-CnAM}
\begin{split}
C^{n} (A;M)
 := \Hom_{k} (A^{\otimes (n+1)},M)
\end{split}
\end{numequation}%

\noindent (where the tensor products are over $k$) with coboundary
operator $\delta $ defined as follows for $f\in C^{n} (A;M)$ and
$a_{i}\in A$ for $0\leq i\leq n+1$.
\begin{numequation}\label{eq-Hochschild-cochain}
\begin{split}
(\delta f) (a_{0}, \dotsb , a_{n+1})
 & := a_{0}f (a_{1}, \dotsb , a_{n+1})  \\
 & \qquad   + \sum_{1\leq i\leq n} (-1)^{i}f (a_{0}, \dotsb, a_{i-1}
            , a_{i}a_{i+1}, a_{i+2},\dotsb , a_{n+1})  \\
 & \qquad    + (-1)^{n}f (a_{0}, \dotsb , a_{n})a_{n+1}.
\end{split}
\end{numequation}%

\noindent Note here that $f$ is $M$-valued, and the first and last
terms above make use of the left and right $A$-module structures on
$M$.  Note also that in no term on the right has the order the $a_{i}$s
changed.  We will see such a change in \cref{eq-HAM} below.


Following Henri Cartan and Sammy Eilenberg in \cite[Chapter IX]{CE}
(where $k$ was no longer assumed to be a field, but a commutative ring
over which $A$ is projective; in more recent literature only flatness
over $k$ is needed), define the {\em enveloping algebra} $A^{e}$ of
$A$ by
\begin{numequation}\label{eq-env-alg}
\begin{split}
A^{e} := A\otimes_{k}A^{\op},
\end{split}
\end{numequation}%

\noindent where $A^{\op}$ denotes $A$ with the opposite
multiplication.  When $A$ is a commutative $k$-algebra, $A^{e}\cong
A\otimes_{k} A$.  

In any case a two-sided $A$-module $M$ becomes a left
$A^{e}$-module by the formula
\begin{displaymath}
(a \otimes b ^{*} )m := a mb  
\qquad \mbox{for $a \in A$, $m\in M$ and $b ^{*}\in A^{\op}$.} 
\end{displaymath}

\noindent It is also a right $A^{e}$-module by the formula
\begin{displaymath}
m(a \otimes b ^{*})  := b  ma. 
\end{displaymath}

\noindent In particular $A$ itself is a two-sided $A^{e}$-module,
leading to an augmentation\linebreak $\epsilon :A^{e}\to A$ defined by
$\epsilon (a \otimes b ^{*})=a b $.

Cartan-Eilenberg \cite[\S IX.4]{CE} then defined the homology of a right
$A^{e}$-module $M$ by
\begin{numequation}\label{eq-Tor}
\begin{split}
H_{n} (A;M) := \Tor^{A^{e}}_{n} (M,A)
\end{split}
\end{numequation}%

\noindent and the cohomology of a left $A^{e}$-module $M$ by 
\begin{displaymath}
H^{n} (A;M) := \Ext^{n}_{A^{e}} (A,M).
\end{displaymath}

\noindent They showed that the latter coincides with the cohomology of
Hochschild's complex of \cref{eq-Hochschild-cochain}.  This will be
generalized topologically in \cref{def-THH-M}.

In both cases we need a projective $A^{e}$-resolution of $A$ as a left
$A^{e}$-module. To define one, let
\begin{displaymath}
S_{n}A:=A^{\otimes (n+2)}
\qquad \aand 
\tilde{S}_{n} (A):=A^{\otimes n},
\end{displaymath}

\noindent where the tensor products are over $k$.  Make $S_{n}A$ a
two-sided $A$-module, {\ie } a left $A^{e}$-module, by
\begin{displaymath}
(a \otimes b ^{*}) (a_{0}\otimes \dotsb \otimes a_{n+1})
 := (a a_{0})\otimes a_{1}\otimes \dotsb \otimes a_{n}
                   \otimes (a_{n+1}b  ).
\end{displaymath}

\noindent  We define $\partial_{n}:S_{n} (A)\to S_{n-1} (A)$ by 
\begin{align*}
\lefteqn{\partial_{n} (a_{0}\otimes \dotsb \otimes a_{n+1})}\qquad\qquad\\
 & := \sum_{0\leq i\leq n} (-1)^{i}a_{0}\otimes \dotsb \otimes a_{i-1}
          \otimes (a_{i}a_{i+1})\otimes a_{i+2}
           \otimes \dotsb \otimes a_{n+1}.
\end{align*}

\noindent We have
\begin{displaymath}
S_{n} (A)=A\otimes_{k} \tilde{S}_{n} (A) \otimes_{k}A
 = A^{e}\otimes_{k}\tilde{S}_{n} (A) 
\end{displaymath}

\noindent and 
\begin{displaymath}
M\otimes_{A^{e}}S_{n} (A)
 = M\otimes_{A^{e}}A^{e}\otimes_{k}\tilde{S}_{n} (A)
 = M\otimes_{k}\tilde{S}_{n} (A).
\end{displaymath}


It follows that $H_{*} (A;M)$ as in \cref{eq-Tor} is the homology of
the complex $M\otimes_{k}\tilde{S} (A)$ in which
\begin{numequation}\label{eq-HAM}
\begin{split}
\partial_{n} (m\otimes a_{1}\otimes \dotsb \otimes a_{n})
 & := ma_{1}\otimes a_{2}\otimes \dotsb \otimes a_{n}  \\
 &\qquad +\sum_{0<i<n} (-1)^{i}m\otimes a_{1}\otimes \dotsb \otimes
       a_{i}a_{i+1}\otimes \dotsb \otimes a_{n}\\
 &\qquad  + (-1)^{n}a_{n}m
\otimes a_{1}\otimes \dotsb \otimes a_{n-1}.
\end{split}
\end{numequation}%

\noindent Note here that in the last term on the right, the order of
the $a_{i}$s is cyclically permuted, unlike in \cref{eq-Hochschild-cochain}.


\begin{defin}\label{def-HHA}
{\bf Hochschild homology.}  The {\bf cyclic bar construction}
or\linebreak {\bf Hochschild complex} $C^{\Hoch} (A; M/k)$ is the chain
complex of \cref{eq-HAM}.  When $k$ is $\Z$, or it is understood from
the context, we drop it from the notation.  When $M$ is $A$ itself, we
denote it by $C^{\Hoch} (A/k)$ and we have
\begin{align*}
\partial_{n} (a_{0}\otimes  \dotsb \otimes a_{n})
 & := \sum_{0\leq i<n} (-1)^{i}a_{0}\otimes  \dotsb \otimes
       a_{i}a_{i+1}\otimes \dotsb \otimes a_{n}\\
 &\qquad\qquad  + (-1)^{n}a_{n}a_{0}
\otimes a_{1}\otimes \dotsb \otimes a_{n-1}.
 \end{align*}

\noindent The homology in this case is {\bf the Hochschild homology of
$A$}, denoted by $\HH_{*} (A/k)$.

The {\bf acyclic Hochschild complex} $C^{\acyc} (A/k)$ has the same chain
groups with boundary operator $\partial_{n}'$ given by
\begin{displaymath}
 \partial_{n}' (a_{0}\otimes  \dotsb \otimes a_{n})
:=\sum_{0\leq i<n} (-1)^{i}a_{0}\otimes  \dotsb \otimes
       a_{i}a_{i+1}\otimes \dotsb \otimes a_{n},
\end{displaymath}

\noindent in which the last term of $\partial_{n} (a_{0}\otimes \dotsb \otimes
a_{n})$ is missing.
\end{defin}


The complex $C^{\Hoch} (A)$ is denoted by $ZA$ by Tom Goodwillie in
\cite{Goodwillie-cyc}, and by $A\natural$ by Alain Connes in
\cite{Connes}.

\begin{ex}\label{ex-easy}{\bf Some easy cases.}
\begin{enumerate}[label={(\roman*)},itemindent=0em]
\item \label{ex-easyi}
For  $M=A=k$, we find that the $C^{\Hoch} (A/k)$ has the form
\begin{displaymath}
\xymatrix
@R=0mm
@C=5mm
{
{0}
  &{1}
    &{2}
      &{3}
        &{4}
          &{5}\\
{k}
  &{k}\ar[l]_(.5){0}
    &{k}\ar[l]_(.5){1}
      &{k}\ar[l]_(.5){0}
        &{k}\ar[l]_(.5){1}
          &{k}\ar[l]_(.5){0}
            &{\dotsb }\ar[l]_(.5){},
}
\end{displaymath}

\noindent leading to 
\begin{displaymath}
\HH_{i} (k/k) = \mycases{    
k      &\mbox{for }i=0\\
0      &\mbox{otherwise.}
}
\end{displaymath}
\item \label{ex-easyii} For $M=A$ with $A$ commutative, it begins with
\begin{displaymath}
\xymatrix
@R=0mm
@C=5mm
{
{0}
  &{1}
    &{2}
      &{}
        &{}
          &{}\\
{A}
  &{A\otimes_{k}A}\ar[l]^(.5){}
    &{A\otimes_{k}A\otimes_{k}A}\ar[l]^(.5){}
      &{\dotsb }\ar[l]^(.5){}\\
{a_{0}a_{1}-a_{1}a_{0}=0}
  &{a_{0}\otimes a_{1}}\ar@{|->}[l]^(.5){}\\
  &{a_{0}a_{1}\otimes a_{2}-a_{0}\otimes a_{1}a_{2}}
    &{a_{0}\otimes a_{1}\otimes a_{2}}\ar@{|->}[l]^(.5){}\\
  & {+a_{2}a_{0}}\otimes a_{1}
}
\end{displaymath}

\noindent leading to $\HH_{0} (A/k)=A$ and $\HH_{1} (A/k)$ being a
certain quotient of $A\otimes A$, namely the $A$-module of
K\"ahler differentials $\Omega^{1}_{R/k}$.  This is the $A$-module
generated by symbols $\dd x$ for $x\in A$, subject to the rules
\begin{align*}
\dd c & = 0 &\mbox{for }c\in  k,    \\
\dd (x+y)&=\dd x+ \dd y&\mbox{for }x,y\in A, \\
\aand 
\dd (xy)&=y\dd x+ x\dd y.
\end{align*}

\item \label{ex-easyiii} For $M=A$ with $A$ noncommutative, the above
shows that
\begin{displaymath}
\HH_{0} (A/k) \cong A/[A,A]\cong A\otimes_{A^{e}}A.
\end{displaymath}

\noindent It turns out that for a finitely generated projective
$A$-module $P$, the trace of its identity map has its value in this
quotient.  The group $\rK_{0} (A)$ is the Grothendieck completion of the
monoid of isomorphism classes of such $P$.  The resulting map
$K_{0} (A)\to \HH_{0} (A)$ is the {\bf Hattori-Stallings trace}
\cite{Hattori, Stallings}.  We will see in \cref{sec-Dennis} that it
refines to the topological Dennis trace
\begin{displaymath}
\Tr : K (A) \to \THH (A).
\end{displaymath}
\end{enumerate}

\noindent We will define $\rK(A)$ in \cref{def-KF} and $\THH$ in
\cref{eq-THH}.   The Dennis trace is is the subject of  \cref{sec-Dennis}.
\end{ex}

For a commutative $k$-algebra $A$, there is a ring structure on
$\HH_{*} (A/k)$ that arises from the fact that the chain complex
$C^{\Hoch} (A)$ is that of a simplicial ring.  This is proved by Achim
Krause and Thomas Nikolaus in \cite[Lemma 2.3]{KN21a}; the
construction involves shuffle maps.  We can also build a differential
graded algebra out of $\Omega^{1}_{A/k}$ by forming the free exterior
algebra $\Omega^{*}_{A/k}:=\Lambda_{A} \Omega^{1}_{A/k}$, also known
as the {\em de Rham complex of $A$ over $k$}.  When $A$ is a
polynomial algebra over $k$ on $n$ variables $x_{i}$,
$\Omega^{*}_{A/k}$ is a graded exterior algebra on the $n$ variables
$\dd x_{i}$.  This multiplicative structure leads to a map
\begin{displaymath}
\Omega^{*}_{A/k} \to \HH_{*} (R/k).
\end{displaymath}

\noindent The theorem of Hochschild, Bertram Kostant and Alexander Rosenberg
\cite{HKR62} says it is an isomorphism when $A$ satisfies a certain
smoothness condition, such as being a polynomial algebra over $k$.

The Hochschild complex is reinterpreted in \cref{eq-simp-chain} as the
chain complex of a simplicial abelian group.  The complex $C^{\acyc}
(A)$ is acyclic because there is a chain homotopy $u:C^{\acyc}
(A)_{n}\to C^{\acyc} (A)_{n+1}$ given by
\begin{numequation}\label{eq-contracting}
\begin{split}
u (a_{0}\otimes  \dotsb \otimes a_{n})
:= 1\otimes  a_{0}\otimes  \dotsb \otimes a_{n}.
\end{split}
\end{numequation}%

\noindent We will see in \cref{eq-contracting2} that this acyclicity
holds in any chain complex associated with a {\em cyclic (in the
categorical sense) abelian group} as in \cref{def-cyc-object}.

\begin{defin}\label{def-cyc-bar-construction} 
Let $E$ be a two-sided $\Gamma $-space, where $\Gamma $ is a grouplike
topological monoid, meaning one for which $\pi_{0}\Gamma $ is a group.
The {\bf topological cyclic bar construction $N^{\cyc}_{\bullet} (E;
\Gamma )$} is defined by formulas similar to those of \cref{eq-HAM}.
It is a simplicial space whose $n$th component is $E\times \Gamma
^{n}$.  When $E$ is $\Gamma $ itself, we denote it by
$N^{\cyc}_{\bullet} (\Gamma )$, and we have
\begin{displaymath}
\tau_{n} (g_{0},\dotsc ,g_{n}) = (g_{n}, g_{0},\dotsc ,g_{n-1})
\qquad \mbox{for }g_{i}\in \Gamma  .
\end{displaymath}
\end{defin}

\subsection{Morita equivalence}
The following is originally due to Kiiti
Morita \cite{Morita58}.

\begin{defin}
Two unital $k$-algebras $R$ and $S$ are {\bf Morita equivalent} if
there is a bimodule $_{R}P_{S}$ (meaning a left $R$-module and a right
$S$-module), a bimodule $_{S}Q_{R}$, an isomorphism of $R$-bimodules
$u : P \otimes_{S} Q \cong R$ and an isomorphism of $S$-bimodules $v:
Q \otimes_{R} P \cong S$.
\end{defin}

\begin{ex}
{\bf Morita equivalence of matrix rings.} 
Let $\mcM_{m} (A)$ denote the ring of $m\times m$ matrices
over $A$. For $R = A$ and $S = \mcM_{m}(A)$, take $P = A^{m}$ (row
vectors of rank $m$) and $Q = A_{m}$ (column vectors).
\end{ex}

\begin{theorem}
\cite[Theorem 1.2.7]{Loday}
If $R$ and $S$ are Morita equivalent $k$-algebras and $M$ is an
$R$-bimodule, then there is a natural isomorphism
\begin{displaymath}
H_{*} (R;M)\cong H_{*} (S; Q\otimes_{R}M\otimes_{R}P),
\end{displaymath}

\noindent for $H_{*} (R;M)$ as in \cref{eq-Tor}.
\end{theorem}

The following is discussed in more detail by Loday in
\cite[1.2]{Loday} and by Weibel in \cite[9.5]{Weibel-H}. We have
homomorphisms
\begin{displaymath}
\inc:A \to \mcM_{m} (A)
\qquad \aand 
\Tra: \mcM_{m} (A)\to A,
\end{displaymath}

\noindent where the former sends $a\in A$ to the square matrix with
$a$ in the upper left corner and zeroes elsewhere, and the latter
sends a matrix to the sum of its diagonal entries. 

Loday \cite[Definition 1.2.1]{Loday} defines a {\em generalized
trace map}
\begin{displaymath}
\Tra: \mcM_{m} (M)\otimes \mcM_{m} (A)^{\otimes n}
      \to M\otimes A^{\otimes n},
\end{displaymath}

\noindent by 
\begin{displaymath}
\Tra (\omega \otimes \alpha \otimes \beta  \otimes \dotsb
          \otimes \alpha ^{(n)} )
:= \sum \omega _{i_{0},i_{1}} \otimes 
       \alpha _{i_{1},i_{2}} \otimes \beta  _{i_{2},i_{3}} 
       \otimes \dotsb \otimes \alpha ^{(n)}_{i_{n},i_{0}},
\end{displaymath}

\noindent where $\alpha ^{(n)}$ is the $n$th letter of the Greek
alphabet, and the sum is over all possible indices $(i_{0},\dotsc ,
i_{n})$.

\begin{theorem}\label{thm-Morita}
{\bf Morita equivalence for matrix tensor products.} \cite[Theorem
1.2.4]{Loday}.  The maps $\inc$ and $\Tra$ above induce inverse
isomorphisms between $H_{*} (A;M)$ and $H_{*} (\mcM_{m} (A);\mcM_{m}
(M))$.  In particular $\HH_{*} (\mcM_{m} (A))$ is naturally isomorphic
to $\HH_{*} (A)$.
\end{theorem}





\bigskip

\subsection{Witt vectors}\label{sec-Witt}
This review follows the
treatment of Jean-Pierre Serre in \cite[\S II.6]{Serre-Local}.  For a
prime $p$ one has {\bf Witt polynomials}
\begin{displaymath}
w_{n} (x_{0}, x_{1},\dotsc
,x_{n})\in \Z[x_{0}, x_{1},\dotsc ,x_{n}]
\end{displaymath}
 
\noindent defined by
\begin{align*}
w_n (x) & := \sum_{i=0}^n p^{i} x_{i}^{p^{n-i}}
     = x_{0}^{p^{n}} + p x_{1}^{p^{n-1}} + \cdots + p^{n} x_n
\qquad \mbox{for }n\geq 0.
\end{align*}

\begin{theorem}\label{thm-Serre-thm-6}
\cite[Theorem II.6]{Serre-Local} Given a second series $(y_{0},
y_{1},\dotsc )$ of indeterminates, for each
\begin{displaymath}
\Phi \in \Z[X,Y]
\end{displaymath}

\noindent there exists a unique sequence of polynomials
\begin{displaymath}
\varphi_{n}\in \Z[x_{0},\dotsc x_{n};y_{0},\dotsc ,y_{n}]
\qquad \mbox{for }n\geq 0 
\end{displaymath}

\noindent such that 
\begin{displaymath}
w_{n} (\varphi_{0},\dotsc ,\varphi_{n})
= \Phi (w_{n} (x),w_{n} ( y)).
\end{displaymath}

In particular we have polynomials $S_{0}, S_{1},\dotsc $ and $P_{0},
P_{1},\dotsc $ associated with\linebreak  $\Phi (X,Y)=X+Y$ and $\Phi  (X,Y)=XY$
respectively.
\end{theorem}

If $A$ is an arbitrary commutative ring with
\begin{displaymath}
a = (a_0, \ldots, a_n, \ldots),
\quad 
b = (b_0, \ldots, b_n, \ldots) \in A^{\mN},
\end{displaymath}

\noindent set 
\[
a \boxplus  b := \bigl(S_0(a,b), \ldots, S_n(a,b), \ldots \bigr)
\qquad\aand 
a \boxtimes b := \bigl(P_0(a,b), \ldots, P_n(a,b), \ldots \bigr).
\]
  
\noindent For example
\begin{align*}
S_{0} (a,b) & = a_{0}+b_{0}  &
S_{1} (a,b) & = a_{1}+b_{1} 
      +\frac{a_{0}^{p}+b_{0}^{p}- (a_{0}+b_{0})^{p}}{p} \\
P_{0} (a,b) & =  a_{0}b_{0}  &
P_{1} (a,b) & =  a_{0}^{p}b_{1}+  a_{1} b_{0}^{p}+pa_{1}b_{1}.  
\end{align*}

\begin{theorem}\label{thm--Serre-thm-7}
\cite[Theorem II.7]{Serre-Local} The laws of composition
defined above make $A^{\mathbb{N}}$ into a commutative unitary ring
called the {\bf ring of Witt vectors with coefficients in $A$} and
denoted by $W(A)$.
\end{theorem}

When $A=\FFp $, $W (A)$ is the $p$-adic integers $\Z_{p}$. For
$A=\FF_{p^{k}}$, $W (A)$ is the degree $k$ extension of $\Z_{p}$
obtained by adjoining $(p^{k}-1)$th roots of unity, the integer lift
of the extension $\FF_{p^{k}}$ of $\FFp $.

\begin{defin}\label{def-WVrF}
{\bf The maps $W_{*}$, $\ANV$, $r$, and $\ANF$.}  
For a Witt vector $a= (a_{0},
a_{1},\dotsc )$, let 
\begin{displaymath}
  W_{*} (a)
  := (w_{0} (a), w_{1} (a), \dotsc )
 = (a_{0}, a_{0}^{p}+pa_{1},\dotsc )
\end{displaymath}

\noindent and
let its {\bf Verschiebung}  or {\bf shift} vector be
\begin{displaymath}
\ANV a := (0,a_{0},a_{1},\dotsc ).  
\end{displaymath}

For $x\in A$, let 
\begin{displaymath}
r (x):= (x,0,0,\dotsc )\in W (A).
\end{displaymath}

When $A$ has characteristic $p$, define the {\bf Frobenius} $\ANF:W
(A)\to W (A)$ by
\begin{displaymath}
\ANF (a_{0}, a_{1}, \dotsc ):= (a_{0}^{p}, a_{1}^{p}, \dotsc ).
\end{displaymath}
\end{defin}

Then we find that 
\begin{align*}
r (xy)
 & = r (x)\boxtimes r (y) = (xy,0,0,\dotsc ) , \\
(a_{0},a_{1},\dotsc )
 & =  r (a_{0})\boxplus \ANV  r (a_{1})\boxplus \ANV ^{2} r (a_{2})
                 \boxplus \dotsb  \\ 
 & =  \sum_{n\geq 0}\ANV ^{n}r (a_{n}), \\ 
r (x)\boxtimes (a_{0},a_{1},\dotsc )
 & = (x a_{0}, x^{p}a_{1},\dotsc ,x^{p^{n}}a_{n},\dotsc ),\\
\aand 
\ANV \ANF= \ANF \ANV  
 & = p.
\end{align*}

\subsection{The simplicial category and simplicial objects}
\label{sec-simp-cat} Simplicial sets were originally defined by
Eilenberg and Joseph Zilber in \cite{EZ}.  Here we use the indexing
conventions of Paul Goerss and Rick Jardine \cite[I.1]{Goerss-Jardine}.

\begin{defin}\label{def-simp-cat}
The {\bf simplical category} $\bdelt$ is that of finite ordered sets\linebreak 
$[n]=\left\{0,1,\dotsc ,n \right\}$ for $n\geq 0$, and order
preserving maps.  Such maps include
\begin{numequation}\label{eq-face-degen}
\begin{split}
d^{i}:[n-1]\to [n],
   &\mbox{ the injective map not having $i$ in its image}\\
\mbox{and } 
s^{i}:[n+1]\to [n],
   &\mbox{ the surjection sending both $i$ and $i+1$ to $i$,} 
\end{split}
\end{numequation}%

\noindent (both for $0\leq i\leq n$) known as coface and codegeneracy maps.
All morphisms in $\bdelt  $ are composites of them. These
satisfy the following {\bf cosimplicial identities}:
\begin{enumerate}
\item [(i)] $d^{j}d^{i}=d^{i}d^{j-1}$ for $i<j$
\item [(ii)] $s^{i}d^{j}=d^{i}s^{j-1}$ for $i<j$
\item [(iii)] $s^{j}d^{i}=\Id$ for $i=j$ and for $i=j+1$
\item [(iv)] $s^{j}d^{i}=d^{i-1}s^{j}$ for $i>j+1$
\item [(v)] $s^{j}s^{i}=s^{i}s^{j+1}$ for $i\leq j$.
\end{enumerate} 
\end{defin}

{\bf Warning.} {\em The symbol $\bdelt$ is not to be confused with
$\Delta $, which we sometimes use to denote a diagonal map.}

\begin{defin}\label{def-simp-object}
  A {\bf simplicial object $X$} in a category
$\mcC$ (sometimes denoted by $X_{\bullet}$) is a $\mcC$-valued
functor on $\bdelt ^{\op}$, in which we denote the image of $[n]$ by
$X_{n}$.  Any such functor comes equipped with face maps
$d_{i}:X_{n}\to X_{n-1}$ and degeneracy maps $s_{i}:X_{n}\to X_{n+1}$
induced by the morphisms $d^{i}$ and $s^{i}$ in $\bdelt$.  We denote
the category of such functors by $\mcC_{\bdelt}$.

When $\mcC=\Set$, the category of sets, an element in the set
$X_{n}$ is called an {\bf $n$-simplex}. It is {\bf degenerate} if it
is in the image of a degeneracy map. Otherwise it is {\em
nondegenerate}.

Similarly a {\bf cosimplicial object $X^{\bullet}$} is a $\mcC$-valued
functor on $\bdelt$, in which we denote the image of $[n]$ by $X^{n}$.

\end{defin}

The corresponding {\bf simplicial identities} are
\begin{numequation}\label{eq-simp-id}
\begin{split}
d_{i}d_{j}
 & = d_{j-1}d_{i}\qquad\hspace{1.2mm} \mbox{for }i<j   \\
d_{i}s_{j}
 & = \mycases{    
s_{j-1}d_{i}
       &\mbox{for }i<j\\
\Id    &\mbox{for }i=j,\,j+1\\
s_{j}d_{i-1}
       &\mbox{for }i>j+1\\
}\hspace{-8cm}  \\ 
s_{i}s_{j}
 & = s_{j+1}s_{i}\qquad\hspace{1.4mm} \mbox{for }i\leq j . 
\end{split}
\end{numequation}%

\begin{defin}\label{def-simp-sets}
{\bf Some simplicial sets.}
The simplicial set $\bdelt^{n}$, the {\bf standard $n$-simplex}, is
defined by
\begin{displaymath}
(\bdelt ^{n})_{k} := \bdelt ([k],[n]) = \bdelt^{\op} ([n],[k]).
\end{displaymath}

\noindent This is the Yoneda functor represented by $[n]$, so we could
denote it by $\yo^{[n]}$. The symbol $\yo$ is the Japanese hiragana
character ``yo,'' the first syllable of Yoneda's name.

In its {\bf boundary $\partial \bdelt ^{n}$}, the set of $k$-simplices
is the set of such morphisms in $\bdelt$ which are not surjective.
 
In its {\bf $i$th face}, the set of $k$-simplices is the set of such
morphisms whose image does not contain $i$.

In the {\bf $i$th horn} $\partial \bdelt ^{n}_{i}\subseteq \partial \bdelt
^{n}$ for $0\leq i\leq n$, the set of $k$-simplices is the set of
nonsurjective morphisms whose image does contain $i$.
\end{defin}

The $i$th horn is usually denoted (for example in
\cite{Goerss-Jardine} and \cite{Lurie:HTT}) by $\Lambda^{n}_{i}$, but
we will  use that symbol differently in \cref{def-r-cyclic}.

The following is an exercise for the reader.

\begin{prop}\label{prop-card}
The cardinality of $(\bdelt^{n})_{k}$ is $\binom{n+k+1}{k+1}$, and the
number of nondegenerate $k$-simplices in $\bdelt^{n}$ is $\binom{n+1}{k+1}$.
\end{prop}

Let $\mcTop $ denote the category of compactly generated weak
Hausdorff spaces.  The {\em topological $n$-simplex} is the space
\begin{numequation}\label{eq-Delta-top}
\begin{split}
\tdelt{n}
:=\left\{(x_{0}, x_{1},\dotsc ,x_{n})\in \reals^{n+1}: x_{i}\geq 0
\mbox{ and }\sum x_{i}=1 \right\}.
\end{split}
\end{numequation}%

\noindent One can check that $\tdelt{2}$ is an equilateral triangle
and $\tdelt{3}$ is a regular tetrahedron.

\begin{defin}
The {\bf geometric realization} $|X|$ of a
simplicial set $X$ is the colimit of the $\mcTop$-valued functor
\begin{displaymath}
[n]\mapsto X_{n}\times \tdelt{n}.
\end{displaymath}

The geometric realization of a simplicial space (as in
\cref{def-simp-object}) is similarly defined.  More generally if $X$
is a simplicial object in a cocomplete category $\mcC $ that is tensored over
$\mcTop$, then $|X|$ is also an object in $\mcC$.
\end{defin}
 
This space turns out to be a quotient of the disjoint union of
geometric realizations of the {\em nondegenerate} topological
simplices of $X$, meaning ones not in the image of any degeneracy map.
The data given by the face maps determine how they are glued together.
In particular, $|\bdelt^{n}| = \tdelt{n}\approx D^{n}$, $|\partial
\bdelt^{n}|\approx S^{n-1}$, and the geometric realizations of both
the the $i$th face and $i$th horn of $\bdelt^{n}$ are homeomorphic to
$D^{n-1}$.

\begin{defin}\label{def-nerve}
{\bf The nerve and classifying space of a small
category.}  For a small category $\mcC$, the {\bf nerve $N_{\bullet}
(\mcC)$} is the simplicial set given by
\begin{displaymath}
N_{n} (\mcC):=\mcCat ([n],\mcC)
\end{displaymath}

\noindent where $\mcCat (-,-)$ denotes the set of functors from one
small category to another and $[n]$ here denotes the linearly ordered
set $\left\{0,\dotsc,n \right\}$ regarded as a category. The {\bf
classifying space $B\mcC$} is the geometric realization of the nerve,
$|N(\mcC)|$.  When the category $\mcC$ is topological, we get a simplicial
space.

For a topological monoid $\Gamma$, we denote by $\mcB \Gamma$  the topological
category with one object and a morphism for each point in $\Gamma $.
Composition of morphisms is determined by the monoid structure of
$\Gamma $.

When $\mcC=\mcB \Gamma$, we denote $N_{\bullet} (\mcC)$ by
$N_{\bullet} (\Gamma)$, leading to $N_{n} (\Gamma)=\Gamma^{n}$.
\end{defin}

In other words, $N_{n} (\mcC)$ and $N^{\cyc}_{n} (\mcC)$ are the sets of
diagrams in $\mcC$ of the form
\begin{displaymath}
\xymatrix
@R=4mm
@C=4mm
{
{c_{0}}\ar[r]^(.5){}
  &{c_{1}}\ar[r]^(.5){}
    &{\dotsb }\ar[r]^(.5){}
      &{c_{n-1}}\ar[r]^(.5){}{}
        &{c_{n}}
}
\end{displaymath}

\noindent and 
\begin{displaymath}
\xymatrix@1
@R=4mm
@C=5.5mm
{
{c_{0}}\ar[r]^(.5){}
  &{c_{1}}\ar[r]^(.5){}
    &{\dotsb }\ar[r]^(.5){}
      &{c_{n-1}}\ar[r]^(.5){}{}
        &{c_{n},}\ar@/^1pc/[llll]^(.5){}
}
\end{displaymath}

\noindent where for each $i$, the $(n+1)$-fold composite morphism
$c_{i}\to c_{i}$ is {\em not} required to be the identity. Of the
$n+1$ face maps $N_{n} (\mcC)\to N_{n-1} (\mcC)$, $n-1$ are obtained
by composing each of the $n-1$ pairs of adjacent arrows above, and the
other two are obtained by ignoring the maps from $c_{0}$ and to
$c_{n}$.  All face maps 
\begin{displaymath}
N^{\cyc}_{n} (\mcC)\to N^{\cyc}_{n-1} (\mcC)
\end{displaymath}

\noindent
are obtained by such composition.  In both cases the $n+1$ degeneracy
maps are obtained by inserting the identity map on $c_{i}$ for each
$i$.

The inclusion functors $[n]\to [n]^{\cyc}$ induce maps of simplicial sets 
\begin{displaymath}
N^{\cyc}_{\bullet} (\mcC ) \to
N_{\bullet} (\mcC ) 
\end{displaymath}

\noindent and their geometric realizations. Composing with the action
of $\mT$ on the cyclic space $|N^{\cyc}_{\bullet} (\mcC )|$, we get maps
\begin{numequation}\label{eq-maps}
\begin{split}
\xymatrix
@R=3mm
@C=10mm
{
{\mT \times |N^{\cyc}_{\bullet} (\mcC )|}\ar[r]^(.55){\mu }
  &{|N^{\cyc}_{\bullet} (\mcC)|} \ar[r]^(.5){\pi }
    &{|N_{\bullet} (\mcC ) |}\\
{|N^{\cyc}_{\bullet} (\mcC )|}\ar[rr]^(.53){f}
  & &{\mcL |N_{\bullet} (\mcC )|}
}
\end{split}
\end{numequation}%

\noindent where $f$ is corresponds to $\pi \mu $ under the
topological adjunction
\begin{displaymath}
\Map (\mT\times X,\,\,Y)\,\,\cong\,\, \Map (X,\,\,\mcL Y).
\end{displaymath}

\begin{theorem}\label{thm-eqvr-equiv}
{\bf A $\mT$-{\eqvr} equivalence.}  \cite[Theorem 7.3.11]{Loday15}.
When $\mcC $ is the one object category associated with a topological
or simplicial group $G$, the map $f$ of \cref{eq-maps} is a
$\mT$-{\eqvr} equivalence.
\end{theorem}




\begin{defin}\label{def-chain-complex}
{\bf The chain complex of a simplicial abelian group.}  Let $C$ be a
simplicial abelian group or more generally a simplicial object in an
abelian category. The chain complex $\Ch (C)$, in which the $n$th
chain group (or abelian object) is $C_{n}$ and the $n$th boundary
operator $\partial_{n}:C_{n}\to C_{n-1}$ for $n>0$ is given by
\begin{numequation}\label{eq-part-n}
\begin{split}
\partial_{n}:=\sum_{0\leq i\leq n} (-1)^{i}d_{i}.
\end{split}
\end{numequation}%

When $C$ is the free abelian group $\Z X$ on a simplicial set $X$, the
chain complex $\Ch ( \Z X)$ is called the {\bf Moore complex} of $X$.

The {\bf normalized chain complex} $N\Ch (C)$ has as its $n$th chain group
\begin{displaymath}
N\Ch (C)_{n} := \bigcap_{i=0}^{n-1}\ker (d_{i})\subseteq C_{n},
\end{displaymath}


\noindent and $n$th boundary operator is  $d_{n}$.

\end{defin}

The geometric realization $|\Ch (C) |$ of the underlying simplicial
set is known to be a generalized {\SESM} space with
\begin{numequation}\label{eq-DK}
\begin{split}
\pi_{*}|C| =H_{*}C.
\end{split}
\end{numequation}%

\noindent This is the {\em Dold-Kan correspondence}, which is nicely
explained by Akhil Mathew in \cite{Mathew-DK}.

\begin{defin}\label{def-Sing}
For a topological space $X$, the {\bf singular simplicial set 
$\Sing_{\bullet} X$} is defined by
\begin{displaymath}
\Sing_{n} X:=\Map (\tdelt{n},X),
\end{displaymath}

\noindent the set of continuous maps $\tdelt{n}\to X$.  Face and
degeneracy maps are defined in terms of maps among the $\tdelt{n}$s.
\end{defin}

The functor $\Sing$ is the right adjoint of geometric realization.
This was stated without proof by Dan Kan in \cite{Kan58}. The singular
chain complex of $X$ is by definition the Moore complex of $\Sing(X)$
as in \cref{def-chain-complex}.

The Hochschild chain complex $C^{\Hoch} (A)$ of \cref{def-HHA} is ${\rm
Ch}(\mathbf{HH}_{\bullet} (A))$, where the simplicial abelian group
$\mathbf{HH}_{\bullet} (A)$ (which Connes denotes by $A^{\natural}$ in
\cite[\S3]{Connes}) is defined by $\mathbf{HH}_{n} (A) =
A^{\otimes (n+1)}$ with face and degeneracy maps 
\begin{numequation}\label{eq-simp-chain}
\begin{split}
d_{i} (a_{0}\otimes \dotsb \otimes a_{n})
 & := \mycases{
a_{0}\otimes \dotsb \otimes a_{i}a_{i+1}\otimes \dotsb \otimes  a_{n}
       &\mbox{for }i<n\\
a_{n}a_{0}\otimes a_{1}\otimes \dotsb \otimes  a_{n-1}
       &\mbox{for }i=n
}  \\
\aand 
s_{i} (a_{0}\otimes \dotsb \otimes a_{n})
 & := \mycases{    
1\otimes a_{0}\otimes \dotsb \otimes a_{n}
       &\mbox{for }i=0\\
a_{0}\otimes \dotsb \otimes a_{i-1}\otimes 1\otimes 
    a_{i}\otimes \dotsb \otimes a_{n}\hspace{-1.1cm}\\
       &\mbox{for }i>0.
}
\end{split}
\end{numequation}%

Suppose we have a symmetric monoidal category $(\mcC,\otimes )$ with a
monoid object $R$, {\ie } an object equipped with a map $R\otimes R\to
R$ with suitable properties that include strict associativity.  Then
we could define a cyclic (and hence simplicial) object
$\mathbf{HH}_{\bullet} (R)$ in $\mcC$ using the formulas of
\cref{eq-simp-chain} and \cref{eq-last-arc}.  Suppose in addition that
$\mcC $ is tensored over the category of topological spaces, meaning
that for an object $C$ in $\mcC $ and a space $X$, we can make sense
of $C\times X$ as another object in $\mcC$.  Suppose further that
$\mcC$ is cocomplete, meaning closed under colimits.  Then we can make
sense of the geometric realization of a simplicial or cyclic object in
$\mcC $ and thus define
\begin{numequation}\label{eq-THH}
\begin{split}
\THH (R) := |\mathbf{HH}_{\bullet} (R)|, 
\end{split}
\end{numequation}%

\noindent an object in $\mcC $, the {\bf topological Hochschild
homology} of the monoid object $R$.  This term was invented by {\Bok}.
For reasons explained \cref{sec-cyc-object}, it comes equipped with an
action of the group $\mT$.

It follows from \cref{eq-DK} that $|\mathbf{HH}_{\bullet} (A)|$ is a
generalized {\SESM} space with
\begin{numequation}\label{eq-bold-HHA}
\begin{split}
\pi_{*}|\mathbf{HH}_{\bullet} (A)|
  =\HH_{*} (A)\qquad \mbox{as in \cref{def-HHA}.} 
\end{split}
\end{numequation}%





\begin{defin}\label{def-EX}
For a discrete set $X$, $E_{\bullet}X$ is the simplicial set with 
\begin{displaymath}
E_{n}X:=\Map ([n],X) = X^{n+1},
\end{displaymath}

\noindent the set of maps to $X$ from a set with $n+1$ elements.
\end{defin}

\begin{prop}\label{prop-EX}
{\bf Contractibility and freeness.}
\begin{enumerate}[label={(\roman*)},itemindent=0em]
\item \label{prop-EXi} 
The geometric realization $|EX|$ is contractible.

\item \label{prop-EXii} 
If $X$ is a discrete (abelian) group $G$, then
$E_{\bullet}G$ is a simplicial (abelian) group, and $|EG|$ is a
contractible free $G$-space with orbit space
\begin{displaymath}
|EG|_{G}=|N (\mcB G) |\qquad \mbox{as in \cref{def-nerve}}.
\end{displaymath}



\end{enumerate}
\end{prop}


\subsection{Connes' cyclic category and cyclic objects}\label{sec-Connes} 

In \cite{Connes} Connes defines the {\em cyclic category} $\blamb $,
which has the same objects as $\bdelt$, but more morphisms.  A formal
description can be found in Loday's book \cite[6.1]{Loday}, where it
is denoted by $\Delta C$. His category $\blamb$ is isomorphic to that
of \cref{def-paracyclic} below. Connes regards $[n]$ (a set with $n+1$
elements) as the set of $(n+1)$th roots of unity sitting in the unit
circle of the complex numbers $\cxs $; see \cite[III.A.$\beta
$]{Connes94}.  His morphisms are homotopy classes of orientation
preserving self-maps of degree 1 of that circle that preserve these
subsets.  One such map is $\tau_{n}:[n] \to [n]$, which denotes
counterclockwise rotation by $2\pi / (n+1)$, so
$\tau_{n}^{n+1}=1_{[n]}$. We denote the corresponding morphism in
$\blamb^{\op}$ by $t_{n}$. There are some obvious identities involving
$t_{n}$ with the morphisms $d_{i}$ and $s_{i}$ of \cref{eq-simp-id}
spelled out by Loday in \cite[ 6.1.2]{Loday}, namely
\begin{numequation}\label{eq-tau-d-s}
\begin{split}
d_{i}t_{n}
 & = \mycases{    
d_{n}  &\mbox{for }i=0\\
t_{n-1}d_{i-1}
       &\mbox{for }1\leq i\leq n
}  \\
s_{i}t_{n}
 & = \mycases{    
t_{n+1}^{2}s_{n}
       &\mbox{for }i=0\\
t_{n+1}s_{i-1}
       &\mbox{for }1\leq i\leq n.
}
\end{split}
\end{numequation}%

\begin{prop}\label{prop-unique}
\cite[Theorem 6.1.3]{Loday}
Any morphism $[n]\to [k]$ in $\blamb$ can be written uniquely as the
composite of some iterate of $\tau_{n}$ with a morphism $[n]\to [k]$
in $\bdelt $.
\end{prop}

In particular there are $n+1$ distinct morphisms $[n]\to [0]$ in
$\blamb $ even though there is just one map between the underlying
sets.  Let 
\begin{displaymath}
\zeta =e^{2\pi \sqrt{-1}/ (n+1)}\in \mT .
\end{displaymath}

\noindent
Then for $0\leq i\leq n$ there is a map sending the arc joining
$\zeta^{i}$ and $\zeta^{i+1}$ to the full unit circle for $[0]$, and
sending its complement to the point $1\in \mT $.  Hence the forgetful
functor from $\bdelt$ to the category of sets extends to
$\blamb$, {\em but not faithfully}.

Next consider morphisms $[n+1]\to [n]$ in $\blamb $ for
which the underlying map of sets is onto.  If the underlying map fixes
0, it is obtained by collapsing one of the $n+2$ arcs (adjoining
adjacent roots unity) in the circle for $[n+1]$ to a single point.
Let $\omega =e^{2\pi \sqrt{-1}/ (n+2)}\in \mT $.  If the arc being collapsed
is not the one from $\omega^{-1}$ to $1$, the {\em last arc}, then the
map corresponds to one of the the simplicial degeneracy operators
$s^{i}$ of \cref{eq-face-degen} for $0\leq i\leq n$.  Thus we denote
the last arc collapse map by $s^{n+1}$, which we will also refer to as
a degeneracy operator.
 
The following should be compared with \cref{def-nerve}. 

\begin{defin}\label{def-cyc-nerve}
For a small category $\mcC$, the {\bf cyclic nerve}
$N^{\cyc}_{\bullet} (\mcC)$ is the simplicial set given by
\begin{displaymath}
N^{\cyc}_{n} (\mcC):=\mcCat ([n]^{\cyc},\mcC)
\end{displaymath} 

\noindent where $[n]^{\cyc}$ is the category $[n]$ equipped with an extra
morphism from $n$ to 0.  We denote its geometric realization by $B^{\cyc}\mcC$.

When $\mcC$ is the one object category $\mcB \Gamma$ associated with a
topological monoid $\Gamma$, we denote $N^{\cyc}_{\bullet} (\mcC)$ by
$N^{\cyc}_{\bullet} (\Gamma)$, leading to $N^{\cyc}_{n}
(\Gamma)=\Gamma^{n+1}$.
\end{defin}

\subsubsection{The work of Bill Dwyer, Mike Hopkins and Dan
Kan}\label{sec-DHK} In \cite{DHK-Cyc}, the authors take a slightly
different but {\eqt} approach to the passage from $\bdelt$ to
$\blamb$.  They define an additional degeneracy map $s^{n+1}:[n+1]\to
[n]$ (our last arc collapse map) and stipulate that
$(s^{n+1}d^{0})^{n+1}= 1_{[n]}$.  Note that $s^{n+1}$ is related to
the other degeneracy maps of \cref{eq-face-degen} via the action of
$\rC _{n+2}$ on the set $[n+1]$.

They give a geometric proof of \cref{prop-geo-r-cyc} for $r=1$. The
key point in their argument is that the simplicial set $U\blamb^{n}$
has $n+1$ nondegenerate $(n+1)$-simplices corresponding to the
composites of the last arc map $s^{n+1}:[n+1]\to [n]$ with powers of
the rotation $\tau_{n}$.

More explicitly, note that the {\em $(n+1)$-dimensional prism} $I\times
|\bdelt^{n}|$ is the space
\begin{displaymath}
\left\{(t;x_{0}, x_{1},\dotsc ,x_{n})\in I\times \reals^{n+1}:  0\leq t\leq 1,\,
    x_{i}\geq 0,\,\sum_{0\leq i\leq n}x_{i}=1 \right\}.
\end{displaymath}

\noindent For each $j$ with $0\leq
j\leq n$, define a subspace $P_{j}\subseteq I\times |\bdelt^{n}|$ by
\begin{displaymath}
P_{j} =\left\{(t; x_{0}, x_{1},\dotsc ,x_{n}):
\sum_{0\leq i\leq j-1}x_{i}
\leq t\leq 
\sum_{0\leq i\leq j}x_{i} \right\}.
\end{displaymath}

\noindent The union of these subspaces is the entire prism, and the
topological $(n+1)$-simplex $P_{j}$ is the convex hull of the set of
$n+2$ points
\begin{displaymath}
\left\{(0;\delta_{k}):0\leq k\leq j \right\}
\cup 
\left\{(1;\delta_{k}):j\leq k\leq n \right\},
\end{displaymath}

\noindent where $\delta_{k}\in \reals^{n+1}$ for $0\leq k\leq n$ is
the vector whose $k$th coordinate is 1 and all others are 0.
The geometric realization of $\blamb^{n}$ is the quotient of the prism
obtained by identifying $(0; x_{0}, x_{1},\dotsc ,x_{n})$ with
$(1;x_{0}, x_{1},\dotsc ,x_{n})$, thereby converting it to
$\mT \times |\bdelt^{n}|$. 

\subsubsection{Connes' cyclic objects}\label{sec-cyc-object} 
\begin{defin}\label{def-cyc-object}  
A {\bf cyclic object $X$} in a category $\mcC $ is a $\mcC $-valued
functor on $\blamb ^{\op}$ and therefore a simplicial object
(by restriction of the functor to $\bdelt^{\op}$) with some additional
structure, which includes an action of the cyclic group $\rC _{n+1}$ on
$X_{n}$ for each $n$.  We denote the category of such functors by
$\mcC_{\blamb}$, and the forgetful functor $\mcC_{\blamb} \to
\mcC_{\bdelt}$ by $U$.  The {\bf geometric realization $|X|$ of a
cyclic set $X$} is that of the underlying simplicial set $UX$, which
we will often denote abusively by $X$.

The cyclic set $\blamb^{n}$ the {\bf standard $n$-cyclex}, is defined by
\begin{numequation}\label{eq-lambda1-n}
\begin{split}
(\blamb^{n})_{k} := \blamb ([k],[n])= \blamb^{\op} ([n],[k]).
\end{split}
\end{numequation}%
\end{defin}

The term {\bf cyclex} (plural {\bf cyclices}) is new and is meant to be the
cyclic analog of {\bf simplex}.  Like the standard $n$-simplex of
\cref{def-simp-sets}, the standard $n$-cyclex is a Yoneda functor.

The above will be generalized in \cref{def-r-cyclic}, for which it is
the case $r=1$.

See \cite[Chapter 7]{Loday} for a formal treatment of cyclic spaces.


This object was denoted by ${\bC _{\bullet}}$ (note the Roman font) in
\cite[Proposition 6.1.9]{Loday}.  For $k=1$, the group $\rC _{2}$
interchanges the degenerate and nondegenerate edges.

Note that the cardinality of the free $\rC _{k+1}$-set $(\blamb^{n})_{k}$
exceeds that of $(\bdelt^{n} )_{k}$ (see \cref{prop-card}) by a factor
of $k+1$ by \cref{prop-unique}.

\begin{prop}\label{prop-geocyc}
{\bf The geometric realization of the standard $n$-cyclex
$\blamb^{n}$} is homeomorphic to $\mT \times \bdelt^{n}_{\rm top}$,
where $\bdelt^{n}_{\rm top}=|\bdelt^{n}|$ is the standard topological
$n$-simplex of \cref{eq-Delta-top}.
\end{prop}

For the proof, 
see \cite[\S2.9]{DHK-Cyc}.

\begin{cor}\label{cor-circle} 
{\bf The cyclic circle or standard 0-cyclex.}  The cyclic set
$\blamb^{0}$ of \cref{eq-lambda-n} is isomorphic as a simplicial set
to the {\bf circle $\bdelt^{1}/\partial \bdelt ^{1}$}.  Its geometric
realization is the circle $\mT$.  $U\blamb^{0}$ is isomorphic to the
simplicial set with a single vertex $x_{0}$ and a single nondegenerate
edge $x_{1}$.  Thus we have
\begin{numequation}\label{eq-S1}
\begin{split}
(U\blamb_{0})_{k}=\mycases{    
\left\{x_{0} \right\}
       &\mbox{for }k=0\\
\left\{x_{1},\,s_{0} (x_{0}) \right\}
       &\mbox{for }k=1\\
\left\{s_{i} (s_{0})^{k-2} (x_{1}):0\leq i\leq k-1 \right\}
\cup \left\{(s_{0})^{k} (x_{0}) \right\}\hspace{-4cm}\\
       &\mbox{for }k\geq 2,
}
\end{split}
\end{numequation}%

\noindent making a total of $(k+1)$ $k$-simplices for each $k\geq 0$.
The cyclic group $\rC _{k+1}$ (which is the automorphism group
$\Aut_{\blamb^{\op}} ([k])$) acts freely on this set, so there is single
orbit for each $k$. 
\end{cor}

\begin{theorem}\label{thm-circle-action}
{\bf The action of the circle group $\mT$.} \cite[Theorem 7.1.4]{Loday} 
Let $X$ be a cyclic space ({\eg} a cyclic set) and let $|X|$ be the
geometric realization of its underlying simplicial space. Then

\begin{enumerate}[label={(\roman*)},itemindent=0em]
\item $|X|$ is endowed with a canonical action of the circle $\mT$, and 
\item $X\mapsto |X|$ is a functor from cyclic spaces to $\mT$-spaces. 
\end{enumerate}
\end{theorem}

\begin{prop}\label{prop-smooth-cyclo}
{\bf The smooth cyclotomy of $B^{\cyc} (\Gamma )$.} For a topological
monoid $\Gamma $, the fixed point set $|B^{\cyc} (\Gamma ) |^{\rC _{r}}$
(as in \cref{def-cyc-nerve}) is $\mT$-{\eqvr}ly homeomorphic to $|B^{\cyc}
(\Gamma ) |$ for each integer $r\geq 1$.
\end{prop}

We sketch the proof here. The space $|B^{\cyc} (\Gamma ) |$
is a quotient of
\begin{displaymath}
\coprod_{n\geq 0}\Gamma^{n+1}\times |\blamb^{n}|
=\coprod_{n\geq 0}\Gamma^{n+1}\times |\bdelt^{n}|\times \mT.
\end{displaymath}

\noindent When $r$ divides $n+1$, let $m=(n+1-r)/r$.  The actions of
$\rC _{r}$ on $\Gamma^{n+1}$ (by permutation of coordinates) and
$|\bdelt^{n}|$ (by permutation of vertices) fix subspaces homeomorphic
to $\Gamma^{m+1}$ and $|\bdelt^{m}|$ respectively. It follows that
$|B^{\cyc} (\Gamma ) |^{\rC _{r}}$ is a quotient of
\begin{displaymath}
\coprod_{m\geq 0}\Gamma^{m+1}\times |\blamb^{m}|
=\coprod_{m\geq 0}\Gamma^{m+1}\times |\bdelt^{m}|\times \mT,
\end{displaymath}

\noindent making it $\mT$-{\eqvr}ly homeomorphic to
$|B^{\cyc} (\Gamma ) |$ as claimed.

\begin{defin}
The {\bf cyclic geometric realization of a cyclic space $X$} is 
\begin{displaymath}
|X|^{\cyc} := E\mT \times_{\mT}|X|.
\end{displaymath}
\end{defin}

In particular for $X=*$, we have $|*|^{\cyc} = B\mT$.

Loday \cite[7.1.5]{Loday} defined a functor $F:\Set_{\bdelt}\to
\Set_{\blamb}$ (with $F\bdelt^{n}=\blamb^{n}$) that is left adjoint to
the forgetful functor $U:\Set_{\blamb}\to \Set_{\bdelt}$.  For a
simplicial set $X$, $|FX|\approx \mT\times |X|$ with free $\mT$-action
by \cite[Lemma 7.1.8]{Loday}, even though $UFX$ is not a product as a
simplicial set.  For a simplicial set $X$ and a cyclic set $Y$, we
have an adjunction isomorphism
\begin{displaymath}
\Set_{\bdelt} (X,UY) \cong  \Set_{\blamb} (FX,Y).
\end{displaymath}

\noindent  When  $X=UY$, this reads
\begin{displaymath}
\Set_{\bdelt} (UY,UY) \cong  \Set_{\blamb} (FUY,Y).
\end{displaymath}

\noindent The morphism set on the left has a distinguished element,
the identity map on $UY$, and the corresponding morphism on the right
is $\epsilon_{Y} :FUY\to Y$, the counit of the adjunction.  On
geometric realizations we have
\begin{displaymath}
\xymatrix
@R=4mm
@C=20mm
{
{|FUY| \approx  \mT\times |UY|  := \mT\times |Y|}
    \ar[r]^(.7){|\epsilon_{Y}|}
  &{|Y|,}
}
\end{displaymath}

\noindent making $|Y|$ a $\mT$-space.

\bigskip

We can extend the simplicial structure on
$\mathbf{HH}_{\bullet} (A)$ as in \cref{eq-simp-chain} to a cyclic
structure by defining the last arc degeneracy map as
\begin{numequation}\label{eq-last-arc}
\begin{split}
s^{n+1} (a_{0}\otimes \dotsb \otimes a_{n})
  = a_{0}\otimes \dotsb \otimes a_{n}\otimes 1  .
\end{split}
\end{numequation}%

Hence $\mathbf{HH}_{\bullet} (A)$ is a cyclic (in the sense of Connes)
$k$-module, and the space $|\mathbf{HH}_{\bullet} (A)|$ of
\cref{eq-bold-HHA} is a $\mT$-space.

\subsection{The paracyclic and $r$-cyclic  categories, and edgewise subdivision}
\label{sec-edgewise-sub}

The following is taken from \cite[page 380]{NS18}.

\begin{defin}\label{def-paracyclic}
The {\bf paracyclic category $\blamb_{\infty }$} has as objects the
linearly ordered sets
\begin{displaymath}
\pow{n}_{\blamb_{\infty }}:= (1/ n )\Z
\end{displaymath}

\noindent on which $\Z$ acts by addition. A morphism is a map of sets
$f:\pow{m}_{\blamb_{\infty }}\to \pow{n}_{\blamb_{\infty }}$
satisfying
\begin{numequation}\label{eq-f-para}
\begin{split}
f (x)\leq f (y)\mbox{ when }x\leq y
\qquad \aand  
f (x+1)= f (x)+1.
\end{split}
\end{numequation}%

\noindent Hence each morphism set has an action of $\Z$ by pointwise
addition.

For an integer $r\geq 1$, the {\bf $r$-cyclic category} $\blamb_{r }$
(denoted simply by $\blamb$ when $r=1$) has the same objects, now
denoted by $\pow{n}_{\blamb_{r } }$ (or simply $\pow{n}_{\blamb }$
when $r=1$), and each morphism set is the quotient of that in
$\blamb_{\infty }$ by the action of $r\Z$.

We will often drop the subscript on $\pow{n}$.
\end{defin}

The conditions of \cref{eq-f-para} mean that $f$ is determined by its
behavior on any interval of length 1.

The category $\blamb=\blamb_{1}$ coincides with Connes' cyclic
category $\blamb$ of \cref{def-cyc-object}.

\begin{remark}\label{rem-lambda}
{\bf The symbols $\blamb$, $[n]$ and $\pow{n}$.} The symbol $\blamb$
is not to be confused with $\Lambda (-) $, which we sometimes use to
denote an exterior algebra, or $\Lsmash$ as in \cref{def-LXR}.

Our notation differs from that of \cite[page 380]{NS18} and
\cite[Example 2.1.5]{McCandless-curves}, in which our
$\pow{n}_{\blamb_{r} }$ is denoted by $[n]_{\blamb_{r} }$.  In this
paper (starting in \cref{sec-simp-cat}) and many others, $[n]$ denotes
the finite ordered set $\left\{0,1,\dotsc ,n \right\}$, which has
$n+1$ elements, while $\pow{n}_{\blamb}$ behaves like a set with $n$
elements.  Thus we write $[n]\times \Z\simeq \pow{n+1}_{\blamb_{\infty } }$ for
the isomorphism written in \cite[Example 2.1.6]{McCandless-curves} as
$[n]\times \Z\simeq [n+1]_{\Lambda }$.
\end{remark}

Thus we have projection functors
\begin{numequation}\label{eq-proj-functor}
\begin{split}
\xymatrix
@R=1mm
@C=4mm
{
{\blamb_{\infty }}\ar[r]^(.5){\Proj_{\infty ,r}}
  &{\blamb_{r}}\ar[r]^(.5){\Proj_{r,r/d}}
    &{\blamb_{r/d}}\\
{\Map_{\blamb_{\infty }}(\pow{m},\pow{n})}
  \ar@{|->}[r]^(.5){}
 &{\Map_{\blamb_{r}}(\pow{m},\pow{n})}
        \ar@{|->}[r]^(.5){}
        \ar@{=}[dd]^(.5){}
   &{\Map_{\blamb_{r/d}}(\pow{m},\pow{n})} \ar@{=}[dd]^(.5){}
\\
{}\\
  &{\Map_{\blamb_{\infty }}(\pow{m},\pow{n})/r\Z} 
    &{\Map_{\blamb_{r}}(\pow{m},\pow{n})/ (\Z/d)} 
}
\end{split}
\end{numequation}%

\noindent for each $r\geq 1$ and each divisor $d$ of $r$.   The
groups $r\Z \subseteq \Z$ and $\Z/d \subseteq \Z/r$ act freely on the
morphism sets in question.

{\Eqt}ly $\blamb_{r }$ for $1\leq r\leq \infty $ is the category which
contains $\bdelt$ and additional morphisms $\tau_{n}:[n]\to [n]$
subject to the relations of \cref{eq-tau-d-s} and the relation
$\tau_{n}^{r (n+1)}= 1_{[n]}$ when $r<\infty$.

The objects of $\blamb$ behave like finite sets and there is an
inclusion functor
\begin{align*}
V:\blamb
 & \to \Fin\\
\pow{n}_{\blamb}
 & \mapsto  \langle n\rangle
\end{align*}

\noindent to the category of finite sets of \cref{def-fin*} sending
$\pow{n}_{\blamb}$, which has $n$ elements, to the unpointed set with
the same cardinality.  We know by \cite[Proposition B.1]{NS18} that it
lifts to a similarly named functor
\begin{displaymath}
V:\blamb \to \Assoc^{\otimes }_{\act },
\end{displaymath}

\noindent where the codomain is the associative operad of
\cref{eq-Ass-act}.  Precomposing with the isomorphism between $\blamb$
and its opposite gives a functor
\begin{numequation}\label{eq-V-op}
\begin{split}
V^{\op}:\blamb^{\op} \to \Assoc^{\otimes }_{\act }.
\end{split}
\end{numequation}%

\noindent The self-duality of $\blamb_{\infty }$, which implies that
of each $\blamb_{r}$, is spelled out in \cite[page 381]{NS18}.

For each $r$, there is an embedding
\begin{numequation}\label{eq-lambda-r}
\begin{split}
\iota _{r} :\bdelt \to \blamb_{r}
\qquad [n]\mapsto \pow{n+1}_{\blamb_{r}}
\end{split}
\end{numequation}%

\noindent Simplicial and $r$-cyclic (paracyclic for $r=\infty $) sets
are contravariant $\mcSet$-valued functors on $\bdelt$ and $\blamb_{r}
$ respectively. See \cref{def-r-cyclic} below. Both have geometric
realizations, and that of an $r$-cyclic set for finite $r$ has a
natural action of the circle $\mT$.  That of a paracyclic set has
natural action of $\reals$.

We will want to look at the fixed point sets of finite subgroups of
$\mT$.  In order to describe these simplicially, the following notion
of subdivision is helpful.

\begin{defin}\label{def-edgewise} 
For each positive integer $r$, the {\bf $r$-fold edgewise subdivision
functor} $\sd_{r}:\bdelt \to \bdelt $ sends $[n-1]$ to $[rn-1]$ for
$n\geq 1$, and each morphism to its $r$th iterated disjoint
union. (Face and degeneracy maps can also be written as $r$-fold
compositions as in \cref{eq-edgewise-face}.) Let $d_{(r)}:\tdelt{n-1}
\to \tdelt{rn-1}$ be the map of topological simplices induced by the
diagonal embedding
\begin{displaymath}
(x_{0},\dotsc ,x_{n-1}) \mapsto \frac{(x_{0},\dotsc ,x_{n-1},
x_{0},\dotsc ,x_{n-1},\dotsc, x_{0},\dotsc ,x_{n-1})}{r}
\end{displaymath}

\noindent where the coefficients in the codomain are repeated $r$
times with
\begin{displaymath}
x_{i}\geq 0 \qquad \aand \sum_{0\leq i<n}x_{i}=1.
\end{displaymath}

For a simplicial set $X$, we define the {\bf $r$th subdivided
simplicial set} $\sd^{*}_{r}X$ to be the composite functor
\begin{numequation}
\begin{split}
\xymatrix
@R=4mm
@C=10mm
{
{\bdelt^{\op}}\ar[r]^(.5){\sd_{r}}
  &{\bdelt^{\op}}\ar[r]^(.5){X}
    &{\Set ,}
}
\end{split}
\end{numequation}%

\noindent so that $(\sd^{*}_{r}X)_{n-1}=X_{rn-1}$. Its $i$th face map 
$d_{i}$ (for $0\leq i\leq n-1$) is the composite of $r$ face maps in $X$,
\begin{numequation}\label{eq-edgewise-face}
\begin{split}
\xymatrix
@R=1mm
@C=15mm
{
{(\sd^{*}_{r}X)_{n-1}}\ar@{=}[ddd]^(.5){}\ar[rrr]^(.5){d_{i}}
  &{}
    &{}
      &{(\sd^{*}_{r}X)_{n-2}}\ar@{=}[ddd]^(.5){}\\
{}\\
{}\\
{X_{rn -1}}\ar[r]^(.55){d_{i+ (r-1)n}}
  &{\dotsb }\ar[r]^(.4){d_{i+n}}
    &{X_{r (n-1)}}\ar[r]^(.45){d_{i}}
      &{X_{r (n-1)-1}}
}
\end{split}
\end{numequation}%

\noindent and its degeneracy maps are defined similarly.
\end{defin}

\begin{lem}\label{lem-subdiv-homeo}
{\bf Edgewise subdivision induces a homeomorphism.} \cite[Lemma
1.1]{BHM93} The map $D_{r}:|\sd^{*}_{r}X|\to |X|$ induced by
\begin{displaymath}
1\times d_{(r)}:
(\sd^{*}_{r}X)_{n-1}\times \bdelt^{n-1}
\to X_{rn-1}\times\bdelt^{rn-1}
\end{displaymath}

\noindent is a homeomorphism.
\end{lem}


\begin{defin}\label{def-r-cyclic}\cite[Definition 1.5]{BHM93}
An {\bf $r$-cyclic (or paracyclic) object} in a category $\mcC $ is a
contravariant $\mcC $-valued functor on $\blamb_{r}$ for finite $r$
(or $\blamb_{\infty }$), and we denote the category of such functors
by $\mcC_{\blamb_{r}}$.  The value of such a functor $X$ on $[n]$ is
denoted by $X_{n}$.  The {\bf geometric realization} $|X|$ of an
$r$-cyclic set $X$ is that of the underlying (via the embedding of
\cref{eq-lambda-r}) simplicial set.

The $r$-cyclic set $\blamb_{r}^{n}$ the {\bf standard $n$-cyclex of
degree $r$}, is defined by
\begin{numequation}\label{eq-lambda-n}
\begin{split}
(\blamb^{n}_{r})_{k} := \blamb_{r} ([k],[n])= \blamb^{\op}_{r} ([n],[k]).
\end{split}
\end{numequation}%

\noindent We will omit the subscript when it is 1.
\end{defin}

\begin{prop}\label{prop-geo-r-cyc}
\cite[Lemma 1.6]{BHM93}
{\bf The geometric realization of $\blamb^{n}_{r}$} is homeomorphic
to $\reals/r\ints \times \bdelt^{n}_{\rm top}$.  The action
of $\tau_{n}$ on $|\blamb^{n}_{r} |$ given by
\begin{numequation}\label{eq-tau-n-action}
\begin{split}
\tau_{n} (\theta ;x_{0},\dotsc ,x_{n}) := (\theta -x_{0};x_{1},\dotsc
,x_{n},x_{0}),
\end{split}
\end{numequation}%

\noindent where $\theta \in \reals/r\ints$.
\end{prop}

\cref{eq-tau-n-action} implies that 
\begin{displaymath}
\tau_{n}^{n+1} (\theta ;x_{0},\dotsc ,x_{n})
  = (\theta-1; x_{0},x_{1},\dotsc,x_{n}),
\end{displaymath}

\noindent making $\tau_{n}^{(n+1)r}$ the identity map for finite $r$.
The duals of the relations of \cref{eq-tau-d-s} imply that
$\tau_{n}^{n+1}$ commutes with $d^{i}$ and $s^{i}$ in that for $0\leq
i\leq n$,
\begin{displaymath}
\tau^{n+1}_{n}d^{i}= d^{i}\tau^{n}_{n-1}
\qquad \aand 
\tau^{n+1}_{n}s^{i}= s^{i}\tau^{n+2}_{n+1}.
\end{displaymath}

\noindent From this it follows that an $r$-cyclic space or set has an
action of the cyclic group $\rC _{r}$.  

For $r=1$, \cref{prop-geo-r-cyc} was first proved by Dwyer, Hopkins
and Kan in \cite[\S2.9]{DHK-Cyc}.  We describe their proof in \cref{sec-DHK}.

\subsection{The epicyclic category}\label{sec-epicyclic}

The ideas of this subsection were first published in \cite{BFG94},
whose authors attribute them to \cite{Goodwillie-letter}.  A more
contemporary reference is \cite{McCandless-curves}, which lists
several other papers on the topic.

We will define analogs $\widetilde{\blamb}_{r}$ of the categories of
\cref{def-paracyclic} having the same objects but more morphisms.
Recall that Connes regards the finite ordered set $[n]$ as the set of
$(n+1)$th roots of unity in the unit circle.  Morphisms in $\blamb$
are homotopy classes of degree one maps of the circle that preserve
the subset. The categories $\blamb_{r}$ and $\blamb_{\infty }$ can be
described in similar terms with the unit circle replaced by its
$r$-fold and universal covers.  Maps in the former case have degree 1
and on the real line they are equivariant with respect to the action
of the integers by addition.  {\em In the epicyclic version we drop
the degree one requirement.  } We allow maps of arbitrary positive
degree $k$ in $\widetilde{\blamb}_{r}$, and a map $f$ in
$\widetilde{\blamb}_{\infty }$ must satisfy $f (x+1)=f (x)+k$.  The
resulting morphism sets still have an action of $\Z$ by pointwise
addition.

\begin{defin}\label{def-epicyclic}
\cite[Definition 1.1]{BFG94}.  The
{\bf epicyclic category $\widetilde{\blamb}_{\infty }$} has as objects
the linearly ordered sets
\begin{displaymath}
\pow{n}_{\widetilde{\blamb}_{\infty }}:= (1/ n )\Z
\end{displaymath}

\noindent on which $\Z$ acts by addition, as in
\cref{def-paracyclic}. A morphism is a map of sets $f:\pow{m}\to
\pow{n}$ satisfying
\begin{numequation}\label{eq-f-para-again}
\begin{split}
f (x)\leq f (y)\mbox{ when }x\leq y
\qquad \aand  
f (x+1)= f (x)+k,
\end{split}
\end{numequation}%

\noindent for a positive integer $k$, the {\bf degree} of $f$. For
$k=1$, this is the same as \cref{eq-f-para}. Hence each morphism set
has an action of $\Z$ by pointwise addition.

For an integer $r\geq 1$, the {\bf $r$-epicyclic category}
$\widetilde{\blamb}_{r }$ (denoted simply by $\widetilde{\blamb}$ when
$r=1$) has the same objects, and each morphism set is the quotient of
that in $\widetilde{\blamb}_{\infty }$ by the action of $r\Z$.
\end{defin}

Equivalently (see \cite[Definition 1.1 ]{BFG94} ),
$\widetilde{\blamb}$ has the same objects and includes the morphisms
of $\blamb$ as in \cref{eq-tau-d-s} with additional morphisms
\begin{numequation}\label{eq-pi-nk}
\begin{split}
\pi^{k}_{n}:[k (n+1)-1]\to [n], \qquad k,n\in \mN,\quad k\geq 1
\end{split}
\end{numequation}%

\noindent defined by $a (n+1)+b\mapsto b$ for $0\leq b\leq n$ and
$0\leq a<k$.  These are subject to the relations

\begin{enumerate}[label={(\roman*)},itemindent=0em]
\item $\pi^{1}_{n} = 1_{[n]}$, 
 $\pi^{\ell }_{n}\pi^{k}_{\ell (n+1)-1}=\pi^{k\ell }_{n}$
\item $\alpha \pi^{k}_{m}=\pi^{k}_{n}\sd_{k} (\alpha )$ 
for $\alpha \in \bdelt ([m],[n])$ for $\sd_{k}$ as in \cref{def-edgewise} 
\item $\tau_{n}\pi^{k}_{n}=\pi^{k}_{n}\sd_{k} (\tau_{n})$ for
$\tau_{n}$  as in \cref{sec-Connes}.
\end{enumerate}

There is an analog of \cref{prop-unique} that says any morphism $[m]\to
[n]$ in $\widetilde{\blamb}$ can be written uniquely as a composite
\begin{numequation}\label{eq-epi-unique}
\begin{split}
\xymatrix
@R=4mm
@C=10mm
{
{[m]}\ar[r]^(.5){(\tau_{m})^{i}}
  &{[m]}\ar[r]^(.35){\alpha  }
    &{[k(n+1)-1]}\ar[r]^(.67){\pi^{k}_{n}}
      &{[n]}
}
\end{split}
\end{numequation}%

\noindent for some $k>0$ with $\alpha  $ being a morphism in $\bdelt$.

A more conceptual definition of $\widetilde{\blamb}$ due to Nikolaus
is \cite[Definition 2.1.9]{McCandless-curves}.

\begin{defin}\label{def-pownr}
{\bf The categories  $\pow{n}_{\blamb_{r }}$.}
The objects $\pow{n}_{\blamb_{r }}=\pow{n}_{\widetilde{\blamb}_{r }}$
of \cref{def-paracyclic,def-epicyclic} for $1\leq r\leq \infty $ are
themselves categories.  For $r=\infty $ it is a poset, which is a
category whose objects are the elements in the set and in which there
is a unique morphism from each element to each larger one.  The sets
of morphisms and objects have an action of the integers by addition.
For each $r<\infty $ we get a new category by passage to orbit sets.
Thus $\pow{n}_{\blamb_{r }}$ has $nr$ objects and the geometric
realization of its nerve is a circle.  That of
$\pow{n}_{\blamb_{\infty }}$ is the real line.
\end{defin}

\begin{defin}\label{def-epicyc-object}
An {\bf epicyclic object $X$} in a category $\mcC $ is a $\mcC
$-valued functor on $\widetilde{\blamb} ^{\op}$ and therefore a
simplicial object (by restriction of the functor to $\bdelt^{\op}$)
with some additional structure, which includes an action of the cyclic
group $\rC _{n+1}$ on $X_{n}$ for each $n$.  We denote the category of
such functors by $\mcC_{\widetilde{\blamb}}$, and the forgetful
functor $\mcC_{\widetilde{\blamb}} \to \mcC_{\bdelt}$ by
$\widetilde{U}$.  The {\bf geometric realization $|X|$ of an epicyclic set
$X$} is that of the underlying simplicial set $\widetilde{U}X$, which
we will often denote abusively by $X$.

\begin{numequation}\label{eq-lambda1-n-epi}
\begin{split}
(\widetilde{\blamb}^{n})_{k}
 := \widetilde{\blamb} ([k],[n])= \widetilde{\blamb}^{\op} ([n],[k]).
\end{split}
\end{numequation}%
\end{defin}


Since $\blamb$ is a wide subcategory (consisting of morphisms of
degree 1) of $\widetilde{\blamb}$, an epicyclic object is also a
cyclic object.  We know by \cref{thm-circle-action} that the geometric
realization of a cyclic comes equipped with an action of the circle
group $\mT$.  The corresponding structure in the epicyclic case
involves the following.

\begin{defin}\label{def-Witt-monoid}
\cite[(0.1)]{BFG94}.  The {\bf {\Wittmonoid}} $\mathscr{M} := \mT \rtimes
\mN^{\times }$, where $\mN^{\times }$ denotes the
multiplicative monoid of positive integers, is the monoid with
multiplication given by
\begin{displaymath}
(z_{1} ,k_{1}) (z_{2} ,k_{2} )
  = (z_{1} z_{2}^{k_{1}}, k_{1}k_{2} )
\qquad \mbox{for $z_{i}\in \mT$ and $k_{i}\in \mN^{\times }$.} 
\end{displaymath}

For a prime number $p$, the {\bf $p$-typical {\Wittmonoid}} is
\begin{displaymath}
\mathscr{M}_{p} :=
\rC_{p^{\infty }} \rtimes \mu_{p^{\mN}},
\end{displaymath}

\noindent where $\rC_{p^{\infty }}\subseteq \mT$ is the Pr\"uffer group and 
\begin{displaymath}
\mu_{p^{\mN}} := \left\{1, p, p^{2},\dotsc  \right\}.
\end{displaymath}
\end{defin}

Jonas McCandless \cite{McCandless-curves} calls $\mathscr{M}$ the
Witt monoid and denotes it by $\bW$. It is studied but not named
in \cite{BFG94}.

$\mathscr{M}$ acts on the free loop space $\mcL X$ with $\mT$ acting by
rotation of loops and $\mN^{\times }$ acting by power maps, hence the
name.  This means that for a point $f:\mT\to X$ in $\mcL X$ and
$(z,k)\in \mathscr{M}$, we have
\begin{numequation}\label{eq-Witt-free-loop}
\begin{split}
(z,k) (f) (u):f (zu^{k})\in X \qquad \mbox{for each $u\in \mT$}.
\end{split}
\end{numequation}%

\begin{theorem}\label{thm-Witt-action}
\cite[Theorems A and B]{BFG94}.  
\begin{enumerate}[label={(\roman*)},itemindent=0em]
\item \label{thm-Witt-actioni}
The geometric realization of an epicyclic space has a canonical,
right action of the {\Wittmonoid} $\mathscr{M}$ of \cref{def-Witt-monoid}.

\item \label{thm-Witt-actionii} The classifying space $B\mathscr{M}$ is
homotopy equivalent to $|N (\widetilde{\blamb}) |$.

\item \label{thm-Witt-actioniii}
The fundamental group of $B\mathscr{M}$ is isomorphic to the multiplicative
group of positive rational numbers. The universal covering of $B\mathscr{M}$
has the homotopy type of the {\SESM} space $K (\rats ,2)$ with
$\pi_{1}$ acting by multiplication of rational numbers.
\end{enumerate}
\end{theorem}

\subsection{Summary of our indexing categories}\label{sec-summary}

We have introduced the categories
\begin{displaymath}
\xymatrix
@R=4mm
@C=4mm
{
{[n-1]}\ar@<1ex>[r]^(.6){d^{i}}
  &{[n]}\ar@<1ex>[l]^(.4){s^{i}}
    &{\bdelt}\ar[d]_(.5){\subseteq }\\
{[n]}\ar[r]^(.5){\tau_{n}}
  &{[n]} 
    &{\blamb}\ar[d]_(.5){\subseteq }
        &&{\blamb_{r/d}}\ar@{->>}[ll]_(.5){\Proj_{r/d,1}}
             &&{\blamb_{r}}\ar[d]^(.5){\subseteq }
                 \ar@{->>}[ll]_(.5){\Proj_{r ,r/d}}
                 &&{\blamb_{\infty }}\ar[d]^(.5){\subseteq }
                         \ar@{->>}[ll]_(.5){\Proj_{\infty ,r}}\\
{[kn+k-1]}\ar[r]^(.7){\pi^{k} _{n}}
  &{[n]} 
   &{\widetilde{\blamb}}
        &&&&{\widetilde{\blamb}_{r}}\ar@{->>}[llll]_(.5){\Proj_{r,1}}
                &&{\widetilde{\blamb}_{\infty }}
                    \ar@{->>}[ll]_(.5){\Proj_{\infty ,r}}
                   &{\pow{n}_{\widetilde{\blamb}_{\infty }}:= (1/ n )\Z}
}
\end{displaymath}

\noindent with the indicated objects and morphisms in
\cref{def-simp-cat,def-paracyclic,def-epicyclic}. These are
respectively the simplicial, cyclic/paracyclic and epicyclic
categories.  The coface maps $d^{i}$ and the codegeneracy maps $s^{i}$
are given in \cref{def-simp-cat}, $\tau_{n}$ is the automorphism of
order $r (n+1)$ given in the first paragraph of \cref{sec-Connes}, and
$\pi^{k}_{n}$ is given in \cref{eq-pi-nk}.  In each case we are
interested in contravariant functors to the category of sets or spaces.

\subsection{The double complexes of Tsygan-Goodwillie and Connes}
\label{sec-double-complex}

The first quadrant version of the following appeared first in less
general form in a paper of Tsygan \cite{Tsygan}.  It was later studied
in \cite{Loday-Quillen} and \cite[2.1.2]{Loday}.  The treatment here
is taken from Goodwillie's paper \cite[II.2]{Goodwillie-cyc} and is
also in Weibel's book \cite[9.6]{Weibel-H}.  To our knowledge,
Goodwillie's paper is the first to treat the case of a general cyclic
object in an abelian category rather than the specific case
$\mathbf{HH}_{\bullet} (A)$ as in \cref{eq-simp-chain}.


\begin{defin}\label{def-Goodwillie-double}
{\bf Tsygan and Goodwillie's first double complex.} Let $X$ be a
cyclic object in an abelian category. Then $C^{\rm per}_{**} (X)$, the
{\bf periodic double complex of $X$}, is the upper half plane double
complex with a contracting chain homotopy $u$ in its oddly indexed
columns 
\begin{numequation}\label{eq-C**}
\begin{split}
\xymatrix
@R=6mm
@C=10mm
{
{}
  &{}\ar[d]^(.4){b}
    &{}\ar@<.5ex>[d]^(.4){-b'}
      &{}\ar[d]^(.4){b}\\
{}
  &{C^{\rm per}_{2i,j+1} (X)}\ar[l]_(.65){N}
                             \ar[d]^(.5){b}
    &{C^{\rm per}_{2i+1,j+1} (X)}\ar[l]_(.5){\epsilon }
                     \ar@<.5ex>[d]^(.5){-b'}\ar@<.5ex>[u]^(.6){u}
      &{C^{\rm per}_{2i+2,j+1} (X)}\ar[l]_(.5){N }\ar[d]^(.5){b}
        &{ }\ar[l]_(.3){\epsilon }\\
{ }
  &{C^{\rm per}_{2i,j} (X)}\ar[l]_(.65){N}\ar[d]^(.6){b}
    &{C^{\rm per}_{2i+1,j} (X)}\ar[l]_(.5){\epsilon }
               \ar@<.5ex>[d]^(.6){-b'}\ar@<.5ex>[u]^(.5){u}
      &{C^{\rm per}_{2i+2,j} (X)}\ar[l]_(.5){N }\ar[d]^(.6){b}
        &{ }\ar[l]_(.3){\epsilon }\\
  &{}
    &{}\ar@<.5ex>[u]^(.4){u}
      &{}
}
\end{split}
\end{numequation}%

\noindent where 
\begin{align*}
C^{\rm per}_{i,j} (X)
 &: =X_{j}
    &&  \mbox{for $i\in \Z$ and $j\geq 0$}\\
b
 & := \sum_{0\leq k\leq j} (-1)^{k} d_{k}, 
    && \mbox{the Hochschild boundary,}\\
b'&  := \sum_{0\leq k <j} (-1)^{k} d_{k},  
    &&  \mbox{the modified Hochschild boundary,}\\
u &:= (-1)^{j}s_{j}, 
    &&  \mbox{the contracting chain homotopy,}\\
\epsilon
 & := 1- (-1)^{j}t_{j}
    &&  \mbox{for $t_{j}$ as in \cref{eq-tau-d-s},} \hspace{-5cm}  \\ 
\aand N
 & := \sum_{0\leq k\leq j}\left((-1)^{j}t_{j} \right)^{k}.
\end{align*}

\noindent The definition of each arrow in the picture depends on the
vertical coordinate, which we suppress from the notation to avoid
clutter, and the parity of the horizontal coordinate. We denote the
evenly and oddly indexed columns by $C^{\Hoch}_{*} (X)$, the {\em
Hochschild complex of $X$}, and $C^{\acyc}_{*} (X)$, the {\em acyclic
complex of $X$}, respectively.  We denote the homology of the former
by $\HH_{*} (X)$.  We will denote the corresponding first ($i\geq 0$)
and third ($i<0$) quadrant double complexes by $C_{**} (X)$ and
$C^{-}_{**} (X)$, so there is a {\SES} of double complexes
\begin{numequation}\label{eq-first-SES}
\begin{split}
0\to C^{-}_{**} (X) \to C^{\rm per}_{**} (X) \to C_{**} (X) \to 0.
\end{split}
\end{numequation}%

When $X$ is $\mathbf{HH}_{\bullet} (A)$ as in \cref{eq-simp-chain}, we
will denote $C^{\rm per}_{*,*} (X)$ by $C^{\rm per}_{*,*} (A)$, and
similarly for the related complexes $C^{\Hoch}_{*}$, $C^{\acyc}_{*}$,
$C^{-}_{*,*}$ and $C_{*,*}$.
\end{defin}

We will see analogs \cref{eq-first-SES} below in \cref{eq-second-SES} and 
\cref{eq-PPP}. The latter two have to do with
$\mT$-spectra rather than cyclic objects in an abelian category.  We
call them {\bf Tate sequences}. See \cref{def-P}.

The acyclic complex is so named because the map
\begin{numequation}\label{eq-contracting2}
\begin{split}
u= (-1)^{j}s_{j}:X_{j}\to X_{j+1}
\end{split}
\end{numequation}%

\noindent satisfies $ b'u+ub'=1_{X_{j}}$, making it a contracting
chain homotopy generalizing \cref{eq-contracting}.

\begin{remark}
{\bf The curious nature of the Tsygan  double complex.}  The Hochschild
complex $C^{\Hoch} (X)$ is the complex $\Ch(X)$ of
\cref{def-chain-complex}.  The acyclic complex $C^{\acyc} (X)$, which has
the same chain objects as $C^{\Hoch} (X)$, is also defined for any
simplicial abelian object $X$. No use of the cyclic structure is made
in defining either.  Only the horizontal arrows in \cref{eq-C**} are
defined in terms of it.

However in the simplicial setting there is no obvious map between the
two chain complexes, so the existence of boundary operator on $\Ch
(X)$ that renders it acyclic appears to be an idle curiosity.
\end{remark}

Let $d^{h}$ and $d^{v}$ denote the horizontal and vertical arrows in
\cref{eq-C**}.  Goodwillie shows that that $d^{h}d^{v}+d^{v}d^{h}$
vanishes on each $C_{i,j} (X)$, generalizing a similar argument in
\cite[Lemma 1.1]{Loday-Quillen}.

\begin{defin}
{\bf Cyclic homology.}  With notation as above, $\HC_{*} (X)$, the
{\bf cyclic homology of $X$}, is the homology of the first quadrant
total complex $\Tot^{\oplus } (C_{**} (X))$.  This is the chain complex defined
by
\begin{displaymath}
\Tot^{\oplus } (C_{**} (X))_{n}:=\bigoplus_{i+j=n\atop i,j\geq 0}C_{i,j} (X)
\end{displaymath}

\noindent with boundary operator $d^{h}+d^{v}$.
\end{defin}

Following \cite[page 191]{Goodwillie-cyc}, the double complex map
$s:C_{i,j} (X)\to C_{i-2,j} (X)$, which is an
isomorphism for $i\geq 2$, leads to a {\SES} 
\begin{displaymath}
0\to \ker (s)\to \Tot^{\oplus} (C_{**} (X))_{*}\to \Tot^{\oplus} (C_{**} (X))_{*-2}\to 0
\end{displaymath}

\noindent and a {\LES} 
\begin{numequation}\label{eq-SBI}
\begin{split}
\dotsb \to \HH_{*} (X)\to \HC_{*} (X)
       \to \HC_{*-2} (X)\to \HH_{*-1} (X)\to \dotsb 
\end{split}
\end{numequation}%

\noindent due in cohomological form to Connes \cite{Connes} and
\cite[III.1.$\gamma $]{Connes94}; see \cite[Theorem
2.2.1]{Loday}. Weibel refers to this as the SBI sequence
\cite[Proposition 9.6.11]{Weibel-H}.  It follows that a map of cyclic
objects inducing an isomorphism in $\HH_{*}$ also induces one in
$\HC_{*}$.

Composing the map $\HH_{*} (X)\to \HC_{*} (X)$ in \cref{eq-SBI} with
the reindexed 
\begin{displaymath}
\HC_{*} (X)\to \HH_{*+1} (X)
\end{displaymath}

\noindent gives the {\bf Connes operator}
\begin{numequation}\label{eq-Connes-operator}
\begin{split}
B:\HH_{j} (X)\to \HH_{j+1} (X).
\end{split}
\end{numequation}%

\noindent  It is induced by the composite 
\begin{displaymath}
\xymatrix
@R=4mm
@C=20mm
{
{X_{j}=C_{2i+2,j} (X)}\ar[r]^(.5){\epsilon uN}
  &{C_{2i,j+1} (X)=X_{j+1}}
}
\end{displaymath}

\noindent in \cref{eq-C**}, which we will also denote by $B$. Note that
\begin{displaymath}
BB= (\epsilon u N) (\epsilon uN) =\epsilon u ( N\epsilon) uN=0,
\end{displaymath}

\noindent and the cochain complex $(\HH_{*} (X), B)$ is the {\bf
de~Rham complex of $X$}.

\begin{defin}
The {\bf second double complex of a cyclic abelian object $X$} is
obtained from that of \cref{eq-C**} by removing the acyclic oddly
indexed columns of \cref{eq-C**} and suitably regrading, thereby
obtaining the half plane complex $\mcB^{\rm per}_{**} (X)$ in which
$\mcB^{\rm per}_{i,j} (X)=X_{j-i}$ for $j\geq i$ and vanishes for
$j<i$.  In the following picture, the horizontal and vertical
coordinates are $i$ and $j$.
\begin{numequation}\label{eq-B**}
\begin{split}
\xymatrix
@R=.5mm
@C=10mm
{
  &  &{-1}&{0}&{1}&{2}\\
  &  &{\vdots}\ar[dd]_(.45){-b}
          &{\vdots}\ar[dd]_(.45){b}
              &{\vdots}\ar[dd]_(.45){-b}
                  &{\vdots}\ar[dd]_(.45){b}\\
\\
{2}
  &{\cdots}
     &{X_{3}}\ar[l]_(.5){B}\ar[ddd]_(.5){-b}
          &{X_{2}}\ar[l]_(.45){B}\ar[ddd]_(.5){b}
              &{X_{1}}\ar[l]_(.45){B}\ar[ddd]_(.5){-b}
                  &{X_{0}}\ar[l]_(.45){B}\\
\\
\\
{1}
  &{\cdots}
     &{X_{2}}\ar[l]_(.5){B}\ar[ddd]_(.5){-b}
         &{X_{1}}\ar[l]_(.45){B}\ar[ddd]_(.5){b}
             &{X_{0}}\ar[l]_(.45){B}\\
\\
\\
{0}
  &{\cdots}
     &{X_{1}}\ar[l]_(.5){B}\ar[ddd]_(.5){-b}
         &{X_{0}}\ar[l]_(.45){B}\\
\\
\\
{-1}
  &{\cdots}
       &{X_{0},}\ar[l]_(.5){B}
}
\end{split}
\end{numequation}%

\noindent where on $X_{j}$,
\begin{align*}
b & = \sum_{0\leq k\leq j} (-1)^{k}d_{k}  \\
\aand 
B & = \epsilon u N
    = ((-1)^{j} + t_{j})s_{j}
    \sum_{0\leq k\leq j} ((-1)^{j}t_{j})^{k}.
\end{align*}

\noindent The horizontal and vertical differentials are $B$ and
$(-1)^{i}b$ respectively.

There is a subcomplex $\mcB^{-}_{*,*} (X)$ obtained by replacing the groups
with $i\geq 0$ by 0, which is concentrated in the second quadrant plus
half of the third one. Its quotient $\mcB_{**} (X)$ is concentrated in
the first quadrant.
\end{defin}

Thus this complex is concentrated in the northwest half plane defined
by $j\geq i$.  It is both a chain complex and a cochain complex, and
is called a {\em mixed complex} by Kassel in \cite{Kassel}.  See
\cite[9.8]{Weibel-H} for further discussion.  It should be compared
with the bicomplex $\mcB(A)$ of \cite[2.1.7]{Loday}.  A cohomological
variant of it is studied by Connes in \cite[3.1.$\gamma$]{Connes94}.

The corresponding $\Tot^{\oplus}$s, in all three versions, are known to be
quasi-isomorphic (up to regrading) to those of \cref{eq-C**}.  We have
\begin{align*}
\Tot^{\oplus} (\mcB_{**} (X))_{n}
 & =\bigoplus_{0\leq i\leq n/2} X_{n-2i},\\
\Tot^{\oplus} (\mcB^{\rm per}_{**} (X))_{n}
 & = \mycases{    
\displaystyle{\bigoplus_{i\geq 0}B_{2i}}
       &\mbox{for $n$ even}\\
\displaystyle{\bigoplus_{i\geq 0}B_{2i+1}}
       &\mbox{for $n$ odd,}
}\\
\aand 
\Tot^{\oplus} (\mcB^{-}_{**} (X))_{n}
 & = \mycases{    
\displaystyle{\bigoplus_{i\geq \max((n+2)/2,0)}B_{2i}}
       &\mbox{for $n$ even}\\
\displaystyle{\bigoplus_{i\geq  \max((n+1)/2,0)}B_{2i+1}}
       &\mbox{for $n$ odd.}
} 
\end{align*}

It follows that we have a {\SES}  of double complexes
similar to \cref{eq-first-SES},
\begin{numequation}\label{eq-second-SES}
\begin{split}
0\to \mcB^{-}_{**} (X)
  \to \mcB^{\rm per}_{**} (X)
  \to \mcB_{**} (X)
  \to 0
\end{split}
\end{numequation}%

\noindent and similarly for their $\Tot^{\oplus}$s.  We denote their
homology groups by $\HC^{-} (X)$, $\Hper (X)$, and $\HC (X)$
respectively.  In each case one can filter the double complex
increasingly by its columns and obtain a {\em Connes spectral
sequence} converging to its homology with input $\HH_{*} (X)$ in which
the first differential is the Connes operator $B$.




\subsection{The free loop space of the circle}\label{sec-LS1}

We recall properties of the free loop space $\mcL S^{1}$ for future
reference.  The circle group $\mT$ acts (as it does on any free loop
space) on by rotation of loops.  Each point in it has a winding number
or degree associated with it, and we denote by $\mcL_{r} S^{1}$ the
space of loops with degree $r$.  The {\Wittmonoid} $\mathscr{M}$ of
\cref{def-Witt-monoid} acts on $\mcL S^{1}$ as explained in
\cref{eq-Witt-free-loop}.

Let 
\begin{displaymath}
p:\reals\to S^{1}\subseteq \cxs 
\qquad \mbox{with}\qquad p (t):=e^{2\pi t\,\sqrt[]{-1}},
\end{displaymath}

\noindent where $S^{1}\subseteq \cxs$ is the unit circle.  Then a loop
of degree $r$ lifts via $p$ to a path  in $\reals$ between two points $|r|$
units apart, for example to a closed path when $r=0$.  The
space of such loops is equivalent to $S^{1}$ with $\mT$ acting via the
$r$th power map.  It follows that we have a $\mT$-{\eqvr} equivalence
\begin{numequation}\label{eq-LS1}
\begin{split}
\mcL S^{1} \simeq \coprod_{r\in \Z}\mT/\rC_{|r|}.
\end{split}
\end{numequation}%

\noindent Here ``$\mT/\rC_{0}$'' is understood to be the circle with
trivial $\mT$-action, there being no such group as $\rC_{0}$.  The
underlying space is {\eqt} to $\Z\times S^{1}$, and $\mT$ acts on the
$r$th component  via the $r$th power map.  This
result is the $G=\Z$ case of the following.

\begin{lem}\label{lem-}
\cite[Lemma 9.1]{KSS} 
Let $G$ be any topological group of CW type. Then there is a fibrewise
homotopy equivalence
\begin{displaymath}
\mcL BG \simeq EG\times_{G}G^{\ad},
\end{displaymath}

\noindent where $G^{\ad}$ denotes $G$ with the conjugation action, of
fibrewise H-spaces over $BG$. In particular if $G$ is abelian,
\begin{displaymath}
\mcL BG \simeq BG\times G.
\end{displaymath}
\end{lem}

For $G=\Z$, $BG\simeq S^{1}$, and this is our description of the space
underlying $\mcL S^{1}$.  Similarly for $G=\Z_{p}$, we have 
\begin{displaymath}
\mcL B\Z_{p}\simeq B\Z_{p}\times \Z_{p}.
\end{displaymath}

\noindent Like any free loop space, this is a $\mT$-space.  In order to
describe it as such, note that as topological spaces,
\begin{displaymath}
\Z_{p}\cong \left\{0 \right\} \amalg 
\coprod_{j\geq 0}      p^{j}\Z_{p}^{\times }.
\end{displaymath}

\noindent This is the stratification of $\Z_{p}$ by $p$-adic
valuation.  Each subspace on the right other than $\left\{0 \right\}$
is open, and all of them are closed.

Since $\Z_{p}$ is abelian, its classifying space is also an abelian
group which we denote by $\mT_{p}$.  The homomorphism $\Z \to \Z_{p}$
induces $\mT\to \mT_{p}$.  For each $j\geq 0$ we have
\begin{displaymath}
\rC_{p^{j}}\subseteq \mT\subseteq \mT_{p}
\end{displaymath}

\noindent with $\mT_{p}/\rC_{p^{j}}\cong \mT_{p}$ as underlying
spaces, since $\Z_{p}$ and $p^{j}\Z_{p}$ are isomorphic as groups.
Then we find that as $\mT$-spaces
\begin{numequation}\label{eq-LBZp}
\begin{split}
\mcL B\Z_{p}\simeq B\Z_{p}  \amalg 
\coprod_{j\geq 0} 
\left(p^{j}\Z_{p}^{\times }\times  \mT_{p}/\rC_{p^{j}} \right),
\end{split}
\end{numequation}%

\noindent where $\mT$ acts trivially on the first summand on the
right.  This is a special case of \cite[Theorem 7.3.11]{Loday}.

\bigskip

\section{The work of {\Bok}, Hsiang and Madsen}\label{sec-BHM} 
Now we quote some results from \cite[\S1]{BHM93} indicating the action
of $\mT$ on spaces related to $\THH$.  Madsen's expository accounts
\cite{Madsen, Madsen94} of this work are very helpful.

\subsection{$\THH$ via {\Bok} functors}

In most papers on the subject, what we are calling a {\Bok} functor is
referred to as a {\bf functor with smash products} or {\bf FSP} and
is usually denoted by the letter $F$.

\begin{defin}\label{def-FSP} 
\cite[Definition 1.1]{Bokstedt} 
A {\bf {\Bok} functor} is an endofunctor $\mcF$  on the category of
pointed topological spaces equipped with natural transformations
\begin{align*}
\mu :\mcF (-)\wedge \mcF (-)
  & \implies \mcF (-\wedge -)   &
 \aand 
\iota :\Id
  & \implies \mcF
\end{align*}

\noindent (which we will denote by $\mu (\mcF)$ and $\iota (\mcF)$
when there is more than one such functor in play) inducing morphisms
\begin{align*}
\mu_{X,Y}
  &:\mcF (X)\wedge \mcF (Y)\to \mcF (X\wedge Y)&
\aand 
\iota_{X}
  &:X\to \mcF (X)
\end{align*}

\noindent for which the following three diagrams commute:
\begin{numequation}\label{eq-Bok1}
\begin{split}
\xymatrix
@R=1mm
@C=20mm
{
  &{\mcF (X)\wedge \mcF (Y)}\ar[dd]^(.5){\mu_{X,Y}}\\
{X\wedge Y}\ar[ur]^(.45){\iota_{X}\wedge \iota_{Y}}
           \ar[dr]_(.45){\iota_{X\wedge Y}}\\
  &{\mcF (X\wedge Y),}
}
\end{split}
\end{numequation}%

\begin{numequation}\label{eq-Bok2}
\begin{split}
\xymatrix
@R=6mm
@C=20mm
{
{\mcF (X)\wedge \mcF (Y)\wedge \mcF (Z)}
      \ar[r]^(.53){\mu_{X,Y}\wedge \mcF (Z)}
      \ar[d]_(.5){\mcF (X)\wedge \mu_{Y,Z}}
  &{\mcF (X\wedge Y)\wedge \mcF (Z)}\ar[d]^(.5){\mu_{X\wedge Y,Z}}\\
{\mcF (X)\wedge \mcF (Y\wedge Z)}\ar[r]^(.5){\mu_{X,Y\wedge Z}}
  &{\mcF (X\wedge Y\wedge Z),}
}
\end{split}
\end{numequation}%

\noindent and 
\begin{numequation}\label{eq-Bok3}
\begin{split}
\xymatrix
@R=6mm
@C=20mm
{
{X\wedge \mcF (Y)}\ar[r]^(.45){\iota_{X}\wedge \mcF (Y)}
               \ar[d]_(.5){t}             
  &{\mcF (X)\wedge \mcF (Y)}\ar[r]^(.5){\mu_{X,Y}}
    &{\mcF (X\wedge Y)}\ar[d]^(.5){\mcF (t)}\\
{\mcF (Y)\wedge X}\ar[r]^(.5){\mcF (Y)\wedge \iota_{X}}
  &{\mcF (Y)\wedge \mcF (X)}\ar[r]^(.5){\mu_{Y,X}}
    &{\mcF (Y\wedge X),}
}
\end{split}
\end{numequation}%

\noindent where $t$ is transposition of smash product factors.  $\mcF$ is {\bf
commutative} if the following diagram also commutes.
\begin{numequation}\label{eq-Bok4}
\begin{split}
\xymatrix
@R=6mm
@C=20mm
{
{\mcF (X)\wedge \mcF (Y)}\ar[r]^(.5){\mu_{X,Y}}\ar[d]_(.5){t}
    &{\mcF (X\wedge Y)}\ar[d]^(.5){\mcF (t)}\\
{\mcF (Y)\wedge \mcF (X)}\ar[r]^(.5){\mu_{Y,X}}
    &{\mcF (Y\wedge X),}
}
\end{split}
\end{numequation}%

The adjoint {\bf stabilization and costabilization maps} $\sigma_{X}$
and $\eta_{X}$ are defined by
\begin{numequation}\label{eq-stabilization}
\begin{split}
\xymatrix
@R=0mm
@C=20mm
{
{S^{1}\wedge \mcF (X)}\ar[r]^(.46){\iota_{S^{1}}\wedge \mcF (X)}
                   \ar@{=}[dd]^(.5){}
  &{\mcF (S^{1})\wedge \mcF (X)}\ar[r]^(.53){\mu_{S^{1},X}}
    &{\mcF (S^{1}\wedge X)}
                   \ar@{=}[dd]^(.5){}\\
{}\\
{\Sigma \mcF (X)}\ar[rr]^(.5){\sigma_{X}}
  & &{\mcF (\Sigma  X)}\\
  &{\perp}\\
{\mcF (X)}\ar[rr]_(.5){\eta _{X}}
  & &{\Omega \mcF (\Sigma  X)}
}
\end{split}
\end{numequation}%

\noindent We require in addition that $\mcF$ preserves connectivity and
that $\pi_{n+i}\mcF (\Sigma^{n}X)$ is independent of $n$ for $n\gg i$.
\end{defin}

\begin{defin}\label{def-cofib-Bok}
A {\Bok} functor $\mcF $ is {\bf cofibrant} if for each CW-complex $X$,
$\mcF (X)$ is again a CW-complex, and the stabilization map
$\sigma_{X}$ is a cofibration.
\end{defin}
 
Nothing like this last definition is found in \cite{BHM93}, where is
no discussion of model structures.  They are only considered in
subsequent work on cyclotomic spectra as orthogonal spectra, starting
with that of Blumberg and Mandell \cite{Blumberg-Mandell-12}.  The
condition of \cref{def-cofib-Bok} insures that the spectrum $\mcF
(\mS)$ is cofibrant in the stable projective model structure on the
categories of sequential, symmetric and orthogonal spectra of
\cref{def-seq-spec,def-symm-orth}.  The cofibrancy condition is the
subject of \cite[Corollary 7.1.37]{HHR:ESHT}, and we will need it in
\cref{def-action}. There is a concise review of the homotopy theory of
(meaning model structures in) the categories of associative and
commutative ring spectra in \cite[\S2.3]{ABGHLM} by Blumberg, Mandell,
Vigleik Angeltveit, Teena Gerhardt, Mike Hill and Tyler Lawson.

\begin{ex}\label{ex-BokF}
{\bf Some {\Bok} functors.}
\begin{enumerate}[label={(\roman*)},itemindent=0em]

\item \label{ex-BokFii}

For a group-like topological monoid $\Gamma $, let 
$\mcFG:=(-\wedge \Gamma _{+})$, making $\mcFG
(\mS)$ the suspension spectrum for $\Gamma _{+}$.  It is denoted by
$\underline{\underline{\Gamma }}$ in \cite[Example 3.2
(i)]{BHM93}. When $\Gamma $ is a single point, $\mcFG$ is the identity
functor $\Id$.  For a pointed space $X$, the Moore loop space
$\Omega^{M} X$ (see \cite[5.1]{Carlsson-Milgram}), which we will
denote simply by $\Omega X$, is such a monoid.  $\mcFG$ is commutative
as when $\Gamma $ is.  We will say that a {\Bok} functor of this form
is {\bf monoidal}.  It is cofibrant as in \cref{def-cofib-Bok} if
$\Gamma $ is a CW-complex.

\item \label{ex-BokFi}
For a discrete ring $R$, let $\mcF_{R}$ be the {\Bok} functor  defined by 
\begin{displaymath}
\mcF_{R} (X) := |R\Sing_{\bullet} (X)/R\Sing_{\bullet}(\pt) |,
\end{displaymath}

\noindent where $\Sing_{\bullet}$ is as in \cref{def-Sing} and
$R\Sing_{\bullet} (-)$ denotes the free simplicial $R$-module on the
indicated simplicial set.  It is denoted by $\underline{\underline{R
}}$ in \cite[Example 3.2 (iii)]{BHM93}. This is a generalized {\SESM}
space with
\begin{displaymath}
\pi_{*}\mcF_{R} (X)\cong \overline{H}_{*} (X;R). 
\end{displaymath}

\noindent This means the spectrum $\mcF_{R} (\mS)$ is the {\SESM}
spectrum for $R$. $\mcF_{R}$ is commutative as in \cref{eq-Bok4} when
$R$ is commutative.  It is cofibrant for any $R$.

\item \label{ex-BokFiii}
For a {\Bok} functor $\mcF$ and integer $m>0$, we can form the
corresponding {\bf matrix {\Bok} functor }
\begin{displaymath}
\mcM_{m} (\mcF) (X)
 := \Map (\langle m \rangle, \langle m \rangle\wedge  \mcF (X)),
\end{displaymath}

\noindent where $\langle m \rangle := \left\{1,\dotsc ,m
\right\}$.  A point in this mapping space is an $(m^{2})$-tuple of
points in $\mcF (X)$.

We need to specify the natural transformations $\mu (\mcM_{m} (\mcF))$
and $\iota (\mcM_{m} (\mcF))$.  For $\alpha_{1}\in \mcM_{m} (\mcF)
(X)$ and $\alpha_{2}\in \mcM_{m} (\mcF) (Y)$, the composite
\begin{displaymath}
\xymatrix
@R=6mm
@C=10mm
{
{\langle m \rangle}\ar[r]^(.35){\alpha_{2}}
  &{\langle m \rangle\wedge \mcF (Y)}
                    \ar[d]^(.5){\alpha_{1}\wedge \mcF (Y)}\\
  &{\langle m \rangle\wedge \mcF (X)\wedge \mcF (Y)}
                    \ar[rr]^(.53){\langle m \rangle\wedge \mu_{X,Y}}
      &&{\langle m \rangle\wedge \mcF (X\wedge Y),}
}
\end{displaymath}

\noindent which lives in $\mcM_{m} (\mcF) (X\wedge Y)$, behaves like
matrix multiplication.  This can be reinterpreted as a map
\begin{displaymath}
\mu (\mcM_{m} (\mcF))_{X,Y}:
\mcM_{m} (\mcF) (X)\wedge \mcM_{m} (\mcF) (Y)\to \mcM_{m} (\mcF)
(X\wedge Y)
\end{displaymath}

\noindent as in \cref{def-FSP}, with $\iota (\mcM_{m} (\mcF))_{X}:X\to
\mcM_{m} (\mcF) (X)$ given by
\begin{displaymath}
(\iota (\mcM_{m} (\mcF))_{X} (x)) (i):=i\wedge \iota (\mcF)_{X} (x)
\qquad \mbox{for $x\in X$ and $i\in \langle m \rangle$.} 
\end{displaymath}

\noindent It is cofibrant when $\mcF $ is.
\end{enumerate}
\end{ex}










\begin{defin}\label{def-THHF} 
\cite[(3.4)]{BHM93} 

\begin{enumerate}[label={(\roman*)},itemindent=0em]
\item \label{def-THHFi} 

The {\bf topological Hochschild homology  $\THH(\mcF)$ of a {\Bok}
functor $\mcF$} is the geometric realization of the simplicial space
$\THH_{\bullet} (\mcF)$ defined by
\begin{numequation}\label{eq-THHn}
\begin{split}
\THH_{n} (\mcF)
:=\hocolim{k} \Omega^{|k|}
       \left( \mcF (S^{k_{0}})\wedge \cdots \wedge \mcF (S^{k_{n}}) \right),
\end{split}
\end{numequation}%

\noindent where the colimit runs over all $(n+1)$-tuples of
nonnegative integers
\begin{displaymath}
k=(k_{0},\dotsc ,k_{n})
\qquad \mbox{with}\qquad  
|k|:=k_{0}+\dotsb +k_{n},
\end{displaymath}

\noindent and the maps are induced by the costabilization map of
\cref{eq-stabilization}.  The face and degeneracy maps are defined as
in \cref{eq-simp-chain}, with the unit 1 being replaced by $S^{0}$.

\item \label{def-THHFii} 
For a pointed space $X$, we define $\THH^{X}_{\bullet} (\mcF)$ similarly by 
\begin{displaymath} 
\THH^{X}_{n} (\mcF)
:=\hocolim{k} \Omega^{|k|}
   \left( \mcF (S^{k_{0}})\wedge \cdots 
   \wedge \mcF (S^{k_{n}})\wedge X \right).
\end{displaymath}

\item \label{def-THHFiii} 
The simplicial space $\THH_{\bullet} (\mcF)$ is a simplicial
infinite loop space, and we denote the corresponding spectrum by
$\tHH (\mcF)$. The spectrum associated with $\THH^{X}_{\bullet} (\mcF)$ is
$\tHH (\mcF)\wedge X$.

\item \label{def-THHFiv} 
When $\mcF =\mcF_{R}$ for a discrete ring $R$ as in
\cref{ex-BokF}\cref{ex-BokFi}, we denote the spectrum $\tHH(\mcF_{R})$
by $\THH(R)$.  It is denoted by $T (R)$ in
\cite[\S4]{Bokstedt-Madsen94}.

\end{enumerate}

\end{defin}

Here we are replacing the associative ring $A$ of
\cref{sec-classical-algebra} by the associative ring spectrum $\mcF
(\mS)$.  There is a spectral sequence converging to $H_{*}\tHH
(\mcF)\wedge X$ (with field coefficients) whose input is $\HH_{*}
(H_{*}\mcF (\mS)\wedge X)$.

\begin{prop}\label{prop-THH-loop}
{\bf $\THH$ and the free loop space.}\cite[Proposition 3.7 for $r=1$]{BHM93} 
 For $\mcF_{\Omega X}$ as in
\cref{ex-BokF}\cref{ex-BokFii}, we have
\begin{displaymath}
\tHH (\mcF_{\Omega X})\simeq \Sigma^{\infty }_{+} (\mcL X).
\end{displaymath}
\end{prop}

$\THH_{\bullet} (\mcF)$ is also a cyclic space as in
\cref{def-cyc-object}, which means that the space $\THH(\mcF)$ and the
spectrum $\tHH(\mcF)$ each have a $\mT$-action.

The following is \cite[Proposition 3.7]{BHM93} in case of monoidal
$\mcF $ (as in \cref{ex-BokF}\cref{ex-BokFii}) and can be deduced from
\cite{ABGHL} for $\mcF $ cofibrant as in \cref{def-cofib-Bok}.

\begin{theorem}\label{def-action}
{\bf The action of $\rC _{r}$ on $\THH$.} For each monoidal or
cofibrant {\Bok} functor $\mcF $ and each integer $r\geq 1$, the
simplicial $\rC _{r}$-space $\sd^{*}_{r}\THH^{X}_{\bullet} (\mcF)$ is
\begin{numequation}\label{eq-sdr-THHn}
\begin{split}
\sd^{*}_{r}\THH^{X}_{n} (\mcF)
=\hocolim{k} \Omega^{|k|\varrho_{r}}
       \left( \mcF (S^{k_{0}})^{\wedge r}\wedge 
      \cdots \wedge \mcF (S^{k_{n}})^{\wedge r}\wedge X \right),
\end{split}
\end{numequation}%

\noindent where  $\varrho_{r}$ denotes the real regular {\rep} of the
cyclic group $\rC _{r}$, $\Omega^{k\varrho_{r}} (-)$ denotes the twisted
loop space of \cref{def-piHX}, and $\rC _{r}$ permutes the factors of
each of the the $r$-fold smash powers of the spaces $\mcF (S^{k_{i}})$.
\end{theorem}

The following is not proved in \cite{BHM93}, so we
will sketch a proof here. 

\begin{prop}\label{prop-fixed-point-sets}
{\bf Fixed point sets.}  For a monoidal or cofibrant {\Bok} functor
$\mcF $ and each subgroup $\rC _{q}\subseteq \rC _{r}$, the fixed
point set $(\sd^{*}_{r}\THH^{X} (\mcF))^{\rC _{q}}$ has a residual
action (\cref{def-residual}) of $\rC _{r/q}$ and is {\eqvr}ly
homeomorphic to $\sd^{*}_{r/q}\THH^{X} (\mcF)$ with the action of
\cref{def-action}.
\end{prop}

\proof We will do this for $q=r$, leaving the general case to the
reader.  The space of \cref{eq-sdr-THHn} is a filtered colimit of
$\rC_{r}$-spaces.  Filtered colimits commute with finite limits, so it
suffices to determine the colimit of fixed point sets.  For each $k$
we have a space of non{\eqvr} maps between $\rC_{r}$-spaces on which
$\rC_{r}$ acts by conjugation.  Its fixed point set is the
corresponding space of {\eqvr} maps. This is the image of the $k$th
space of \cref{eq-THHn} under the norm map that sends a pointed space
$X$ to its $r$-fold smash product with $\rC_{r}$ permuting the
factors.  It follows that
\begin{displaymath}
(\sd^{*}_{r}\THH^{X}_{n} (\mcF))^{\rC_{r}} = \THH^{X}_{n} (\mcF),
\end{displaymath}

\noindent and similarly
\begin{displaymath}
(\sd^{*}_{r}\THH^{X}_{n} (\mcF))^{\rC_{q}} = \sd^{*}_{r/q}\THH^{X}_{n} (\mcF).
\qedhere
\end{displaymath}

\begin{cor}\label{cor-}
{\bf Smooth cyclotomy of $\THH$.}  For each monoidal or cofibrant
{\Bok} functor $\mcF $, $\THH (\mcF )$ is smoothly cyclotomic as in
\cref{def-smooth-cyclo}.
\end{cor}

\begin{ex}\label{ex-THH-Id}
{\bf $\THH$ of the identity functor.} \cite[Example 2.8]{Madsen94} It
follows from \cref{def-THHF} that $\Id (\mS)=\mcF_{\pt} (\mS) \simeq
\mS$.  The identification of its fixed point sets under finite cyclic
$p$-groups is more interesting.  We have
\begin{displaymath}
\THH (\Id)^{\rC_{p^{n}}}\simeq \prod_{i=0}^{n}Q_{+} (B\rC_{p^{i}}),
\end{displaymath}

\noindent where 
\begin{numequation}\label{eq-Q+}
\begin{split}
Q_{+} (X) :=\hocolim{j}\Omega^{j} \Sigma^{j} (X_{+}).
\end{split}
\end{numequation}%

\noindent  The maps $\Frob$ and $\Res$ of \cref{def-residual}
 are given by 
\begin{align*}
\Frob^{}_{} (x_{0},\dotsc ,x_{n})
 & = (x_{0}+T (x_{1}), T (x_{2}),\dotsc ,T (x_{n}))  \\
\aand 
\Res^{}_{} (x_{0},\dotsc ,x_{n})
 & =    (x_{0},\dotsc ,x_{n-1})
\end{align*}

\noindent where $ T : Q_{+} (B\rC_{p^{i}}) \to Q_{+} (B\rC_{p^{i-1}})$
is the transfer mapping associated to the degree $p$
covering $B\rC_{p^{i-1}}\to B\rC_{p^{i}}$.

It follows that the spectrum $\tHH (\Id )$ is the sphere spectrum
$\mS$.
\end{ex}

\begin{theorem}\label{thm-Bokstedt}
{\bf $\THH$ of $\Z$ and $\Z/p$.} \cite[Theorem 1.1]{Bokstedt2} For a
discrete ring $R$, let $\mcF_{R}$ be as in
\cref{ex-BokF}\cref{ex-BokFi}.
\begin{enumerate}[label={(\roman*)},itemindent=0em]
\item \label{thm-Bokstedti} 
For each prime $p$,
\begin{align*}
\THH (\mcF_{\Z/p})
 & \cong  \prod_{i=0}^{\infty }{^{'}}K (\Z/p, 2i)\\
\mbox{and } 
\pi_{*}\THH (\mcF_{\Z/p})
 & = \Z/p[b]\mbox{ with }b\in \pi_{2} ,
\end{align*}

\noindent where $\prod '$ denotes the restricted product, meaning the
colimit of finite products.  

\item \label{thm-Bokstedtii} 
\begin{displaymath}
\THH (\mcF_{\Z}) \cong  \Z\times  \prod_{i=1}^{\infty }{^{'}}K (\Z/i, 2i-1),
\end{displaymath}

\item \label{thm-Bokstedtvii} \cite[Theorems 1.3 and 1.4]{BCS} 
$\THH (\Z/p)$ and $\THH (\Z)$ (as in
\cref{def-THHF}\cref{def-THHFiv}) are the corresponding
{\SESM} spectra with
\begin{align*}
\THH (\Z/p)
 & \simeq H\Z/p \otimes \Omega S^{3}_{+}   \\
\aand 
\THH (\Z)
 & \simeq H\Z  \otimes \tau_{\geq 3}\Omega S^{3}_{+}.
\end{align*}

\noindent Here $\tau_{\geq 3}\Omega S^{3}$ (denoted by $\Omega
(S^{3}\langle 3\rangle)$ in \cite{BCS}) denotes the 3-connective cover
of $\Omega S^{3}$, which is the fiber of the evident map to $K
(\Z,2)$.  A classical calculation involving the Serre spectral
sequence for the fiber sequence
\begin{displaymath}
S^{1}\to \tau_{\geq 3}\Omega S^{3}\to \Omega S^{3}
\end{displaymath}

\noindent (see \cite[Lemma 1.2.3]{Rav:MU} for a closely related
computation) shows that its integer homology coincides with the
homotopy groups of \cref{thm-Bokstedtii}.

\item \label{thm-Bokstedtiii} 
The map between the these spaces induced by $\Z\to \Z/p$ preserves the
product decomposition, is the obvious reduction for $i=0$, is trivial
for $i$ not divisible by $p$, and is the Bockstein $K (\Z/pj,
2pj-1)\to K (\Z/p, 2pj)$ for $i=pj$ with $j>0$.

\item \label{thm-Bokstedtiv} In the decomposition of
\cref{thm-Bokstedti}, let $\iota_{2i}\in H^{2i} (K (\Z/p,2i);\Z/p)$ be the
fundamental class.  Then the coproduct in cohomology is given by
\begin{displaymath}
\iota_{2i}\mapsto \sum_{0\leq k\leq i}\iota_{2k}\otimes \iota_{2i-2k},
\end{displaymath}

\noindent which is dual to the multiplication of \cref{thm-Bokstedti}.

\item \label{thm-Bokstedtv} In the decomposition of
\cref{thm-Bokstedtii}, let 
\begin{displaymath}
\iota_{2pj-1}\in H^{2pj-1} (K (\Z/pj,2pj-1);\Z/p)
\end{displaymath}

\noindent be the fundamental class.  Then the coproduct in cohomology
is given by
\begin{displaymath}
\iota_{2pj-1}\mapsto \iota_{2pj-1}\otimes 1 +1\otimes \iota_{2pj-1} 
  +  \sum_{0<k<j}\beta (\iota_{2pk-1}\otimes \iota_{2p (j-k)-1}),
\end{displaymath}

\noindent where $\beta:H^{2pj-2}\to H^{2pj-1} $ denotes the mod $p$ Bockstein.

\end{enumerate}
\end{theorem}

In \cref{sec-THHZp} we will describe the cyclotomic structure
(\cref{def-cyclo-spec2}) of the spectrum $\THH (\Z/p)$.
The above is merely a description of its underlying homotopy type.
We will see that the action of $\mT$ on it is nontrivial.

Here is a sketch of a proof of \cref{thm-Bokstedt} due to Blumberg,
Ralph Cohen and Christian Schlichtkrull \cite{BCS}, which we learned
from Tomer Schlank.  It is also treated by Krause and Nikolaus in
\cite[\S4]{KN21a}, who in addition give a description of $\THH
(\Z/p^{j})$ for all $j> 0$, and  by Joseph Hlavinka in \cite[\S3]{Talbot24}. 

For any space $X$, one has
\begin{displaymath}
\mcL \Omega X\simeq \Omega^{2} X\times \Omega X.
\end{displaymath}

\noindent Thus for $X=S^{3}$, we have $\mcL \Omega S^{3}\simeq
\Omega^{2} S^{3}\times \Omega S^{3}$.  A theorem of Mahowald
\cite{Mah:Ring} says there is a map $\Omega^{2}S^{3}\to BO$ for which
the Thom spectrum is $H\FF_{2}$.  It is the double loop functor
applied to a map $S^{3}\to B^{3}O$. Applying the functor $\mcL \Omega
$ gives a map
\begin{displaymath}
\xymatrix
@R=4mm
@C=10mm
{
{\mcL \Omega S^{3}}\ar[r]^(.5){}
  &{\mcL \Omega B^{3}O\simeq BO\times BBO}\ar[r]^(.75){p_{1}}
    &{ BO.}
}
\end{displaymath}

\noindent An odd primary analog due to Mike Hopkins can be found in
\cite[Lemma 3.3]{MRS3}.  






Let $Y=\tau_{\geq 4}S^{3}$ (sometimes denoted by
$S^{3}\langle 3 \rangle$) be the 4-connective (or 3-connected) cover
of $S^{3}$, the fiber of the map $S^{3}\to K (\Z,3)$.  We know that
the Thom spectrum associated with $\Omega^{2}Y$ is $H\Z$. By studying
the {\SSS} for the fiber sequence
\begin{displaymath}
\xymatrix
@R=4mm
@C=10mm
{
{S^{1}}\ar[r]^(.5){}
  &{\Omega Y}\ar[r]^(.5){}
    &{\Omega S^{3},}
}
\end{displaymath}

\noindent one can show
\begin{displaymath}
H_{i} (\Omega Y; \Z) = \mycases{    
\Z
       &\mbox{for $i=0$}\\
\Z/m
       &\mbox{for $i=2m-1$ and $m>1$}\\
0      &\mbox{otherwise,}
}
\end{displaymath}

\noindent and this leads to \cref{thm-Bokstedt}\cref{thm-Bokstedtii}.

\begin{defin}\label{def-KF}
The {\bf algebraic $K$-theory of a {\Bok} functor $\mcF$.}
\cite[Definition 2.3]{Bokstedt}.  Let $\GL_{m} (\mcF)$ be the union of
the invertible components of
\begin{displaymath}
\hocolim{k}\Omega^{k} \mcM_{m}(\mcF)(S^{k}),
\end{displaymath}

\noindent where $\mcM_{m}(\mcF)$ is the matrix {\Bok} functor of
\cref{ex-BokF}\cref{ex-BokFiii}.  It is a group-like topological
monoid, and
\begin{displaymath}
\rK(\mcF)
 := \Omega_{0} B\left(\coprod_{m}B\GL_{m} (\mcF) \right)
  = \Omega_{0} B\left(\coprod_{m}
       \left|N_{\bullet}\GL_{m} (\mcF)\right| \right),
\end{displaymath}

\noindent where $\Omega_{0}$ on the right indicates the degree zero
component of the indicated loop space, in which $\pi_{0}$ is $\Z$.
\end{defin}
  
As in Quillen's original definition of algebraic $K$-theory, the
space on the right is {\eqt} to $B\GL (\mcF)^{+}$, the plus
construction on the filtered homotopy colimit of the spaces $B\GL_{m}
(\mcF)$ for $m\geq 0$, which is described by Weibel in
\cite[IV.1]{Weibel-K}.

\begin{ex}\label{ex-QW}
{\bf The $K$-theories of Quillen and Waldhausen.}
\begin{enumerate}[label={(\roman*)},itemindent=0em]

\item \label{ex-QWi} For $\mcF_{R}$ as in \cref{ex-BokF}\cref{ex-BokFi},
$\pi_{*}\rK(\mcF_{R})=\rK_{*} (R)$ as defined by Quillen in
\cite{Quillen73}.  This is discussed in \cite[\S2.6]{Madsen}. 

\item \label{ex-QWii} For $\mcF_{\Omega X}$ as in
\cref{ex-BokF}\cref{ex-BokFii}, $\rK(\mcF_{\Omega X})$ is Waldhausen's
space $A (X)$ as in \cite{Wald2}. This is discussed in
\cite[\S6]{Madsen94}.
\end{enumerate}
\end{ex}

\begin{defin}\label{def-KF-cyc}
The {\bf cyclic $K$-theory of a {\Bok} functor
$\mcF$.} \cite[(5.7)]{BHM93}.  
\begin{displaymath}
\rK^{\cyc}(\mcF)
 := \Omega_{0} B\left(\coprod_{m}
       \left|N^{\cyc}_{\bullet}\GL_{m} (\mcF)\right| \right),
\end{displaymath}

\noindent where $\Omega_{0}$ on the right indicates the degree zero
component of the indicated loop space.
\end{defin}
  
\subsection{The Dennis trace}\label{sec-Dennis}
The topological Dennis trace $\Tr: \rK(\mcF)\to \THH (\mcF)$, was
first defined in algebraic form by Keith Dennis in an unpublished
paper in 1976. For a topological monoid $\Gamma $, we have simplicial
sets $N_{\bullet} (\Gamma )$ and $N^{\cyc}_{\bullet} (\Gamma )$ as in
\cref{def-nerve,def-cyc-nerve}.  There is a map
\begin{numequation}\label{eq-I}
\begin{split}
I_{\bullet}:N_{\bullet} (\Gamma ) \to N^{\cyc}_{\bullet} (\Gamma )
\quad \mbox{with }I (\gamma_{1},\dotsc ,\gamma_{n})
= \left(\left(\gamma_{1},\cdots ,\gamma_{n} \right)^{-1}, 
                      \gamma_{1},\dotsc ,\gamma_{n}\right) 
\end{split}
\end{numequation}%

\noindent whose geometric realization is homotopic to the inclusion of
$B\Gamma $ into the free loop space $\mcL B\Gamma $ as the constant
loops.

For $\Gamma =\GL_{m} (\mcF)$ as in \cref{def-KF}, we denote the map of
\cref{eq-I} by $I^{(m)}_{\bullet}$.  Consider the composite
\begin{displaymath}
\xymatrix
@R=4mm
@C=10mm
{
{N_{\bullet}\GL_{m} (\mcF)}\ar[r]^(.48){I^{(m)}_{\bullet}}
  &{N^{\cyc}_{\bullet}\GL_{m} (\mcF)}\ar[r]^(.48){S^{(m)}_{\bullet}}
    &{\THH_{\bullet} (\mcM_{m} (\mcF))}
}
\end{displaymath}

\noindent where
\begin{numequation}\label{eq-S}
\begin{split}
S^{(m)}_{n} (f_{0},\dotsc ,f_{n})=f_{0}\wedge \dotsb \wedge f_{n}
\qquad \mbox{for }f_{i}:S^{k_{i}}\to \mcM_{m} (\mcF) (S^{k_{i}}). 
\end{split}
\end{numequation}%

\noindent   Passing to
geometric realizations gives a map
\begin{numequation}\label{eq-GL-to-THH}
\begin{split}
S^{(m)}I^{(m)}:B\GL_{m} (\mcF)\to \THH (\mcM_{m} (\mcF)).
\end{split}
\end{numequation}%

The following analog of \cref{thm-Morita} was proved in \cite{Bokstedt}.

\begin{theorem}\label{thm-Morita-Bok}
{\bf Morita invariance of {\Bok} functors.}  For each {\Bok}
functor $\mcF$, there is an equivalence
\begin{displaymath}
\Tra:   \Omega B\left(\coprod_{m\geq 0}\THH (M_{m} (\mcF)) \right)
\to \THH (\mcF)\times \Z.
\end{displaymath}
\end{theorem}


\begin{defin}\label{def-Dennis-trace}
{\bf The Dennis trace of a {\Bok} functor $\mcF$} is the following
composite.
\begin{displaymath}
\xymatrix
@R=8mm
@C=40mm
{
{\Omega_{0} B\left(\coprod_{m}B\GL_{m} (\mcF) \right)}
          \ar[r]^(.48){\Omega_{0} B\left(\coprod_{m}S^{(m)}I^{(m)}\right)}
  &{\Omega_{0} B\left(\coprod_{m}\THH (\mcM_{m} (\mcF)) \right)}
                       \ar[d]^(.5){\Tra}_(.5){\simeq }\\
{\rK(\mcF)}\ar@{=}[u]^(.5){}
  &{\THH (\mcF),}
}
\end{displaymath}

\noindent where each of the coproducts is over $m\geq 0$, and the
equivalence on the right is that of \cref{thm-Morita-Bok}.
\end{defin}

\subsection{The cyclotomic trace}\label{sec-TC-BHM} 
Now we need to
consider the action of $\mT$ and its finite subgroups $\rC _{r}$ for
$r\geq 1$ on the cyclic space $\THH(\mcF)$ of \cref{def-THHF}.  

First we need some elementary definitions.

We will now make use of the edgewise subdivision of
\cref{def-edgewise}.  Recall that the geometric realization of
$r$-fold edgewise subdivision of a simplicial set $X$ is homeomorphic
to $|X|$ by \cref{lem-subdiv-homeo}.  Thus we get a residual action
map as in \cref{def-residual}\cref{def-residualii},
\begin{align*}
\Res _{p}:= \Res ^{\rC_{p^{n}}}_{\rC_{p}}
   &:(\sd^{*}_{p^{n}}\THH (\mcF))^{\rC _{p^{n}}}
  \to ((\sd^{*}_{p^{n}}\THH (\mcF))^{\rC_{p}})^{\rC _{p^{n}/p}}.
\end{align*}

\noindent Recall that 
\begin{displaymath}
(\sd^{*}_{p^{n}}\THH (\mcF))^{\rC _{p}}\simeq 
\sd^{*}_{p^{n-1}}\THH (\mcF)
\end{displaymath}

\noindent by \cref{prop-fixed-point-sets}.  The map $\Res _{p}:=\Res
_{p}^{p^{n}}$ figures in the original definition of $\TopC (\mcF ,p)$,
\cite[Definition 5.12 (i)]{BHM93}, where it is denoted by $\Phi_{p}$.

Since $\sd^{*}_{p^{n}}\THH (\mcF)\approx\THH (\mcF)$, we have  
\begin{displaymath}
\Res _{p} :\THH (\mcF)^{\rC _{p^{n}}}\to \THH (\mcF)^{\rC _{p^{n}/p}}.
\end{displaymath}

We can also regard $\rC _{p^{n}/p}$ as a subgroup of $\rC _{p^{n}}$, yielding a
restriction map as in \cref{def-residual}\cref{def-residuali}.
\begin{displaymath}
\Frob_{p}:=\Frob^{\rC _{p^{n}}}_{\rC _{p^{n}/p}} :\THH (\mcF)^{\rC _{p^{n}}}
     \to \THH (\mcF)^{\rC _{p^{n}/p}},
\end{displaymath}

\noindent which is denoted by $D$ in \cite[(5.10)]{BHM93}. 

Thus we get maps
\begin{numequation}\label{eq-Phi-D}
\begin{split}
\Res  ,\Frob :\THH (\mcF)^{\rC _{p^{n}}}\to \THH (\mcF)^{\rC _{p^{n-1}}}
\end{split}
\end{numequation}%

\noindent as in \cref{def-residual} for which the diagram
\begin{numequation}\label{eq-FR}
\begin{split}
\xymatrix
@R=4mm
@C=10mm
{
{\THH (\mcF)^{\rC _{p^{n+1}}}}\ar[r]^(.5){\Res  }_(.5){\cong  }
                              \ar[d]_(.5){\Frob   }
  &{\THH (\mcF)^{\rC _{p^{n}}}}\ar[d]^(.5){\Frob  }\\
{\THH (\mcF)^{\rC _{p^{n}}}}\ar[r]^(.47){\Res  }_(.47){\cong  }
  &{\THH (\mcF)^{\rC _{p^{n-1}}}}
}
\end{split}
\end{numequation}%

\noindent commutes as a special case of \cref{eq-RF}.

These induce maps of filtered homotopy limits
\begin{numequation}\label{eq-two-diagrams}
\begin{split}
\xymatrix
@R=4mm
@C=10mm
{
{\holim{\Res }\THH (\mcF)^{\rC _{p^{n}}}}\ar@<.5ex>[r]^(.5){\Frob  }
                               \ar@<-.5ex>[r]_(.5){1 }
  &{\holim{\Res  }\THH (\mcF)^{\rC _{p^{n}}}}\\
{\holim{\Frob  }\THH (\mcF)^{\rC _{p^{n}}}}\ar@<.5ex>[r]^(.5){\Res   }
                                   \ar@<-.5ex>[r]_(.5){1 }
  &{\holim{\Frob  }\THH (\mcF)^{\rC _{p^{n}}},}
}
\end{split}
\end{numequation}%

\noindent and these two diagrams have {\eqt} equalizers. 

\begin{defin}\label{def-TCF}   
\cite[Definition 5.12 ]{BHM93} 
\begin{enumerate}[label={(\roman*)},itemindent=0em]
\item \label{def-TCFi}
The {\bf topological cyclic homology at $p$ of a {\Bok} functor
$\mcF$}, $\TopC (\mcF,p)$, is the equalizer of either diagram of
\cref{eq-two-diagrams}.  Equivalently it is the homotopy limit of the
diagram
\begin{displaymath}
\xymatrix
@R=4mm
@C=10mm
{
{\dotsb }
            \ar@<1ex>[r]^(.35){\Res } \ar@<-1ex>[r]_(.35){\Frob}
  &{\THH (\mcF )^{\rC _{p^{2}}}}
            \ar@<1ex>[r]^(.5){\Res } \ar@<-1ex>[r]_(.5){\Frob}
     &{\THH (\mcF )^{\rC _{p}}}
            \ar@<1ex>[r]^(.5){\Res } \ar@<-1ex>[r]_(.5){\Frob}
        &{\THH (\mcF ).}
}
\end{displaymath}
   
\item \label{def-TCFii} The {\bf cyclotomic trace at $p$} is the
infinite loop map
\begin{displaymath}
\xymatrix
@R=4mm
@C=10mm
{
{\rK (\mcF )}\ar[r]^(.3){I}\ar@/_1.5pc/[rr]_(.5){\trc}
  &{(\holim{\Frob}\rK^{\cyc} (\mcF )^{\rC_{p^{n}}})^{h\Res }}
             \ar[r]^(.64){S}
    &{\TopC} (\mcF ,p)
}
\end{displaymath}

\noindent where $\rK^{\cyc}(\mcF)$ is as in \cref{def-KF-cyc}, and the
maps $I$ and $S$ are as in \cref{eq-I} and \cref{eq-S}.  The middle
object above is the homotopy limit of the diagram
\begin{displaymath}
\xymatrix
@R=4mm
@C=10mm
{
{\dotsb }
            \ar@<1ex>[r]^(.35){\Res } \ar@<-1ex>[r]_(.35){\Frob}
  &{\rK^{\cyc} (\mcF )^{\rC _{p^{2}}}}
            \ar@<1ex>[r]^(.5){\Res } \ar@<-1ex>[r]_(.5){\Frob}
     &{\rK^{\cyc} (\mcF )^{\rC _{p}}}
            \ar@<1ex>[r]^(.5){\Res } \ar@<-1ex>[r]_(.5){\Frob}
        &{\rK^{\cyc} (\mcF ).}
}
\end{displaymath}
\end{enumerate}
\end{defin}

This is the first of several definitions of $\TopC$.  In \cref{def-TC}
it is defined as a mapping object in the {\qcat} of cyclotomic
spectra; see \cref{rem-TC-fixed}.  The equivalence of these two is due
to Blumberg and Mandell \cite[Theorem 1.4]{blumberg_cyclotomic}.  A
new formula for computing it is given in \cref{prop-TC}.  A formula
relating it to $\TR$ is given in \cref{thm-TC-TR}.

\begin{remark}\label{rem-TC-free-loop-space}
{\bf $\TopC$ of a free loop space.}
We have homeomorphisms
\begin{displaymath}
\varphi_{p}^{-1}:\mcL X^{\rC_{p^{n}}}\to \rho_{p}^{*}\mcL X^{\rC_{p^{n+1}}}
\end{displaymath}

\noindent for $\varphi_{p}$ as in \cref{ex-free-loop}, induced by the
$p$th power map on $\mT$.  If we replace $\THH (\mcF )$ by $\mcL X$ in 
\cref{eq-FR}, the horizontal maps become homeomorphisms and 
\begin{displaymath}
\holim{\Res}\mcL X^{\rC_{p^{n}}}
 = \lim{\Res}\mcL X^{\rC_{p^{n}}}\cong \mcL X. 
\end{displaymath}
 
\noindent Then the first equalizer of the analog of \cref{eq-two-diagrams}
would consist only of constant loops.

The other limit,
\begin{displaymath}
\lim{\Frob}\mcL X^{\rC_{p^{n}}},
\end{displaymath}

\noindent is the space constant loops, on which $\Res$ is the
identity.  Thus the free loop space analog of $\TopC$ appears to be
the space of constant loops, so {\em $\TopC$ seems to undo the free loop
functor $\mcL$}.
\end{remark}

\section{Spectra and equivariant stable homotopy theory}
\label{sec-orth-spec}

A {\bf spectrum $X$} was originally defined to be a
sequence of pointed spaces $X_{n}$ for $n\geq 0$ with {\bf structure
maps} $\epsilon^{X}_{n}:\Sigma X_{n}\to X_{n+1}$, or {\eqt}ly {\bf
costructure maps} $\eta ^{X}_{n}:X_{n}\to \Omega X_{n+1}$.  Since this
is only one of several ways to define spectra these days, we will call
such  objects  {\bf sequential spectra}.

A sequential spectrum in which each structure map is a weak {\equi} is
called a {\bf suspension spectrum}, and is often denoted by
$\Sigma^{\infty }X_{0}$. One in which each costructure map is a weak
{\equi } is called an {\bf $\Omega$-spectrum}.  The objects of Lurie's
{\qcat} of spectra (see \cite[\S9]{Rav:Whatis} and \cite[Definition
1.4.3.1]{Lurie:HA}) are sequential $\Omega $-spectra, so we will
sometimes refer to them as {\bf Lurie spectra}.

If each $X_{n}$ has a $G$-action for some compact Lie group $G$, with
each structure map being {\eqvr} (where $\Sigma X_{n}$ and $\Omega
X_{n}$ are understood to be the ordinary suspension and loop space of
the pointed $G$-space $X_{n}$), the resulting object is called a {\em
sequential spectrum with $G$-action}.  This means that for each closed
subgroup $H\subseteq G$ one gets a spectrum $X^{H}$ in which the $n$th
space is the fixed point set $X_{n}^{H}$.  This notion is different
from that of a {\bf $G$-spectrum}, which we will discuss below.

A {\bf map of sequential spectra (with $G$-action)} $f:X\to Y$ is a
collection of ($G$-{\eqvr}) pointed maps $f_{n}:X_{n}\to Y_{n}$
compatible with the structure maps for $X$ and $Y$.  Such a map is
said to be a ($G$-{\eqvr}) {\bf stable {\equi}} if it induces an
isomorphism of stable homotopy groups (of fixed point sets for each
closed subgroup $H\subseteq G$) as in \cref{eq-stable-pi*-H*}.  This
condition is weaker than requiring each map $f_{n}$ to be a
($G$-{\eqvr}) {\equi}.

\begin{defin}\label{def-seq-spec}
We denote the category of sequential spectra by $\Sp_{\rm seq}$, and
the category of those with $G$-action by $\Sp^{BG}_{\rm seq}$.
\end{defin}

A sequential spectrum has stable homotopy (or simply homotopy) and
reduced homology groups defined by
\begin{numequation}\label{eq-stable-pi*-H*}
\begin{split}
\pi _{i}X
 & : = \colim{n}\pi _{i}\Omega^{n}X_{n}
     = \colim{n}\pi _{n+i}X_{n}      \\
\aand 
H_{i}X
 & := \colim{n}H_{n+i}X_{n} .
\end{split}
\end{numequation}%

\noindent When $X$ has a $G$-action, we can make similar definitions for
each closed subgroup $H\subseteq G$ in terms of fixed point sets, the
definitions above being those for the trivial subgroup.  
The definition does not require $X_{n}$ to be $(n-1)$-connected, so
the groups of \cref{eq-stable-pi*-H*} could be nonzero for $i<0$.

This was our only definition until the late 80's.  It suffered
from the lack of a convenient way to define the smash product of two
spectra.

A sequential spectrum can be reinterpreted as a $\mcT$-valued functor
on a certain $\mcT$-enriched indexing category $\msJ^{\bf N}$ with one
object for each   finite
set $\mathbf{n}=\left\{1,2,\dotsc n \right\} $ for $n\geq 0$ with
$\mathbf{0}$ being the empty set $\varnothing$.  Its morphism spaces are
\begin{numequation}\label{eq-JN}
\begin{split}
\msJ^{\bf N } (\mathbf{m},\mathbf{n} )
  = \mycases{    
S^{n-m} &\mbox{for $0\leq m\leq n$}\\
\pt       &\mbox{for $m>n\geq 0$}
}
\end{split}
\end{numequation}%

\noindent For more on this point of view, see \cite[Chapter
7]{HHR:ESHT}, specifically \cite[Definition 7.2.4]{HHR:ESHT}.  From
this perspective, the difficulty with the smash product is related to
the absence of a symmetric monoidal structure on $\msJ^{\bf N}$.  See
\cite[Remark 7.2.11]{HHR:ESHT} for more on this lack of symmetry.

One solution is to define a new indexing category which {\em is}
symmetric monoidal.  Two such categories are $\msJ^{\Sigma } $ and
$\msJ^{O}$; we will often omit the superscript in the latter.  They
have the same objects as $\msJ^{\bf N }$ with morphism spaces
\begin{numequation}\label{eq-symm-orth}
\begin{split}
\msJ^{\Sigma } (\mathbf{m},\mathbf{n} ) & = \mycases{
\Sigma_{n+}\smashove{\Sigma_{n-m}} S^{n-m}
        &\mbox{for $0\leq m\leq n$}\\
\pt     &\mbox{for $m>n\geq 0$}
}\\
\msJ^{O } (\mathbf{m},\mathbf{n} )
 & = \mycases{    
O (n)_{+}\smashove{O_{n-m}} S^{n-m}  
        &\mbox{for $0\leq m\leq n$}\\
\pt     &\mbox{for $m>n\geq 0$,}
}
\end{split}
\end{numequation}%

\noindent where $\Sigma_{n+}$ and $O (n)_{+}$ denote the symmetric and
orthgonal groups with disjoint basepoint.

\begin{defin}\label{def-symm-orth}
{\bf Symmetric and orthogonal spectra.}  The categories of
$\mcT$-valued functors on the indexing categories of
\cref{eq-symm-orth} are those of {\bf symmetric spectra} of
\cite{HoveySS}, which we denote by $\Sp^{\Sigma }$, and {\bf
orthogonal spectra} of \cite{MandellMay}, which we
denote by $\Sp^{O}$,  respectively.
\end{defin}

\begin{prop}\label{prop-forget}
There are forgetful functors $\Sp^{O} \to \Sp^{\Sigma }\to \Sp_{\rm seq}$
induced by the evident functors $\msJ^{O} \leftarrow \msJ^{\Sigma
}\leftarrow \msJ^{\mathbf{N}} $.
\end{prop}

\subsection{The Mandell-May category}\label{sec-MM-cat} 

Now we define the Mandell-May category $\msJ_{G}$  for a
compact Lie group $G$, which is enriched over $\mcT_{G}$. The case of
trivial $G$ is the category $\msJ^{O}$ of \cref{eq-symm-orth}.

\begin{defin}\label{def-MM-cat}
{\bf The Mandell-May category $\msJ_{G}$} has 
finite dimensional orthogonal {\rep}s $V$ of $G$ as objects. These
are actual {\rep}s, not virtual ones. For two such {\rep}s $V$ and $W$
we have a pointed morphism $G$-space $\msJ_{G} (V,W)$ defined as
follows.  Let $O (V,W)$ be the space of orthogonal embeddings
$t:V\subseteq W$, which need not be {\eqvr}. It is empty when
$|W|<|V|$. Otherwise each such embedding gives us an orthogonal complement $t
(V)^{\perp }\subseteq W$. This enables us to define a vector bundle
over $O (V,W)$ sitting inside the trivial bundle $O (V,W)\times W$.
{\bf The morphism space $\msJ_{G} (V,W)$ is its Thom space, which is a
pointed $G$-space.}  A point in it other than the base point is a pair
$(t,x)$ where $t:V\to W$ is an embedding as above and $x\in t
(V)^{\perp }$.  Composition of morphisms $U\to V\to W$ leads to a map 
\begin{displaymath}
\msJ_{G} (V,W)\wedge \msJ_{G} (U,V)\to \msJ_{G} (U,W)
\end{displaymath}

\noindent induced by composition of orthogonal embeddings.
\end{defin}

The following is immediate.

\begin{prop}\label{prop-JG}
{\bf Properties of $\msJ_{G}$.}
\begin{enumerate}[label={(\roman*)},itemindent=0em]
\item \label{prop-JGi}
 $\msJ_{G} (V,W)$ is a point when $|W|<|V|$,
\item \label{prop-JGii}
 $\msJ_{G} (0,W)=S^{W}$.
\item \label{prop-JGiii}
$\msJ_{G} (V,V)=O (V,V)_{+}$, the orthogonal group $O (V)$ of
$V$ with disjoint base point.
\item \label{prop-JGiv}
The inclusion of $V$ into $V\oplus W$ induces a map 
\begin{displaymath}
S^{V}=\msJ_{G} (0,V)\to \msJ_{G} (W,V\oplus V).
\end{displaymath}
\item \label{prop-JGv} $\msJ_{G}$ has a symmetric monoidal structure
(which is not closed) related to direct sum of {\rep}s.

\end{enumerate}
\end{prop}

The following is part of \cite[Definition V.4.1]{MandellMay}. 

\begin{defin}\label{def-phiGN}
{\bf The relative Mandell-May category.} Let $N\subseteq G$ be a
normal subgroup.  
\begin{enumerate}[label={(\roman*)},itemindent=0em]

\item \label{def-phiGNi} The $\mcT_{G/N}$ enriched category
$\msJ_{G,N}$ has as objects the representations $V$ of $G$, and the
morphism space $\msJ_{G,N} (V,W)$ is the fixed point space $\msJ_{G}
(V,W)^{N}$ with its residual action of $G/N$.  A nonbase point in this
space is a pair $(t,x)$, where $t:V \to W$ is an $N$-{\eqvr}
orthogonal embedding and $x\in (t (V)^{\perp })^{N}$.

\item \label{def-phiGNii} We have a functor 
\begin{displaymath}
\xymatrix
@R=4mm
@C=10mm
{
{\msJ_{G,N}}\ar[r]^(.5){\phi^{G,N}}
  &{\msJ_{G/N}}\\
{V}\ar@{|->}[r]
    &{V^{N}}\\
{(t,x)\in \msJ_{G,N} (V,W)=\msJ_{G} (V,W)^{N}}\ar@{|->}[r]
      &{(t^{N}, x)\in \msJ_{G/N} (V^{N}, W^{N})}
}
\end{displaymath}

\noindent where $t^{N}$ denotes the restriction of $t$ to $V^{N}$, and
since $x\in t (V)^{\perp}$ is fixed by $N$, it lies $W^{N}$ and hence
in the orthogonal complement of $t (V^{N})$ in $W^{N}$.
Thus for each $V$ and $W$, we have a map 
\begin{numequation}\label{eq-phi-GN-VW}
\begin{split}
\phi^{G,N}_{V,W}:\msJ_{G,N} (V,W)\to \msJ_{G/N} (V^{N}, W^{N}).
\end{split}
\end{numequation}%

\item \label{def-phiGNiii} The functor $\nu_{G,N}:\msJ_{G/N}\to
\msJ_{G,N}$ (the right adjoint of $\phi^{G,N}$) sends a {\rep} $W$ of
$G/N$ to the {\rep} of $G$ obtained by precomposition with the
homomorphism $G\to G/N$.

\end{enumerate}

\end{defin}

Note here that $O (V,W)^{N}$ is the space of $N$-{\eqvr} orthogonal
embeddings $V\to W$.  

We will be interested in the case where $G$ is the circle group $\mT$
and $N$ is a finite (and hence cyclic) subgroup.

\begin{remark}\label{rem-non-normal}
{\bf Non-normal subgroups.}  We could replace the pair $(G,N)$ for a
normal subgroup $N\subseteq G$ by $(N (H), W (H))$ (normalizer and
Weyl group) for arbitrary subgroup $H\subseteq G$.  We leave these
details to the reader.  The same goes for the discussion of fixed
point spectra in \cref{def-fixed}.
\end{remark}

\subsection{Defining orthogonal $G$-spectra}\label{sec-G-spectra}
The relative form of the following is \cite[Definition
V.4.2]{MandellMay} with different terminology.

\begin{defin}\label{def-orth-spec}
An {\bf orthogonal $G$-spectrum} is a $\mcT_{G}$-enriched
$\mcT_{G}$-valued functor on the Mandell-May category
$\msJ_{G}$ of \cref{def-MM-cat}.  We denote by $\Sp_{G}$ ($\,\Sp^{G}\!$) the
category of such functors (and {\eqvr} maps).  For such a functor $X$
we denote by $X_{V}$ its value on the object $V$ in $\msJ_{G}$.  For
$V=\reals^{n}$ with trivial $G$-action, we denote this space by
$X_{n}$.

We say that an orthogonal $G$-spectrum $X$ {\bf has trivial
$G$-action} if the same holds for $X_{n}$ for each $n\geq 0$. (In this
case the action on $X_{V}$ is necessarily nontrivial when the action
on $V$ is.)

More generally an {\bf orthogonal $(G,N)$-spectrum} is a
$\mcT_{G/N}$-enriched $\mcT_{G/N}$-valued functor on the relative
Mandell-May category $\msJ_{G,N}$ of
\cref{def-phiGN}\cref{def-phiGNi}.

A morphism $f:X\to Y$ of orthogonal $G$-spectra is a natural
transformation of functors.  This consists of a collection of suitably
compatible maps $f_{V}:X_{V}\to Y_{V}$, which are {\eqvr} if we are in
the category $\Sp^{G}$, but need not be if we are in $\Sp_{G}$.
\end{defin}

Hence an orthogonal $(G,N)$-spectrum $X$ consists of a collection
$\left\{X_{V} \right\}$ of pointed $G/N$-spaces (for each {\rep} $V$
of $G$) along with structure maps
\begin{numequation}\label{eq-structure}
\begin{split}
\epsilon^{X}_{V,W}:\msJ_{G,N} (V,W)\wedge X_{V}\to X_{W}
\end{split}
\end{numequation}%

\noindent with suitable properties that are spelled out in the
references cited above.  In particular each $X_{V}$ comes equipped
with an action of the orthogonal group $O (V)$.  An orthogonal
$(G,e)$-spectrum is an orthogonal $G$-spectrum.  The sphere spectrum
$\mathbb{S}$ is defined by $\mathbb{S}_{V}=S^{V}$.

It follows that the morphism space $\Sp^{G} (X,Y)$ is a certain
subspace of the product (over all $V$) of the spaces $\mcT^{G}
(X_{V},Y_{V})$. It can be described categorically as an enriched
end, and similarly for $\Sp_{G} (X,Y)$.  See \cite[\S9.1.C]{HHR:ESHT}. 

\begin{remark}\label{rem-naive}
{\bf Genuine and naive $G$-spectra.}  It is known (\cite[Lemma
V.1.1]{MandellMay}, reproved as \cite[Lemma 9.1.8]{HHR:ESHT}) that the
space $X_{V}$ depends only on the dimension $|V|$ of $V$, although the
group actions on the spaces $X_{V}$ and $X_{|V|}$ will differ.  For
example if $G$ acts trivially on $X_{|V|}$ but nontrivially on $V$,
then it will act nontrivially on
\begin{displaymath}
X_{V} = O (|V|,V)_{+}\wedge_{O (|V|)}X_{|V|}.
\end{displaymath}

\noindent
For this reason Schwede in \cite{Schwede:ESHT},
which is the reference for orthogonal $G$-spectra used in \cite{NS18},
defines such objects $X$ solely in terms of the spaces $X_{n}$.  He
discusses the relation between the two definitions in \cite[Remark
2.7]{Schwede:ESHT}.  

It follows that $\Sp^{G}$ as we have defined it is {\eqt} to the
category of $\mcT^{G}$-valued functors on $\msJ $, the Mandell-May
category for the trivial group.  Such objects are often called {\em
naive $G$-spectra}, while $\mcT^{G}$-valued functors on $\msJ_{G} $
are called {\bf genuine $G$-spectra}.  We will denote the former
category by $\SpGn$.  While $\Sp^{G}$ and $\SpGn$
are {\eqt} as categories, they come with different homotopical
structures, meaning different classes of stable equivalences.  See
\cite[Theorem 9.3.10 and Example 9.3.11]{HHR:ESHT} for more
discussion.  It turns out that $\SpGn$ has more genuine
{\weq }s than naive ones.
\end{remark}

Naive $G$-spectra are
sometimes called {\bf Borel $G$-spectra}, and the category of such is
sometimes denoted by $\Sp^{hG}$.  Recall that for a $G$-space $X$, the
homotopy fixed point set $X^{hG}$ is the mapping map $\Top^{G}
(EG,X)$.  When $G$ acts trivially on $X$, this is the same as $\Top
(BG, X)$.  Hence if $\qmcC$ is an {\qcat} on which $G$ acts trivially,
it would be reasonable to denote $\qFun (\mcB G, \qmcC)= \qFun
({\color{cyan}{BG}}, \qmcC)$ by $\qmcC^{hG}$.  For $\qmcC=\qSp$, this
the {\qcat} of Lurie spectra with $G$-action.  See \cite[page
3]{Ayala-MG-R} (where naive $G$-spectra are called {\bf homotopy
$G$-spectra}) for more discussion.



The categories $\Sp_{G}$ and $\Sp^{G}$ are both tensored and
cotensored over $\mcT_{G}$ and $\mcT^{G}$ respectively.  This means
that for a pointed $G$-space $K$ and an orthogonal $G$-spectrum $X$,
we can make sense of both the tensor product $X\wedge K$
and the cotensor product $X^{K}$.  These spectra are defined by 
\begin{numequation}\label{eq-tensor-cotensor}
\begin{split}
(X\wedge K)_{V}
 &  := X_{V}\wedge K\\
\aand 
\left(X^{K} \right)_{V}
 & := \mcT_{G} (K,X_{V}) \mbox{ or } \mcT^{G} (K,X_{V}).   
\end{split}
\end{numequation}%

\begin{defin}\label{def-Yoneda-spectrum}
For a {\rep} $V$ of $G$, the {\bf $V$th Yoneda spectrum} $\yo^{V}$ defined by
\begin{displaymath}
(\yo^{V})_{W}=\msJ_{G} (V,W).
\end{displaymath}

\noindent Note that the space on the right is a point when $|W|<|V|$,
and that $\yo^{0}=\mathbb{S}$.  As in \cref{def-orth-spec}, we denote
$\yo^{\reals^{n}}$ by $\yo^{n}$.
\end{defin}

$\yo^{V}$ is a twisted desuspension of the sphere spectrum $\mS$, and
$\Sigma^{V}\yo^{V}\simeq \mS$.

As noted before, the symbol $\yo$ is the Japanese hiragana character
``yo,'' the first syllable of Yoneda's name.  This spectrum is denoted
by $S^{-V}$ in \cite{HHR} and \cite{HHR:ESHT} (where $G$ is assumed to
be finite), while the ``shift desuspension functor'' $(\yo^{V}\wedge
-)$, in which the variable may be either an orthogonal $G$-spectrum or
a pointed $G$-space, is denoted by $F_{V}$ in both \cite[page
11]{MandellMay} and \cite[Definition 1.3]{MMSS}.

Both $\Sp_{G}$ and $\Sp^{G}$ are known to have a closed symmetric monoidal
structure defined in terms of the Day convolution; see
\cite[\S3.3]{HHR:ESHT}.  This means that the morphism space $\Sp_{G}
(V,W)$ is the $0$th space of a morphism or function spectrum $F_{G}
(X,Y)$ defined by
\begin{displaymath}
F_{G} (X,Y)_{V} = \Sp_{G} (\yo^{V}\wedge X,Y).
\end{displaymath}

Again we refer the reader to \cite[\S9.1]{HHR:ESHT} for details.  Even
though it was written with only finite groups $G$ in mind, most
of it holds as well for $G$ a compact Lie group. One feature of the
finite group case {\em not} present in the compact Lie group case is the
existence of a finite dimensional {\rep} (the regular one) containing
every irreducible one as a summand. The group $\mT$ has infinitely
many distinct irreducible {\rep}s.

\begin{defin}\label{def-pi-stable}
{\bf {\Eqvr} stable homotopy groups and stable {\eqiv}s.}  An orthogonal $G$-spectrum has
{\bf stable homotopy groups $\pi^{H}_{\star} (-)$} graded over $RO
(G)$ for each subgroup $H\subseteq G$.  For finite dimensional
orthogonal {\rep}s $V$ and $V'$ of $G$,
\begin{numequation}\label{eq-pi-stable}
\begin{split}
\pi^{H}_{V-V'}X
 :
 = \colim{W}\pi^{H}_{V}\Omega^{W}X_{W+V'},
\end{split}
\end{numequation}%

\noindent where $\pi^{H}_{V}$ and $\Omega^{W}$ on the right are as in
\cref{def-piHX}, and the colimit is over the category
of finite dimensional orthogonal {\rep}s $W$ of $G$ and {\eqvr}
inclusions.

We also have {\bf geometric homotopy groups} 
\begin{numequation}\label{eq-pi-stable=Phi}
\begin{split}
\pi^{\Phi H}_{V-V'}X
 :
 = \colim{W}\pi_{V}\Omega^{W^{H}}X_{W+V'}^{H}.
\end{split}
\end{numequation}%

An {\eqvr} map $f:X\to Y$ is a {\bf stable {\eqiv}} if it induces
isomorphisms in $\pi^{H}_{\star} (-)$ for all closed subgroups
$H\subseteq G$.
\end{defin}

For finite $G$, let $\varrho_{G}$ denote the {\bf regular {\rep} of $G$},
meaning its group ring over $\reals$ (a vector space of dimension
$|G|$) on which $G$ acts by left multiplication.  Then we know that every
finite dimensional {\rep} of $G$ is a summand of some finite multiple of
$\varrho_{G}$.  It follows that the colimits of \cref{eq-pi-stable} and
\cref{eq-pi-stable=Phi} can be simplified to the sequential colimits
\begin{align*}
\pi^{H}_{V-V'}X
 & = \colim{n}\pi^{H}_{V}\Omega^{n\varrho_{G}}X_{n\varrho_{G}+V'}\\
\aand 
\pi^{\Phi H}_{V-V'}X
 & = \colim{n}\pi_{V}\Omega^{n\varrho_{G}^{H}}X_{n\varrho_{G}+V'}^{H}.
\end{align*}

\noindent Such a simplification is {\em not} available for a compact
but infinite Lie group such as $\mT$.

A $G$-spectrum has two flavors of fixed points for each normal
subgroup $N\subseteq G$.

\begin{defin}\label{def-fixed}
{\bf Fixed point spectra.} Let $X$ be an orthogonal $G$-spectrum $X$
and let  $N\subseteq G$ be a normal subgroup.
\begin{enumerate}[label={(\roman*)},itemindent=0em]
\item \label{def-fixedi}
The $G/N$-spectrum $X^{N}$, its {\bf categorical fixed point spectrum}, is
defined by
\begin{displaymath}
(X^{N})_{W} = (X_{q (W)})^{N}
\end{displaymath}

\noindent for each {\rep} $W$ of $G/N$, where $q=q_{G/N}:\msJ_{G/N}\to
\msJ_{G,N}$ is the right adjoint of \cref{def-phiGN}\cref{def-phiGNiii}.

\item \label{def-fixedii}
The $(G,N)$-spectrum $\fix^{N}X$ is defined by 
\begin{displaymath}
(\fix^{N}X)_{V} = (X_{V})^{N}
\end{displaymath}

\noindent for each {\rep} $V$ of $G$. The $G/N$-spectrum $\Phi^{N}X$,
its {\bf geometric fixed point spectrum}, is the left Kan extension of
$\fix^{N}X$ along the left adjoint functor
\begin{displaymath}
\phi_{G,N}:\msJ_{G,N}\to \msJ_{G/N}
\end{displaymath}

\noindent of
\cref{def-phiGN}.
\end{enumerate}
\end{defin}

\subsection{The Loday functor}\label{sec-Loday-functor}
 
In this subsection we use the symbol $\uot $ for the categorical
tensor product, reserving $\otimes $ for its usual meaning as
algebraic tensor product or smash product.

A cocomplete category $\mcC $ is tensored over sets.  For an object
$W$ in $\mcC $ and a finite set $F$ with cardinality $f$, $W\uot F$ is
the $f$-fold coproduct of $W$. For a simplicial set $X$ of finite
type, $W\uot X$ is a simplicial object in $\mcC $ with $(W\uot
X)_{n}=W\uot X_{n}$, {\em provided that we can define the relevant
maps among the various finite coproducts of $W$}.  When in addition
$\mcC $ is tensored over topological spaces, we get a geometric
realization $|W\uot X|$ in $\mcC $.  Its homotopy type (assuming this
makes sense in $\mcC$) for a given $W$ depends only on the homotopy
type of $|X|$, meaning that weakly {\eqt} simplicial sets give weakly
{\eqt} objects in this construction.

Of particular interest is the case where $X$ is the simplicial circle
or standard 0-cyclex $\blamb^{0}=\bdelt^{1}/\partial \bdelt^{1}$ of
\cref{cor-circle}.  Recall that its $n$th component has $n+1$ elements
that are permuted by the cyclic group $\rC_{n+1}$.

In the category of commutative $k$-algebras for a discrete commutative
ring $k$, the coproduct is tensor product over $k$, denoted as usual
by $\otimes_{k}$.  The algebra multiplication gives us the maps we
need between finite coproducts to give us the desired simplicial
structure.  Thus for a commutative algebra $A$ and a simplicial set
$X$ we get a simplicial $k$-algebra which we denote by $A\uot_{k} X$,
which is also a simplicial abelian group, omitting the subscript when
$k=\Z$.  We denote the associated chain complex as in \cref{eq-part-n}
by $\Ch (A\uot_{k} X)$. For example the chain complex $\Ch (A\uot_{k}
\bdelt ^{0})$ is
\begin{displaymath}
\xymatrix
@R=0mm
@C=5mm
{
{0}
  &{1}
    &{2}
      &{3}
        &{4}\\
{A}
  &{A}\ar[l]_(.5){0}
    &{A}\ar[l]_(.5){1}
      &{A}\ar[l]_(.5){0}
        &{A}\ar[l]_(.5){1}
          &{\dotsb, }\ar[l]^(.5){}
}
\end{displaymath}

\noindent so its homology is $A$ concentrated in degree 0.  More
interestingly, 
\begin{displaymath}
A\uot_{k}  \blamb^{0}=\mathbf{HH}_{\bullet} (A),
\end{displaymath}

\noindent the simplicial abelian group of \cref{eq-simp-chain}.

In the category of commutative ring spectra which are algebras over a
fixed commutative ring spectrum $R$, the coproduct is the smash
product over $R$, denoted by $\otimes_{R}$.  When $A$ is such an
algebra, we denote its categorical tensor product with a simplicial
set $X$ by $A\uot_{R} X$.  When $R=\mS$, we omit it.

The following notation is due to Brun, Carlsson and Dundas \cite{BCD},
who study the Loday functor in detail.

\begin{defin}\label{def-LXR}
For a commutative ring spectrum $A$ and a simplicial set $X$, the {\bf
Loday functor} $\Lsmash_{X}A$ is the geometric realization of the
simplicial ring spectrum $A\uot X$.  More generally when $A$ is a
commutative algebra over a commutative ring spectrum $R$, the {\bf
Loday functor relative to $R$} is
\begin{displaymath}
\Lsmash_{X}A/R:= |A\uot _{R}X |.
\end{displaymath}
\end{defin}

Our $\Lsmash_{X}A$ is denoted by $\Lambda_{X}A$ and called ``smash $X$
of $A$'' in \cite{BCD}, and by $\mcL_{X} (A)$ in \cite{HHLRZ}.  In the
latter it is generalized (in a different direction from the above) to
$\Lsmash_{X,Y} (A,B;C)$, for $Y\subseteq X$ and commutative rings
$A\to B\to C$.  This relative version places $C$ over the basepoint of
$X$, $B$ over all points of $Y$ that are not the basepoint and $A$
over the complement of $Y$ in $X$.

\begin{remark}\label{rem-assoc-R-algebras}
{\bf The category of associative $R$-algebras.} When $A$ is an
associative (but not commutative) $R$-algebra and $X=\blamb^{0}$, we
still have a simplicial ring spectrum in which the $n$th component is
\begin{displaymath}
A\otimes_{R} A\otimes_{R}A\otimes_{R}\dotsb \otimes_{R}A
\qquad \mbox{with $n+1$ factors} 
\end{displaymath}

\noindent whose geometric realization is $\THH(A/R)$.  However the
$(n+1)$-fold coproduct in the category of {\em associative}
$R$-algebras is not the $(n+1)$-fold relative smash product, but
something far more complicated.  This is analogous to the fact that
while the coproduct in the category of abelian groups is the direct
product, that in the category of arbitrary groups is the far more
complicated free product.  Hence we cannot identify $\THH(A/R)$ with
$|A\uot_{R} \blamb^{0}|$.
\end{remark}

Note that $\Lsmash_{X}A$ is covariantly functorial in both $A$ and
$X$.  Jim McClure, Roland Schw\"anzel and RainerVogt \cite[Theorem
B]{MSV97} show that when $A$ is an $\EE_{\infty }$-ring spectrum and
$X=\blamb^{0} $ (which they denote simply by $S^{1}$),
$\Lsmash_{X} A=\THH (A)$.  In \cite[Theorem G]{MSV97} they show that
any map $A\to B$ to an $\EE_{\infty }$-ring spectrum with $\mT$-action
extends uniquely to $\THH(A)$.

Angeltveit and Rognes \cite[\S3]{AR05} show
that it has a Hopf algebra structure. Brun, Fiedorowicz and Vogt
\cite{BFV} show that for an $\EE_{n}$-ring spectrum $A$ with
$n\geq 2$, $\THH (A)$ can be defined in the same way and is an
$\EE_{n-1}$-ring spectrum.

In \cite{BCS}, Andrew Blumberg, Cohen and Schlichtkrull determine the
$\THH$ of certain Thom spectra.  Let $BF$ be the classifying space for
stable spherical fibrations and suppose we have a map $f:X\to BF$
which is the loop map associated to a map of connected based spaces
$Bf: BX \to B^{2}F $.  Then the resulting Thom spectrum $T (f)$ is an
$\EE_{1}$-ring.  They identify $\THH (T (f))$ as a certain Thom
spectrum associated with the free loop space $\mcL (BX)$.  When $f$ is
a triple loop map, \cite[Theorem 3]{BCS} says
\begin{displaymath}
\THH (T (f)) \simeq T (f)\wedge BX_{+}.
\end{displaymath}

\noindent  For example,
\begin{displaymath}
\THH (MU) \simeq MU\wedge SU_{+}.
\end{displaymath}

In \cite{Hahn-Wilson} the authors study the relative version
$\THH(\BPn/MU)$.  $MU$ is a commutative ring spectrum, but the
$\EE_{3}$-ring spectrum $\BPn$ is known not to admit an
$\EE_{\infty}$-structure by work of Lawson \cite{Lawson18} for $p=2$
and Andrew Senger \cite{senger24} for odd primes.

\subsection[The Greenlees-May or Tate diagram]
{The Greenlees-May or Tate diagram} \label{sec-Tate-diagram}

The standard reference for the material in this subsection is John
Greenlees and Peter May's classic book \cite{GreeMay-Tate}.  As noted
above and illustrated in \cref{ex-S1}, each compact Lie group $G$ has
a contractible free $G$-space $EG$. Let $EG_{+}$ denote $EG$ with a
disjoint base point, and consider the base point preserving map
$j:EG_{+}\to S^{0}$ that sends all of $EG$ to the nonbase point of
$S^{0}$.  The resulting cofiber sequence
\begin{numequation}\label{eq-useful}
\begin{split}
\xymatrix
@R=4mm
@C=10mm
{
{EG_{+}}\ar[r]^(.5){j}
  &{S^{0}}\ar[r]^(.5){k}
    &{\widetilde{E}G},
}
\end{split}
\end{numequation}%

\noindent is useful to us.  Since $EG$ is contractible, the first map
is underlain by a homotopy {\equi}, making the underlying space of
$\widetilde{E}G$ contractible. On the other hand for any nontrivial
subgroup $H\subseteq G$, the fixed point set $(EG_{_{+}})^{H}$ is a
point since $H$ acts freely away from the basepoint.  It follows that
\begin{displaymath}
(EG_{+})^{H}
\simeq \mycases{
S^{0}   &\mbox{for }H=e\\
\pt       &\mbox{otherwise}
}
\qquad \mbox{while}\qquad  
\widetilde{E}G^{H}
\simeq \mycases{
\pt       &\mbox{for }H=e\\
S^{0}  &\mbox{otherwise.}
}
\end{displaymath}

The cofiber sequence of \cref{eq-useful} is not to be confused with
the {\bf isotropy separation sequence} (see \cite[\S9.11A]{HHR:ESHT})
\begin{displaymath}
\begin{split}
\xymatrix
@R=4mm
@C=10mm
{
{E\mcP_{+}}\ar[r]^(.5){}
  &{S^{0}}\ar[r]^(.5){}
    &{\widetilde{E}\mcP },
}
\end{split}
\end{displaymath}

\noindent where $\mcP $ denotes the family of proper subgroups of $G$,
and $E\mcP$ is a $G$-space for which $E\mcP^{G}$ is empty but
$E\mcP^{H}$ is weakly contractible for each proper subgroup
$H\subseteq G$.  The two cofiber sequences coincide when $G=\rC _{p}$.

For an orthogonal $G$-spectrum $X$, cotensoring (as in
\cref{eq-tensor-cotensor}) with the first map of \cref{eq-useful}
gives a map
\begin{numequation}\label{eq-cotensored-map}
\begin{split}
\xymatrix
@R=4mm
@C=20mm
{
{X=X^{S^{0}}}\ar[r]^(.55){\epsilon_{X} :=X^{j}}
  &{X^{EG_{+}}}
}
\end{split}
\end{numequation}%

\noindent known as {\bf Borel completion}.

We now smash \cref{eq-useful} with \cref{eq-cotensored-map} and get
the {\bf Greenlees-May diagram} of \cite[Diagram (D)]{GreeMay-Tate},
\begin{numequation}\label{eq-Tate-diagram}
\begin{split}
\xymatrix
@R=2mm
@C=20mm
{
{X\wedge EG_{+}}\ar[r]^(.6){X\wedge j}
                 \ar[dd]_(.45){\epsilon_{X} \wedge EG_{+}}
   &{X}\ar[r]^(.4){X\wedge k}
       \ar[dd]_(.45){\epsilon_{X} }
     &{X\wedge \widetilde{E}G}
                 \ar[dd]^(.45){\epsilon_{X} \wedge \widetilde{E}G}
{}\\ \\
{X^{EG_{+}}\wedge EG_{+}}\ar[r]^(.6){X^{EG_{+}}\wedge j}
   &{X^{EG_{+}}}\ar[r]^(.4){X^{EG_{+}}\wedge k}
     &{X^{EG_{+}}\wedge \widetilde{E}G}\\
{f_{G} (X)}\ar@{:}@{=}[u]^(.5){}
   &{c_{G} (X)}\ar@{=}[u]^(.5){}
     &{t_{G} (X)}\ar@{=}[u]^(.5){} 
}
\end{split}
\end{numequation}%

\noindent in which each row is a cofiber sequence and the spectra in
the bottom row are defined by the indicated equalities. Here the
letters $f$, $c$ and $t$ stand for ``free,'' ``complete,'' and
``Tate.''  The left vertical map is an {\equi} by \cite[Proposition
1.2]{GreeMay-Tate}.  This means that {\em the square on the right is a
homotopy pullback diagram}.

Greenlees and May \cite[Introduction]{GreeMay-Tate} call $t_{G}(X)$
(which they denote by $t(X)$ without the subscript) the {\bf Tate
$G$-spectrum associated to $X$}, and its fixed point set
\begin{numequation}\label{eq-tG}
\begin{split}
X^{tG}:=t_{G}(X)^{G}
\end{split}
\end{numequation}%

\noindent {\bf the Tate construction of $X$}. It is surprisingly
interesting even when the action of $G$ on $X$ is trivial as in
\cref{def-orth-spec}.

When $X$ is a naive $G$-spectrum, \cite[Proposition 2.1]{GreeMay-Tate}
says that the fixed point sets of $f_{G}(X)$ and $c_{G} (X)$ can by be
identified with $(\Sigma^{\Ad (G)}X)_{hG}$, the homotopy orbit
spectrum of the twisted suspension (\cref{def-piHX}) of $X_{hG}$ of
\cref{eq-hos}, and $X^{hG}$, the homotopy fixed point set of
\cref{eq-hfps}.  Here $\Ad(G)$ denotes the {\bf adjoint {\rep} of
$G$}, whose degree is the dimension of $G$.  In particular it is
trivial when $G$ is finite.  It is defined in terms of the action of
$G$ on its Lie algebra, and that action is trivial when $G$ is
abelian.  Thus for the circle group $\mT$ acting on a spectrum $X$, we
have
\begin{numequation}\label{eq-AdT}
\begin{split}
(f_{\mT}X)^{\mT} 
 \simeq (\Sigma^{\Ad (\mT)}X)_{h\mT} 
 \simeq \Sigma X_{h\mT}.
\end{split}
\end{numequation}%

The map 
\begin{numequation}\label{eq-NormG}
\begin{split}
\Nm_{G}:(\Sigma^{\Ad (G)}X)_{hG}\simeq f_{G} (X)^{G}
    \to  c_{G} (X)^{G}\simeq X^{hG}
\end{split}
\end{numequation}%

\noindent is known as {\bf the norm} and has the $X^{tG}$ as its
cofiber. For finite $G$ it can be described as follows.

For a module $M$ over the integral group ring of $G$ there is a map 
\begin{displaymath}
\Nm_{G}:M_{G}\to M^{G},
\end{displaymath}

\noindent originally due to John Tate, which we describe in
\cref{sec-alg-norm}.  There is no comparable map $X_{G}\to X^{G}$ for
a $G$-space $X$, but there is one for a $G$-spectrum $X$ that we
describe in \cref{sec-ESHT-norm}.


\subsubsection{The norm map in algebra}\label{sec-alg-norm}
The original reference for this paragraph is \cite[Chapter XII]{CE},
whose authors attributed the ideas to Tate.  Let $\ZG$ denote the
integer group ring of a finite group $G$.  For each $\gamma \in G$, we
denote by $[\gamma ]$ the corresponding basis element of $\ZG$.  We
have the augmentation map
\begin{numequation}\label{eq-aug}
\begin{split}
\nabla  :\ZG \to \Z
\quad \mbox{given by}\quad  
[\gamma ]\mapsto 1
\quad \mbox{for each }\gamma \in G,
\end{split}
\end{numequation}%

\noindent
and we denote its kernel, the
augmentation ideal, by $I$.  A $\ZG$-module $M$ is an abelian group
with an action of $G$.  As such it has an orbit module
\begin{displaymath}
M_{G} = M/IM = M\otimes_{\ZG} \Z
\end{displaymath}

\noindent and a fixed point set
\begin{displaymath}
M^{G} = \left\{m\in M:[\gamma ]m=m \,\, \forall \gamma \in G \right\}
      = \Hom_{\ZG} (\Z,M)
      \subseteq M.
\end{displaymath}

\noindent  Now consider the {\bf Tate norm map} $\Nm_{G}:M \to M$ given by 
\begin{displaymath}
m\mapsto \sum_{\gamma \in G}[\gamma ]m.
\end{displaymath}

\noindent It sends each element $m\in M$ to the sum of the elements in
its $G$-orbit.  Hence two elements in an orbit have the same image
under it, so it factors through the quotient $M_{G}$.  Its image is
fixed by $G$ and thus contained in $M^{G}$, and we have 
\begin{displaymath}
\Nm_{G}: M_{G} \to M^{G}.
\end{displaymath}

\noindent When $M=\ZG$, it is an isomorphism between $\ZG/I$ and the
subgroup of $\ZG$ generated by the norm of the unit element.  Hence
its kernel and cokernel both vanish.  When $M=\Z$ with trivial
$G$-action, this map is multiplication by $|G|$, the order of $G$.  We
will describe its kernel and cokernel in \cref{eq-4term}.



Recall that
\begin{align*}
H^{i} (G;M)
 & := \Ext_{\ZG}^{i}\left(\Z,M \right),  \\
H_{i} (G;M)
 & := \Tor^{\ZG}_{i}\left(\Z,M \right),
\end{align*}

\noindent and both can be computed using a free (or projective)
$\ZG$-resolution of $\Z$ of the form
\begin{numequation}\label{eq-free-res}
\begin{split}
\xymatrix
@R=3mm
@C=8mm
{
{\dotsb }  \ar[r]^(.5){}
  &{P_{2}}\ar[r]^(.5){}
    &{P_{1}}\ar[r]^(.5){}
      &{P_{0}=\ZG}\ar[r]^(.65){\nabla }
        &{\Z}\ar[r]^(.5){}
          &{0,}
}
\end{split}
\end{numequation}%

\noindent where $\nabla$ is as in \cref{eq-aug}.

Then $H^{*} (G;M)$ is the cohomology of
the cochain complex
\begin{displaymath}
\dotsb \leftarrow \Hom_{\ZG} (P_{2},M) 
       \leftarrow \Hom_{\ZG} (P_{1},M) 
       \leftarrow \Hom_{\ZG} (P_{0},M)=M
\end{displaymath}

\noindent and  $H_{*} (G;M)$ is the homology of the chain complex
\begin{displaymath}
\dotsb \rightarrow  P_{2}\otimes_{\ZG}M 
       \rightarrow  P_{1}\otimes_{\ZG}M 
\rightarrow  P_{0}\otimes_{\ZG}M=M.
\end{displaymath}

\noindent The augmentation $\nabla  :\ZG \to \Z$ of
\cref{eq-aug} induces maps $H_{0} (G;M)\to M_{G}$ and $M^{G}\to
H^{0} (G;M)$. These lead to a 4-term exact sequence
\begin{numequation}\label{eq-4term}
\begin{split}
0\to H_{0} (G;M)
  \to M_{G}
  \xrightarrow{\Nm_{G}}M^{G}
  \to H^{0} (G;M)
  \to 0.
\end{split}
\end{numequation}%

\noindent {\bf Tate cohomology }
$\widehat{H}^{*}$ is defined by 
\begin{numequation}\label{eq-Tate-H*}
\begin{split}
\widehat{H}^{i} (G;M)=\mycases{
H^{i} (G;M)
       &\mbox{for }i>0\\
H_{-1-i} (G;M)
       &\mbox{for }i\leq 0,
}
\end{split}
\end{numequation}%

\noindent so we can write the  4-term exact sequence as 
\begin{displaymath}
0\to \widehat{H}^{-1} (G;M)
  \to M_{G}
  \xrightarrow{\Nm_{G}}M^{G}
  \to \widehat{H}^{0} (G;M)
  \to 0.
\end{displaymath}

When $M$ is a ring, there is a multiplication in $\widehat{H}^{*}
(G;M)$ that makes it a module over $H^{*} (G;M)$, its nonnegatively
graded part.

\begin{ex}\label{ex-CpZp}
{\bf The case $G=\rC_{p}$ and  $M=\Z/p$.} We find that 
\begin{displaymath}
\widehat{H}^{i} (\rC_{p}; \Z/p) = \Z/p
\qquad \mbox{for all $\Z$.} 
\end{displaymath}

\noindent Given generators $a\in \widehat{H}^{1}$ and $x\in
\widehat{H}^{2}$, $\widehat{H}^{2i+\epsilon }$ is generated by
$a^{\epsilon }x^{i}$ for $\epsilon =0,1$ and all $i\in \Z$.  As graded
rings we have (for $p>2$)
\begin{align*}
H^{*} (\rC_{p}; \Z/p)
 & \cong    \Lambda (a)\otimes \Z/p[x] \\
\aand 
\widehat{H}^{*} (\rC_{p}; \Z/p)
 & \cong   \Lambda (a)\otimes \Z/p[x, x^{-1}],
\end{align*}

\noindent and for $p=2$ there is a similar description with $x=a^{2}$.

The case $p=2$ is discussed further in \cref{sec-stunted}.

\end{ex}

\subsubsection{The norm map in {\eqvr} stable homotopy theory.}
\label{sec-ESHT-norm} For a finite group $G$ acting on a space $X$
(pointed or not) we do {\em not} get a similar map $X_{G}\to X^{G}$
because in general there is no way to sum the points in a $G$-orbit of
$X$. However we can do this in the stable world as follows. For a naive
$G$-spectrum $X$ there are $G$-maps
\begin{numequation}\label{eq-stable-norm}
\begin{split}
\xymatrix
@R=-1mm
@C=4mm
{
{EG_{+}\wedge_{G}X}\ar@{=}[r]^(.5){}
  &{X_{hG}}\ar@{-->}[rrr]^(.5){\Nm_{G}}
           \ar[dddd]_(.5){j\wedge_{G}X}
    & &&{X^{hG}}\ar@{=}[r]^(.5){}
        &{F (EG_{+},X)^{G}} \\
{}\\
{}\\
{}\\
{S^{0}\wedge_{G}X}\ar@{=}[r]^(.5){}
  &{X_{G}}\ar@{-->}[rrr]^(.5){}
    & &&{X^{G}}\ar@{=}[r]^(.5){}\ar[dddd]^(.5){}
              \ar[uuuu]_(.5){\epsilon_{X}^{G} }
        &{F (S^{0},X)^{G}}\\
{}\\
{}\\
{}\\
{x}\ar@{|->}[ddddd]^(.5){}
  &{X}\ar[dddd]_(.4){\rm diag}\ar@{-->}[rrruuuu]^(.5){}\ar[uuuu]^(.5){}
    & &&{X}
        &{\displaystyle{\sum_{\gamma \in G}\gamma (x)}}  \\
{}\\
{}\\
{}\\
  &{\displaystyle{\prod_{\gamma \in G}X}}\ar@<.5ex>@{-->}[rrr]_(.5){}
    & &&{\displaystyle{\bigvee_{\gamma \in G}X}}\ar[uuuu]_(.5){\rm fold}
                            \ar@<.5ex>[lll]^(.5){\simeq }\\
{(\gamma (x):\gamma \in G)}\ar@{|->}[rrrrr]^(.5){}
  & & && &{?} \ar@{|->}[uuuuu]^(.5){}    
}
\end{split}
\end{numequation}%


\noindent where $G$ acts as usual on $X$, and on the product and
coproduct by permuting factors and summands.  The diagonal map is
twisted as indicated, and the fold map is the usual one. The map from
the $G$-indexed coproduct to the $G$-indexed product is functorial,
and it is an equivalence because we are in the stable category, and
therefore it has an (admittedly mysterious) inverse whose composite
with the fold map is the indicated addition up to homotopy.  The sum
in the lower right is fixed by $G$, so the composite lifts to the
fixed point spectrum $X^{G}$ as indicated.  This lifting factors
through the orbit spectrum $X_{G}$ since $G$ acts trivially on
$X^{G}$.  The resulting composite map $X_{hG}\to X^{hG}$ is the norm.


This norm is {\em not} that of \cite{HHR}, which is a functor that
converts an $H$-spectrum to a $G$-spectrum for $G$ finite and
$H\subseteq G$.  A generalization of it to the case $e\subseteq
\mT$ is used in \cite{ABGHLM} to study the cyclotomic structure on
$\THH (R)$.

As noted in \cite[Definition I.1.13]{NS18}, \cref{eq-stable-norm}
makes sense in any stable {\qcat} in which limits and colimits indexed
by $G$ are defined.  The stability condition guarantees the
equivalence of the product and coproduct there.

Thus the fixed point diagram for \cref{eq-Tate-diagram}, the {\bf
Greenlees-May fixed point diagram}, is
\begin{numequation}\label{eq-Tate2}
\begin{split}
\xymatrix
@R=8mm
@C=15mm
{
{X_{hG}}\ar[r]^(.5){} \ar@{=}[d]_(.45){}
   &{X^{G}}\ar[r]^(.5){} \ar[d]_(.45){\epsilon_{X}^{G}}\ar@{}[dr]|-(.15){\pb}
     &{(X\wedge \widetilde{E}G)^{G}}\ar[d]^(.45){} \\
{X_{hG}}\ar[r]^(.5){\Nm_{G}^{X}}
   &{X^{hG}}\ar[r]^(.5){s_{X}^{G}}
     &{X^{tG}.}
}
\end{split}
\end{numequation}%

\noindent {\em We remind the reader that $G$ is assumed to be finite. }
We will omit the index $X$ in the three maps when it is
clear from the context.

\begin{defin}\label{def-can}
When $G$ acts trivially on $X$ (making $X^{G}=X$), the {\bf Tate map}
\begin{displaymath}
\theta^{G}_{X}:X \to X^{tG}
\end{displaymath}

\noindent is the composite $s^{G}_{X}\epsilon_{X}^{G}$ in
\cref{eq-Tate2}.  We will often drop the subscript $X$.
\end{defin}


\begin{defin}\label{def-TCTP}
{\bf The functors $\TopC^{-}$ and $\TP$.}
For a $\mT$-spectrum $X$,
\begin{displaymath}
\TopC^{-}(X)
 :=  X^{h\mT}
\qquad \aand 
\TP(X) :=  X^{t\mT}.
\end{displaymath}

\noindent 
For a $\rC_{p^{\infty }}$-spectrum $X$,
\begin{displaymath}
\TopC^{-}(X,p)
 :=  X^{h\rC_{p^{\infty }}}
\qquad \aand 
\TP(X,p) :=  X^{t\rC_{p^{\infty }}}.
\end{displaymath}
\end{defin}

For a cyclotomic spectrum $X$ (see \cref{def-cyclo-spec}), we have a 
diagram of cofiber sequences
\begin{numequation}\label{eq-Klein-norm}
\begin{split}
\xymatrix
@R=4mm
@C=15mm
{
{}
   &{\TopC^{-} (X)}\ar@{=}[d]^(.5){}
     &{\TP ( X)}\ar@{=}[d]^(.5){}\\
{\Sigma X_{h\mT}}\ar[r]^(.5){\Nm_{\mT}}\ar[d]^(.5){}
   &{X^{h\mT}}\ar[r]^(.5){s_{X}^{\mT}}\ar[d]^(.5){}
     &{X^{t\mT}}\ar[d]^(.5){}\\
{X_{h\rC_{p^{j}}}}\ar[r]^(.5){\Nm_{\rC_{p^{j}}}}
   &{X^{h\rC_{p^{j}}}}\ar[r]^(.5){s_{X}^{\rC_{p^{j}}}}
     &{X^{t\rC_{p^{j}}},}
}
\end{split}
\end{numequation}%

\noindent where the equalities are those of \cref{def-TCTP}, and the
norm maps are those of \cref{eq-NormG} and \cref{eq-Tate2}.  The one
for $\mT$ was constructed using different methods by John Klein in
\cite[\S3]{Klein-dual} and is defined for any $\mT$-spectrum $X$; the
cyclotomic structure is not needed for it.  It is known \cite[Lemma
2.18]{Bokstedt-Madsen94} that after $p$-adic completion $X^{h\mT}$ is
the homotopy limit of the $X^{h\rC_{p^{j}}}$ when $X$ has finite type.

The left vertical map in \cref{eq-Klein-norm} is constructed as follows.  Let 
\begin{displaymath}
H:=\rC_{p^{j}}\subseteq \mT =:G.
\end{displaymath}

\noindent The spaces in \cref{eq-Tate-diagram} for $G=\mT$ are defined
in terms of a contractible from $G$-space $EG$, which is also a
contractible free $H$-space.  Hence \cref{eq-Tate-diagram} is also a
Greenlees-May diagram for $H$.  Consider the inclusion map 
\begin{displaymath}
f_{G} (X)^{G} \to f_{G} (X)^{H}\simeq f_{H} (X)^{H}.
\end{displaymath}

\noindent These spectra are identified by \cite[Proposition
2.1]{GreeMay-Tate} as
\begin{displaymath}
\Sigma X_{h\mT}= (\Sigma^{\Ad (G)}X)_{hG} \to 
(\Sigma^{\Ad (H)}X)_{hH} = X_{h\orC_{p^{j}}}.
\end{displaymath}

\subsection{The case $G=\rC _{r}$}

When $G=\rC _{r}$ we denote the map $s^{G}_{X}$ of \cref{eq-Tate2} by
$s^{r}_{X}$, and we have
\begin{displaymath}
(X\wedge \widetilde{E}\rC _{r})^{\rC _{r}}=\Phi^{\rC _{r}}X,
\end{displaymath}

\noindent the geometric fixed point spectrum of
\cref{def-fixed}\cref{def-fixedii}.

The right square in \cref{eq-Tate2} is the pullback diagram
\begin{numequation}\label{eq-TateCp}
\begin{split}
\xymatrix
@R=10mm
@C=15mm
{
{X^{\rC _{r}}}\ar[r]^(.5){} \ar[d]_(.45){\epsilon_{X}^{\rC _{r}}}
              \ar@{}[dr]|-(.15){\pb}
     &{\Phi^{\rC _{r}}X}\ar[d]^(.45){} \\
{X^{h\rC _{r}}}\ar[r]^(.5){s_{X}^{r}}
     &{X^{t\rC _{r}}.}
}
\end{split}
\end{numequation}%

Blumberg and Mandell (see \cref{def-cyclo-spec}) define a cyclotomic spectrum
to be a $\mT$-spectrum equipped with an equivalence
\begin{numequation}\label{eq-BM-def}
\begin{split}
\Phi_{p}:\rho_{p}^{*}\Phi^{\rC _{p}}X\to X
\end{split}
\end{numequation}%

\noindent for each prime $p$, where $\Phi^{\rC _{p}}X$ is the geometric
fixed point spectrum of \cref{def-fixed}\cref{def-fixedii}.  There is
a residual action of $\omT:=\mT/\rC _{p}\cong \mT$ on $\Phi^{\rC
_{p}}X$, and $\rho_{p}^{*}$ is as in \cref{def-rho-n}.  The map is
required to be $\mT$-{\eqvr}.  

In \cref{def-cyclo-spec2} Nikolaus and Scholze define it in terms of a
$\mT$-{\eqvr} map
\begin{numequation}\label{eq-NS-def}
\begin{split}
\varphi_{p}:X\to \rho_{p}^{*}X^{t\rC _{p}}. 
\end{split}
\end{numequation}%

Nikolaus and Scholze show that the two definitions agree when $X$ is
bounded below.  Both \cref{eq-BM-def} and \cref{eq-NS-def} should be
compared with \cref{def-smooth-cyclo}.  We will discuss this further
in \cref{sec-NS}.



\bigskip

\subsubsection{The group $\rC_{2}$ and stunted projective spaces}
\label{sec-stunted}
When $X$ has trivial group action and $p=2$, we have
\begin{numequation}\label{eq-RP-infinity}
\begin{split}
X^{t\rC _{2}}\simeq \holim{i\to\infty } (\RP_{-i}\wedge \Sigma X),
\end{split}
\end{numequation}%

\noindent where $\RP_{-i}$ denotes the Thom spectrum for the
$(-i)$-fold Whitney sum of the canonical real real bundle over $\reals
P^{\infty }$.  There is an analogous statement for the group
$\rC _{p}$ for an odd prime $p$.  These are proved as \cite[Theorem
16.1]{GreeMay-Tate}. 

To explain \cref{eq-RP-infinity}, we need the following notation.  For
$0\leq i<j$, let
\begin{displaymath}
\RP_{i}^{j}:=\Sigma^{\infty }\RP^{j}/\RP^{i-1},
\end{displaymath}

\noindent the suspension spectrum of the {\bf stunted real projective space}, which is known to have the following properties.
\begin{itemize}
\item [$\bullet$] It has a CW-structure with a single cell in each
dimension ranging from $i$ to $j$.

\item [$\bullet$] For $i<j<k$ there is a cofiber sequence
\begin{numequation}\label{eq-cofiber-seq}
\begin{split}
\xymatrix
@R=4mm
@C=10mm
{
{\RP_{i}^{j}}\ar[r]^(.5){}
  &{\RP_{i}^{k}}\ar[r]^(.5){}
    &{\RP_{j+1}^{k}}
}
\end{split}
\end{numequation}%

\item [$\bullet$] $\RP_{i}^{j}$ is the suspension spectrum of the Thom
space for the $i$-fold Whitney sum of the canonical line bundle over
$\RP^{j-i}$.

\item [$\bullet$] There is an integer $\ell$ depending only on $j-i$
such that
\begin{displaymath}
\RP_{i+2^{\ell}}^{j+2^{\ell}}\simeq \Sigma^{2^{\ell}}\RP_{i}^{j}
\end{displaymath}

\noindent This was first proved by James in \cite{James3} and is known
as {\bf James periodicity}.  (The value of $\ell $ is known and is
roughly $(j-i)/2$, but we do not need it here.)  It means that for
each finite $k>0$, the homotopy type of $\Sigma^{-i}\RP_{i}^{i+k}$
varies periodically with $i$.
\end{itemize}
  
We can use James periodicity to make sense of $\RP_{i}^{j}$ for {\em
any integers} $i$ and $j$ with $i\leq j$, and define
\begin{numequation}\label{eq-RP-inf}
\begin{split}
\RP_{i}= \RP_{i}^{\infty } & := \hocolim{j\to \infty }\RP_{i}^{j},  \\
\RP_{-\infty }^{j}
 & := \holim{i\to \infty } \RP_{-i}^{j } ,\\
\aand 
\RP_{-\infty }^{\infty }
 & := \holim{i\to \infty } \RP_{-i}^{\infty }.
\end{split}
\end{numequation}%

Note that when $G$ acts trivially on a space $X$, its homotopy orbit
space is $BG\times X$, and its homotopy fixed point set is $\Map_{}
(BG, X)$. Hence for a $\rC _{2}$-spectrum $X$ with trivial group with
trivial action we have
\begin{align*}
X_{h\rC _{2}}
 & \simeq  \RP_{0}\wedge X  \\
\aand 
X^{h\rC _{2}}
 & \simeq  \map (\RP_{0},X) 
   \simeq  \map (\hocolim{m\to \infty}\RP_{0}^{m-1},X) \\ 
 & \simeq \holim{m\to \infty} \map (\RP_{0}^{m-1},X)\\
 & \simeq \holim{m\to \infty} (\map (\RP_{0}^{m-1},\mS)\wedge X)
\end{align*}

\noindent Now we claim that the Spanier-Whitehead dual (see
\cite[\S8.0B]{HHR:ESHT}) of $\RP_{0}^{m-1}$ is
\begin{numequation}\label{eq-claim}
\begin{split}
\DD\RP_{0}^{m-1}:=\map (\RP_{0}^{m-1},\mS)\simeq \Sigma \RP^{-1}_{-m};
\end{split}
\end{numequation}%

\noindent see \cref{eq-SW-dual}.  Assuming this and that $X$ is
finite, we have
\begin{align*}
X^{h\rC _{2}}
 & \simeq (\holim{m\to \infty }\Sigma \RP^{-1}_{-m} \wedge X)\\
 & \simeq \Sigma \RP^{-1 }_{-\infty } \wedge X.
\end{align*}

\noindent It follows that for finite $X$ and $G=\rC _{2}$, the bottom row of 
\cref{eq-Tate2} is the smash product of $X$ with
\begin{numequation}\label{eq-RP-cofib}
\begin{split}
\xymatrix
@R=4mm
@C=10mm
{
{\RP_{0}^{\infty }}\ar[r]^(.5){}\ar@{=}[d]^(.5){}
  &{\Sigma \RP^{-1 }_{-\infty }}
              \ar[r]^(.5){}\ar@{=}[d]^(.5){}
    &{\Sigma \RP^{\infty  }_{-\infty }}
              \ar[r]^(.5){}\ar@{=}[d]^(.5){}
      &{\Sigma \RP_{0}^{\infty }}\ar@{=}[d]^(.5){}\\
{\mS_{h\rC_{2}}}\ar[r]^(.5){}
  &{\mS^{h\rC_{2}}}\ar[r]^(.5){}
    &{\mS^{t\rC _{2}}}\ar[r]^(.5){}
      &{\Sigma \mS_{h\rC_{2}},}
}
\end{split}
\end{numequation}%
 
\noindent where the last map in the top row is the extension of the
cofiber sequence to the right.  It turns out that the last two maps of
the top row form the suspended limiting/colimiting case of the cofiber
sequence of \cref{eq-cofiber-seq} with $i=-\infty $, $j=-1$ and
$k=\infty $.  This gives us \cref{eq-RP-infinity} once we have proved
\cref{eq-claim}.

Atiyah Duality \cite{Atiyah-Dual} tells us that the Spanier-Whitehead
dual of the suspension spectrum of a closed manifold $M$ (such as
$\RP^{m-1}$) is up to suspension the Thom spectrum $M^{\nu }$ of its
normal bundle.  (See \cite[\S8.0B, especially Theorem
8.0.12]{HHR:ESHT} for more discussion.) We know that the tangent
bundle of $\RP^{m-1}$ is stably {\eqt} to the $m$-fold Whitney sum of
the canonical line bundle; see \cite[Theorem 4.5]{MilSta}. This means
its normal Thom spectrum is $\RP^{-1}_{-m}$.  The suspension on the
right-hand side of \cref{eq-claim} is needed to insure that the top
cell is in dimension 0.

We also know that the top cell of the normal Thom spectrum is
spherical.  This means there is a map $\mS \to \Sigma \RP^{-1}_{-m}$
that is nontrivial on mod 2 homology.  It is compatible with the maps
in the limit, leading to a diagram
\begin{displaymath}
\xymatrix
@R=4mm
@C=10mm
{
{\mS}\ar[r]^(.5){}\ar[dr]^(.5){}
  &{\Sigma \RP^{-1}_{-m}}\ar[r]^(.5){}
    &{\Sigma \RP^{\infty }_{-m}}\\
  &{\Sigma \RP^{-1}_{-\infty }}\ar[r]^(.5){}
                               \ar[u]^(.5){}
    &{\Sigma \RP^{\infty }_{-\infty }}\ar[u]^(.5){}
                                      \ar@{=}[r]^(.5){}
      &{\mS^{t\rC _{2}}}
}
\end{displaymath}

\noindent The Segal Conjecture says that the composite map $\mS \to
\mS^{t\rC _{2}}$ is 2-adic completion.

\subsubsection{The Segal conjecture}\label{sec-Segal-conjecture} 
For more information on this topic, we recommend the recent
account by Kaif Hilman \cite{Hilman}.  The conjecture was first stated
by Graeme Segal in \cite{Segal:ESHT} and later proved by Gunnar
Carlsson in \cite{Carlsson:Segal}.  In its simplest form it describes
the 0th stable cohomotopy group of the classifying space $BG$ of a
finite group $G$ with an isomorphism
\begin{numequation}\label{eq-pi0}
\begin{split}
\lim{k} \pi^{0}_{S}\left(BG^{(k)}_{+} \right)
\cong \widehat{A} (G). 
\end{split}
\end{numequation}%

The homotopy limit of \cref{eq-RP-infinity} was first examined by
Adams in \cite{Ad:RP}. It was the subject of extensive computations by
Mahowald in the 70s and 80s, some of which were eventually reported in
\cite{MR:Root}.  That paper and \cite{MahShi2} were the first
discussions of {\em blueshift phenomena} in chromatic homotopy theory
(not to be confused with the redshift of \cref{sec-Bok-Mad}), although
the term ``blueshift'' was not coined until later.  As Mahowald and
Shick said in their opening paragraph,

\begin{quote}
Calculations of various sorts lead to the following slogan:

The root invariant of $v_{n}$-periodic homotopy is $v_{n}$-torsion.
\end{quote}

\noindent Mahowald's computations also paved the way to W. H. Lin's
proof of the Segal Conjecture for the group $\rC _{2}$ in
\cite{Lin:RP} and \cite{LDMA}, and of Gunawardena's proof of it for
groups of odd prime order \cite{JHCG}.

For a pointed space $X$, $\pi^{0}_{S} (X)$ denotes the group of
homotopy classes of maps from the suspension spectrum of $X$,
$\Sigma^{\infty }X$, to the sphere spectrum,
$\mathbb{S}=\Sigma^{\infty }S^{0}$.  Like ordinary cohomology, stable
cohomotopy has a ring structure (cup product) induced by the diagonal
inclusion $X\to X \times X$.

The space $BG^{(k)}_{+}$ is the $k$-skeleton of $BG_{+}$, the
classifying space of $G$ with a disjoint basepoint.  The functor
$\pi^{0}_{S} (-)$ is contravariant, so it converts the sequential
colimit associated with the skeletal filtration of $BG_{+}$ to a limit
of abelian groups. While differing choices of a contractible free
$G$-space $EG$ (with orbit space $BG$) can lead to differing skeleta
$BG^{(k)}_{+}$, it is known that the limit in question is independent
of such choices.

$A (G)$ denotes the {\em Burnside ring of $G$}, originally defined by
William Burnside in \cite{Burnside2}.  In modern language it is the
Grothendieck group of the abelian monoid (under disjoint union) of
isomorphism classes of finite $G$-sets under disjoint union, with
multiplication induced by Cartesian product.  Additively it is the
free abelian group on the set of conjugacy classes of finite subgroups
$H\subseteq G$.  It admits an {\em augmentation} map $\epsilon :A
(G)\to \Z$, defined by sending a finite $G$-set to its cardinality,
with kernel $I$, the {\em augmentation ideal}. $\widehat{A} (G)$
denotes the $I$-adic completion of $A (G)$.

There is a stronger form of the conjecture, of which \cref{eq-pi0} is
just the tip of the iceberg.  It has to do with the function spectrum $F
(\Sigma^{\infty }BG_{+}, \mathbb{S})$, or {\eqvt}ly the function
$G$-spectrum $F^{G} (\Sigma^{\infty }EG_{+}, \mathbb{S})$ with $G$
acting trivially on $\mathbb{S}$.  It can also be thought of as the
Spanier-Whitehead dual of $\Sigma^{\infty }BG_{+}$.  The group of
\cref{eq-pi0} is its 0th homotopy group.  See \cite[\S4]{Hilman} for
details.

\begin{theorem}\label{thm-Tate}
{\bf Properties of the Tate construction.}
\begin{enumerate}[label={(\roman*)},itemindent=0em]
\item \label{thm-Tatei} \cite{JFA-Segal} 
The Segal Conjecture implies that for finite $X$ with trivial
$\rC _{p}$-action, $X^{t\rC _{p}}$ is the $p$-adic completion of $X$.

\item \label{thm-Tateii} 
Let $H/p$ be the mod $p$ {\SESM} spectrum. Then
\begin{displaymath}
(H/p)^{t\rC _{p}}
 = \holim{k\leq 0}\bigvee_{i\geq k}\Sigma^{i}H/p.
\end{displaymath}

\noindent For $p=2$ we can write this as 
\begin{displaymath}
\holim{i} (\RP_{-i}\wedge \Sigma H/2),
\end{displaymath}

\noindent which is wildly different from 
\begin{displaymath}
\holim{i} (\RP_{-i})\wedge \Sigma H/2 = H/2.
\end{displaymath}

\item \label{thm-Tateiii} \cite{AMS} Let $E (n)$ be the $n$th
Johnson-Wilson spectrum for the prime $p$ and $n>0$. Then $E (n)^{t\rC
_{p}}$ is a certain completion of the coproduct of all even supensions
of $E (n-1)$. This is another instance of chromatic blueshift.
\end{enumerate}
\end{theorem}

\cref{thm-Tateii} above is a spectacular example of the failure of the smash
product to preserve limits.

\subsubsection{More about $G=\mT $}

\begin{defin}\label{def-rep-T}
{\bf Real orthogonal representations of $\mT$.}  Let $\lambda $ denote
the complex numbers with the usual multiplication by elements of
$\mT$, regarded as a 2-dimensional real vector space.  More generally
for each integer $r\geq 0$, let $\lambda^{r}$ denote complex numbers
with $\omega \in \mT $ acting by multiplication by $\omega^{r}$, again
regarded as a 2-dimensional real vector space. 
Using the notation of \cref{def-piHX}, let
\begin{align*}
\mT (r)&:=S (\lambda^{r}),&
\mS[\mT/\rC _{r}]
       &:=\Sigma^{\infty }_{+}\mT(r) &
\aand 
\mS^{(r)}
       &:=\Sigma^{\infty }S^{\lambda^{r}}.
\end{align*}

\noindent The map $a_{\lambda^{r}}:S^{0}\to S^{\lambda^{r}}$ of
\cref{def-piHX} leads to a map $a_{(r)}:\mS^{- (r)}\to \mS$, the {\bf
Euler class}.
\end{defin}

The notation $\mS[\mT/\rC _{r}]$ (but not $\mT (r)$) is abusive for
$r=0$ since there is no group $\rC_{0}$.

\begin{prop}\label{prop-Cpinfty}
{\bf The action of the Pr\"uffer group.}  For $r>0$, let 
$r=sp^{j}$ with $s$ not divisible by $p$.  Then the space
$\mT/\rC_{r}$ is $\rC_{p^{\infty }}$-{\eqvr}ly isomorphic to
$\mT/\rC_{p^{j}}$.

\end{prop}

In the action of $\mT$ on $\lambda^{r}$, the subgroup $\rC _{r}\subseteq
\mT$ acts trivially, so the action factors through $\mT/\rC _{r}$.  The
underlying spectrum of $\mS^{(r)}$ is $\Sigma^{2}\mS=\Sigma^{\infty
}S^{2}$, and it is tensor invertible.  The {\rep} $\lambda^{r}$ is
denoted in \cite[Notation 4.1]{Blumberg-Mandell-12} by $\cxs (r)$.  We
are mostly using the notation of \cite[\S2.1.1]{BHLS}.

For each positive integer $r$, the circle group $\mT$ has a subgroup
isomorphic to the cyclic group $\rC _{r}$. 
The group $\mT$ is isomorphic to its quotient by any finite subgroup,
a property not enjoyed by any nontrivial finite group.  For a
$\mT$-space $X$, we have a residual action (\cref{def-residual}) of
$\mT /\rC _{r}\iso \mT $ on the fixed point set $X^{\rC _{r}}$.

The following is part of \cite[Definition 4.2]{Blumberg-Mandell-12}.


\begin{defin}\label{def-cyclo-spec}
A {\bf cyclotomic spectrum} is an orthogonal $\mT$-spectrum $X$ along
with $\mT$-{\eqvr} maps
\begin{displaymath}
r_{p,V}:\rho_{p}^{*} \left(X_{V}^{\rC _{p}} \right)
        \to X_{\rho_{p}^{*} (V^{\rC _{p}})}.
\end{displaymath}

\noindent for each prime $p>0$ and each {\rep} $V$ of $\mT$ (where
$\rho_{p}^{*}$ is as in \cref{def-rho-n}) satisfying certain
conditions implying that there is an isomorphism
\begin{displaymath}
\rho_{p}^{*}\Phi^{\rC _{p}}X\to X,
\end{displaymath}

\noindent where $\Phi ^{\rC _{p}}X$ is the geometric fixed point spectrum
of \cref{def-fixed}\cref{def-fixedii}.
\end{defin}

\subsection{Tate resolutions}
\label{sec-Tate-resolutions} 
Applying the functor $\Hom_{\ZG} (-,\Z)$ to
\cref{eq-free-res}, we get
\begin{numequation}\label{eq-P-}
\begin{split}
\xymatrix
@R=3mm
@C=6mm
{
{0}\ar[r]^(.5){}
  &{\Hom_{\ZG} (\Z,\Z)}\ar[r]^(.5){}\ar@{=}[d]^(.5){}
    &{\Hom_{\ZG} (P_{0},\Z)}\ar[r]^(.55){}\ar@{=}[d]^(.5){}
      &{\Hom_{\ZG} (P_{1},\Z)}\ar[r]^(.5){}
        &{\dotsb } \\
  &{\Z,} \ar[r]^(.5){\Delta }
    &{\ZG}
}
\end{split}
\end{numequation}%

\noindent where $\Delta $ is the diagonal or coaugmentation
\begin{displaymath}
1\mapsto \sum_{\gamma \in G}[\gamma ].
\end{displaymath}

\noindent We define $P_{-i}:=\Hom_{\ZG} (P_{i-1},\Z)$ for $i>0$.  Then
we can splice \cref{eq-free-res} with \cref{eq-P-} and get
\begin{numequation}\label{eq-P}
\begin{split}
\xymatrix
@R=3mm
@C=8mm
{
{\dotsb }  \ar[r]^(.5){}
  &{P_{2}}\ar[r]^(.5){}
    &{P_{1}}\ar[r]^(.5){}
      &{P_{0}}\ar[r]^(.5){ }\ar@{=}[d]^(.5){}
        &{P_{-1}}\ar[r]^(.5){}\ar@{=}[d]^(.5){}
          &{P_{-2}}\ar[r]^(.5){}
            &{P_{-3}}\ar[r]^(.5){}
              &{\dotsb }\\
  & & &{\ZG}\ar[r]^(.5){\Delta \nabla  }
        &{\ZG} 
}
\end{split}
\end{numequation}%

\noindent where the augmentation $\nabla:\ZG\to\Z$ is as in
\cref{eq-free-res}.

\begin{defin}\label{def-P}
A {\bf Tate resolution $P$ of $\Z$ over a finite group $G$} is a
{\LES} of the form \cref{eq-P}.  $P^{-}\subseteq P$ is the subsequence
obtained by replacing $P_{i}$ by 0 for $i\geq 0$, and
$P^{+}:=P/P^{-}$, the quotient obtained by killing $P_{i}$ for $i<0$.
For a $G$-CW spectrum $X$ with cellular chain complex (of
$\ZG$-modules) $C (X)$, let
\begin{displaymath}
P (X):=P\otimes_{\ZG} C (X),
\end{displaymath}

\noindent with $P^{-} (X)$ and $P^{+} (X)$ defined similarly. The {\bf
  Tate sequence for $X$} is following {\SES} of chain complexes over
$\ZG$.
\begin{numequation}\label{eq-PPP}
\begin{split}
0\to P^{-} (X)\to P (X)\to P^{+} (X)\to 0
\end{split}
\end{numequation}%
\end{defin}

\section{{\qcats} and the work of Nikolaus-Scholze}
\label{sec-NS}

Most of \cite{NS18} is written in the language of {\qcats}.  For a
very brief introduction to them we refer the reader to
\cite{Rav:Whatis}, where other references can be found.  As in that
paper, we will write {\qcats} which are not ordinary categories in the
color {\color{cyan}cyan}. 

\subsection{Elementary {\qcatal} notions}\label{sec-elem-qcat}
Stable {\qcats}, of which the {\qcat} of
spectra $\qSp$ is the marquee example, are defined in
\cref{def-qcat-props}\cref{def-qcat-propsiv}, \cite[\S9]{Rav:Whatis}
and in \cite[Definition 1.1.1.9]{Lurie:HA}.



In a pointed {\qcat} $\qmcC$, the suspension $\Sigma X$ is the pushout
of the diagram $0\leftarrow X \rightarrow 0$, while the loop object
$\Omega Y$ is the pullback of the diagram $0 \rightarrow Y \leftarrow
0$.  In a stable {\qcat} (see
\cref{def-qcat-props}\cref{def-qcat-propsiv}), a commutative square
diagram is a pushout iff it is a pullback, so $X\simeq \Omega \Sigma
X$ and $\Sigma \Omega Y\simeq Y$.  This makes the functors $\Sigma $
and $\Omega $ both invertible in the homotopy category of a stable
{\qcat} $\qmcC$.

\begin{defin}\label{def-qcats}
{\bf Some familiar {\qcats}.}   

\begin{enumerate}[label={(\roman*)},itemindent=0em]

\item \label{def-qcatsi} $\qmcS$ is Lurie's {\qcat} of spaces (meaning
Kan complexes, aka $\infty $-groupoids) \cite[Definition
1.2.16.1]{Lurie:HTT}, which is described illustratively in
\cite[\S5]{Rav:Whatis}.  It is sometimes called the category of {\bf
anima}.  This term is used in writings on condensed mathematics, for
which the original references appear to be
\cite{clausen-scholze-condensed} and \cite{clausen-scholze-analytic},
even though the term is not used in the former.

\item \label{def-qcatsiii} For a space (or Kan complex) $K$, $\qK$
  denotes the {\qcat} in which objects are points (0-simplices) in
  $K$, 1-morphisms are paths (1-simplices), 2-morphisms are homotopies
  between paths (2-simplices), and so on.  The existence inverse paths
  makes $\qK$ an {\em $\infty$-groupoid}, meaning that its homotopy
  category is a groupoid. See \cite[\S1.1.1 and Proposition
  1.2.5.1]{Lurie:HTT}.

\item \label{def-qcatsv} For objects $X$ and $Y$ in an {\qcat}
$\qmcC$, we denote the space (or Kan complex) of morphisms $X\to Y$ by
$\Map_{\qmcC} (X,Y)$.

\item \label{def-qcatsii} $\qSp$ denotes Lurie's {\qcat} of spectra as
in \cite[Definition 1.4.3.1]{Lurie:HA} and in \cite[\S9]{Rav:Whatis}.
It is the homotopy limit of the diagram obtained by iterating the loop
functor on {\qcat} of pointed spaces $\qmcS_{*}$.  We will give a
different definition of the {\qcat} of orthogonal $G$-spectra in
\cref{def-qcat-Gspec}.

\item \label{def-qcatsiv} 
For {\qcats} $\qmcC$ and $\qmcD$, $\qFun (\qmcC, \qmcD)$
denotes the {\qcat} of functors $\qmcC\to \qmcD$; see
\cite[Proposition 1.2.7.3]{Lurie:HTT}.  

\item \label{def-qcatsvi} 
When $\qmcC=\qK$ for a space or Kan complex
$K$, we will write $\qFun (\qmcC, \qmcD)$ as $\qmcD^{K}$.  In
particular, for a group $G$ with classifying space $BG$ (or {\eqt}ly
the nerve of the one object category $\mcB G$), the {\qcat}
$\qmcC^{BG}$ is that of {\bf objects in $\qmcC$ with $G$-action}
\cite[Definition I.1.2]{NS18}.

\item \label{def-qcatsvii} 
Replacing $\qmcC$ by the nerve $N
(\mcC^{\op})$ for on ordinary category $\mcC $, the {\qcat} of functors
becomes
\begin{displaymath}
\qFun (N (\mcC^{\op}), \qmcD)=: \mathpzc{P}_{\qmcD} (\mcC ),
\end{displaymath}

\noindent the {\bf category of $\qmcD$-valued presheaves (or
contravariant functors) on $\mcC $}.  Note that for a group $G$ the
one object category $\mcB G$ is isomorphic to its opposite (since each
morphisms is invertible), so $\qmcC^{BG}\cong \mathpzc{P}_{\qmcD}
(\mcC )$.

\item \label{def-qcatsix} 
An {\qcat} is {\bf perfect} if it is small,
idempotent-complete, and stable. $\qCatw^{\mathrm{perf}}$ is the {\bf
{\qcat} of perfect {\qcats}}.


\end{enumerate}
\end{defin}




Recall that in classical category theory an adjunction for a pair of functors
\begin{displaymath}
\xymatrix
@R=4mm
@C=20mm
{
{\mcC }\ar@<1.5ex>[r]^(.5){F}_(.5){\perp }
  &{\mcD }\ar@<1ex>[l]^(.5){U}
}
\end{displaymath}

\noindent is a natural isomorphism between the morphism sets $\mcC
(X, UY)$ and $\mcD (FX,Y)$ for objects $X$ in $\mcC $ and $Y$ in $\mcD
$. For {\qcats} $\qmcC$ and $\qmcD$ we are looking instead for a natural
equivalence between morphism spaces.

\begin{defin}\label{def-adjoint}
{\bf Adjoint functors of {\qcats}.}  \cite[Definition 2.1.1]{CSY22}
and \cite[\S5.2, specifically Definition 5.2.2.1]{Lurie:HTT}.  For a
functor  $f : \qmcC \to \qmcD$,
\begin{enumerate}[label={(\roman*)},itemindent=0em]

\item \label{def-adjointi} a {\bf left adjoint of $f$} is a pair
$(f_{!},\eta )$ where where $f_{!} : \qmcD \to \qmcC$ is a functor and
$\eta : \Id_{\qmcD} \implies f \circ f_{!}$ is a unit natural
transformation in the sense of \cite[Definition 5.2.2.7]{Lurie:HTT}, and

\item \label{def-adjointii} a {\bf right adjoint of $f$} is a pair
$(f_{*},\epsilon )$ where where $f_{*} : \qmcD \to \qmcC$ is a functor
and $\epsilon : f_{*} \circ f \implies\Id_{\qmcC}$ is a counit natural
transformation in the sense of  \cite[the dual of Definition
5.2.2.7]{Lurie:HTT}.
\end{enumerate}
\end{defin}

Adjoints of composite functors are suitably defined composites of
adjoints as in \cite[Definition 2.1.3]{CSY22}.

\begin{defin}\label{def-stab-cat}
{\bf The stabilization of an {\qcat}.} As noted above, a sequential
$\Omega $-spectrum $X$ is a sequence of pointed spaces $X_{n}$ for
$n\geq 0$ and pointed {\weq }s $X_{n}\to \Omega X_{n+1}$.  In a
pointed {\qcat} $\qmcC$ with finite limits, the loop of an object $Y$
is the pullback of the diagram $0 \rightarrow Y \leftarrow 0$, so we
can define a {\bf spectrum object in $\qmcC$} in the same way we do it
in the pointed {\qcat} of spaces $\qmcS_{*}$. In
\cite[\S1.4]{Lurie:HA} Lurie denotes the category of such objects by
${\rm Sp} (\qmcC)$.  In \cite{BGT} it is denoted by $\rm Stab
(\qmcC)$, the {\em stabilization of $\qmcC$}.  In any case there is a
functor $\Omega^{\infty }:{\rm Sp} (\qmcC) \to \qmcC$ sending a
spectrum $X$ object to the object $X_{0}$ in $\qmcC$.  When $\qmcC$ is
presentable (see \cite[Chapter 5]{Lurie:HTT}), this functor admits a
left adjoint $\Sigma^{\infty }$ \cite[Proposition 1.4.4.4]{Lurie:HA}.
\end{defin}

\begin{remark}\label{rem-G-action}
{\bf Spectra with $G$-action and $G$-spectra.}  The ordinary
categories of spectra with $G$-action ($\Sp^{BG}$ as in
\cref{def-seq-spec}), sometimes called {\bf naive $G$-spectra}, and of
orthogonal $G$-spectra as in \cref{def-orth-spec}, are different.  The
same goes for the {\qcats} $\qSp^{BG}$ of \cref{def-qcats}\cref{def-qcatsvi}  
and $\qSp^{G}$ of
\cref{def-qcat-Gspec}.
\end{remark}

\subsection{Limits and colimits}\label{sec-lim-colim}
In \cref{def-qcats}\cref{def-qcatsiv}, let $f:K\to L$ be a map of
Kan complexes.  It leads to a pullback functor
$f^{*}:\qmcC^{L}\to\qmcC^{K}$ between functor categories induced by
precomposition.  For suitable $\qmcC$ and $f$, this functor has left
and right adjoints denoted by $f_{!}$ and $f_{*}$ \cite[Notation
6.1.6.1]{Lurie:HA}.  {\Eqt}ly they are the functors sending each
functor $X:\qK \to \qmcC$ (a $K$-shaped diagram in $\qmcC$) to its
left and right Kan extensions along $f$.

When $L$ is a point, these functors from $\qmcC^{K}$ to $\qmcC$ send
$X$ to its colimit and limit.  One advantage of working with {\qcats}
is that homotopy limits and homotopy colimits, the subject of
Bousfield and Kan's ``yellow monster'' \cite{BK1}, are the same as
ordinary limits and colimits.  The standard reference for this is
\cite[4.2.4]{Lurie:HTT}.  See \cite[\S9]{Rav:Whatis} for a simple
illustration.

The following is originally due to Joyal \cite[Definition
4.5]{Joyal:QCKC} and is quoted by Lurie as \cite[Definition
1.2.13.4]{Lurie:HTT}.

\begin{defin}\label{def-lim-colim}
{\bf Limits and colimits in an {\qcat}.}  For a simplicial set $K$ and
an {\qcat} $\qmcC$, let $f:K\to\qmcC $ be a simplicial map.  This is
an object in $\qmcC^{K}$.  When $K$ is a Kan complex, this is a
$\qmcC$-valued functor on the {\qcat} $\qK$ of
\cref{def-qcats}\cref{def-qcatsiii}.  An {\bf object $X$ in $\qmcC$
over $f$} is one equipped with compatible maps from $X$ to the images
under $f$ of all vertices in $K$.  {\Eqt}ly it is a natural
transformation from the $X$-valued constant functor on $K$ to $f$.
The collection of all such objects and suitable morphisms (and higher
morphisms) between them is an {\qcat} $\qmcC_{/f}$ \cite[Proposition
1.2.9.3]{Lurie:HTT}, the {\bf over category of $f$}. A {\bf limit of
$f$}, $\lim{K}f$, is a terminal object in it ({\ie } an object for
which the mapping space $\Map_{\qmcC_{/f}} (Y,\lim{K}f )$ is
contractible for all $Y$), regarded as an object in $\qmcC$.  Dually,
one can construct an {\bf under category} $\qmcC_{f/}$ of objects
under $f$, and define a colimit of $f$ to be an initial object in it
regarded as an object in $\qmcC$.  Such terminal and initial objects
may or may not exist depending on $\qmcC$.
\end{defin}

In \cite[\S4.2.4]{Lurie:HTT} Lurie shows that \cref{def-lim-colim}
agrees with the classical theory of homotopy (co)limits when we
specialize to the case where $\qmcC$ is the nerve of a topological
category.

\begin{defin}\label{def-assembly}
{\bf Induced maps to/from a limit/colimit.}  Now suppose that in
addition to the data of \cref{def-lim-colim}, we have an
$\infty$-functor $F:\qmcC\to \qmcD$ and that $\qmcD$ has the needed
limits or colimits.  Then $F$ induces a functor
\begin{displaymath}
F_{/f}:\qmcC_{/f}\to \qmcD_{/Ff}.
\end{displaymath}

\noindent Then we have objects $F (\lim{K}f)$, the functor of the
limit, and $\lim{K} (Ff)$, the limiting value of the functor, in
$\qmcD$.  Since the latter is a terminal object, we have a
unique (up to homotopy) map $\epsilon :F (\lim{K}f)\to \lim{K} (Ff)$,
the {\bf coassembly map of $F$ and $f$}.  If it is an {\eqiv}, we say
that $F$ {\bf preserves} the limit of $f$.  Dually we have the {\bf
assembly map}, $\eta :\colim{K} (Ff)\to F (\colim{K}f) $.
\end{defin}

Lurie discusses this situation in \cite[page 48]{Lurie:HTT} but does
not give names to the two maps.  

\begin{defin}\label{ex-assembly}
{\bf Assembly and coassembly maps for $G$-spectra.}  For compact Lie
group $G$, let $\mcB G$ be the one object topological category of
\cref{def-nerve}.  Let $K=N (\mcB G)$ (so $|K|=BG$, the classifying
space of $G$) and let $\qmcC$ be the {\qcat} of spaces or spectra.
Then $f:BG\to \qmcC$ defines a $G$-action on an object $X$ in
$\qmcC$. The limit and colimit are $X^{hG}$ and $X_{hG}$, the homotopy
fixed points \cref{eq-hfps} and homotopy orbits \cref{eq-hos}.  Thus
we have the unit map $\eta :(FX)_{hG}\to F (X_{hG})$ and the counit
map $\epsilon :F (X^{hG})\to (FX)^{hG}$ of $F$ and $X$, the assembly
and coassembly maps. The map of \cref{eq-loc-coassembly} is an example
of a coassembly map.
\end{defin}

{\bf The norm map.}  We can ask for a norm map $\Nm_{G}:X_{hG} \to
X^{hG}$ as in \cref{eq-Tate2} when $G$ is finite.  In
\cite[6.1.6]{Lurie:HA} Lurie shows that for any finite group $G$, the
norm map $\Nm_{G}$ exists and has a cofiber whenever $\qmcC$ is a
stable {\qcat} with countable limits and colimits. This cofiber is the
{\bf categorical Tate construction $X^{tG}$}. In \cite{Rav:gjmcyc-ext}
we will consider situations in which it is defined for spaces 
that are $\pi$-finite (meaning all homotopy groups are finite and
only finitely many of them are nontrivial), such as $BG$, and
certain {\qcats} (such as that of $K (n)$-local spectra) for which it
is always an {\eqiv}, rendering the Tate construction contractible.

\subsection{Additional structures on {\qcats}}\label{sec-additional}

The first four parts of the following are similar to \cite[Definition
I.1.1]{NS18}.

\begin{defin}\label{def-qcat-props}
{\bf Some properties of {\qcats}.}  Let $\qmcC$ be a {\qcat.} 

\begin{enumerate}[label={(\roman*)},itemindent=0em]

\item \label{def-qcat-propsi}\cite[Definition 1.1.1.1]{Lurie:HA}
$\qmcC$ is {\bf pointed} if it has an object, usually denoted by 0,
that is both initial and final.

\item \label{def-qcat-propsii}\cite[Definition 6.1.6.13]{Lurie:HA}
$\qmcC$ is {\bf semiadditive} (called {\em preadditive} by
Nikolaus and Scholze) if it is pointed, has finite products
and coproducts, and for any two objects $X,Y \in \qmcC$, the map
\begin{displaymath}
\left(\begin{array}[]{cc}
1_{X} & 0\\
0     &1_{Y} 
\end{array} \right)    :X \sqcup Y\to X\times Y
\end{displaymath}

\noindent is an equivalence. Here, $0\in \Map_{\qmcC}(X, Y )$ denotes
the composition\linebreak $X\to 0\to Y$ for any zero object $ 0\in
\qmcC$.  In this case we write $X\oplus Y$ for this object, and note
that $\pi_{0}\Map_{\qmcC}(X, Y )$ has a natural commutative monoid
structure.

\item \label{def-qcat-propsiii} $\qmcC$ is {\bf additive} if it is
semiadditive and $\pi_{0}\Map_{\qmcC}(X, Y )$ has a natural abelian
group structure.

\item \label{def-qcat-propsiv} 
\cite[Definition 1.1.1.9]{Lurie:HA}
$\qmcC$ is {\bf stable} if it is additive and the loop functor $\Omega
:\qmcC \to \qmcC$ sending $X$ to $0\times_{X} 0$ is an equivalence.  A
functor between stable {\qcats} is {\bf exact} if it preserves finite
colimits, or equivalently, finite limits, \cite[Proposition
1.1.4.1]{Lurie:HA}.

\item \label{def-qcat-propsiva}\cite[Definition 2.1]{Mathew-Galois} 
A {\bf stable homotopy theory} is a stable {\qcat} which is also a
symmetric monoidal {\qcat} $(\qmcC, \otimes, \mathbf{1})$
(\cref{def-symm-mon-qcat}) in which the tensor product commutes with
all colimits.

\item \label{def-qcat-propsvi}
\cite[4.4.5]{Lurie:HTT} An object $Y$ in
$\qmcC$ is a {\bf retract} of an object $X$ if there exists a
2-simplex $\bdelt^{2}\to \qmcC$ corresponding to a diagram
\begin{displaymath}
\xymatrix
@R=6mm
@C=10mm
{
{}
  &{X}\ar[dr]^(.5){r}
    &{}\\
{Y}\ar[ur]^(.5){i}\ar[rr]^(.5){{\rm id}_{Y}}
      &{}
        &{Y,}
}
\end{displaymath}

\noindent or {\eqvt}ly if $Y$ is a retract of $X$ in the classical
sense in the homotopy category $h\qmcC$.  It follows that the
endomorphism $i\circ r:X\to X$ is idempotent and that the object $Y$
is both the equalizer and coequalizer of the pair $({\rm id}_{X},
i\circ r)$.  $\qmcC$ is {\bf idempotent complete} if every idempotent
endomorphism (that is, every $\qmcC$-valued functor on the one object
category equipped with an idempotent map) arises in this way, meaning
that said equalizer/coequalizer always exists.  An {\bf idempotent
completion of $\qmcC$} is a fully faithful functor $f:\qmcC\to \qmcD$ where $\qmcD=:\Ind (\qmcC)$ is idempotent complete and each of its objects
is a retract of one in the image of $f$.  See \cref{rem-idem-comp}.

\item \label{def-qcat-propsv} 
$\qmcC$ is {\bf exact} if it is small
and stable. The {\qcat} of exact {\qcats} and exact functors is
denoted by ${\color{cyan}\rm Cat}^{\rm ex}_{\infty }$.  The {\qcat} of
idempotent complete exact {\qcats} and exact functors is denoted by
${\color{cyan}\rm Cat}^{\rm perf}_{\infty }$. The relevant example of
such is the category of {\bf compact left (or right) modules over a
ring spectrum $R$}, which we denote by $\qMod_{R} (\qSp)$ or simply
$\qMod_{R}$.  See \cite[\S4.2 and \S7.1.1]{Lurie:HA} for more
discussion.  An exact functor $F:\qmcA \to \qmcB$ between small stable
{\qcats} is a {\bf Morita equivalence} if it induces an equivalence
their idempotent completions.
\end{enumerate}
\end{defin}

\begin{defin}\label{def-loc-inv}
A functor
\[
E \colon \qCatw^{\mathrm{perf}}\longrightarrow \qmcD,
\]
valued in a stable presentable $\infty$-category $\qmcD$, is  a
{\bf localizing invariant} if it satisfies the following conditions:
\begin{enumerate}[label={(\roman*)},itemindent=0em]
    \item \textbf{Morita invariance.}
    If $F \colon \qmcC \to \qmcC$ is an exact functor that induces an
    equivalence on idempotent completions (see \cref{def-qcat-props}\cref{def-qcat-propsvi})
    \[
    \Ind(\qmcC) \xrightarrow{\;\simeq\;} \Ind(\qmcD),
    \]
    then $E(F)$ is an equivalence.

    \item \textbf{Exactness.}
    For every exact sequence of small stable $\infty$-categories
    \[
    {\qmcA} \longrightarrow {\qmcB} \longrightarrow {\qmcC},
    \]
    the induced sequence
    \[
    E({\qmcA}) \longrightarrow E({\qmcB}) \longrightarrow E({\qmcC})
    \]
    is a fiber (equivalently, cofiber) sequence in ${\qmcD}$.

   \item \textbf{Filtered colimits.}
    The functor $E$ preserves filtered colimits:
    \[
    E\!\left(\colim{i} {\qmcC}_i\right)
    \;\simeq\;
    \colim{i} E({\qmcC}_i).
    \]
\end{enumerate}
\end{defin}

We know (see \cite[\S5.1.3]{Lurie:HTT}) that for any {\qcat} $\qmcC $
there is a mapping space functor
\begin{displaymath}
\Map_{\qmcC } :\qmcC^{\op}\times \qmcC \to \qmcS,
\end{displaymath}

\noindent meaning that the set of morphisms $X\to Y$ comes equipped
with a topology.  When $\qmcC $ is pointed, we get a pointed morphism
space.  When $\qmcC $ is stable as in \cref{def-qcat-propsiv}, this
functor lifts to
\begin{numequation}\label{eq-map-spec}
\begin{split}
\map_{\qmcC } :\qmcC^{\op}\times \qmcC \to {\qSp}
\qquad \mbox{with}\qquad 
\Omega^{\infty } \map_{\qmcC } (X,Y)= \Map_{\qmcC } (X,Y).
\end{split}
\end{numequation}%

\noindent In other words, in addition to a space one has a {\bf
spectrum} $\map_{\qmcC} (X,Y)$ (sometimes called the {\bf function
spectrum}) of maps between objects in a stable {\qcat}.  The 0th space
of this mapping spectrum is the previously defined mapping space. In
the category $\qSp$ itself we have
\begin{displaymath}
\map_{\qSp} (\mS ,X)\simeq X,
\end{displaymath}

\noindent but this need not be the case in a stable subcategory such
as $\qCycSp$ or $\qCycSp_{p}$ as in \cref{def-cyc-spectra}.  We also
define the Spanier-Whitehead dual of a spectrum $X$ by
\begin{numequation}\label{eq-SW-dual}
\begin{split}
\DD X:=\map_{\qSp} (X,\mS).
\end{split}
\end{numequation}%

\noindent For a space $X$, we will write $\DD\Sigma^{\infty }X$
abusively as $\DD X$.

\begin{remark}\label{rem-idem-comp}
{\bf Idempotent completion.}
Lurie shows in \cite[Proposition 5.1.4.2]{Lurie:HTT} that the
idempotent completion of \cref{def-qcat-propsvi} always exists, and in
\cite[Corollary 1.1.3.7]{Lurie:HA} that it is stable when $\qmcC$ is.
\cite[Propositions 5.1.4.9]{Lurie:HTT} implies the inclusion functor
${\color{cyan}\rm Cat}^{\rm perf}_{\infty }\to {\color{cyan}\rm
Cat}^{\rm ex}_{\infty }$ has a left adjoint which is the idempotent
completion functor $\Ind$ denoted in \cite{BGT} by ${\rm Idem}$.
\end{remark}




\begin{defin}\label{def-R-lin}
{\bf $R$-linearity.} Let $R$ be an $\mathbb{E}_n$-ring spectrum for $n
\geq 1$. A stable {\qcat} $\qmcC$ is {\bf $R$-linear} if:
\begin{itemize}
  \item It is enriched over $\qMod_R$, the {\qcat} of $R$-module
spectra as in \cref{def-qcat-props}\cref{def-qcat-propsv}, meaning
that for objects $X, Y \in \qmcC$, the mapping spectrum
$\map_{\qmcC}(X, Y)$ is an $R$-module.

\item Tensoring with $R$-modules acts on $\qmcC$, and this action is
compatible with the stable structure.
\end{itemize}

\noindent ${\color{cyan}\rm Cat}^{\rm perf}_{R,\infty }$ denotes the
{\qcat} small, idempotent-complete $R$-linear stable\linebreak  {\qcats} and
$R$-linear functors between them.
\end{defin}

\begin{defin}\cite[Definition II.1.4]{NS18}\label{def-lax-eq} 
Let $\qmcC $ and $\qmcD $ be {\qcats} with functors $F_{0},F_{1}:\qmcC
\to \qmcD $. The {\bf lax equalizer} $\LEq{F_{0}}{F_{1}}$
of $F_{0}$ and $F_{1}$, also written as 
\begin{displaymath}
{\color{cyan} \rm LEq}
 \left(\!\!
\xymatrix
@R=4mm
@C=10mm
{
{\qmcC}\ar@<.5ex>[r]^(.5){F_{0}}\ar@<-.5ex>[r]_(.5){F_{1}}
  &{\qmcD}}
\!\!\right),
\end{displaymath}


\noindent is the pullback in the following
diagram of simplicial sets.
\begin{displaymath}
\xymatrix 
@R=6mm
@C=15mm
{
{\LEq{F_{0}}{F_{1}}}\ar[r]^(.5){}\ar[d]^(.5){}\ar@{}[dr]|-(.15){\pb}
  &{{\qmcD}^{\bdelt^{1}}}\ar[d]^(.5){(\ev_{0},\ev_{1})}\\
{\qmcC }\ar[r]^(.5){(F_{0},F_{1})}
  &{{\qmcD }\times {\qmcD} }
}
\end{displaymath}

\noindent In particular, an object of $\LEq{F_{0}}{F_{1}}$ is a pair
$(c,f)$, where $c$ is an object in $\qmcC $ and $f: F_{0}(c)\to
F_{1}(c)$ is a morphism in $\qmcD$.

Similarly the {\bf equalizer} 
$\Eq{F_{0}}{F_{1}}$
is the pullback of the diagram above with the right column replaced by
the diagonal functor $\Delta: {\qmcD} \to {\qmcD} \times {\qmcD} $.
An object in it 
is an object $c$ in $\qmcC $ on which the two functors agree. 
\end{defin}

 The word ``lax'' above refers to the fact that the two images of $c$
in $\qmcD $ are related by a morphism rather than equality.

\begin{lem}\label{lem-Tate-orbit-fix}
{\bf The Tate orbit and fixed point lemmas.} \cite[Lemmas I.2.1 and
I.2.2]{NS18} Let $X$ be a spectrum with an action of
$\rC_{p^{2}}$. Recall that the quotient group $\orC_{p}:=\rC_{p^{2}}/\rC_{p}$
acts residually on $X_{h\rC_{p}}$ and $X^{h\rC_{p}}$.

\begin{enumerate}[label={(\roman*)},itemindent=0em]

\item \label{lem-Tate-orbit-fixi}
If $X$ is bounded below, meaning that $\pi_{i}X=0$ for $i\ll 0$, then  
\begin{displaymath}
(X_{h\rC_{p}})^{t \orC_{p}}\simeq \pt.
\end{displaymath}

\item \label{lem-Tate-orbit-fixii}
If $X$ is bounded above, meaning that $\pi_{i}X=0$ for $i\gg 0$, then  
\begin{displaymath}
(X^{h\rC_{p}})^{t \orC_{p}}\simeq \pt.
\end{displaymath}
\end{enumerate}
\end{lem}








\subsection{The {\qcatal} construction of $\THH$}\label{sec-THH-qcat}
Here we recall the results of \cite[\S III.1]{NS18}, referring the
reader to that paper for the proofs.

\begin{prop}\label{prop-exactness-Tp}
{\bf Exactness of $T_{p}$.}
\cite[Proposition III.1.1]{NS18}
The functor 
\begin{displaymath}
T_{p}:\qSp\to \qSp\qquad X\mapsto (X^{\otimes p})^{t\rC _{p}},
\end{displaymath}

\noindent where $\rC _{p}$ acts on $X^{\otimes p}$ by permuting its
factors, is exact as in \cref{def-qcat-props}\cref{def-qcat-propsiv}.
\end{prop}

This functor is studied by Lun{\o}e-Nielsen and Rognes in
\cite{LN-Rognes}, where they call it the {\bf topological Singer
construction}.

\begin{prop}\label{prop-III.1.2}
\cite[Proposition III.1.2]{NS18}. Let $\qFun^{\Ex} (\qSp, \qSp)$ be
the category of exact  endofunctors of $\qSp$,
and let $\Id_{\qSp}\in \qFun^{\Ex} (\qSp, \qSp)$ be the identity
functor.  For any $F\in \qFun^{\Ex} (\qSp, \qSp)$, evaluation at the
sphere spectrum $\mS$ induces an equivalence
\begin{displaymath}
\xymatrix
@R=4mm
@C=10mm
{
{\Map_{\qFun^{\Ex} (\qSp, \qSp)} (\Id_{\qSp}, F)}
       \ar[r]^(.65){\simeq }
  &{\Omega^{\infty }F (\mS).}
}
\end{displaymath}

\end{prop}

\begin{defin}\label{def-Tate-diagonal}
\cite[Definition III.1.4]{NS18}
{\bf The Tate diagonal} is the natural transformation
\begin{align*}
\Delta_{p} :\Id_{\qSp}
 &\to T_{p}   \\
X & \mapsto  (X^{\otimes p})^{t\rC _{p}}  
\end{align*}

\noindent 
of endofunctors of $\qSp$ which, under the
equivalence of \cref{prop-III.1.2}, corresponds to the map
\begin{displaymath}
\mS \to \mS^{h\rC _{p}} \to \mS^{t\rC _{p}}.
\end{displaymath}
\end{defin}

Recall that the Segal Conjecture for the group $\rC _{p}$ implies that
the map above is $p$-adic completion.

\begin{theorem}\label{thm-generalized-Segal-conj}
{\bf The generalized  Segal Conjecture.} \cite[Theorem III.1.7]{NS18}
For any bounded below spectrum $X$, the Tate diagonal map 
\begin{displaymath}
\Delta_{p}:X\to (X^{\otimes p})^{t\rC _{p}}
\end{displaymath}

\noindent is $p$-adic completion.
\end{theorem}

The following makes use of the definition of a symmetric monoidal
{\qcat} and related notions in \cref{sec-SM-qcat}.

\begin{defin}\label{def-Alg-E1}
\cite[Definition III.2.1]{NS18} The {\bf {\qcat} $\Alg_{\EE_{1}}
(\qSp)$ of $\EE_{1}$-ring spectra} is the {\qcat} of functors
$R^{\otimes }$ from $N (\Assoc^{\otimes })$ to $\qSpot$ over $N
(\Fin_{*})$ as shown below, where the sequence $(R,\dotsc ,R)$ has $n$
coordinates and the triangle commutes.
\begin{displaymath}
\xymatrix
@R=1mm
@C=8mm
{
  & & &{( R,\dotsc ,R)}\ar@/^3pc/@{|->}[dddd]^(.5){}
\\
{\langle n \rangle_{\Assoc}}\ar@/^/@{|->}[rrru]^(.5){}
                          \ar@{|->}[dddrrr]^(.5){}
  &{N (\Assoc^{\otimes  } )}\ar[rr]^(.55){R^{\otimes }}\ar[ddrr]^(.5){}
    & &{\qSpot}\ar[dd]^(.5){\pr}\\
{}\\
  & & &{N (\Fin_{*})} \\
  & & &{\langle n \rangle_{*}}   
}
\end{displaymath}
\end{defin}


\begin{defin}\label{def-THH-qcat}
\cite[Definition III.2.3]{NS18} 
{\bf The {\qcatal} form of $\THH$}
sends an $\EE_{1}$-ring spectrum $R^{\otimes }$ as above to the geometric
realization of the cyclic spectrum
\begin{displaymath}
\xymatrix
@R=1mm
@C=15mm
{
{N (\blamb^{\op})}\ar[r]^(.5){N (V^{\op})}
  &{N (\Assoc^{\otimes }_{\act })}\ar[r]^(.6){R^{\otimes }}
    &{\qSpot}\ar[r]^(.5){\otimes }
      &{\qSp}\\
{[n]_{\blamb}}\ar@{|->}[rrr]^(.5){}
  & & &{R^{\otimes  (n+1)},}
}
\end{displaymath}

\noindent where $V^{\op}$ is as in \cref{eq-V-op}.
\end{defin}

The following is needed by Hahn and Wilson \cite{Hahn-Wilson} in their
study of the relative spectrum $\THH (\BPn/MU)$.

\begin{defin}\label{def-relative-THH}
{\bf Relative $\THH$.}  Let $A$ be an $\EE_{\infty }$-ring spectrum,
and let $\qMod_{A}$ be the {\qcat} of $A$-module spectra. It is a
special case of the category ${\rm LMod}_{A} (\qSp)$ of
\cite[Notation 7.1.1.1]{Lurie:HA}.  (Lurie only requires $A$ to be
associative, so he has to distinguish between left and right modules.
Since our $A$ is commutative, modules over it are two sided.)
$\qMod_{A}$ is symmetric monoidal under the relative smash product
$X\otimes_{A}Y$, which is defined to be the coequalizer of the two
maps 
\begin{displaymath}
\xymatrix
@R=4mm
@C=10mm
{
{X\otimes A\otimes Y}\ar@<.5ex>[r]^(.5){}\ar@<-.5ex>[r]^(.5){}
  &{X\otimes Y}
}
\end{displaymath}

\noindent induced by the right $A$-module structure on $X$ and the left one
on $Y$.

For an $\EE_{1}$-algebra $R^{\otimes }$ in $\qMod_{A}$, $\THH (R/A)$
the geometric realization of the cyclic spectrumin the category of
$A$-modules,
\begin{displaymath}
\xymatrix
@R=1mm
@C=15mm
{
{N (\blamb^{\op})}\ar[r]^(.5){N (V^{\op})}
  &{N (\Assoc^{\otimes }_{\act })}\ar[r]^(.6){R^{\otimes_{A}}}
    &{\qMod}_{A}^{\color{cyan}\otimes }\ar[r]^(.5){\otimes }
      &{\qMod_{A}}\\
{[n]_{\blamb}}\ar@{|->}[rrr]^(.5){}
  & & &{R^{\otimes_{A}  (n+1)}.}
}
\end{displaymath}
\end{defin}

\subsection{The {\qcat} of cyclotomic spectra}\label{sec-qcat-cycsp}

Defining the {\qcat} $\qSp^{G}$ of orthogonal $G$-spectra involves
some technicalities that would distract us here, so we postpone that
discussion to \cref{sec-qcat-Gspec}, specifically
\cref{def-qcat-Gspec}.  The definition of a symmetric monoidal
structure on an {\qcat} $\qmcC$ is also complicated and is the subject
of \cref{sec-SM-qcat}, specifically \cref{def-symm-mon-qcat}.

\begin{defin}\label{def-cyclo-spec2}
\cite[Definition II.1.1]{NS18}

\begin{enumerate}[label={(\roman*)},itemindent=0em]
\item \label{def-cyclo-spec2i} 
A {\bf cyclotomic spectrum} $X $ is an
object in $\qSp^{\mT}$ (see \cref{def-qcat-Gspec}) with a
$\mT$-{\eqvr} {\bf cyclotomic structure map} (which Nikolaus and
Scholze call the {\em Frobenius map})
\begin{displaymath}
\varphi_{p}:X\to X^{t\rC _{p}},
\end{displaymath}

 \noindent where $X^{t\rC _{p}}$ is as in \cref{eq-Tate2}, for each prime
$p$.  We will write this as $(X,(\varphi_{p})_{p\in \PP})$, where
$\PP$ denotes the set of primes.

\item \label{def-cyclo-spec2ii} A {\bf $p$-cyclotomic spectrum} $X $
is an object in $\qSp^{\rC _{p^{\infty }}}$ (where $\rC _{p^{\infty
}}\subset \mT $ is the Pr\"uffer group consisting of the $p^{i}$th
roots of unity for all $i\geq 0$) with a $\rC _{p^{\infty}}$-{\eqvr}
map $\varphi_{p}$ as above.

\item \label{def-cyclo-spec2iii} 
A {\bf $p$-cyclotomic spectrum with
Frobenius lift} $X $ is an object in $\qSp^{\rC _{p^{\infty }}}$ with
a $\rC _{p^{\infty}}$-{\eqvr} map $F:X\to X^{h\rC _{p}}$ such that
$\varphi_{p} = s_{X}^{\rC_{p}}F$ for $s_{X}^{\rC_{p}}$ as in
\cref{eq-Tate2}.  This map $F$, which is a lifting of $\varphi_{p}$,
is not to be confused with the restriction map $\Frob^{G}_{H}$ of
\cref{def-residual}\cref{def-residuali}.
\end{enumerate}
\end{defin}

The word ``Frobenius'' (as a map) has several meanings in the
literature.  It has long been used for the $p$th power map in an
algebra in characteristic $p$.  Nikolaus-Scholze use it for the
cyclotomic structure map while McCandless uses it for a lifting of the
latter to $X^{h\rC_{p}}$ as in \cref{def-cyclo-spec2iii} above.
McCandless also has a more general notion of Frobenius lift in an
{\qcat} in \cite[Definition 2.1.2]{McCandless-curves} that involves
the {\Wittmonoid} $\mathscr{M}:=\mT \rtimes \mN^{\times }$ of
\cref{def-Witt-monoid}; see \cref{def-Frob-lifts}.

\begin{defin}\label{def-cyc-spectra}
\cite[Definition II.1.6]{NS18} and \cite[Definition 2.3]{BHLS}.

\begin{enumerate}[label={(\roman*)},itemindent=0em]
\item \label{def-cyc-spectrai} 
The {\bf {\qcat} $\qCycSp$ of cyclotomic spectra} is the lax equalizer (as
in \cref{def-lax-eq}) in which $\qmcC $ and $\qmcD $ are each
$\qSp^{\mT }$, and the two functors are the identity and the one
induced by $\prod_{p}\varphi_{p}$.

\item \label{def-cyc-spectraii}
That of {\bf $p$-cyclotomic spectra}, $\qCycSp_{p}$, is the lax equalizer in
which $\qmcC $ and $\qmcD $ are each $\qSp^{\rC _{p^{\infty }}}$, and the
two functors are the identity and the one induced by $\varphi_{p}$.

\item \label{def-cyc-spectraiii}
That of {\bf $p$-cyclotomic spectra with Frobenius lifts},
$\qCycSp_{p}^{\rm Fr}$, is the lax equalizer in which $\qmcC $ and
$\qmcD $ are each $\qSp^{\rC _{p^{\infty }}}$, and the two functors are
the identity and the one induced by $F$ as in
\cref{def-cyclo-spec2}\cref{def-cyclo-spec2iii}.
\end{enumerate}
\end{defin}

\begin{remark}\label{rem-presentable-stable-qcats}
All three of the above are presentable stable {\qcats} whose objects
are spectra with additional structure.  In each case the forgetful
functor to $\qSp$ is exact and preserves {\equi}s and small colimits.
Hence all three are enriched over spectra as in \cref{eq-map-spec}.
In each case function spectrum $\map_{\qmcC } (X,Y)$ differs
substantially from the underlying function spectra $\map_{\qSp }
(X,Y)$.
\end{remark}

\begin{prop}\label{prop-forgetful-functor}\cite[Proposition 10.3]{KN21}
The forgetful functor %
\begin{displaymath}
  U:\qCycSp_{p}^{\rm Fr} \to \qCycSp_{p}
  \qquad (X,F)\mapsto (X,s_{X}^{p}F)
\end{displaymath}
 
\noindent is left adjoint to the functor $\TR$ of \cref{def-TR}.
\end{prop}

Nikolaus and Scholze show that the following are examples.

\begin{ex}\label{ex-cycl-spectra}
{\bf Some cyclotomic spectra.}

\begin{enumerate}[label={(\roman*)},itemindent=0em]
\item \label{ex-cycl-spectrai} For an $\EE_{1}$-ring spectrum $R$,
$\THH(R)$ is cyclotomic.  It is defined as the geometric realization
of a cyclic spectrum spelled out in \cite[Definition III.2.3]{NS18}
and in \cref{thm-THH-cyclo}.

\item \label{ex-cycl-spectraii} 
Every spectrum $X$ has a trivial
action of $\mT$ and a trivial cyclotomic structure whose $p$th
component is given by the composite
\begin{displaymath}
\xymatrix
@R=4mm
@C=10mm
{
{X}\ar[r]^(.5){}
  &{X^{h\rC_{p}}}\ar[r]^(.5){s_{X}^{\rC_{p}}}
    &{X^{t\rC_{p}}}
}
\end{displaymath}

\noindent where the first map is pullback along $B\rC_{p+}\to
S^{0}$ (see \cref{eq-triv-action}) and the second one is as in
\cref{eq-Tate2}.  The Segal Conjecture implies that this composite is
$p$-adic completion when $X$ is finite.  We sometimes
denote the resulting cyclotomic spectrum by $X^{\triv}$.


\item \label{ex-cycl-spectraiii} 
Every cyclotomic spectrum is a
$p$-cyclotomic spectrum by restriction.





\end{enumerate}
\end{ex}

\begin{theorem}\label{thm-THH-cyclo}
{\bf The cyclotomic structure on $\THH (A)$} as in \cref{def-THH-qcat}
is the composite map from the geometric realization of the top row to
that of the bottom row in the following diagram of cyclic spectra.
\begin{numequation}\label{eq-Tate-diag}
\begin{tikzcd}[column sep=8mm,row sep=4mm]
{\dotsb }\arrow[r,shift left=1.2ex]\arrow[r,shift left=.4ex]
         \arrow[r,shift right=1.2ex]\arrow[r,shift right=.4ex]
  &{A^{\otimes 3}}\arrow[loop above, "\rC _{3}"]\arrow{d}{\Delta_{p}}
        \arrow[r,shift left=.8ex]\arrow[r]
         \arrow[r,shift right=.8ex]
    &{A^{\otimes 2}}\arrow[loop above, "\rC _{2}"]\arrow{d}{\Delta_{p}}
         \arrow[r,shift left=.4ex]
         \arrow[r,shift right=.4ex]
      &{A}\arrow{d}{\Delta_{p}}\\
{\dotsb }\arrow[r,shift left=1.2ex]\arrow[r,shift left=.4ex]
         \arrow[r,shift right=1.2ex]\arrow[r,shift right=.4ex]
  &{((A^{\otimes 3})^{\otimes p})^{t\rC _{p}}}\arrow{d}{\otimes }
       \arrow[r,shift left=.8ex]\arrow[r]
         \arrow[r,shift right=.8ex]
    &{((A^{\otimes 2})^{\otimes p})^{t\rC _{p}}}\arrow{d}{\otimes }
         \arrow[r,shift left=.4ex]
         \arrow[r,shift right=.4ex]
      &{(A^{\otimes p})^{t\rC _{p}}}\arrow{d}{\otimes }\\
{\dotsb }\arrow[r,shift left=1.2ex]\arrow[r,shift left=.4ex]
         \arrow[r,shift right=1.2ex]\arrow[r,shift right=.4ex]
  &{(A^{\otimes 3})^{t\rC _{p}}}
       \arrow[r,shift left=.8ex]\arrow[r]
         \arrow[r,shift right=.8ex]
    &{(A^{\otimes 2})^{t\rC _{p}}}
         \arrow[r,shift left=.4ex]
         \arrow[r,shift right=.4ex]
      &{(A)^{t\rC _{p}},}
\end{tikzcd}
\end{numequation}%

\noindent where $\Delta_{p}$ is the Tate diagonal of
\cref{def-Tate-diagonal} and the action of $\rC_{p}$ on the spectra in
the bottom row is trivial.
\end{theorem}



\proof
 The geometric realizations of the first and third rows of
\cref{eq-Tate-diag} are $\THH(A)$ and $\THH(A)^{t\rC _{p}}$, and the
composite map between them is the desired cyclotomic structure on the
former.
The second row requires further explanation.
The $p$-fold subdivision endofunctor $\sd_{p}$ of $\bdelt$ (as in
\cref{def-edgewise}) induces one on $\blamb_{p}$ (in which $\bdelt$ is
a wide subcategory), which we follow by the projection onto $\blamb$
to give a functor which we abusively denote by 
\begin{displaymath}
\xymatrix
@R=0mm
@C=10mm
{
{\blamb_{p}^{\op}}\ar[r]^(.5){\sd_{p}}
  &{\blamb^{\op}}\\
{[n-1]_{\blamb_{p}}=\pow{n}_{\blamb_{p}}}\ar@{|->}[r]^(.5){}
    &{\pow{pn}_{\blamb}=[pn-1]_{\blamb} }
}
\end{displaymath}

\noindent The object on the lower right is a free $\rC _{p}$-set. We
denote the category of such sets by $\Free_{\rC _{p}}$ and a free
$\rC _{p}$-set with $pn$ elements by $\pow{pn}$.  For such a set $S$, we
denote its orbit set by $\overline{S} $.  Thus we have a functor
$\Free_{\rC _{p}}\to \Fin$ sending $S$ to $\overline{S} $, and for each
$n>0$ a composite

\begin{numequation}\label{eq-ITp}
\begin{split}
\xymatrix
@R=6mm
@C=0mm
{
&{N (\blamb_{p}^{\op})}\ar[d]^(.5){}
  &&&&{\pow{n}_{\blamb_{p}}}\ar@{|->}[d]^(.5){}\\
&{N (\Free_{\rC _{p}})\underset{N (\Fin)}{\times}N (\Assoc^{\otimes }_{\act})
           \hspace{-1cm}}
        \ar[d]_(.52){A^{\otimes }}
  &&&&{(\pow{pn}, \langle pn \rangle_{\Assoc})}
         \ar@{|->}[d]^(.5){}\\
&{N (\Free_{\rC _{p}})\underset{N (\Fin)}{\times}\qSp^{\otimes }_{\act}}
      \ar[d]_(.5){\Theta }\ar[ldd]_(.5){I}\ar[rdd]^(.5){\widetilde{T}_{p}}
  &&&&{(\pow{pn}, (A^{\otimes p},\dotsc ,A^{\otimes p}))
           \hspace{-1cm}}
         \ar@{|->}[d]^(.5){}\\
&{(\qSp^{\otimes}_{\act})^{B\rC _{p}}}
         \ar[rd]^(.5){}\ar[ld]^(.5){}
  &&&&{(A^{\otimes p},\dotsc ,A^{\otimes p}))}
         \ar@{|->}[ld]^(.5){}\ar@{|->}[rd]^(.5){}\\
{\qSp}\ar@{-->}[rr]^(.5){}
  &&{\qSp}
    &&{A^{\otimes np}} \ar@{|->}[rr]^(.5){}
     &&{(A^{\otimes n}})^{t\rC _{p}}
}
\end{split}
\end{numequation}%

\noindent Here an object in the third category is a pair $(S,
(X_{\overline{s} })_{\overline{s}\in \overline{S}})$ where $S$ is a
free $\rC _{p}$-set and each $X_{\overline{s}}$ is a spectrum.  This
category admits a functor $\Theta $ to
$(\qSp^{\otimes}_{\act})^{B\rC _{p}}$ by \cite[Proposition
III.3.6]{NS18}, where it is not named.  The functors $I$ and
$\widetilde{T}_{p}$ are given by
\begin{displaymath}
(S,(X_{\overline{s} })_{\overline{s}\in \overline{S}})
  \mapsto \bigotimes_{\overline{s}\in \overline{S}} X_{\overline{s} }
\qquad \aand 
(S,(X_{\overline{s} })_{\overline{s}\in \overline{S}})
  \mapsto \left(\bigotimes_{s\in S}X_{s} \right)^{t\rC _{p}}\hspace{-5mm},  
\end{displaymath}

\noindent where the action of $\rC _{p}$ on the smash product on the
right is induced by its free action on $S$.  Both functors are lax
symmetric monoidal.  In \cite[Lemma III.3.7]{NS18} they show that $I$
is an initial object in the {\qcat} of all such functors satisfying a
certain exactness condition.  \cite[Corollary III.3.8]{NS18} asserts
there is a unique natural transformation from $I$ to
$\widetilde{T}_{p}$, which is indicated by the broken arrow in the
bottom row of \cref{eq-ITp}.

In the image of $\pow{n}_{\blamb_{p}}$ in the fourth category of
\cref{eq-ITp}, there are $n$ spectrum coordinates, and $\rC _{p}$ acts on
each of them by permuting its factors.

Thus we
have a composite
\begin{numequation}\label{eq-big-comp}
\begin{split}
\xymatrix
@R=1mm
@C=1mm
{
&{\pow{n}_{\blamb_{p}}}\ar@{|->}[rr]^(.5){}
                       \ar@/_2.61pc/@{|->}[dddddd]^(.5){}
  &&{(\pow{pn}, (A^{\otimes p},\dotsc ,A^{\otimes p}))}
                      \ar@{|->}[rrr]^(.5){}
    && &{(A^{\otimes p},\dotsc ,A^{\otimes p})}\ar@{|->}[ddd]^(.5){}\\
&{N (\blamb_{p}^{\op})}\ar[rr]^(.35){A^{\otimes }}
                       \ar[dddd]^(.5){N (\Proj_{p,1})}
  &&{N (\Free_{\rC _{p}})\underset{N (\Fin)}{\times}\qSp^{\otimes }_{\act}}
         \ar[rr]^(.75){\Theta }
         \ar[rrdddd]_(.5){\widetilde{T}_{p}}
    &&{(\qSp^{\otimes}_{\act})^{B\rC _{p}}\hspace{-1cm}}
                      \ar[dd]^(.45){(\otimes)^{B\rC _{p}}\hspace{-1cm} }\\
{}\\
& && &&{\qSp^{B\rC _{p}}}
                      \ar[dd]^(.45){(-)^{t\rC _{p}}}
       &{A^{\otimes np}}\ar@{|->}[ddd]^(.5){}\\
{}\\
&{N (\blamb^{\op})}\ar[rrrr]^(.5){}
  && &&{\qSp}\\
&{\pow{n}_{\blamb}=[n-1]_{\blamb}}\ar@{|->}[rrrrr]^(.5){}
 && &&
      &{(A^{\otimes np})^{t\rC _{p}},}
}
\end{split}
\end{numequation}%

\noindent where $\Proj_{p,1}$ is the projection functor of
\cref{eq-proj-functor}.  The middle row of \cref{eq-Tate-diag} is
defined by the bottom functor $N (\blamb^{\op})\to \qSp$ above.
\qed

\subsection{Polygonic spectra}\label{sec-polygonic}
Polygonic spectra, which are controlled by a truncation set
(\cref{def-polygonic}) of natural numbers $T$ , are first defined and
studied by Krause, McCandless and Nikolaus in \cite{KMN23}.  The case
where $T=\left\{1,p \right\}$ is reviewed in \cite[\S2.1.3]{BHLS}, and
it figures in their definition of the Dehn twist trivialization of
\cite[\S4.2]{BHLS}.

Here is the topological analog of \cref{eq-Tor} and \cref{eq-HAM}.

\begin{defin}\label{def-THH-M}
The {\bf topological Hochschild homology} $\THH(R;M)$ of an
$\EE_{1}$-ring spectrum $R$ {\bf with coefficients} in an $R$-bimodule
$M$ is the geometric realization of a simplicial spectrum informally
depicted by
\begin{displaymath}
\xymatrix
@R=4mm
@C=10mm
{
{\dots }\ar@<1.5ex>[r]^(.5){}
        \ar@<.5ex>[r]^(.5){}
        \ar@<-.5ex>[r]^(.5){}
        \ar@<-1.5ex>[r]^(.5){}
  &{M\otimes R\otimes R}\ar@<1ex>[r]^(.5){}
                        \ar[r]^(.5){}
                        \ar@<-1ex>[r]^(.5){}
    &{M\otimes R}\ar@<.5ex>[r]^(.5){}
                 \ar@<-.5ex>[r]^(.5){}
      &{M.}
}
\end{displaymath}
\end{defin}

The simplicial structure here does not extend to a cyclic one, so the
spectrum $\THH (R,M)$ does not have a cyclotomic structure.  The point
of polygonic spectra, which we now define, is to provide a substitute
for such a structure, as explained  in \cite{KMN23}.

For them, part of the data is the following.

\begin{defin}\label{def-polygonic}
\cite[Definition 2.1 and Example
2.2]{KMN23}.  A {\bf truncation set} $T$ is a set of positive integers
closed under divisors.  Examples include
\begin{align*}
\ltrunc n \rtrunc
 & := \left\{m\in \mN_{>0}:m|n \right\} &&\mbox{for }n\in \mN_{>0}   \\
\ltrunc \infty  \rtrunc
 & := \left\{m\in \mN_{>0}\right\} \\
\ltrunc p^{\infty } \rtrunc
 & :=\left\{p^{k}:k\geq 0 \right\}&&\mbox{for $p$ prime} \\
T/n
 & := \left\{t\in  T:nt\in T \right\} 
     &&\mbox{for a truncation set $T$ and $n\in \mN_{>0}$.} 
\end{align*}
\end{defin}

In \cite{KMN23} and \cite{BHLS} the first three truncation sets are
denoted by $\langle - \rangle$.  We use the notation 
$\ltrunc - \rtrunc$ instead to avoid conflicting with the notation of
\cref{def-fin*}.

\begin{defin}\label{def--polygonic}\cite[Definition 2.4]{KMN23}.  
A {\bf $T$-typical polygonic spectrum} $X_{\pentagon}$ consists of the
following data:
\begin{enumerate}[label={(\roman*)},itemindent=0em]

\item A $T$-indexed collection of spectra $\left\{X_{t}:t\in T
\right\}$, where each $X_{t}$ is an object of $\qSp^{B\rC _{t}}$.

\item For every prime number $p$ and integer $t$ with $pt\in T$, a
$\rC _{t}$-equivariant map of spectra $\varphi_{p,t}:X_{t}\to
(X_{pt})^{t\rC _{p}}$, where the Tate construction on $X_{pt}$
carries the residual action of $\rC _{pt}/\rC _{p}\cong \rC _{t}$.
\end{enumerate}

For $T=\ltrunc p^{k} \rtrunc$, $T=\ltrunc p^{\infty } \rtrunc$ and
$T=\ltrunc \infty  \rtrunc$, we call these $p^{k}$-, 
$p^{\infty }$- and {\bf integer polygonic spectra.}
\end{defin}

\begin{ex}\label{ex-poly-spectra}
{\bf Some polygonic spectra.} \cite[Example 1.2]{KMN23}
\begin{enumerate}[label={(\roman*)},itemindent=0em]

\item \label{ex-poly-spectra0} A $\ltrunc 1 \rtrunc$-typical polygonic
spectrum is an ordinary spectrum.
 
\item \label{ex-poly-spectrai} For a cyclotomic spectrum $X$ with
  structure maps $\varphi_{p}:X\to X^{t\rC _{p}}$ for each prime $p$,
  for any $T$ we can define a constant $T$-typical polygonic spectrum
  $(i_{T}X)_{\pentagon}$ by $(i_{T}X)_{t}=X$ for all $t\in T$.  Since
  $X$ is a $\mT$-spectrum it is also a $\rC_{t}$-spectrum for each
  $t\in T$, and the required maps $\varphi_{p,t}$ are induced by the
  structure maps $\varphi_{p}$. See \cref{thm-cyclo-poly}.

\item \label{ex-poly-spectraii} 
The topological Hochschild homology of
an $\EE_1$-ring spectrum $R$ with coefficients in an $R$-bimodule $M$
canonically admits the structure of an integer polygonic spectrum
$\THH(R;M)_{\pentagon}$ with
\begin{displaymath}
\THH(R;M)_{t} := \THH(R, M^{\otimes_{R}t})
\end{displaymath}

\noindent for every $t\geq 1$. This is the motivating example for
\cite{KMN23}.  Their main result is \cref{thm-KMN23-main}.

\item \label{ex-poly-spectraiii} 
$\mL:=L_{\ltrunc p^{\infty } \rtrunc}\mS$
is the $p^{\infty }$-polygonic spectrum whose $p^{j}$th component (in
the notation of \cref{def-rep-T}) is
\begin{displaymath}
\mL _{p^{j}}:=
\mS[\mT/\rC _{p^{j}}].
\end{displaymath}

\noindent This spectrum is $L_{\ltrunc p^{\infty }\rtrunc}i_{\ltrunc
p^{\infty } \rtrunc}\mS$ in the notation of \cref{thm-cyclo-poly}, as
noted in the footnote to \cite[Example 2.5]{BHLS}.

It follows that we have a diagram
\begin{numequation}\label{eq-TR-diagram}
\begin{split}
\xymatrix
@R=4mm
@C=5mm
{
{\mL_{1}}\ar[dr]^(.5){\varphi_{p,1}}
  & &{\mL_{p}}\ar[dr]^(.5){\varphi_{p,p}}
            \ar[dl]^(.35){\theta^{\rC_{p}}}
      & &{\mL_{p^{2}}}\ar[dr]^(.5){\varphi_{p,p^{2}}}
            \ar[dl]^(.35){\theta^{\rC_{p}}}
          & &{\dotsb }\ar[dl]^(.45){}\\
  &{\mL_{p}^{t\rC_{p}}}
    & &{\mL_{p^{2}}^{t\rC_{p}}}
        & &{\mL_{p^3}^{t\rC_{p}}}
}
\end{split}
\end{numequation}%

\noindent in which $\theta^{\rC_{p}}$ is the map of \cref{def-can}.
The image of this under the functor $\Upsilon$ of
\cref{def-U-poly}, namely the wedge of all the $\mS[\mT/\rC_{p^{j}}]$,
corepresents the functor $\TR$ of \cref{def-TR} in $\qCycSp_{+}$, the
{\qcat} of bounded below cyclotomic spectra, by \cite[Lemma
2.6]{BHLS}.



\item \label{ex-poly-spectraiv}
 For a $p$-polygonic spectrum 
\begin{displaymath}
X_{\pentagon}= (X_{1}, X_{p},  \varphi_{p,1}:X_{1}\to X_{p}^{t\rC _{p}}),
\end{displaymath}

\noindent \cite[page 15]{BHLS} defines
\begin{displaymath}
X_{\pentagon}^{\Phi \rC _{p}}:=X_{1},\quad 
X_{\pentagon}^{\Phi e}:=X_{p},\quad 
X_{\pentagon}^{h\rC _{p}}:=X_{p}^{h\rC _{p}},\quad 
X_{\pentagon}^{t\rC _{p}}:=X_{p}^{t\rC _{p}},
\end{displaymath}

\noindent and $X_{\pentagon}^{\rC _{p}}$ is the pullback in 
\begin{displaymath}
\xymatrix
@R=8mm
@C=15mm
{
{X_{\pentagon}^{\rC _{p}}}\ar[r]^(.5){}\ar[d]^(.5){}
           \ar@{}[dr]|-(.15){\pb}
  &{X_{1}}\ar[d]^(.5){\varphi_{p,1}}\\
{X_{p}^{h\rC _{p}}}\ar[r]^(.5){s_{X_{p}}^{\rC_{p}}}
    &{X_{p}^{t\rC _{p}},}
}
\end{displaymath}

\noindent where $s^{\rC _{p}}_{X_{p}}$ is the map of \cref{eq-Tate2}.
\end{enumerate} 
\end{ex}

\begin{defin}\label{def-U-poly}
{\bf The cyclotomic spectrum associated with a $p^{\infty }$-polygonic
spectrum.} \cite[Construction 2.20]{KMN23} For the data 
\begin{displaymath}
X_{\pentagon}
=\left\{X_{p^{j}},
\varphi_{p,p^{j}}:X_{p^{j}}\to (X_{p^{j+1}})^{t\rC_{p}}
   :j\geq 0 \right\},
\end{displaymath}

\noindent we define
\begin{displaymath}
\Upsilon X_{\pentagon}
:=\bigoplus_{j\geq 0}X_{p^{j}} 
\end{displaymath}

\noindent with structure map $\varphi
^{\Upsilon X_{\pentagon}}_{p}:\Upsilon X_{\pentagon}\to
(\Upsilon X_{\pentagon})^{t\rC_{p}}$ induced by the maps
$\varphi_{p,p^{j}}$.

\end{defin}

One can define the cyclotomic spectrum associated with an integer polygonic
spectrum in a similar way.

\begin{defin}\label{def-PgcSp}
 {\bf {\qCats} of polygonic spectra.} \cite[Definition 2.6 and Example
2.9]{KMN23}.  The {\qcat} of
$T$-typical polygonic spectra $\qPgcSp_{T}$ is the lax equalizer as in
\cref{def-lax-eq} where
\begin{align*}
\qmcC & = \prod_{t\in T} \qSp^{B\rC _{t}},\\
\qmcD & = \prod_{p\in T\atop 
           p\,\, {\rm prime} }\prod_{s\in T/p } \qSp^{B\rC _{s}},
\end{align*}

\noindent and the $p$th components of the two functors $\qmcC\to \qmcD $
are the identity and the functor
\begin{displaymath}
F_{p}:\prod_{t\in T} \qSp^{B\rC _{t}}
 \to \prod_{s\in T/p } \qSp^{B\rC _{s}}
\end{displaymath}

\noindent which for each $s\in T/p$ is induced by the composite
\begin{displaymath}
\xymatrix
@R=4mm
@C=20mm
{
\displaystyle{\prod_{u\in T}\qSp^{B\rC_{u}}}\ar[r]^(.55){\pr}
  &{\qSp^{B\rC_{ps}}}\ar[r]^(.5){(-)^{t\rC_{ps}}}
    &{\qSp^{B\rC_{s}}.}
}
\end{displaymath}

For truncation sets $T'\subseteq T$ there is a restriction functor 
\begin{displaymath}
\qPgcSp_{T} \to \qPgcSp_{T'}.
\end{displaymath}

\noindent We abbreviate $\qPgcSp_{\ltrunc \infty \rtrunc}$ by
$\qPgcSp$, and (as indicated in \cite[Definition 2.8]{BHLS}).  The
{\qcat} of $p$-polygonic spectra $\qPgcSp_{\ltrunc p \rtrunc}$ is the
pullback in
\begin{displaymath}
\xymatrix
@R=6mm
@C=15mm
{
{\qPgcSp_{\ltrunc p \rtrunc}}
           \ar[r]^(.5){}\ar[d]^(.5){}\ar@{}[dr]|-(.15){\pb}
  &{\qSp^{\bdelt^{1}}}\ar[d]^(.5){{\rm ev_{1}}}\\
{\qSp^{B\rC _{p}}}\ar[r]^(.5){(-)^{t\rC _{p}}}
    &{\qSp.}
}
\end{displaymath}

\end{defin}

\begin{theorem}\label{thm-KMN23-main}
{\bf Polygonic spectra and $\THH$ with coefficients.}  \cite[Theorem
6.31]{KMN23} Let $R$ denote an $\mathbb{E}_1$-ring and let
$\qBMod_{R}$ denote the {\qcat} of two sided $R$-module spectra. The
topological Hochschild homology functor with coefficients
\[
\mathrm{THH}(R;-) \colon \qBMod_R \to \qSp
\]

\noindent of \cref{def-THH-M} canonically refines to a functor of
$\infty$-categories
\[
\mathrm{THH}(R;-)_{\pentagon} \colon \qBMod_R \to \qPgcSp
\]

\noindent with $\mathrm{THH}(R;-)_{t}$ as in
\cref{ex-poly-spectra}\cref{ex-poly-spectraii}.

We get a similar functor to $\qPgcSp_{T}$ for any $T$ by restriction
of truncation sets.
\end{theorem}

\begin{defin}\label{def-res-pent}
{\bf Some functors to and from $\qPgcSp_{\ltrunc p \rtrunc}$.}
\cite[Definition 2.9]{BHLS} 

\begin{enumerate}[label={(\roman*)},itemindent=0em]
\item \label{def-res-penti}
The functor 
$\res_{\pentagon}:\qCycSp\to \qPgcSp_{\ltrunc p \rtrunc}$
sends a cyclotomic spectrum $X$ to the
$p$-polygonic spectrum
\begin{displaymath}
 (X, X,  \varphi_{p}:X\to X^{t\rC _{p}}).
\end{displaymath}

\item \label{def-res-pentii}
We denote by  
\begin{displaymath}
\res_{\varphi } :\qPgcSp_{\ltrunc p \rtrunc}\to
\qSp^{\bdelt^{1}}
\end{displaymath}

\noindent the lax symmetric monoidal functor sending $(X_{1},
X_{p},\,\, \varphi_{p,1}:X_{1} \to X_{p}^{t\rC _{p}})$ to the morphism
$\varphi_{p,1}$.
\end{enumerate}
\end{defin}

\begin{defin}\label{def-EQ}
Let $\EQ$ be the {\bf equalizer category} $\bullet \rightrightarrows
\bullet$, and let $\qmcC^{\EQ}$ denote the category of equalizer
diagrams in $\qmcC$. 
\end{defin}

An object $R \rightrightarrows S$ in $\Alg (\qSp)^{\EQ}$ is an
$R$-bimodule structure on the ring spectrum $S$, so the functor
$\mathrm{THH}(R;-)_{\pentagon}$ of \cref{thm-KMN23-main} leads to
\[
\mathrm{THH}_{\pentagon} \colon \Alg (\qSp)^{\EQ} \to \qPgcSp.
\]

\begin{lem}\label{lem-BHLS-2.16}
\cite[Lemma 2.16]{BHLS} 
There is a commuting diagram of lax symmetric monoidal functors
\begin{displaymath}
\xymatrix
@R=4mm
@C=15mm
{
{\Alg(\qSp)}\ar[r]^(.5){\THH}\ar[dd]^(.5){}
  &{\qCycSp}\ar[dr]^(.5){}
            \ar[dd]^(.5){\mathrm{res}_{\pentagon}}\\
  & &{\qSp^{\Delta^1}}\\
{\Alg(\qSp)^{\EQ}}\ar[r]_(.5){\THH_{\pentagon}}
  &{\qPgcSp_{\ltrunc p \rtrunc},}\ar[ur]_(.5){\mathrm{res}_\varphi}
}
\end{displaymath}

\noindent where the left vertical map sends an $\mathbb{E}_1$-algebra
$R$ to the constant diagram \( R \rightrightarrows R.  \)
\end{lem}

\begin{lem}\label{lem-BHLS-2.17}
\cite[Lemma 2.17]{BHLS}
Suppose we are given $(A,B) \in \Alg(\qSp)^{E\mathbb{Q}}$ and 
$V \in \Alg(\qSp)$ such that the underlying spectrum of $V$ is a 
dualizable.
Then the natural map
\[
 \THH_{\pentagon}(\mathbb{S};V)
  \;\otimes\;
\THH_{\pentagon}(A;B)
\;\longrightarrow\;
\THH_{\pentagon}\bigl(A;\, V \otimes B\bigr),
\]

\noindent coming from the lax symmetric monoidal structure on
$\mathrm{res}_\varphi\, \THH(-;-)$, is an isomorphism.
\end{lem}

\begin{defin}\label{def-spectral-line}
\cite[Notation 1.1.1]{McCandless-curves}
The {\bf spectral affine line} is the $\EE_{\infty }$-ring
\begin{displaymath}
\mS[x]:= \Sigma^{\infty }_{+}\mN
\end{displaymath}

\noindent with multiplication induced by addition in $\mN$.
$\mS[x]/x^{n}$ is the following pushout in $\EE_{\infty }$-rings
\begin{numequation}\label{eq-nth-pushout}
\begin{split}
\xymatrix
@R=8mm
@C=20mm
{
{\mS[x]}\ar[r]^(.5){x\mapsto x^{n}}
        \ar[d]_(.5){x\mapsto 0}
        \ar@{}[dr]|-(.85){\po}
  &{\mS[x]}\ar[d]^(.5){}\\
{\mS}\ar[r]^(.5){}
  &{\mS[x]/x^{n}}
}
\end{split}
\end{numequation}%

\noindent The reduced topological Hochschild homologies of $\mS[x]$
and $\mS[x]/x^{n}$, $\widetilde{\THH} (\mS[x])$ and
$\widetilde{\THH}(\mS[x]/x^{n})$, are the fibers of the maps of
cyclotomic spectra
\begin{align*}
\THH(\mS[x])
 & \to \THH(\mS)=\mS & \mbox{induced by
the left map in \cref{eq-nth-pushout}} \\
\aand \THH(\mS[x]/x^{n})
 & \to \mS& \mbox{induced by a similar map $\mS[x]/x^{n}\to \mS$}
\end{align*}

\noindent with the trivial cyclotomic structure on $\mS$.

For a cyclotomic spectrum $X$,
\begin{align*}
X\widehat{\otimes  }\widetilde{\THH}^{\rm cont}
   (\mS\llbracket x\rrbracket )
 & := \lim{n} (X\otimes \widetilde{\THH}(\mS[x]/x^{n}))\\
\aand 
X\widehat{\otimes}\,\Omega \widetilde{\THH}^{\rm cont}
  (\mS\llbracket x\rrbracket)
 & := \lim{n} (X\otimes \Omega \widetilde{\THH}(\mS[x]/x^{n})).
\end{align*}

\noindent where the limits on the right are induced by the evident maps
$\mS[x]/x^{n+1}\to\mS[x]/x^{n}$.

For a connective $\EE_{1}^{}$-ring $R$,
\begin{displaymath}
R[x]/x^{n}:= R\otimes \mS[x]/t^{n}.
\end{displaymath}
\end{defin}

Note that the $\EE_{\infty }$-ring $\mS[x]$ is not the free
$\EE_{\infty }$-ring on one generator. However the underlying
$\EE_{1}$-ring of $\mS[x]$ is the free $\EE_{1}$-ring on one generator
since $\mN$ is the free $\EE_{1}$-monoid on one generator in the
1-category of spaces.

On \cite[page 5]{McCandless-curves} McCandless asserts without proof that
\begin{displaymath}
\widetilde{\THH} (\mS[x])
\simeq \bigoplus_{r>0}\mS[\mT/\rC_{r}].
\end{displaymath}

\noindent The computation is discussed by Catherine Li in
\cite[Example 2.11]{Talbot24}.

\begin{theorem}\label{thm-cyclo-poly}
{\bf The cyclotomic/polygonic adjunction.}  \cite[Theorem 2.24]{KMN23}
For a truncation set $T$, let $i_{T}:\qCycSp:\to \qPgcSp_{T}$ be the
functor given in \cref{ex-poly-spectra}\cref{ex-poly-spectrai}.  It
has left and right adjoints $L_{T}$ and $R_{T}$.  We omit the
subscript when $T=\ltrunc \infty  \rtrunc$.

For a cyclotomic spectrum $X$
\begin{align*}
Li (X) & \simeq X \otimes \widetilde{\THH} (\mS[x])   \\
\aand 
Ri (X) & \simeq X \,\widehat{\otimes  }\,\,
           \Omega \widetilde{\THH}^{\rm cont} (\mS \llbracket x\rrbracket)
          :=\Omega\, \lim{n} \left(X \otimes 
          \widetilde{\THH} (\mS[x]/x^{n}) \right).
\end{align*}
\end{theorem}

A formula for $R_{T}$ is given in \cite[Definition 2.23]{KMN23}.

\subsection{Epicyclic spaces and spectra}\label{sec-epicyc}

Recall \cref{thm-circle-action}, which says that the geometric
realization of a cyclic space has an action of the circle group $\mT$.
In this subsection we will state a similar result for epicyclic spaces
in {\qcatal} language, where by \cref{def-epicyc-object} an epicyclic
object in a category $\mcC $ is a $\mcC $-valued functor on
$\widetilde{\blamb} ^{\op}$.

\begin{defin}\label{def-epi-qcat}
\cite[Definition 2.1.14]{McCandless-curves}, {\bf An epicyclic object
in an {\qcat} $\qmcC$} is a map of simplicial sets $N
(\widetilde{\blamb}^{\op})\to \qmcC$, that is a $\qmcC$-valued
presheaf on $\widetilde{\blamb}$ as in
\cref{def-qcats}\cref{def-qcatsvii}.
\end{defin}

For the {\Wittmonoid} $\mathscr{M}$ of \cref{def-Witt-monoid}, we have
a one object topological category $\mcB\mathscr{M}$ with $\mathscr{M}$
as endomorphism space.  It contains the self dual $\mcB \mT$ as a
subcategory. Recall that a functor $\mcB \mT\to \qmcC$ is an object
$X$ in $\qmcC$ equipped with an action of $\mT$ and hence with each of
its finite subgroups $\rC_{r}$. When $\qmcC$ has finite limits and
colimits, the limit and colimit of the functor restricted to $\mcB
\rC_{r}$ are $X^{h\rC_{r}}$ and $X_{h\rC_{r}}$.

\begin{defin}\label{def-Frob-lifts}
\cite[Definition 2.1.2]{McCandless-curves}.
For an {\qcat} $\qmcC$, the {\bf
{\qcat} of objects in $\qmcC$ with Frobenius lifts} is
\begin{displaymath}
\qmcC^{\Fr} := \qFun (\mcB\mathscr{M}^{\op}, \qmcC)
    =\mathpzc{P}_{\qmcC}(\mcB\mathscr{M}),
\end{displaymath}

\noindent the category of $\qmcC$-valued presheaves on $\mcB\mathscr{M}$,
where $\msM$ is the {\Wittmonoid} of \cref{def-Witt-monoid}.  We omit the
subscript when $\qmcC=\qmcS$, the {\qcat} of topological spaces.
\end{defin}

Such a functor defines an object $X$ in $\qmcC$ with a $\mT$-action
since $\mcB \mT \subseteq \mcB \mathscr{M}$.  If $\qmcC$ has finite limits and
colimits, there are $\mT$-{\eqvr} maps $\psi_{k}:X\to
\rho^{*}_{k}X^{h\rC_{r}}$ for each $k>0$, for $\rho^{*}_{k}$ as in
\cref{def-rho-n}.  These are compatible in that for two integers
$k_{1}, k_{2}>0$, we have a commutative diagram
\begin{displaymath}
\xymatrix
@R=8mm
@C=10mm
{
{\rho_{k_{1}}^{*}X^{h\rC_{1}}}
           \ar[d]_(.5){\rho_{k_{1}}^{*} (\psi_{k_{2}})^{h\rC_{k_{1}}}}
  &{X}\ar[l]_(.4){\psi_{k_{1}}}
      \ar[r]^(.4){\psi_{k_{2}}}
    &{\rho_{k_{2}}^{*}X^{h\rC_{2}}}
           \ar[d]^(.5){\rho_{k_{2}}^{*} (\psi_{k_{1}})^{h\rC_{2}}}\\
{\rho_{k_{1}k_{2}}^{*} (X^{h\rC_{k_{1}}})^{h\rC_{k_{2}}}}
    \ar[rr]^(.5){\simeq }
  & &{\rho_{k_{2}k_{1}}^{*} (X^{h\rC_{k_{2}}})^{h\rC_{k_{1}}}.}
}
\end{displaymath}

\begin{prop}\label{prop-conservative}
\cite[Proposition 2.2.1]{McCandless-curves}.
The forgetful functor
\begin{displaymath}
\mathpzc{P}(\mcB\mathscr{M}) \to \mathpzc{P}(\mcB\mT)
\end{displaymath}

\noindent
from spaces
with Frobenius lefts to $\mT$-spaces is conservative and preserves
small limits and colimits.
\end{prop}

\begin{theorem}\label{thm-geo-epi}
\cite[Proposition 2.1.15]{McCandless-curves}.
{\bf Geometric realizations of cyclic and epicyclic objects.}  
Let $\qmcC$ be an {\qcat} that admits geometric realizations. 
Then there is a commutative diagram
\begin{displaymath}
\xymatrix
@R=8mm
@C=20mm
{
  {\mathpzc{P}_{\qmcC} (\widetilde{\blamb})}
     \ar[r]^(.5){}\ar[d]^(.5){}
    &{\mathpzc{P}_{\qmcC} (\mcB\mathscr{M})=\qmcC^{\Fr}}
     \ar[d]^(.5){}\\
  {\mathpzc{P}_{\qmcC} (\blamb)}
     \ar[r]^(.5){}
    &{\mathpzc{P}_{\qmcC} (\mcB\mT)=\qmcC^{B\mT}}
}
\end{displaymath}

\noindent in which the vertical arrows are induced by the inclusions
$\blamb\to \widetilde{\blamb}$ and $\mT\to \mathscr{M}$, and each horizontal
arrow is geometric realization of the corresponding simplicial object
(induced by the inclusion functors $\bdelt \to \blamb \to
\widetilde{\blamb}$) in $\qmcC$.
\end{theorem}

\begin{ex}\label{ex-free-loop-epi}
\cite[Example 2.2.2]{McCandless-curves}.
{\bf The free loop space.}  For a space $X$ we define an epicyclic space by 
\begin{displaymath}
\pow{n}_{\widetilde{\blamb}} \mapsto 
\Map_{\qmcS} (|N (\pow{n}_{\widetilde{\blamb}} ) |, X)\simeq \mcL(X),
\end{displaymath}

\noindent so we may regard the free loop space $\mcL X$ as a space
with Frobenius lift by \cref{thm-geo-epi}.  This functor is constant
on the objects of $\widetilde{\blamb}^{\op}$ but not on its morphisms.
The self-map of $\mcL X$ determined by the cyclic operator $\tau_{n}$
is induced by the self-map of $\mT$ given by rotation by $2\pi /
(n+1)$ as in \cref{sec-Connes}. The self-map of $\mcL X$ determined by
the epicyclic operator $\pi^{k}_{*}$ of \cref{eq-pi-nk} is induced by
the self-map of $\mT$ given by the $k$th power map.  This is
compatible with the action of $\mathscr{M}$ on $\mcL X$ of
\cref{eq-Witt-free-loop}.
\end{ex}

\begin{defin}\label{def-E1-monoid}
An {\bf $\EE_{1}$-monoid $M$ is a map of simplicial sets $N
(T^{\op}_{\Assoc})\to \qmcS$}, where $T_{\Assoc}$ is the opposite category
of Lawvere's algebraic theory of monoids, which is described in
\cref{sec-univ-alg}.  Its
{\bf epicyclic bar construction} $B^{\epi}M$ as in \cite[Definition 
2.2.4]{McCandless-curves} is the geometric realization of the epicyclic
space given by the composite functor
\begin{displaymath}
\xymatrix
@R=0mm
@C=10mm
{
{N (\widetilde{\blamb}^{\op})}\ar[r]^(.5){j^{\op}}
  &{N (T^{\op}_{\Assoc})}\ar[r]^(.65){M}
    &{\qmcS}\\
{\pow{n}}\ar@{|->}[r]^(.5){}
  &{\langle n \rangle,}
}
\end{displaymath}

\noindent while its {\bf cyclic bar construction} $B^{\cyc}M$ as in
\cref{def-cyc-nerve} is that of
\begin{displaymath}
\xymatrix
@R=0mm
@C=10mm
{
{N (\blamb^{\op})}\ar[r]^(.5){}
  &{N (\widetilde{\blamb}^{\op})}\ar[r]^(.5){j^{\op}}
    &{N (T^{\op}_{\Assoc})}\ar[r]^(.65){M}
      &{\qmcS.}
}
\end{displaymath}

The {\bf epicyclic topological Hochschild homology} $\THH^{\epi}
(\qmcC)$ of a small {\qcat} $\qmcC$ is the geometric realization of the
epicyclic space given by
\begin{displaymath}
\pow{n}\mapsto \qFun (\pow{n}, \qmcC)^{\simeq },
\end{displaymath}

\noindent where the superscript on the right denotes the maximal Kan
complex within the given simplicial set.
\end{defin}

\begin{lem}\label{lem-Bepi-THHepi}
\cite[Proposition 2.2.7]{McCandless-curves} {\bf The epicyclic bar
construction and $\THH$.}  For an $\EE_{1}$-monoid $M$, let $\mcB M$
denote the {\qcat} with one object and $M$ as endomorphisms. Then
there is an equivalence of spaces with Frobenius lifts
\begin{displaymath}
B^{\epi}M \to \THH^{\epi} (\mcB M).
\end{displaymath}
\end{lem}

\subsection{Topological cyclic homology $\TopC$ and
  friends}\label{sec-TC-et-al}

The first two parts of the following appears to be quite different
from the original definition of $\TopC$ given in \cite{BHM93},
\cref{def-TCF}\cref{def-TCFi}.  The equivalence of the two is due to
Blumberg and Mandell and is the main theorem of
\cite{blumberg_cyclotomic}.

\begin{defin}\label{def-TC}
 \cite[Definition II.1.8]{NS18} and \cite[Definition 2.4]{BHLS}. Let $
(X, (\varphi_{p} )_{p\in \PP})$ be a cyclotomic spectrum as in
\cref{def-cyclo-spec2}.

\begin{enumerate}[label={(\roman*)},itemindent=0em]
\item \label{def-TCi} The {\bf integral topological cyclic homology
$\TopC(X) $ of $X$} is the mapping spectrum
\begin{displaymath}
\map_{\qCycSp}(\mS^{\triv}, X)
\in \qSp
\end{displaymath}

\noindent for $\mS^{\triv}$ as in
\cref{ex-cycl-spectra}\cref{ex-cycl-spectraii}.

\item \label{def-TCii} The {\bf $p$-typical topological cyclic
homology $\TopC(X, p)$} is the mapping spectrum
\begin{displaymath}
\map_{\qCycSp_{p}}(\mS^{\triv} , X)\in \qSp_{p}.
\end{displaymath}

\item \label{def-TCiii}
Let $R\in \Alg_{\EE_{1}}(\qSp)$ be an associative ring
spectrum. (For a discrete ring $R$, the associative ring spectrum is
the {\SESM} spectrum $HR$.) Then
\begin{align*}
\TopC(R)&:=\TopC(\THH(R)),&
\TopC(R,p)&:=\TopC(\THH(R), p),\\
\TopC^{-}(R)&:=\TopC^{-}(\THH(R)),&
\TopC^{-}(R,p)&:=\TopC^{-}(\THH(R),p),\\
\TP(R)&:=\TP(\THH(R)) &
\aand \TP(R,p)&:=\TP(\THH(R),p)
\end{align*}

\noindent for $\TopC^{-} (-)$, $\TopC^{-} (-,p)$, $\TP (-)$ and
$\TP (-,p)$ as in \cref{def-TCTP}, namely
\begin{align*}
\TopC^{-}(X)
 & :=  X^{h\mT},&
\TopC^{-}(X,p)
 &:=  X^{h\rC_{p^{\infty }}}\\
\TP(X)&:=  X^{t\mT}&
\qquad \aand 
\TP(X,p)&:=  X^{t\rC_{p^{\infty }}}.
\end{align*}
\end{enumerate}
\end{defin}

\begin{remark}\label{rem-TC-fixed}
{\bf $\TopC $ as a type of fixed point functor.}  A cyclotomic
spectrum $X$ is a $\mT$-spectrum with additional structure, as is the
sphere spectrum $\mS$.  In the category $\mcT^{G}$ of pointed
$G$-spaces, the fixed point set $X^{G}$ of a pointed $G$-space $X$ is
the morphism object $\mcT^{G} (S^{0}, X)$ by \cref{prop-fpms}, where
the symmetric monoidal unit $S^{0}$ has trivial $G$-action.
\cref{def-TC} has the same flavor. See \cref{rem-T-action}.
\end{remark}

From \cref{def-TC}\cref{def-TCi} we deduce that 
\begin{align*}
\map_{\qCycSp}(\mS^{\triv}, X)
 & =  \TopC(X)  
  = \map_{\qSp}(\mS, \TopC(X)  )  ,
\end{align*}

\noindent so we have the  adjunction of  \cite[Proposition IV.4.14]{NS18}
\begin{numequation}\label{eq-CycSp-adjunct}
\begin{split}
\xymatrix
@R=4mm
@C=20mm
{
{\qSp}\ar@<1.2ex>[r]^(.5){(-)^{\triv}}_(.45){\perp }
  &{\qCycSp.}\ar@<1ex>[l]^(.55){\TopC}
}
\end{split}
\end{numequation}%

If $X$ is an $\EE_{\infty }$-ring spectrum, so is $\TopC (X)$, and
there is a similar adjunction to \cref{eq-CycSp-adjunct} between the
corresponding categories of $\EE_{\infty }$-ring objects.

The following is the best known result of Nikolaus-Scholze and is a major
breakthrough.

\begin{theorem}\label{prop-TC}
{\bf How to compute $\TopC$.}
\cite[Proposition II.1.9]{NS18}
\begin{enumerate}[label={(\roman*)},itemindent=0em]
\item \label{prop-TCi} Let $(X,(\varphi_{p})_{p\in \PP})$ be a
cyclotomic spectrum.  There is a functorial fiber sequence
\begin{displaymath}
\xymatrix
@R=4mm
@C=8mm
{
{\TopC(X)}\ar[r]^(.5){}
  &{X^{h\mT}}\ar[rrr]^(.45){(\varphi_{p}^{h\mT }-\can)_{p\in \PP}}
    & &&{\prod_{p\in \PP}} (X^{t\rC _{p}})^{h\mT },\\
}
\end{displaymath}

\noindent where the maps are given by
\begin{displaymath}
\varphi_{p}^{h\mT }:X^{h\mT }\to (X^{t\rC _{p}})^{h\mT }
\end{displaymath}

\noindent and the $p$th component of $\can$ is the composite in the
diagram
\begin{numequation}\label{eq-can-p}
\begin{split}
\xymatrix
@R=8mm
@C=2mm
{
{\TopC^{-} (X)}\ar@{=}[r]^(.5){}
  &{X^{h\mT}}\ar[rrrr]^(.45){\can_{p}}\ar[d]_(.45){\simeq }
    &&&&{\left(X^{t\rC_{p}} \right)^{h\mT}}\\
  &{\left(X^{h\rC_{p}} \right)^{h (\mT/\rC_{p})}}
             \ar[rrrr]^(.55){\simeq }
    &&&&{\left(X^{h\rC_{p}} \right)^{h\mT}}
                \ar[u]_(.55){\left(s^{\rC_{p}}_{X} \right)^{h\mT},}
}
\end{split}
\end{numequation}%

\noindent where the left vertical map is the residual action
equivalence of \cref{def-residual}, the lower equivalence comes from
the $p$-th root isomorphism\linebreak  $\mT /\rC _{p}\cong \mT$ of
\cref{def-rho-n}, and $s^{\rC_{p}}_{X}$ is as in \cref{eq-Tate2}.



\item \label{prop-TCii}  
Let $(X,\varphi_{p})$ be a
$p$-cyclotomic spectrum.  There is a functorial fiber sequence
\begin{displaymath}
\xymatrix
@R=4mm
@C=8mm
{
{\TopC(X,p)}\ar[r]^(.5){}
  &{X^{h\rC _{p^{\infty}}}}
       \ar[rrr]^(.45){(\varphi_{p})^{h\rC _{p^{\infty}}}-\can_{p}}
    & &&{(X^{t\rC _{p}})^{h\rC _{p^{\infty}}}}
}
\end{displaymath}

\noindent with notation as in \cref{prop-TCi}.  
\end{enumerate}
\end{theorem}

The last statement above identifies $\TopC(X,p)$ as an equalizer. A
similar identification in terms of $\TR$, the subject of
\cref{sec-TR}, is \cref{thm-TC-TR}.

For a cyclotomic spectrum $X$, we have $\mT$-{\eqvr} maps
$\varphi_{p}:X\to X^{t\rC _{p}}$ from \cref{def-cyclo-spec2} and
$s_{p}:X^{h\rC _{p}}\to X^{t\rC _{p}}$ from \cref{eq-Tate2}.  Applying the
functor $(-)^{h\rC _{p^{i}}}$ to each gives maps
\begin{align*}
\xymatrix
@R=4mm
@C=20mm
{
{X^{h\rC _{p^{i}}}}\ar[r]^(.35){\varphi_{p^{i+1}}}
  &{(X^{t\rC _{p}})^{h\rC _{p^{i}}}\phantom{=X^{t\rC _{p^{i+1}}}}}
}\\
\aand 
\xymatrix
@R=4mm
@C=20mm
{
{(X^{h\rC _{p}})^{h\rC _{p^{i}}}=X^{h\rC _{p^{i+1}}}}
          \ar[r]^(.6){s_{p^{i+1}}}
  &{(X^{t\rC _{p}})^{h\rC _{p^{i}}},}
}
\end{align*}

\noindent and $(X^{t\rC _{p}})^{h\rC _{p^{i}}}\simeq X^{t\rC
_{p^{i+1}}}$ when $X$ is bounded below by \cite[Lemma II.4.1]{NS18}.


\subsection{The cyclotomic spectra $\THH (\Z/p)$ and $\THH (\Z)$}
\label{sec-THHZp}

Here we will report on the results of \cite[\S IV.4]{NS18}, which the
reader should consult for details.  They denote our $\THH (\Z/p)$ (as
in \cref{def-TC}\cref{def-TCiii}) by $\THH (H\FFp )$.  $\TopC (\Z/p)$
is first computed in \cite[Theorem B]{HM1}, and redone by Nikolaus and
Scholze using techniques from \cite{Bokstedt-Madsen94}.

The latter paper makes use of three ``skeleton spectral sequences,''
one for each of the spectra in the bottom row of \cref{eq-Klein-norm}.
The Tate version is a whole plane spectral sequence that may or may
not converge.  The upper and lower half plane portions of it are
related to the homotopy orbit and homotopy fixed point spectra.  This
is explained nicely following the proof of \cite[Theorem
2.15]{Bokstedt-Madsen94} and later by Alice Hedenlund and John Rognes
in \cite{HedRog}.

\subsubsection{Toward $\TopC (\Z/p)$}\label{sec-TC(Zp)}
It follows from \cref{prop-TC} that in order to find $\TopC (\Z/p)$,
we must first determine $\TopC^{-}(\Z/p) :=\THH (\Z/p)^{h\mT}$.

\begin{prop}\label{prop-THH-Zp-hT}
\cite[Proposition IV.4.6]{NS18}
\begin{displaymath}
\pi_{i}\THH (\Z/p)^{h\mT}=\mycases{    
\Z_{p} &\mbox{for $i$ even}\\
0      &\mbox{otherwise}
}
\end{displaymath}

\noindent and as a ring,
\begin{displaymath}
\pi_{*}\THH (\Z/p)^{h\mT}
= \Z_{p}[\tilde{u},v]/ (\tilde{u}v-p)
\end{displaymath}

\noindent with $\tilde{u}\in \pi_{2}$ and $v\in \pi_{-2}$, where
$\tilde{u}$ maps to the class $u\in \pi_{2}\THH (\Z/p)$ of
\cref{thm-Bokstedt}\cref{thm-Bokstedti}.
\end{prop}

The proof makes use of the homotopy fixed point spectral sequence, one
of the three skeleton spectral sequences alluded to above, for which
\begin{displaymath}
E_{2}^{i,j} = H^{i} (B\mT; \pi_{-j}\THH (\Z/p))
\implies \pi_{-i-j}\THH (\Z/p)^{h\mT}.
\end{displaymath}

\noindent See \cite[Chapter 5]{HedRog} for more details. They identify
the $E_{2}^{i,j}$ (which they denote by $E^{2}_{-i,-j}$) in
\cite[Theorem 5.14]{HedRog}.

Since $H^{*}B\mT=H^{*}\cxs P^{\infty }$ and $\pi_{*}\THH (\Z/p)$ are
both even dimensional, $E_{2}^{i,j}$ is nontrivial only when $i$ and
$j$ are both even, and the spectral sequence collapses.  The equation
$\tilde{u}v=p$ is the subject of \cite[Lemma IV.4.7]{NS18}.

\begin{cor}\label{rem-nontriv}
  The action of  $\mT$ on $\THH (\Z/p)$ is nontrivial. 
\end{cor}

\proof
  For any
$\mT$-action (including the trivial one) on $H/p$, the homotopy fixed
point spectral sequence collapses without any additive extensions for
degree reasons, showing that $\pi_{*} (H/p)^{h\mT}$ is an $\FFp
$-vector space, namely
\begin{displaymath}
\pi_{i} (H/p)^{h\mT}\cong H^{-i} (\cxs P^{\infty };\Z/p).
\end{displaymath}

\noindent If the $\mT$-action were trivial on
\begin{numequation}\label{eq-THH-split}
\begin{split}
\THH (\Z/p)\simeq \bigoplus_{j\geq 0}\Sigma^{2j}H/p,
\end{split}
\end{numequation}%

\noindent then $\pi_{*}\THH (\Z/p)^{h\mT}$ would also be a graded
$\FFp $-vector space.  By \cref{prop-THH-Zp-hT}, we find that this
graded abelian group is torsion free instead.  It follows that the
splitting of \cref{eq-THH-split} is {\em not} $\mT$-{\eqvr}.
\qed\bigskip 

There is a similar spectral sequence for the Tate construction (see
\cite[Chapter 6]{HedRog}), which yields the following.

\begin{cor}\label{cor-THH-Zp-tT}
\cite[Propositions IV.4.6 and IV.4.9]{NS18}
As a ring,
\begin{displaymath}
\pi_{*}\THH (\Z/p)^{t\mT}
= \Z_{p}[v^{\pm 1}]
\end{displaymath}

\noindent and for all even integers $i$, the map 
\begin{displaymath}
\pi_{i} \THH(\Z/p)^{h\mT} \cong \Z_{p} 
\rightarrow
\pi_{i} \THH(\Z/p)^{t\mT} \cong \Z_{p} 
\end{displaymath}

\noindent induced by $\varphi_{p}^{h\mT}$ is injective. For $i\leq 0$, it
is an isomorphism, while for $i=2j\geq 0$, it has image $p^{j}\Z_{p}$.
\end{cor}

In other words, the ring of \cref{prop-THH-Zp-hT} is embedded into
that of \cref{cor-THH-Zp-tT} by sending $\tilde{u}$ to $pv^{-1}$.
With this result in hand, an easy application of \cref{prop-TC} gives
the following.

\begin{cor}\label{cor-TC-Zp}
\cite[Corollary IV.4.10]{NS18}
The homotopy type of $\TopC(\Z/p)$ is
\begin{displaymath}
\TopC(\Z/p)\simeq H\Z_{p}\oplus \Sigma^{-1}H\Z_{p}.
\end{displaymath}
\end{cor}

This is proved by looking at the {\LES} of homotopy groups associated
with the fiber sequence of \cref{prop-TC}\cref{prop-TCii} for $X=\THH
(\Z/p)$. The second and third spectra have even dimensional homotopy
groups, so  for each integer $i$ we get a 4-term sequence
\begin{displaymath}
\xymatrix
@R=8mm
@C=6mm
{
{0}\ar[r]^(.5){}
  &{\pi_{2i}\TopC (\Z/p)}
   \ar[r]^(.5){}
    &{\pi_{2i}\THH(\Z/p)^{h\mT}=\Z_{p}}
                                \ar[d]_(.5){\can - \varphi^{h\mT}_{p} }\\
  & &{\Z_{p}=\pi_{2i}\THH(\Z/p)^{t\mT}}\ar[r]^(.5){}
      &{\pi_{2i-1}\TopC (\Z/p)}\ar[r]^(.5){}
        &{0.}
}
\end{displaymath}

\noindent When $i\neq 0$ the vertical homomorphism is the
difference between the identity on $\Z_{p}$ and a map divisible by
$p$, and therefore an isomorphism. For $i=0$, it is trivial, and the
result follows.

\begin{remark}\label{rem-simplest}
{\bf The simplest instance of chromatic redshift.}
Recall that $H/p$ has {\fp}-type $-1$ (as explained in \cref{ex-BPn}),
and we see that 
\begin{displaymath}
\TopC(\Z/p):=\TopC (\THH(\Z/p))
\end{displaymath}

\noindent has {\fp}-type $0$.  Note that $\THH (\Z/p)$ itself does not
have any {\fp}-type because its smash product with any nontrivial
finite spectrum has infinitely many nontrivial homotopy groups. 
\end{remark}

It follows from \cref{cor-TC-Zp} that there is a unique map (up to
unit $p$-adic scalar)
\begin{numequation}\label{eq-unique-mP}
\begin{split}
H\Z_{p}\to \TopC(\Z/p) :=\TopC (\THH(\Z/p))
\end{split}
\end{numequation}%

\noindent that is surjective in homotopy since $H\Z_{p}$ is the
connective cover $\tau_{\geq 0}\TopC(\Z/p)$.  Under the adjunction of
\cref{eq-CycSp-adjunct} this determines a map of cyclotomic spectra\linebreak 
$H\Z_{p}^{\triv}\to \THH(\Z/p)$.

Note that since the underlying spectrum of $\THH(\Z/p)$ is a wedge of
evenly suspended copies of $H\FFp $, a map to it from $H\Z_{p}$ is
given by a sequence of even dimensional mod $p$ cohomology classes in
the latter with suitable multiplicative properties.  We do not know
which such sequence our map corresponds to, but we do know this.

\begin{prop}\label{prop-T-Cp-equivariant-equiv}
{\bf A $\mathbb{T}/\rC_{p}$-equivariant equivalence.} \cite[Corollary
IV.4.13]{NS18} The $\mathbb{T}$-equivariant map of
$\mathbb{E}_\infty$-algebras
\[
  H\mathbb{Z}_p^{\triv} \longrightarrow \mathrm{THH}(\Z/p)
\]
adjoint to \cref{eq-unique-mP} induces a
$\mathbb{T}/\rC_{p}$-equivariant equivalence of
$\mathbb{E}_\infty$-algebras
\[
  H\mathbb{Z}_p^{t\rC_{p}}
  \;\simeq\; \mathrm{THH}(\Z/p)^{t\rC_{p}}
   =: \TP(\Z/p).
\]
In particular,
\[
  \pi_* \mathrm{THH}(\Z/p)^{t\rC_{p}} \;\cong\; \mathbb{F}_p[v^{\pm 1}].
\]

Moreover, the $\mathbb{T} \simeq \mathbb{T}/\rC_{p}$-equivariant map of
$\mathbb{E}_\infty$-algebras
\[
  \varphi_p \colon \mathrm{THH}(\Z/p) \longrightarrow 
  \mathrm{THH}(\Z/p)^{t\rC_{p}}
\]
identifies $\mathrm{THH}(\Z/p)$ with the connective cover
\[
  \tau_{\ge 0}\,\mathrm{THH}(\Z/p)^{t\rC_{p}}
  \;\simeq\;
  \tau_{\ge 0}\,H\mathbb{Z}_p^{t\rC_{p}}.
\]
\end{prop}

\subsubsection{The {\Bok}-Madsen computation of $\TopC (\Z_{p})$}
\label{sec-Bok-Mad}

Nikolaus and Scholze do not determine $\TopC(H\Z_{p})$, but it is
treated earlier by {\Bok} and Madsen in \cite{Bokstedt-Madsen94}.  In
his MathSciNet review Peter May says
\begin{quote}
This difficult and important paper gives a pioneering application of
equivariant stable homotopy theory to calculations in algebraic
$K$-theory.
\end{quote}

\noindent The authors determine not just the homotopy groups but the
homotopy type of $\TopC (\Z_{p})$, relating it to $BU$ and $\Im\, J$.
They show (without using that terminology since it had not been
invented yet) that it has {\fp}-type $1$ as in
\cref{def-MR99}\cref{def-MR99i}.  See \cite[(0.7)]{Bokstedt-Madsen94}
for a precise statement. The groups ``agree with [the answer]
predicted by a generalized version of the Lichtenbaum-Quillen
conjecture formulated by Dwyer and Friedlander \cite{DF85}.''  {\Bok}
and Madsen's proof relies on \cite[Conjecture 4.3]{Bokstedt-Madsen94},
which was proved while their paper was in press by Stavros Tsalidis in
\cite{Tsalidis}.

\cite{Bokstedt-Madsen94} also anticipates the {\em redshift philosophy}
introduced by Christian Ausoni and John Rognes in
\cite{Ausoni-Rognes1}, which says that algebraic $K$-theory and the
related functor $\TopC$ each raise chromatic height (or {\fp}-type as
in \cref{def-MR99}) by one.  In this case the input is the integer
{\SESM} spectrum, which has {\fp}-type 0, while the spectra associated
with $BU$ and $\Im\, J$ have {\fp}-type 1.  The case where the input
is the mod $p$ {\SESM} spectrum is discussed in \cref{sec-TC(Zp)}.

The first theorem about redshift at all heights is that of Dylan
Wilson and Hahn \cite{Hahn-Wilson} proved 28 years later.

There is a complementary notion of {\em chromatic blueshift} having to
do with the Tate construction's lowering chromatic height; see
\cref{sec-Segal-conjecture}.


\subsection{$t$-structures and boundedness}\label{sec-t-structures}
The original definition of a $t$-structure on a triangulated category
is due to Alexander Beilinson, Joseph Bernstein and Pierre Deligne,
\cite[Definition 1.3.1]{BBD82}.  A {\bf $t$-structure} on a stable
{\qcat} $\qmcC$ is a system of full sub-{\qcats} $\qmcC_{\geq n}$ and
$\qmcC_{\leq n}$ for $n\in \Z$ with certain properties spelled out in
\cite[Definition 1.2.1.4]{Lurie:HA} and in \cite[Appendix A]{AN21}.
These properties imply that the subcategories are determined by
$\qmcC_{\geq 0}$ and $\qmcC_{\leq 0}$, the {\bf aisle} and {\bf
co-aisle}. $\qmcC_{\geq 0}$ is an example of a {\bf prestable
{\qcat}}, a notion studied in \cite[Appendix C]{Lurie:SAG}, that
abstracts the properties of the {\qcat} of connective spectra.

$\qmcC_{\geq n}$ and $\qmcC_{\leq n}$ are sometimes called the
subcategories of {\bf $n$-connective objects} and {\bf $n$-coconnective
objects}.  Objects that are in $\qmcC_{\geq n}$ for some $n$ are said
to be {\bf bounded below}, and those in $\qmcC_{\leq n}$ for some $n$
are said to be {\bf bounded above}. The full subcategories of such
objects are denoted by $\qmcC_{+}$ and $\qmcC_{-}$.  The
subcategory\linebreak $\qmcC_{\geq 0}\cap \qmcC_{\leq 0}$ is called
{\bf the heart} $\qmcC^{\heartsuit}$.  Its homotopy category is known
to be abelian.  For $m\leq n$, let
\begin{displaymath}
\qmcC_{[m,n]}:=\qmcC_{\geq m}\cap\qmcC_{\leq n}.
\end{displaymath}

Let $\qmcC$ be a stable {\qcat} $\qmcC$ with $t$-structure. One has
has a {\bf truncation functor} $\tau_{\leq n}:\qmcC \to \qmcC_{\leq
n}$ generalizing the classical $n$th Postnikov section.  The fiber of
the map $X\to \tau_{\leq n}X$ is denoted by $\tau_{>n}X:=\tau_{\geq
n+1}X$, the generalization of the $n$-connected cover of $X$.  The
functor $\tau_{>n}:\qmcC\to \qmcC_{>n}$ is also called a truncation
functor.

We need the following informal definition, which is discussed in much
more detail by Weibel in \cite[Chapter 10]{Weibel-H}.

\begin{defin}\label{def-derived}
{\bf The derived category $\mcD (\mcA)$ of an abelian category
$\mcA$.}  Let $\Ch (\mcA)$ denote the category of chain complexes with
objects in $\mcA$.  One can define homology and chain homotopy in the
usual way.  Let $\mcK (\mcA) $ be the category whose objects are chain
complexes and whose morphisms are chain homotopy classes of chain
maps.  A {\bf quasi-isomorphism} is a morphism in $\Ch (\mcA)$ or
$\mcK (\mcA)$ that induces an isomorphism in homology. (This notion is
weaker than chain homotopy equivalence.) $\mcD (\mcA)$ is the category
obtained from $\mcK (\mcA) $ by formally inverting all
quasi-isomorphisms.
\end{defin}

The stable {\qcat} of spectra $\qSp$ has the {\bf Postnikov
$t$-structure} in which the subcategories are those of spectra with
trivial homotopy groups in dimensions outside the range indicated by the
subscripts. Its heart is the category of {\SESM} spectra with homotopy
groups concentrated in dimension 0, which is known to be {\eqt} to the
derived category of abelian groups $\mcD (\Ab)$ of \cref{def-derived}.
Its homotopy category is $\Ab$.  There is a similar $t$-structure on
$\qSp^{BG}$ for any compact Lie group $G$, in which connectivity and
coconnectivity is that of the underlying spectrum.

\begin{defin}\label{def-t-exact}
{\bf Exactness with respect to $t$-structures.} Let $\qmcC$ and
$\qmcD$ be stable {\qcats} each equipped with a $t$-structure.  A
functor $\qmcC \to \qmcD$ is\begin{itemize}
\item [$\bullet$] {\bf left exact} if sends $\qmcC_{\geq 0}$ to
$\qmcD_{\geq 0}$,

\item [$\bullet$] {\bf right exact} if sends $\qmcC_{\leq 0}$ to
$\qmcD_{\leq 0}$ and 

\item [$\bullet$] {\bf $t$-exact} if both conditions hold, meaning
that it sends $\qmcC^{\heartsuit}$ to $\qmcD^{\heartsuit}$.
\end{itemize}
\end{defin}

\begin{defin}\label{def-can-van}
{\bf Canonical vanishing.}
Suppose we are given a $\mT$-spectrum $X$ which is bounded below.  We
say that:
\begin{enumerate}[label={(\roman*)},itemindent=0em]
\item \label{def-can-vani}
  \cite[Definition 2.18(1)]{BHLS} 
  Recall the map $S_{X}^{G}:X^{hG}\to X^{tG}$
of \cref{eq-Tate2} for $G\subseteq \mT$.  $X$ satisfies {\bf weak
canonical vanishing} with parameter $b$ ($\mathrm{WCV}(\le b)$ for
short) if for each $H=\rC_{p^{j}}$ with $1 \le j \le \infty$, the composite
\begin{displaymath}
\xymatrix
@R=4mm
@C=10mm
{
{\tau_{>b}X^{hH}}\ar[r]^(.5){}
  &{X^{hH}}\ar[r]^(.5){s_{X}^{H}}
    &{X^{tH}}
}
\end{displaymath}

\noindent is null.

\item \label{def-can-vanii}
  \cite[Definition 2.18(2)]{BHLS} 
Recall the Euler class of \cref{def-rep-T},\linebreak 
$a_{(1)}:\mS^{-(1)}\to\mS$, and let
\begin{displaymath}
a_{(1)}^{d}:X \to \Sigma^{d (1)}X
\end{displaymath}

\noindent be the smash product of its $d$th smash power with the
twisted suspension $\Sigma^{d (1)}X$.  $X$ satisfies {\bf strong
canonical vanishing} with parameter $b$\linebreak ($\mathrm{SCV}(\le
b)$ for short) if there exists some $d \ge 0$ for which the
composition
\[
\tau_{> b} X \longrightarrow X 
\;\xrightarrow{ a_{(1)}^{d}  }\;
\Sigma^{d(1)} X
\]

\noindent is $\mT$-equivariantly null.
\end{enumerate}
\end{defin}

They show in \cite[Lemma 2.19]{BHLS} that \cref{def-can-vanii} implies
\cref{def-can-vani}. 

\subsection{The surprising Antieau-Nikolaus $t$-structure on
  $\qCycSp$} \label{sec-AN-t-structure} There is a $t$-structure on
$\qCycSp_{p}$ due to Nikolaus and Ben Antieau \cite{AN21}.  In
\cite{BHLS} it is introduced in \S2.1.4, and discussed further in
\S2.2 and \S2.4.

The Antieau-Nikolaus connectivity of a cyclotomic spectrum is the
Postnikov connectivity of the underlying spectrum, but cyclotomic
coconnectivity is much more interesting, as \cref{thm-AN2}
illustrates.  In any case $\qCycSp_{p,+}$ (denoted in \cite{AN21} by
$\mathbf{CycSp}_{p}^{-}$) is the full subcategory of cyclotomic
spectra underlain by spectra that are bounded below in the Postnikov
sense.

\begin{theorem}\label{thm-AN1}
\cite[Theorem 1]{AN21} The {\qcat} $\qCycSp_{p,+}$ of connective
$p$-typical cyclotomic spectra is the connective part of a unique
$t$-structure on $\qCycSp_p$, the {\qcat} of $p$-typical cyclotomic
spectra as in \cref{def-cyc-spectra}\cref{def-cyc-spectraii}.  The heart
$\qCycSp_p^{\heartsuit}$ is equivalent to the abelian category of
derived $V$-complete $p$-typical Cartier modules, to be defined in
\cref{def-Cart-mod}\cref{def-Cart-modi}.
\end{theorem}

In their words,
\begin{quote}
The existence and uniqueness of such a $t$-structure is a formal
consequence of the fact that $\qCycSp_{p,+}$ is presentable and
is closed under colimits and extensions in $\qCycSp_p$.  The difficult
part of the theorem is the identification of the heart.

Recall {\Bok}'s [\cref{thm-Bokstedt}], which says that
\begin{displaymath}
\pi_{*}\THH(\mathbb{F}_{p}) \cong \mathbb{F}_{p}[b],
\end{displaymath}

\noindent a polynomial ring on a degree~$2$ generator.  More
generally, using the vanishing of the cotangent complex, one deduces
that
\begin{displaymath}
\pi_{*}\THH(k) \cong k[b]
\end{displaymath}

\noindent for any perfect ring~$k$.
Our interest in the cyclotomic $t$-structure was piqued by the
discovery of the next result.
\end{quote}

\begin{theorem}\label{thm-AN2}
\cite[Theorem 2]{AN21}
If $k$ is a perfect ring of characteristic~$p$, then 
$\THH(k) \in \qCycSp_p^{\heartsuit}$.
\end{theorem}

Again in the words of \cite{AN21},
\begin{quote}
Despite the higher homotopy groups, $\THH(k)$ is discrete [meaning in
the heart of the $t$-structure] as a cyclotomic spectrum.  On the
Cartier module side of the story, when $k$ is a perfect ring of
characteristic~$p$, $\THH(k)$ corresponds to $W(k)$, the ring of
$p$-typical Witt vectors over~$k$, with its Witt vector Verschiebung
and Frobenius operations.  The fact that $\THH(k)$ is in
$\CycSp_p^{\heartsuit}$ is consistent with the fact, due to
Hesselholt--Madsen \cite[Theorem~B]{HM1}, that for perfect fields of
characteristic~$p$, that $\pi_i \TopC(k) = 0$ for $i>0$ [see
\cref{cor-TC-Zp}].  However, the theorem is much stronger: it says
that for any cyclotomic spectrum $X$ with $\pi_i X = 0$ for $i<0$, one
has
\[
  \Map_{\qCycSp_p}\bigl(\Sigma^{k}X,\, \THH(k)\bigr) \simeq \pt
     \qquad \text{for } k>0.
\]
\end{quote}

For any $p$-typical cyclotomic spectrum $X$, the homotopy groups with
respect to the $t$-structure of \cref{thm-AN1} are denoted by
$\pi^{\mathrm{cyc}}_{i} X$.  These are objects in the homotopy category of
$\qCycSp_p^{\heartsuit} \subseteq \qCycSp_p$.  Thus, they can be considered
either as $p$-typical cyclotomic spectra, with underlying spectrum with
$S^{1}$-action and Frobenius
\[
  \varphi \colon \pi^{\mathrm{cyc}}_{i} X \longrightarrow
  \bigl(\pi^{\mathrm{cyc}}_{i} X\bigr)^{t C_{p}},
\]
or as derived $V$-complete $p$-typical Cartier modules under the
equivalence of \cref{thm-AN1}.

It is known that this $t$-structure is not compatible with filtered
colimits; see \cite[Example 3.27]{AN21}.

\subsubsection{Cartier modules}\label{sec-cartier-modules}

Before describing the Antieau-Nikolaus $t$-structure more explicitly,
we discuss the relevant abelian category, the homotopy category of its
heart, and a related category of spectra.

\begin{defin}\label{def-Cart-mod}
{\bf Two flavors of Cartier modules.} \cite[Definition 3.1]{AN21} 
\begin{enumerate}[label={(\roman*)},itemindent=0em]

\item \label{def-Cart-modi} A {\bf $p$-typical Cartier
module} is an abelian group $M$ equipped with endomorphisms $\ANV$ and
$\ANF$ (the {\bf Verschiebung} and {\bf Frobenius} maps) such that
$\ANF\ANV=p$. (We do not require that $\ANV\ANF=p$.)

Such a module $M$ is {\bf derived $\ANV$-complete} \cite[Definition
3.24]{AN21} if the map
\begin{displaymath}
M\to \lim{n}M/\ANV^{n} 
\end{displaymath}

\noindent is an equivalence in the derived category of abelian groups
$\mcD (\Ab)$ as in \cref{def-derived}.  We will sometimes refer to
such an object as a $p$-typical {\bf algebraic} Cartier module to
distinguish it from what comes next.

\item \label{def-Cart-modii} A {\bf $p$-typical topological Cartier
module} is a $\mT$-spectrum $X$ with a $\mT$-{\eqvr} factorization of
the $\rC _{p}$-norm of \cref{eq-stable-norm}
\begin{numequation}\label{eq-Fv-norm}
\begin{split}
\xymatrix
@R=4mm
@C=10mm
{
{\rho_{p}^{*}X_{h\rC _{p}}}\ar[r]^(.6){\ANV}
  &{X}\ar[r]^(.4){\ANF}
    &{\rho_{p}^{*}X^{h\rC _{p}}},
}
\end{split}
\end{numequation}%

\noindent where $X_{h\rC _{p}}$ and $X^{h\rC _{p}}$ each have a
residual action of $\omT:=\mT/\rC _{p}\cong \mT$. The isomorphism
$\rho _{p}:\mT\to \omT$ is the $p$th root map of \cref{def-rho-n}.  

$X$ is {\bf $\ANV$-complete} \cite[Definition 3.20]{AN21} if the limit of
the tower
\begin{displaymath}
\xymatrix
@R=4mm
@C=15mm
{
{\dotsb }\ar[r]^(.45){\rho_{p^{2}}^{*}\ANV}
  &{\rho_{p^{2}}^{*}X_{h\rC_{p^{2}}}}\ar[r]^(.5){\rho_{p}^{*}\ANV}
    &{\rho_{p}^{*}X_{h\rC_{p}}}\ar[r]^(.6){\ANV}
      &{X}
}
\end{displaymath}

\noindent is contractible.
\end{enumerate}
\end{defin}

The maps $\ANF$ and $\ANV$ here are related to the maps in
\cref{def-WVrF} and to the maps $\pi _{k,L}$ and $\pi_{k,R}$ of
\cref{def-TR}, but the former is not the restriction $F$ of
\cref{def-residual}.  We use the bold font to avoid confusion.

\begin{remark}\label{rem-Cartier-module}\cite[\S4.1]{AN21}
A $p$-typical topological Cartier module is a $p$-cyclotomic
spectrum with additional structure.  The Frobenius
map $\ANF:X\to X^{h\rC _{p}}$ can be composed with the map
$s^{p}:X^{h\rC _{p}}\to X^{t\rC _{p}}$ of \cref{eq-Tate2} to give a
cyclotomic structure on $X$.  This is the same as a spectrum with
Frobenius lifts as in \cite[Definition 2.1.2]{McCandless-curves} and
\cref{def-cyc-spectra}\cref{def-cyc-spectraiii}.
\end{remark}

\begin{ex}\label{ex-top-Cart}
{\bf Some Cartier modules.}

\begin{enumerate}[label={(\roman*)},itemindent=0em]
\item \label{ex-top-Carti} For a $p$-typical algebraic Cartier module $M$ as in
\cref{def-Cart-mod}\cref{def-Cart-modi}, consider the {\SESM} spectrum
$HM$ with trivial $\mT$-action.  Then $\ANV$ and $\ANF$ induce endomorphisms
$H\ANV$ and $H\ANF$ of $HM$, and we have
\begin{displaymath}
\xymatrix
@R=6mm
@C=2mm
{
{HM_{h\rC _{p}}}\ar[rr]^(.5){}\ar[d]^(.5){\simeq }
  &&{(HM_{h\rC _{p}})_{\leq 0}}\ar[d]^(.5){\simeq }
    &{}
      &{(HM^{h\rC _{p}})_{\geq 0}}\ar[r]^(.5){}
        &{HM^{h\rC _{p}}}\\
{HM\wedge B}
  &&{HM}\ar[r]^(.5){H\ANV }
    &{HM}\ar[r]^(.5){H\ANF }
      &{HM}\ar[u]^(.5){\simeq  }
        &{\map_{\qSp }(B, HM),}
              \ar[u]_(.5){\simeq }
}
\end{displaymath}

\noindent where $B:=\Sigma^{\infty }B\rC_{p+}$, $\map_{\qSp } (-,-)$
is as in \cref{eq-map-spec}, and the composite\linebreak $HM_{h\rC
_{p}}\to HM^{h\rC _{p}}$ is the norm $\Nm_{\rC_{p}}$ of
\cref{eq-NormG}.  This makes $HM$ a $p$-typical topological Cartier module.

\item \label{ex-top-Cartii} For a $p$-typical topological Cartier
module $X$, one has endomorphisms $\ANV$ and $\ANF$ in $\pi^{\mathrm{cyc}}_{*}X$
induced by the composites
\begin{displaymath}
\xymatrix
@R=4mm
@C=6mm
{
{X}\ar[r]^(.45){i_{*}}
  &{X_{h\rC _{p}}}\ar[r]^(.55){\ANV}
    &{X}
}
\qquad \aand 
\xymatrix
@R=4mm
@C=6mm
{
{X}\ar[r]^(.45){\ANF}
  &{X^{h\rC _{p}}}\ar[r]^(.55){i^{*}}
    &{X.}
}
\end{displaymath}

\noindent for $i_{*}$ and $i^{*}$ as in \cref{eq-orbit-spaces} and
\cref{eq-fp-spaces}.
These make $\pi_{*}X$ a graded $p$-typical Cartier module.

\item \label{ex-top-Cartiv} For a $p$-typical topological Cartier
module $X$, let $X/\ANV$ denote the cofiber of the $\mT$-{\eqvr} map
$\ANV:\rho_{p}^{*}X_{h\rC_{p}}\to X$, where $\rho _{p}^{*}$ is as in
\cref{def-rho-n}.  It is a $p$-cyclotomic spectrum but need not be a
topological Cartier module. Consider the diagram
\begin{displaymath}
\xymatrix
@R=2mm
@C=10mm
{
{\rho_{p}^{*}X_{h\rC_{p}}}\ar@{=}[dd]^(.5){}\ar[r]^(.5){\ANV}
  &{X}\ar[dd]_(.5){\ANF}\ar[r]^(.5){j }
    &{X/\ANV}\ar[dd]_(.5){\ANF'}\ar[dr]^(.4){\varphi_{p}}\\
  & &  &{(X/\ANV)^{t\rC_{p}}}\\ 
{\rho_{p}^{*}X_{h\rC_{p}}}\ar[r]^(.5){\Nm_{\rC_{p}}^{X}}
  &{\rho_{p}^{*}X^{h\rC_{p}}}\ar[r]^(.5){s_{p}^{X}}
    &{X^{t\rC_{p}}}\ar[ur]^(.4){ }_(.45){j^{t\rC_{p}}}
}
\end{displaymath}

\noindent where the top and bottom rows are cofiber sequences and
$\ANF'$ is induced by $\ANF$.  Then the map
$\varphi_{p}:=j^{t\rC_{p}}\ANF'$ is a cyclotomic structure on
$X/\ANV$.  The fiber of $j^{t\rC_{p}}$ is $(X_{h\rC_{p}})^{t
(\orC_{p})}$ (where $\orC_{p}:=\rC_{p^{2}}/\rC_{p}$ acting residually
on $X_{h\rC_{p}}$ as in \cref{def-residual}\cref{def-residualii}),
which is contractible by
\cref{lem-Tate-orbit-fix}\cref{lem-Tate-orbit-fixi} when $X$ is
bounded below.
\end{enumerate}
\end{ex}

\begin{defin}\label{def-TCart}\cite[Definition 3.6]{AN21} 
The {\bf {\qcat} of $p$-typical topological Cartier modules}
$\qTCart_{p}$ is the pullback
\begin{displaymath}
\xymatrix
@R=1mm
@C=15mm
{
{\qTCart_{p}}\ar[r]^(.5){}\ar[ddd]^(.5){}\ar@{}[dddr]|-(.15){\pb}
  &{\left(\qSp^{\mT} \right)^{\bdelt^{2}}}
            \ar[ddd]^(.5){(\ev_{1},d_{1})}\\
{}\\
{}\\
{\qSp^{\mT}}\ar[r]^(.4){(id, \Nm_{\rC _{p}})}
  &{\qSp^{\mT}\times \left(\qSp^{\mT} \right)^{\bdelt^{1}}}\\
{X}\ar@{|->}[r]^(.5){}
  &{(X,\Nm_{\rC _{p}}:X_{h\rC _{p}}\to X^{h\rC _{p}})}
}
\end{displaymath}

\noindent where the right arrow has the form
\begin{displaymath}
\xymatrix
@R=2mm
@C=5mm
{
{}
  &{X_{1}}\ar[ddr]^(.5){f}\\
{}& &{}\ar@{|->}[rr]^(.5){}
      & &{(X_{1}, n:X_{0}\to X_{2}).}\\
{X_{0}}\ar[uur]^(.5){v}\ar[rr]^(.5){n}
  & &{X_{2}}
}
\end{displaymath}

\noindent We will write an object in this category as $(X,\ANV_{X},\ANF_{X},
\sigma _{X})$, where $X$ is a\linebreak  $\mT$-spectrum,
\begin{displaymath}
\ANV_{X}:X_{h\rC _{p}}\to X,
\qquad 
\ANF_{X}:X\to X^{h\rC _{p}},
\end{displaymath}

\noindent and $\sigma_{X}$ is a 2-simplex corresponding to a
factorization $\Nm_{\rC _{p}}\simeq \ANF_{X}\cdot \ANV_{X}$.
\end{defin}

The following is a consequence of the above and a similar statement
about topological Cartier modules proved by Mathew as
\cite[Proposition 6.5]{Mathew-TR}.

\begin{lem}\label{lem-BHLS-2.12}
\cite[Lemma 2.12]{BHLS} A filtered colimit of cyclotomic spectra
bounded in the range $[a, b]$ is itself bounded in the range
$[a,b+3]$.
\end{lem}

\begin{defin}\label{def-Segal-cond}
  \cite[Definition 2.25]{BHLS} 
For $X \in \qCycSp$ and $b \in \mathbb{Z}$, we say that $X$ satisfies
the {\bf Segal condition} with parameter $b$ ($\Segal(\le b)$ for
short) if the fiber of the map
\[
X \xrightarrow{\ \varphi\ } X^{t C_p}
\]
is cyclotomically $b$-truncated.  
\end{defin}

This condition is so named because the Segal conjecture implies that
the map in question is an equivalence when $X$ is a $p$-local finite
complex with trivial cyclotomic structure.

The conditions of \cref{def-Segal-cond,def-can-van}
are related as follows.
 
\begin{prop}\label{prop-BHLS2.30}\cite[Proposition 2.30]{BHLS}
  Let $X \in \qCycSp_{p}$ be a bounded below.
\begin{enumerate}[label={(\roman*)},itemindent=0em]
\item If $X$ satisfies ${\rm SCV}(\leq b)$, then $X$ satisfies
  ${\rm WCV}(\leq b)$.

  \item If $X$ satisfies ${\rm WCV}(\leq b)$ and $\Segal(\leq b)$,
    then $X \in \qCycSp_{\leq b}$.

  \item If $X \in \qCycSp_{\leq b}$, then $X$ satisfies $\Segal(\leq b)$.

  \item If $X \in \qCycSp_{[c,b]}$ and $p^m$ acts by zero on $X$, then
    $X$ satisfies
  \[
    {\rm SCV}\bigl(\leq b + 2(b - c + 1)m\bigr).
  \]
\end{enumerate}
 \end{prop}

\subsection {Topological restriction homology $\TR$}\label{sec-TR}

In \cite[\S2]{BHLS} this functor is discussed and figures in the proofs
of \cite[Theorem 3.22]{BHLS} (our \cref{thm-failure}) and \cite[Lemma
4.26]{BHLS}, which is needed for \cite[Theorem C]{BHLS}, our
\cite[Theorem 5.3]{Rav:gjmcyc-ext}.

For a cyclotomic $\Omega $-spectrum $X$, Blumberg and Mandell
\cite[Definition 6.3]{BM15} define
\begin{displaymath}
\TR^{k} (X,p):=X^{\rC_{p^{k}}}\mbox{ for }k\geq 0
\qquad \aand \TR (X,p) :=\holim{k} \TR^{k} (X),
\end{displaymath}

\noindent which they call a ``mapping microscope.''  They also define
a global version.

When $X$ is bounded below, the above definition is equivalent to the
following one by \cite[Theorem 3.3.12]{McCandless-curves}, which is
used in \cite{BHLS}, where we leave the prime implicit.

\begin{defin}\label{def-TR}
{\bf Topological restriction homology} \cite[Example 3.4 and
Construction 3.18]{AN21} and \cite[Definition 2.4]{BHLS}.  For a
cyclotomic spectrum $X$, the {\bf spectrum $\TR(X)$} (which turns out
to be cyclotomic) is the limit of the diagram
\begin{numequation}\label{eq-TR}
\begin{split}
\xymatrix
@R=2mm
@C=2mm
{
{X}\ar[ddr]^(.5){\varphi_{p}}
  & &{X^{h\rC _{p}}}\ar[ddl]^(.4){s^{p}}
                 \ar[ddr]^(.5){\varphi_{(p^{2})}}
      & &{{X^{h\rC _{p^{2}}}}}\ar[ddl]^(.35){s^{(p^{2})}}
                 \ar[ddr]^(.5){\varphi_{(p^{3})}}
          & &{{X^{h\rC _{p^{3}}}}}\ar[ddl]^(.35){s^{(p^{3})}}
                               \ar[ddr]^(.5){\varphi_{(p^{4})}}    \\
{}\\
  &{X^{t\rC _{p}}}
   & &{(X^{t\rC _{p}})^{h\rC _{p}}}
        & &{(X^{t\rC _{p}})^{h\rC _{p^{2}}}}
            & &{\dotsb, }
}
\end{split}
\end{numequation}%

\noindent where $\varphi_{p}$ is the cyclotomic structure map for $X$
as in \cref{def-cyclo-spec2}\cref{def-cyclo-spec2i} and $s^{p}$ is as
in \cref{eq-TateCp}. The other maps are defined by
\begin{displaymath}
\varphi_{(p^{i+1})}:=(\varphi_{p})^{h\rC _{p^{i}}}
\qquad \aand 
s^{(p^{i+1})}:= (s^{p})^{h\rC _{p^{i}}}
\qquad \mbox{for }i>0 .
\end{displaymath}

For each $k\geq 0$, the {\bf spectrum $\TR^{k} (X)$}  is the limit of the
following subdiagram of \cref{eq-TR}.
\begin{numequation}\label{eq-TRk}
\begin{split}
\xymatrix
@R=2mm
@C=2mm
{
{X}\ar[ddr]^(.5){\varphi_{p}}
  & &{X^{h\rC _{p}}}\ar[ddl]^(.4){s^{p}}
                 \ar[ddr]^(.5){\varphi_{(p^{2})}}
      & &{\dots}\ar[ddl]^(.35){}
                 \ar[ddr]^(.5){}
          & &{{X^{h\rC _{p^{k}}}}}\ar[ddl]^(.35){s^{(p^{k})}}
                \\
{}\\
  &{X^{t\rC _{p}}}
   & &{(X^{t\rC _{p}})^{h\rC _{p}}}
        & &{(X^{t\rC _{p}})^{h\rC _{p^{k-1}}}}
}
\end{split}
\end{numequation}%

For $k>0$, let $\pi_{k,L}:\TR^{k} (X)\to \TR^{k-1} (X)$ be projection
away from the last two factors, and let
$\pi_{k,R}:\TR^{k} (X)\to \TR^{k-1} (X)^{h\rC _{p}}$ be projection
away from the first two factors, as in \cref{eq-V} and \cref{eq-F}.
\end{defin}

The limit of \cref{eq-TRk} consists of sequences $(x_{0}, x_{1},\dotsc
,x_{k})$, where $x_{i}\in X^{h\rC _{p^{i}}}$ with $s_{p^{i}}
(x_{i})=\varphi_{p^{i}} (x_{i-1})$ for $1\leq i\leq k$.  It is the
equalizer of
\begin{displaymath}
\xymatrix
@R=4mm
@C=10mm
{
{X\vee X^{h\rC _{p} }\vee \dotsb \vee X^{h\rC _{p^{k} }}}
    \ar@<.75ex>[r]^(.4){s} \ar@<-.75ex>[r]_(.4){\varphi }
  &{X^{t\rC _{p}}\vee (X^{t\rC _{p}})^{h\rC_{p}}\vee
        \dotsb \vee (X^{t\rC _{p}})^{h\rC _{p^{k-1}}}}
}
\end{displaymath}

\noindent in which $s (x_{0})$ and $\varphi (x_{k})$ are each
understood to be the basepoint. The limit of \cref{eq-TR} consists of
infinite sequences $(x_{0}, x_{1},\dotsc)$ satisfying similar
conditions.

$\TR^{k} (X)$ as defined above is denoted by $X^{\rC_{p^{k}}}$ in
\cite[Definition 9.1]{KN21}.  They justify this notation by proving in
\cite[Proposition 9.2]{KN21} that when $X$ is bounded below, there is
an orthogonal $\mT$-spectrum whose genuine $\rC_{p^{k}}$ fixed point
is $\TR^{k} (X)$.

There is a $\mT$-{\eqvr} map $\pi:\TR (X)\to X $ given by 
\begin{displaymath}
(x_{0}, x_{1},\dotsc)\mapsto x_{0}
\end{displaymath}

\noindent and  a $\mT$-{\eqvr} equivalence
\begin{displaymath}
\Phi :\TR(X)\to X\times_{X^{t\rC_{p}}}\TR (X)^{h\rC_{p}}
\end{displaymath}

\noindent given by the pullback diagram
\begin{numequation}\label{eq-TR-pullback}
\begin{split}
\xymatrix
@R=1mm
@C=5mm
{
{(x_{0},x_{1},x_{2},\dotsc )}\ar@{|->}[ddddd]^(.5){}
                             \ar@{|->}[rrrr]^(.5){}
  & & & &{(x_{1},x_{2},\dotsc )}\ar@{|->}[dddd]^(.5){}\\
  &{\TR(X)}\ar[rr]^(.45){\pi_{k,R}}
           \ar[ddd]_(.45){\pi }
           \ar@{}[dddrr]|-(.15){\pb}
    & &{\TR (X)^{h\rC_{p}}}\ar[ddd]^(.45){s^{p}\pi^{h\rC_{p}}}\\
{}\\
{}\\
  &{X}\ar[rr]^(.5){\varphi_{p}}
    & &{X^{t\rC_{p}}}
        &{s^{p} (x_{1})}\ar@{=}[d]^(.5){}\\
{x_{0}}\ar@{|->}[rrrr]^(.5){}
  & & & &{\varphi_{p} (x_{0}).}
}
\end{split}
\end{numequation}%

\begin{remark}\label{rem-TR-Cartier-modules}
{\bf $\TR$ and topological Cartier modules.}
The maps $\pi_{k,L}$ and $\pi_{k,R}$ are $\ANF$ and $\ANV$ in
\cref{def-Cart-mod}.  For a cyclotomic spectrum $X$, the
spectrum $\TR (X)$ of \cref{def-TR} has a topological Cartier module
structure induced by maps between diagrams similar to that of
\cref{eq-TR} as follows. Recall that $\TR(X)$ is defined to be the
limit of that diagram.

For $\ANV$ we have
\begin{numequation}\label{eq-V}
\begin{split}
\xymatrix
@R=0mm
@C=3mm
{
{\pt}\ar[ddr]^(.5){}\ar[ddd]^(.5){}
  & &{X_{h\rC _{p}}}\ar[ddl]^(.4){}\ar[ddd]^(.5){}
                 \ar[ddr]^(.5){}
      & &{(X^{h\rC _{p}})_{h\rC _{p}}}\ar[ddl]_(.45){}
      \ar[ddr]^(.5){}        \ar[ddd]^(.5){}
          & & {\dotsb }\ar[ddl]^(.5){} \\
{}\\
  &{\pt}\ar[ddd]^(.5){}
   & &{(X^{t\rC _{p}})_{h\rC _{p}}}\ar[ddd]^(.5){}
        & &{((X^{t\rC _{p}})^{h\rC _{p}})_{h\rC _{p}}}\ar[ddd]^(.5){}
\\
{X}\ar[ddr]_(.45){}
  & &{X^{h\rC _{p}}}\ar[ddl]^(.4){}
                 \ar[ddr]_(.45){}
      & &{{X^{h\rC _{p^{2}}}}}\ar[ddl]^(.35){}
                 \ar[ddr]_(.45){}
          & & {\dotsb }\ar[ddl]^(.5){}  
\\
{}\\
  &{X^{t\rC _{p}}}
   & &{(X^{t\rC _{p}})^{h\rC _{p}}}
        & &{(X^{t\rC _{p}})^{h\rC _{p^{2}}}}
}
\end{split}
\end{numequation}

\noindent where each vertical arrow save the first two is
$\Nm_{\rC _{p}}$. The bottom two rows are those of \cref{eq-TR}, and the
top two rows, ignoring the first entry of each, comprise its image
under the functor $(-)_{h\rC _{p}}$.

For $\ANF$ we have
\begin{numequation}\label{eq-F}
\begin{split}
\xymatrix
@R=0mm
@C=2mm
{
  {X}\ar[ddr]^(.5){}  \ar[ddd]^(.5){}
& &{X^{h\rC _{p}}}\ar[ddl]_(.5){}   \ar[ddd]^(.5){}
                 \ar[ddr]^(.5){}
      & &{{X^{h\rC _{p^{2}}}}}\ar[ddl]_(.5){}
      \ar[ddr]^(.5){}       \ar[ddd]^(.5){}
          & &{{X^{h\rC _{p^{3}}}}}\ar[ddl]_(.5){}
                               \ar[ddr]^(.35){}\ar[ddd]^(.5){}    \\
{}\\
  &{X^{t\rC _{p}}}\ar[ddd]^(.5){} 
   & &{(X^{t\rC _{p}})^{h\rC _{p}}}\ar[ddd]^(.5){} 
        & &{(X^{t\rC _{p}})^{h\rC _{p^{2}}}}\ar[ddd]^(.5){} 
            & &{\dotsb }\\
{\pt}\ar[ddr]^(.5){}
  & &{X^{h\rC _{p}}}\ar[ddl]^(.4){}
                 \ar[ddr]^(.5){}
      & &{{X^{h\rC _{p^{2}}}}}\ar[ddl]^(.35){}
                 \ar[ddr]^(.5){}
          & &{{X^{h\rC _{p^{3}}}}}\ar[ddl]^(.35){}
                               \ar[ddr]^(.35){}    \\
{}\\
  &{\pt}
   & &{(X^{t\rC _{p}})^{h\rC _{p}}}
        & &{(X^{t\rC _{p}})^{h\rC _{p^{2}}}}
            & &{\dotsb, }
                }
\end{split}
\end{numequation}

\noindent where each vertical arrow save the first two is the identity
map.  Here the top two rows are those of \cref{eq-TR}, and the bottom
two rows, ignoring the first entry of each, form its image under the
functor $(-)^{h\rC _{p}}$.

Thus we can regard $\TR$ as a functor from $\qCycSp_{p}$ to
$\qTCart_{p}$ that preserves bounded below objects.
\end{remark}

A global form of a polygonal analog of $\TR$ is defined in
\cite[Definition 1.3]{KMN23}.  It is corepresented in $\qPgcSp$ by
$i\mS$, the constant $\mS$-valued  polygonal spectrum.

It follows that
\begin{displaymath}
\TR (X)=\lim{\pi_{k,L}}\TR^{k} (X),
\end{displaymath}

\noindent and the maps $\pi_{k,R}$ assemble into a map
$\pi_{R}:\TR(X)\to \TR (X)^{h\rC_{p}}$.  The maps $\pi_{k,L}$ assemble
into a map $\pi_{L}:\TR(X)\to \TR (X)$.

\begin{theorem}\label{thm-AN3.21}
{\bf A fully faithful right adjoint.} \cite[Theorem 3.21]{AN21} and
\cite[Lemma 2.11]{BHLS}. The functor $\TR:\qCycSp_{p,+} \to
\qTCart_{p,+}$ on the bounded below subcategory is fully faithful
and $t$-exact with left adjoint $X \mapsto X/\ANV$. The essential
image is the full subcategory of $\ANV$-complete bounded below
p-typical topological Cartier modules.
\end{theorem}

This implies the following, which can be viewed as an {\qcatal} reformulation
of \cite[Definition 5.12(i) and Remark 5.13]{BHM93}.

\begin{theorem}\label{thm-TC-TR}
{\bf The relation between $\TopC$ and $\TR$.}
$\TopC (X,p)$
is the equalizer in
\begin{displaymath}
\xymatrix
@R=4mm
@C=10mm
{
{\TopC (X,p)}\ar[r]^(.5){}
  &{\TR (X)}\ar@<.75ex>[r]^(.5){\pi_{L}}\ar@<-.75ex>[r]_(.5){1}
    &{\TR (X).}
}
\end{displaymath}
\end{theorem}

\section{$\THH$ of cochains on the circle: the foot in the
door}\label{sec-BHLS}

In this section we will explain how the machinery developed above can
be used to approach the {\TC}.  More details will be given in
\cite{Rav:gjmcyc-ext}.

\subsection{$\mS$-cochains on the circle}\label{sec-cochains-circle} The category (or
{\qcat}) of
spectra is cotensored over that of topological spaces, pointed or not.
This means that for a spectrum $E$ and a space $X$ one can define a
spectrum $E^{X}$ in which the $n$th component is the mapping space
$\Map (X, E_{n})$ if $X$ is pointed, and $\Map (X_{+}, E_{n})$
otherwise.  The same goes for both naive and genuine $G$-spectra, for
a pointed or unpointed $G$-space $X$.

The authors of \cite{BHLS} refer to $\mS^{X}$ for an unpointed space
$X$ as the spectrum of {\bf $\mS$-cochains on $X$}.  When $X$ is a finite
CW-complex, it is more traditionally known as $\DD X$, the {\bf
Spanier-Whitehead dual of $X$} as in \cref{eq-SW-dual}.


Let $R\in \Alg (\qSp)^{B\Z}$ be a $p$-complete $\EE_{1}$-ring spectrum with
an action of the integers $\Z$.  Then one has a diagram
of cyclotomic spectra
\begin{numequation}\label{eq-subgroups}
\begin{split}
\THH (R^{h\Z})
   \to\THH (R^{h (p\Z)})
      \to\THH (R^{h (p^{2}\Z)})
          \to \dotsb \to \THH (R).
\end{split}
\end{numequation}%

\noindent It is the image of the evident diagram of subgroups of $\Z$,
\begin{displaymath}
\xymatrix
@R=4mm
@C=14mm
{
{\Z\,}
  &{\,p\Z\,}\ar@{_{(}->} [l]^(.5){}
    &{\,p^{2}\Z\,}\ar@{_{(}->} [l]^(.5){}
      &{\,\dotsb\, }\ar@{_{(}->} [l]^(.5){}
        &{\,e}\ar@{_{(}->} [l]^(.5){}
}
\end{displaymath}

\noindent under the contravariant functor $\THH(R^{h (-)})$.  There is
a similar diagram with $\Z$ replaced by the $p$-adic integers
$\Z_{p}$.  The two are used interchangeably in \cite{BHLS} since they
assume that their spectra (including $\mS$) are $p$-adically complete.
Half of that paper (\S3, 4 and 5) is devoted to the study of the cyclotomic
spectra in \cref{eq-subgroups}.

When the action of $\Z$ on $R$ is trivial, then by \cref{eq-triv-action}
we have
\begin{displaymath}
R^{h (p^{i}\Z)}\simeq R^{B (p^{i}\Z)_{+}}
\simeq R\vee \Sigma^{-1}R.
\end{displaymath}

\noindent When $R=\mS$, this spectrum is that of {\bf cochains on the
circle} (the circle being the space $B (p^{i}\Z)\cong B (\Z)$), also
known as the {\bf dual circle}. It is the subject of \cite[\S3]{BHLS},
where they say
\begin{quote}
The key idea governing our analysis is that, since $\mS^{B\Z}$ is the
$\mS$-cochains on $B\Z_{p}$, $\THH(\mS^{B\Z})$ is controlled by the
geometry the free loop space of $B\Z_{p}$.
\end{quote}

That free loop space is described in \cref{sec-LS1}.  The connection
between $\THH$ and the free loop space is evident in
\cref{prop-THH-loop}, where the suspension spectrum of $\mcL X$ is
related to a {\Bok} functor.  The underlying cyclotomic spectrum of
$\THH(\mS^{B\Z})$ is studied by Malkiewich in \cite{malkiewic},
and the underlying commutative algebra is studied by Levy and David
Jongwon Lee in \cite[Lemma 4.6]{Lee-Levy}.

Malkiewich studies $\mS^{B\Z}$, which he calls $\DD S^{1}$, as an
$\EE_{1}$-ring spectrum and proves the following.

\begin{theorem}\label{thm-dual-circle}
{\bf $\THH$ of the dual circle.} 
\cite[Corollary 1.3]{malkiewic} 
As genuine $\mT$-spectra,
\begin{numequation}\label{eq-THHDS1}
\begin{split}
\THH (\mS^{B (p^{k}\Z)}) \simeq \THH (\mS^{B\Z}) \simeq \mS \oplus 
    \bigoplus_{r\geq 1}\Sigma^{-1}\mS[\mT/\rC_{r}]
\end{split}
\end{numequation}%

\noindent for $\mS[\mT/\rC_{r}]$ as in \cref{def-rep-T}.

Completing both $\Z$ and $\mS$ at $p$ gives 
\begin{displaymath}
\mS_{p} \oplus 
    \bigoplus_{k\geq 1}\left(\prod_{s\in \Z_{p}^{\times }}
         \Sigma^{-1}\mS[\mT/\rC_{p^{k}}]_{p} \right).
\end{displaymath}
\end{theorem}
 
In both cases the CW-spectrum in question has a single cell in
dimension $-1$ and infinitely many in dimension 0.  Note that by
\cref{eq-LS1},
\begin{align*}
\Sigma^{\infty -1}_{+}\mcL S^{1}
 &  \simeq \bigoplus_{r\in \Z}\Sigma^{-1}\mS[\mT/\rC_{|r|}]  \\
 &  \simeq \left(\bigoplus_{r\leq -1}\Sigma^{-1}\mS[\mT/\rC_{|r|}] \right)
     \vee \mS \vee 
     \left(\bigoplus_{r\geq 1}\Sigma^{-1}\mS[\mT/\rC_{|r|}] \right) ,
\end{align*}

\noindent which is similar to the expression of \cref{eq-THHDS1}.

On the other hand we know by \cref{ex-THH-Id} that $\THH (\mS)\simeq
\mS$.  This means the coassembly map
\begin{displaymath}
\epsilon :\THH (\mS^{B\Z})\to \THH (\mS)^{B\Z}
\simeq \mS^{B\Z}
\simeq \mS \oplus  \Sigma^{-1}\mS
\end{displaymath}

\noindent is very far from an equivalence. {\em This proposition will
enable us to show that various $\TopC$ coassembly maps are not
isomorphisms.}


For the next theorem we need some notation.

\begin{itemize}
\item [$\bullet$]  For a profinite set $X$, let
$C^{0} (X)$ denote the ring of continuous (meaning locally constant)
$\FFp $-valued functions on $X$.

\item [$\bullet$] For a commutative perfect $\FFp $-algebra $A$ (such
as $C^{0} (X)$ or a finite field), let $\mW (A)$ denote the
commutative ring spectrum of spherical Witt vectors as in
\cite[Example 5.2.7]{Lurie:Ell2}, which is typically a countable coproduct
of $p$-adic sphere spectra.  Let $\Perf_{p}$ denote the category of
commutative perfect $\FFp $-algebras, making $\mW$ a functor from it
to $\CAlg\qSp$, the category of commutative ring spectra. 
It is known (\cite[Proposition 2.2]{BSY}) to have a right adjoint, the
{\bf tilting functor} $(-)^{\flat}$, which is inverse limit along
Frobenius on $\pi_{0} (-)/p$.

\item [$\bullet$] Let $L_{\ltrunc p^{\infty }\rtrunc}\mS_{p}$ be the
polygonic spectrum of
\cref{ex-poly-spectra}\cref{ex-poly-spectraiii}. Then the cyclotomic
spectrum 
\begin{numequation}\label{eq-mL}
\begin{split}
  \mL
  :=\Upsilon  L_{\ltrunc p^{\infty }\rtrunc}\mS_{p}
 = \bigoplus_{j\geq 0} \mS[\mT/\rC_{p^{j}}], 
\end{split}
\end{numequation}%

\noindent for $\Upsilon $ as in \cref{def-U-poly}, is the
corepresenting object in $\qCycSp$ for $\TR$.  This is proved in
\cite[Lemma 2.6]{BHLS}. Note that by \cref{eq-LS1}, $\mL$ is a summand
of $\Sigma^{\infty }_{+}\mcL S^{1}$.
\end{itemize}

\begin{ex}\label{ex-C0Zp}
{\bf Continuous $\FFp $-valued functions on the $p$-adic integers.}  Each $a\in \Z_{p}$ can be written uniquely as 
\begin{displaymath}
a=\sum_{j\geq 0}a_{j}p^{j}\qquad \mbox{with }a_{j}^{p}=a_{j}. 
\end{displaymath}

\noindent By abuse of notation we can regard each coefficient $a_{j}$
as a continuous $\FFp $-valued function (by reducing mod $p$) that
factors though the quotient group $\Z/p^{j+1}$. It follows that
\begin{align*}
C^{0} (\Z/p^{j})
 & \cong  \FFp [a_{0},a_{1},\dotsc ,a_{j-1}]/ (a_{k}^{p}-a_{k}:0\leq k< j),\\
C^{0} (\Z_{p})
 & \cong  \colim{j}C_{0} (\Z/p^{j})
   \cong  \FFp [a_{0},a_{1},\dotsc ]/ (a_{k}^{p}-a_{k}:k\geq 0),\\
C^{0} (p^{k}\Z_{p})
 & \cong  C_{0} (\Z_{p})/ (a_{0},\dotsc ,a_{k-1}),\\
\aand C^{0} (\Z_{p}^{\times })
 & \cong  (a_{0})\subseteq  C^{0} (\Z_{p}).
\end{align*}

\noindent The ring $C^{0} (\Z/p^{j})$ has rank $p^{j}$ as an $\FFp
$-vector space.
\end{ex}

\begin{prop}\label{thm-3.19}
{\bf The fibers of two coassembly maps.}
\begin{enumerate}[label={(\roman*)},itemindent=0em]
\item \label{prop-3.19i}\cite[Proposition 3.19]{BHLS} 
The fiber of the
coassembly map  (see \cref{ex-assembly})
\begin{displaymath}
\epsilon :\THH(\mS_{p}^{B\Z_{p}})\to
\THH(\mS_{p})^{B\Z_{p}}
\end{displaymath}

\noindent is
\begin{displaymath}
\Sigma^{-1}\mL\otimes\mW (C^{0} (\Z^{\times }_{p})) 
\end{displaymath}

\noindent for $\mL$ as in \cref{eq-mL} and $C^{0} (\Z^{\times }_{p})$
as in \cref{ex-C0Zp}.

\item \label{prop-3.19ii}\cite[Corollary 3.21]{BHLS} For a connective
ring spectrum $R$, the thick subcategory generated by the fiber of
$\epsilon :\TopC(R^{B\Z_{p}})\to \TopC(R)^{B\Z_{p}}$ contains
\begin{displaymath}
\mW (C^{0} (\Z^{\times }_{p}))\otimes \TopC(R).
\end{displaymath}

\noindent In particular $\epsilon $ is not an equivalence when
$\TopC(R)$ is nontrivial.
\end{enumerate}
\end{prop}

An important step toward the disproof of the {\TC} is the
following, which is a consequence of \cref{thm-3.19}.

\begin{theorem}\label{thm-failure}
{\bf The $T (n+1)$-local $K$-theory coassembly map for the trivial
$\Z$-action.}  \cite[Theorem 3.22 for $X=\mS$]{BHLS} Let $R$ be a $T
(n)$-local $\EE_{1}$-ring spectrum for $n\geq 1$.  If $\rK_{T (n+1)}
(R)$ (see \cref{def-KT-KK}) is nontrivial, then the coassembly map
(\cref{ex-assembly}) for the trivial action of $\Z$ on $R$,
\begin{displaymath}
\epsilon :\rK_{T (n+1)} (R^{B\Z}) \to \rK_{T (n+1)} (R)^{B\Z},
\end{displaymath}

\noindent is {\em not} an equivalence.
\end{theorem}

\subsection{$K$-theory and $\TopC$}\label{sec-K-theory-TC}

In \cite[\S6.1]{BHLS} they explain why $T(n + 1)$-localized algebraic
$K$-theory and $T(n + 1)$-localized $\TopC$ coincide for examples of
interest. For connective rings,  we have the following.

\begin{theorem}
\label{thm-purity}
  {\cite[Purity Theorem]{LMMT}} and
  {\cite[Cor.\ 4.11]{CMNN2}}.
  Let $R$ be a connective
$\EE_1$-algebra. For $n \ge 1$, the $(T(n)\oplus T(n+1))$-localization
map and the cyclotomic trace induce equivalences
\begin{displaymath}
\xymatrix
@R=4mm
@C=10mm
{
{\rK_{T(n+1)} \!\bigl(L_{T(n)\oplus T(n+1)} R\bigr)}
  &{\rK_{T(n+1)} (R)}\ar[r]^(.47){\trc}_(.47){\simeq}
                      \ar[l]_(.37){\eta}^(.37){\simeq}
    &{\TopC_{T(n+1)} (R).}
 }
\end{displaymath}

\end{theorem}

To disprove the telescope conjecture, we will need to understand the
topological cyclic homologies of fixed points of $\Z$-actions on
connective $\EE_{1}^{}$-algebras. Such fixed points are
(-1)-connective, but often not 0-connective, so the above theorem does
not apply.

In order to get around this we will use \cref{thm-Levy-trunc} below,
which requires the following.

\begin{defin}\label{def-trunc-inv}\cite[Definition 3.1]{Land-Tamme} 
A localizing invariant as in \cref{def-loc-inv}
\[
E \colon  \qCatw^{\mathrm{perf}} \longrightarrow
\qSp
\]

\noindent (with source as in \cref{def-qcats}\cref{def-qcatsix}) is a
{\bf truncating invariant} if for every connective \(\EE_1\)-ring
spectrum $R$, the canonical map
\[
E(R) \longrightarrow E(H\pi_0(R))
\]

\noindent is an equivalence.  Here $E (R)$ is understood to be the
value of $E$ on the {\qcat} $\qMod_{R}$ of
\cref{def-qcat-props}\cref{def-qcat-propsv} and similarly for the
{\SESM} ring spectrum $H\pi_0(R)$.
\end{defin}
 
In other words, the truncating invariant $E$ does not see the positive
dimensional homotopy groups of $R$.

\begin{theorem}\label{thm-Levy-trunc}
\cite[Theorem B]{Levy:K(1)-local} Let \(R_{0}\) and \(R_{1}\) be connective
\(\EE_1\)-algebras with \(\mathbb{Z}\)-action.  Let
\[
f \colon R_{0} \longrightarrow R_{1}
\]
be a \(1\)-connective (meaning it induces an isomorphism in
$\pi_{0}$), \(\mathbb{Z}\)-equivariant \(\EE_1\)-algebra map. For any
truncating invariant \(E\), the induced map
\[
E(R_{0}^{h\mathbb{Z}}) \longrightarrow E(R_{1}^{h\mathbb{Z}})
\]
is an equivalence.
\end{theorem}

\begin{cor}\label{cor-BHLS6.3}
\cite[Corollary 6.3]{BHLS} 
For $n \ge 1$, let $R$ be a $T(n+1)$-acyclic, connective
$\EE_1$-algebra with a $\mathbb{Z}$-action.  The coassembly map
$\epsilon $, the $T(n)$-localization map $\eta $, and the
cyclotomic trace $\trc$ fit into a commuting diagram
\begin{displaymath}
\xymatrix
@C=5em
@R=3em{
  \rK_{T(n+1)}(L_{T(n)}R^{h\mathbb{Z}})    
    \ar[d]_{\epsilon}
&
  \rK_{T(n+1)}(R^{h\mathbb{Z}})
    \ar[d]^{\epsilon }
    \ar[l]_(.45){\rK_{T(n+1)} (\eta  )}^(.45){\simeq}
    \ar[r]^{\trc}_{\simeq}
&
  \TopC_{T(n+1)}(R^{h\mathbb{Z}})\ar[d]_{\epsilon}
\\
  \rK_{T(n+1)}(L_{T(n)}R)^{h\mathbb{Z}}
&
  \rK_{T(n+1)}(R)^{h\mathbb{Z}}
    \ar[r]^{\trc}_{\simeq}
    \ar[l]_(.45){\rK_{T(n+1)} (\eta  )}^(.45){\simeq}
&
  \TopC_{T(n+1)}(R)^{h\mathbb{Z}}
}
\end{displaymath}

\noindent
where each horizontal map is an equivalence.
\end{cor}

The coassembly maps in \cref{cor-BHLS6.3} need not be equivalences.
Indeed \cref{thm-failure} says the middle one is not an equivalence
when the group action is trivial and $\rK_{T(n+1)}(R)$ is nontrivial.

If we knew that the $K (n+1)$-local analog of the coassembly map of
\cref{thm-failure} was an equivalence for $n\geq 1$, we would know
that the {\TC} is false.  What we do know is two steps removed from
this. There is a {\em particular $R$}, namely $L_{T (n)}BP\langle n
\rangle$, with a {\em nontrivial} action of $\Z$ for which the
coassembly map is a $K (n+1)$-local but not a $T (n+1)$-local
equivalence. This will be discussed in \cite{Rav:gjmcyc-ext}, where we
will see that a crucial ingredient is \cite[Theorem C]{BMCSY23}.

\subsection{Coming attractions}\label{sec-coming-attractions}

In the next paper we will talk about how to relax the triviality
hypothesis for the group action of \cref{thm-failure}. 

\begin{defin}\label{def-loc-uni}
{\bf Local unipotence.}  For an action of $\Z$ on an abelian group
$A$, let $\Psi :A\to A$ be the automorphism induced by a generator of
$\Z$. Then the action is {\bf unipotent} if $\Psi -1$ is a nilpotent
endomorphism of $A$.  An action of $\Z$ on a spectrum $R$ is {\bf
locally unipotent} if the induced action onn each
homotopy group is unipotent.
\end{defin}

Such group actions are studied in \cite[\S4]{BHLS}.  The specific
$\Z$-action on $\BPn$ by Adams operations is the subject of
\cite[\S5]{BHLS}.  

In \cite[\S6]{BHLS} they show that for \(n \geq 1\) and \(k \geq 0\),
the \(T(n+1)\)-localized coassembly map
\[
\epsilon :
\rK_{T (n+1)}\!\left(BP\langle n\rangle^{h p^{k}\mathbb{Z}}\right)
\;\longrightarrow\;
\rK_{T (n+1)}\!\left(BP\langle n\rangle\right)^{h p^{k}\mathbb{Z}}
\]

\noindent is not an equivalence, but becomes one after
\(K(n+1)\)-localization.  Thus, the functors $\rK_{T (n+1)}$ and
$\rK_{K (n+1)}$ differ, so the telescope conjecture fails.

Here we repeat the words of \cite{BHLS}  that we quoted in \cref{sec-roadmap}.

\begin{quote}
We do this by looking at the coassembly map from two highly divergent
perspectives, which are connected via trace theorems:
\begin{enumerate}
  \item From the perspective of locally unipotent $\mathbb{Z}$-actions
on ring spectra, the results of \cite[\S4]{BHLS} tell us that the coassembly
map cannot be an isomorphism.

  \item From the perspective of cyclotomic redshift of \cite{BMCSY23},
the map 

\[ L_{T(n)}\BPn ^{h p^k \mathbb{Z}}
\;\longrightarrow\; L_{T(n)}\BPn
\] 

\noindent splits after base change to the maximal abelian extension of
the $K(n)$-local sphere, and therefore the coassembly map is a
$K(n+1)$-local isomorphism.
\end{enumerate}
\end{quote}

\appendix

\preto\appendix{%
  \renewcommand{\thesection}{\Alph{section}}%
  \renewcommand{\theequation}{\thesection.\arabic{equation}}%
  \renewcommand{\theHequation}{\thesection.\arabic{equation}}%
  \setcounter{equation}{0}%
}

\section{Some {\eqvr} homotopy theory}\label[appendix]{sec-EHT}

\subsection{Residual action and restriction}\label{sec-residual-action}
\begin{defin}\label{def-residual}
Let $G$ be a group acting on a space or spectrum $X$, and let
$H\subseteq G$ be a subgroup.

\begin{enumerate}[label={(\roman*)},itemindent=0em]
\item \label{def-residualii} When $H$ is normal in $G$, the action of
$G$ on $X$ induces {\bf residual actions} of $G/H$ on $X^{H}$ and
$X_{H}$, and a {\bf residual action homeomorphisms}
\begin{displaymath}
\Res ^{G}_{H}:X^{G}\to (X^{H})^{G/H}
\qquad \aand 
\widehat{\Res} ^{G}_{H}:(X_{H})_{G/H}\to X_{G}.
\end{displaymath}

\noindent We also have  {\bf residual action homotopy equivalences}
\begin{displaymath}
h\Res ^{G}_{H}:X^{hG}\to (X^{hH})^{h (G/H)}
\qquad \aand 
h\widehat{\Res} ^{G}_{H}:(X_{hH})_{h (G/H)}\to X_{hG}.
\end{displaymath}

\item \label{def-residuali} 
For arbitrary $H\subseteq G$, the {\bf
restriction maps}
\begin{align*}
\Frob^{G}_{H}&:X^{G}\to X^{H}&
\mbox{and} &&
\widehat{\Frob}^{G}_{H}&:X_{H}\to X_{G}\\
h\Frob^{G}_{H}&:X^{hG}\to X^{hH}&
\mbox{and} &&
h\widehat{\Frob}^{G}_{H}&:X_{hH}\to X_{hG}
\end{align*}


\noindent are defined by the fact that any point fixed by $G$ is fixed
by its subgroup $H$ and any $H$-orbit is part of a $G$-orbit.
\end{enumerate}

In both cases the indices may be omitted when they are clear from the
context.  In this paper the groups will always be finite cyclic $p$-groups,
and we will sometimes write
\begin{displaymath}
\Res =\Res _{p}:= \Res ^{\rC_{p^{n+1}}}_{\rC_{p}}
\qquad \aand 
\Frob=\Frob_{p}:=\Frob^{\rC _{p^{n+1}}}_{\rC _{p^{n}}}.
\end{displaymath}
\end{defin}

This notation was introduced by Madsen in \cite{Madsen} and has been
used in most other papers on this subject since then. Since the words
`residual' and `restriction' both begin with the same letter, the
choice is not obvious, and occasionally one sees $\Res^{}_{}$ used for
the restriction map instead.  The maps were originally denoted by
$\Phi $ and $D$ respectively in \cite{BHM93}, \cite{Bokstedt-Madsen94}
and \cite{Madsen94}.  The letters in \cref{def-residual} stand for
{\bf restriction} (in a different sense) and {\bf Frobenius}, as
explained by Hesselholt and Madsen in \cite[page 3]{HM1}.  They are
related to similarly named maps in the theory of Witt vectors, which
Serre describes in \cite[\S II.6]{Serre-Local}.

The following is elementary.

\begin{prop}\label{prop-RF}
{\bf $\Frob$ and $\Res$ commute.}  Let $K\lhd H\lhd G$ be groups
acting on a space or spectrum $X$.  Then the following diagram
commutes.
\begin{displaymath}
\xymatrix
@R=10mm
@C=20mm
{
{X^{G}}\ar[r]^(.45){\Res^{G}_{K}}_(.45){\cong }
       \ar[d]_(.5){\Frob^{G}_{H}}
  &{(X^{K})^{G/K}}\ar[d]^(.5){\Frob^{G/K}_{H/K}}\\
{X^{H}}\ar[r]^(.45){\Res^{H}_{K}}_(.45){\cong }
  &{(X^{K})^{H/K}.}
}
\end{displaymath}
\end{prop}

For a prime $p$, integer $n>0$ and groups $\rC_{p}\lhd
\rC_{p^{n}}\lhd \rC_{p^{n+1}}$, \cref{prop-RF} gives
\begin{numequation}\label{eq-RF}
\begin{split}
\xymatrix 
@R=6mm
@C=16mm
{
{X^{{\rC}_{p^{n+1}}}}\ar[r]^(.4){\Res}_(.4){\cong }
       \ar[d]_(.5){\Frob}
  &{(X^{{\rC}_{p}})^{{\rC}_{p^{n+1}} / {\rC}_{p}}}
    \ar[d]^(.5){\Frob}\\
{X^{{\rC}_{p^{n}}}}\ar[r]^(.4){\Res}_(.4){\cong }
  &{(X^{{\rC}_{p}})^{{\rC}_{p^{n  }}/{\rC}_{p}}.}
}
\end{split}
\end{numequation}%

\begin{remark}\label{rem-T-action}
When the group is $\mT$, we have fixed point sets for all of its finite
subgroups, and there is a global analog of \cref{eq-RF} that we do
not need here.  These maps are also discussed by Blumberg and Mandell in
\cite[Definition 6.1]{BM15} where they are used to define functors
$\TR$, $\TF$ and $\TopC$ on fibrant cyclotomic spectra as homotopy
limits obtained by iterating $\rR$, $\rF$ or both. They give
$p$-typical versions, which they denote by $\TR (-;p)$ and so on, in
\cite[Definition 6.3]{BM15}.  Then they show that the right derived
functors of $\TR (-;p)$ (\cite[Theorem 6.5]{BM15}) and $\TopC (-;p)$
(\cite[Theorem 6.8]{BM15}) are corepresentable by $p$-cyclotomic
spectra described in \cite[Constructions 6.4 and 6.6]{BM15}.  They do
this for the global $\TR$ in \cite[Theorem 6.12]{BM15}, which
\cite{McCandless-curves} cites as justification for his claim about
$\THH$ of the spectral affine line.  Their proof makes use of the
definition of $\TR$ as a homotopy limit (or ``mapping microscope'') of
fixed point sets.

In \cite[Construction 5.11]{BM15} they define mapping spaces
$F^{h}_{Cyc} (X,Y)$ in the $p$-cyclotomic category that anticipates
\cref{def-cyc-spectra}.  They show it is the categorical mapping space
when $X$ is cofibrant and $Y$ is fibrant.  Presumably when we pass the
the {\qcatal} world, we need not worry about derived functors or
fibrancy and cofibrancy.  See \cref{rem-TC-fixed}.
\end{remark}

\subsection{Fixed point spaces and orbit spaces}\label{sec-fixed-orbit-spaces}

For a compact Lie group $G$ (such as $\mT$), let $\mcT_{G}$
denote the category of pointed $G$-spaces (assumed to be compactly
generated and weak Hausdorff for technical reasons) where the
basepoint is fixed by $G$, and continuous (but not necessarily
{\eqvr}) pointed maps. This category is enriched over itself: for
pointed $G$-spaces $X$ and $Y$, the morphism set $\mcT_{G}
(X,Y)$ equipped with the compact-open topology is itself a pointed
$G$-space, where for an element $\gamma \in G$ and a map
$f:X\to Y$, $\gamma (f)$ is defined to be $\gamma f\gamma^{-1} $.  The
map $f$ is {\eqvr} iff $\gamma (f)=f$ for all $\gamma \in G$, so the
fixed point set $\mcT_{G} (X,Y)^{G}$ (which does not have a
$G$-action) is the pointed space of {\eqvr} maps from $X$ to $Y$.
Thus we can define the category $\mcT^{G}$ of pointed
$G$-spaces and {\eqvr} maps by $\mcT^{G} (X,Y):=\mcT_{G}
(X,Y)^{G}$. It is enriched over the category of pointed spaces
$\mcT$ rather than over $\mcT_{G}$ or itself.

A pointed $G$-space can also be regarded as a $\mcT$-valued functor on
$\mathcal{B}G$, the one object topological category with an
automorphism for each element of $G$, with composition given by the
binary operation of $G$.  (A category is topological if each of its
morphism sets has a topology with suitable continuity conditions.)
One then sees that the fixed point space $X^{G}$ and orbit space
$X_{G}$ are the limit and colimit of this functor.

\begin{prop}\label{prop-fpms}
{\bf The fixed point set as a mapping space.}  For a pointed $G$-space
$X$, the fixed point set $X^{G}$ is also $\mcT^{G} (S^{0},X)$, where
the group action on $S^{0}$ is trivial. More generally for each
subgroup $H\subseteq G$, $X^{H}=\mcT^{G} (G/H_{+},X)$.
\end{prop}

The analog of this in the {\qcat} $\qSp^{BG}$ of
\cref{def-qcats}\cref{def-qcatsvi} is
\begin{displaymath}
\map_{\qSp^{BG}} (\Sigma^{\infty }G/H_{+}, Y)\cong Y^{hH}.
\end{displaymath}

\noindent In particular for a naive $\mT$-spectrum $Y$,
\begin{displaymath}
\map_{\qSp^{B\mT}} (\mS[\mT/\rC _{p^{k}} ], Y)\cong Y^{h\rC _{p^{k}}},
\end{displaymath}

\noindent where $\mS[\mT/\rC _{p^{k}} ]$ is as in \cref{def-rep-T}.

There is a way to construct a contractible $G$-space with free
$G$-action, commonly denoted by $EG$; it is unique up to {\eqvr}
homotopy equivalence.  Its orbit space is the classifying space $BG$.

\begin{ex}\label{ex-S1}
 {\bf The case $G=\mT $.} We regard $\mT$ as the multiplicative
group of complex numbers on the unit circle. For each $k>0$, it acts
freely by scalar multiplication of $S^{2k-1}$, the space of unit
vectors in $\cxs^{k}$.  The orbit space $S^{2k-1}_{\mT}$ is the complex
projective space $\cxs P^{k-1}$.  

We have {\eqvr} maps
\begin{numequation}\label{eq-odd-spheres}
\begin{split}
\xymatrix
@R=4mm
@C=10mm
{
{\cdots }\ar[r]^(.5){}
  &{S^{2k-1}}\ar[r]^(.5){}
    &{S^{2k+1}}\ar[r]^(.5){}
      &{S^{2k+3}}\ar[r]^(.5){}
        &{\cdots}
}
\end{split}
\end{numequation}%

\noindent with contractible colimit.  This is our contractible free
$\mT$-space $E\mT$, It also serves as a contractible $H$-space for any
subgroup $H\subseteq \mT$.  For $H=\rC _{r}$, the orbit space
$S^{2k-1}_{\rC _{r}}$ is a lens space, and $E\mT_{\rC _{r}}$ is the {\SESM}
space $B\rC _{r}=K (\rC _{r}, 1)$.

The diagram of $\mT$-orbit spaces for  \cref{eq-odd-spheres} is
\begin{displaymath}
\xymatrix
@R=4mm
@C=10mm
{
{\cdots }\ar[r]^(.5){}
  &{\CP^{k-1}}\ar[r]^(.5){}
    &{\CP^{k}}\ar[r]^(.5){}
      &{\CP^{k+1}}\ar[r]^(.5){}
        &{\cdots,}
}
\end{displaymath}

\noindent where $\CP^{k}$ is $k$-dimensional complex projective space,
the space of complex lines through the origin in $\cxs^{k+1}$.  The
colimit is $\CP^{\infty }=K (\Z, 2)$.

If we replace $\mT$ by its subgroup $\rC _{2}=\left\{\pm 1 \right\}$, the
orbit space diagram becomes
\begin{displaymath}
\xymatrix
@R=4mm
@C=10mm
{
{\cdots }\ar[r]^(.5){}
  &{\RP^{2k-1}}\ar[r]^(.5){}
    &{\RP^{2k+1}}\ar[r]^(.5){}
      &{\RP^{2k+3}}\ar[r]^(.5){}
        &{\cdots,}
}
\end{displaymath}

\noindent for which the colimit is 
\begin{numequation}\label{eq-RP}
\begin{split}
\RP^{\infty }=\colim{m}\RP^{m}=B\rC _{2},
\end{split}
\end{numequation}%

\noindent where $\RP^{m}$ is $m$-dimensional real projective space,
the space of real lines through the origin in $\reals^{m+1}$. Note
that $\cxs^{k+1}$ as a real vector space is $\reals^{2k+2}$.
\end{ex}

Then we define the {\em homotopy fixed point space} as
\begin{numequation}\label{eq-hfps}
\begin{split}
X^{hG} := \mcT^{G} (EG_{+},X),
\end{split}
\end{numequation}%

\noindent where $EG_{+}$ denotes $EG$ with a disjoint base point. This
may or may not be homotopy equivalent to $X^{G}$.  Its homotopy type
is known to be independent of the choice of $EG$.  The map
$j:EG_{+}\to S^{0}$ of \cref{eq-useful} induces a map to $X^{hG}$ from
$\mcT^{G} (S^{0},X)=X^{G}$, so one might think of $X^{hG}$ as a
fattened up version of $X^{G}$.

Consider the $G$-{\eqvr} maps
\begin{numequation}\label{eq-GES}
\begin{split}
\xymatrix
@R=4mm
@C=10mm
{
{G_{+}}\ar[r]^(.5){i}
  &{EG_{+}}\ar[r]^(.5){j}
    &{S^{0},}
}
\end{split}
\end{numequation}%

\noindent where $i$sends $G$ to some orbit of $EG$ and $j$ send $EG$
to the nonbase point in $S^{0}$ as in \cref{eq-useful}. Applying the
functor $\mcT^{G} (-,X)$ gives maps
\begin{numequation}\label{eq-fp-spaces}
\begin{split}
\xymatrix
@R=4mm
@C=10mm
{
{X}
  &{X^{hG}}\ar[l]_(.45){i^{*}}
    &{X^{G}.}\ar[l]_(.4){j^{*}}
}
\end{split}
\end{numequation}%

When the action of $G$ on $X$ is trivial, an {\eqvr} map $EG\to X$
factors through the orbit space $EG_{G}=BG$, so
\begin{numequation}\label{eq-triv-action}
\begin{split}
X^{hG}\simeq \mcT (BG_{+}, X).
\end{split}
\end{numequation}%
 
One also has the {\em homotopy orbit space}
\begin{numequation}\label{eq-hos}
\begin{split}
X_{hG}: = EG_{+}\wedge_{G}X = EG\times_{G}X,
\end{split}
\end{numequation}%

\noindent which is defined as follows.  The smash product
$EG_{+}\wedge X$ is the quotient of $EG\times X$ by $EG\times \left\{
x_{0} \right\}$ for the base point $x_{0}\in X$.  The diagonal action
of $G$ on $EG\times X$ induces one on this quotient, and $X_{hG}$ is
its orbit space.  The projection map $EG\times X\to EG$ leads to a map
\begin{displaymath}
X_{hG}\to EG \times _{G} (*) = (EG)_{G}= BG
\end{displaymath}

\noindent with fiber $X_{G}$.  Here the map
$j:EG_{+}\to S^{0}$ induces one from $X_{hG}$ to
\begin{displaymath}
S^{0}\wedge_{G}X=X_{G},
\end{displaymath}

\noindent so $X_{hG}$ is a space over $X_{G}$.

Applying the functor $-\wedge_{G}X$ to \cref{eq-GES} gives maps
\begin{numequation}\label{eq-orbit-spaces}
\begin{split}
\xymatrix
@R=4mm
@C=10mm
{
{X}\ar[r]^(.45){i_{*}}
  &{X_{hG}}\ar[r]^(.5){j_{*}}
    &{X_{G}.}
}
\end{split}
\end{numequation}%

\subsection{{\Eqvr} homotopy groups}\label{sec-eqvr-homotopy-groups}

\begin{defin}\label{def-piHX}
{\bf Twisted loop spaces, twisted suspensions and {\eqvr} homotopy
groups.}  For a finite dimensional orthogonal {\rep} $V$ of a compact
Lie group $G$, we denote by $S^{V}$ the one point compactification of
$V$ (its {\bf {\rep} sphere}), and by $S (V)$ its space of unit
vectors.  For a pointed $G$-space $X$, let
\begin{displaymath}
\Omega^{V}X := \mcT_{G} (S^{V},X)
\qquad \aand 
\Sigma^{V}X := S^{V}\wedge X,
\end{displaymath}

\noindent the {\bf twisted loop space} and {\bf twisted suspension} of
$X$.  When $V^{G}=0$, we denote the inclusion of fixed points
$S^{0}\to S^{V}$ by $a_{V}$.

For each closed subgroup $H\subseteq G$, let
\begin{displaymath}
\pi^{H}_{V}X:=\pi_{0}\mcT^{H} (S^{V},X),
\end{displaymath}

\noindent the set of homotopy classes of $H$-{\eqvr} maps $S^{V}\to
X$, where we are regarding $S^{V}$ and $X$ as pointed $H$-spaces by
restricting the $G$-action to $H$.  We omit the superscript $H$ when
it is the trivial subgroup. If the action of $H$ on $V$ is trivial, we
write $\pi^{H}_{V}X$ as $\pi^{H}_{|V|}X$.
\end{defin}

The set $\pi^{H}_{V}X$ has a natural
group structure when the fixed point vector space $V^{H}$ is
nontrivial.  This group is abelian when $|V^{H}|>1$.  Thus we have
homotopy groups indexed by such {\rep}s for each subgroup.  The set
$\pi^{H}_{V}X$ depends only on the action of $H$ (rather than $G$) on
$V$ and $X$.

A theorem of Bredon \cite{Bredon} says that a map $f:X\to Y$ of path
connected pointed $G$-CW complexes (see \cite[Definition
8.4.4]{HHR:ESHT}) is an {\eqvr} homotopy {\equi} iff it induces an
isomorphism is $\pi_{*}^{H}$ (meaning the integer graded groups) for
each $H\subseteq G$.

It is known that for any module $M$ over the finite group $\pi_{0}G$,
there is an {\SESM} $G$-spectrum $\HM$ with
\begin{numequation}\label{eq-HM}
\begin{split}
\pi^{H}_{i}\HM=\mycases{    
M       
       &\mbox{for }i=0\\
0      &\mbox{for nonzero integers $i$}
}
\end{split}
\end{numequation}%

\noindent for each closed subgroup $H\subseteq G$.  A similar
statement is true if we replace $M$ by a Mackey functor $\uM$, but we
do not need this here.  Note that we are requiring only that $\HM$ has
just one nontrivial {\em integer graded} homotopy group.  In general
there will be other {\rep}s $V$ for which $\pi^{H}_{V}HM$ is
nontrivial.  For the case of finite $G$, see \cite[\S9.1H]{HHR:ESHT}.







From the sequence
\begin{displaymath}
\rC _{m}\subseteq \rC _{mn}\subseteq \mT
\end{displaymath}

\noindent we get 
\begin{displaymath}
\rC _{mn}/\rC _{m}\to \mT/\rC _{m}\to \mT/\rC _{mn}
\end{displaymath}

\noindent corepresenting
\begin{displaymath}
(i^{\mT}_{\rC _{mn}}X)^{\rC _{m}} \leftarrow X^{\rC _{m}}
 \leftarrow X^{\rC _{mn}}
\end{displaymath}

\section{The {\qcat} of genuine $G$-spectra}\label[appendix]{sec-qcat-Gspec}

In \cite[Definition II.2.5]{NS18} Nikolaus and Scholze say
${\color{cyan}G\Sp}$ is the {\qcat} obtained from the simplicial set
$N (\Sp^{G})$ (for $\Sp^{G}$ as in \cref{def-orth-spec}) by inverting
stable equivalences. This process is treated in Lurie's Kerodon
[\url{https://kerodon.net}] as follows.

\begin{defin}\label{def-FCD}
{\rm [\url{https://kerodon.net/tag/01M5}]} Let $F:\mcC \to \mcD $ be a
functor between categories and let $\mcW $ be a collection of
morphisms in $\mcC $.  We say that {\bf $F$ exhibits $\mcD $ as a
strict localization of $\mcC $ with respect to $\mcW $} if, for every
category $\mcE$, precomposition with F induces a bijection
\begin{displaymath}
\xymatrix
@R=10mm
@C=10mm
{
{\left\{\mbox {Functors } \mcD \to \mcE \right\}}\ar[d]^(.5){}\\
{\left\{\mbox {Functors $\mcC \to \mcE$ carrying each 
  $w\in \mcW$  to an isomorphism in $\mcE $}   \right\}.}
}
\end{displaymath}
\end{defin}

It turns out that such a category $\mcD $ is uniquely determined by
the category $\mcC $ and the morphism collection $\mcW $, and we will
denote it by $\mcC [\mcW^{-1}]$. Explicitly, the category $\mcC
[\mcW^{-1}]$ can be constructed from $\mcC $ by adjoining a new
morphism $w^{-1}:Y\to X$ for each morphism $w:X \to Y$ of $\mcW $, and
imposing the relations $w^{-1}\cdot w= 1_{X}$ and $w\cdot w^{-1}=
1_{Y}$.  

\bigskip Now we generalize the situation above, replacing the
categories $\mcC $ and $\mcD$ by simplicial sets $\widetilde{\mcC} $
and $\widetilde{\mcD} $ (which Lurie calls $\mcC $ and $\mcD $, like
the categories above) not required to be nerves, meaning simplicial
sets to which each map of an inner horn $\partial \bdelt ^{n}_{i}$
with $0<i<n$ (see \cref{def-simp-sets}) has a unique extension to the
full simplex $\bdelt^{n}$.  The domain simplicial set
$\widetilde{\mcC} $ comes equipped with a collection of edges
$\widetilde{\mcW} $.  The target category $\mcE $ above gets replaced
by an {\qcat} $\qmcE$.

\cref{prop-exist-loc} says that for each $\widetilde{\mcC} $ and
$\widetilde{\mcW} $ there is an {\qcat} $\qmcD$ that does the
job that $\mcD $ does in \cref{def-FCD}.

\begin{defin}\label{def-loc-qcat}
{\rm [\url{https://kerodon.net/tag/01MP}]} Let $F :
\widetilde{\mcC}\to \widetilde{\mcD} $ be a morphism of simplicial
sets and let $\widetilde{\mcW} $ be a collection of edges in
$\widetilde{\mcC}$. We say that {\bf $F$ exhibits $\widetilde{\mcD} $
as a localization of $\widetilde{\mcC} $ with respect to
$\widetilde{\mcW} $} if, for every {\qcat} $\qmcE$, the precomposition
map $\qFun(\widetilde{\mcD}, {\qmcE}) \circ F \to \qFun(\widetilde{\mcC}
, \qmcE)$ is fully faithful, and its essential image is the full
subcategory of functors that send edges in $\mcW $ to isomorphisms in
$\qmcE$.
\end{defin}

Lurie denotes this essential image by
\begin{displaymath}
\qFun(\widetilde{\mcC} [\widetilde{\mcW} ^{-1} ], {\qmcE})
\end{displaymath}

\noindent without defining $\widetilde{\mcC}[\widetilde{\mcW}^{-1}]$.
The next result says there is an {\qcat} having the properties one
might expect of this undefined simplicial set.

\begin{prop}\label{prop-exist-loc}
{\bf Existence of localizations.}  
{\rm [\url{https://kerodon.net/tag/01N0}]} 
Let $\widetilde{\mcC}$ be a simplicial set and let
$\widetilde{\mcW} $ be a collection of edges of $\widetilde{\mcC}
$. Then there exists an {\qcat} $\qmcD$ and a morphism of simplicial
sets $F:\widetilde{\mcC} \to \qmcD$ which exhibits $\qmcD$ as a
localization of $\widetilde{\mcC}$ with respect to $\widetilde{\mcW}$.
\end{prop}

\begin{remark}
{\bf Not in HTT.}  
As far as we can tell,  no similar result is proved in
\cite{Lurie:HTT}.  In \S5.2.7 Lurie says the following.

\begin{quote}
Suppose we are given an {\qcat} $\qmcC$ and a collection $S$ of
morphisms of $\qmcC$ which we would like to invert. In other words, we
wish to find an {\qcat} $S^{-1}\qmcC$ equipped with a functor $\eta
:\qmcC \to S^{-1}\qmcC$ which carries each morphism in $S$ to an
equivalence and is in some sense universal with respect to these
properties. One can give a general construction of $S^{-1}\qmcC$ using
the formalism of \S3.1.1. \dots 
However, this construction is
generally very difficult to analyze, and the properties of
$S^{-1}\qmcC$ are very difficult to control. For example, it might be
the case that $\qmcC$ is locally small and $S^{-1}\qmcC$ is not.

Under suitable hypotheses on $S$ (see \S5.5.4), there is a drastically
simpler approach: we can find the desired {\qcat} $S^{-1}\qmcC$ {\em
inside} $\qmcC$ as the full subcategory of $S$-local objects of
$\qmcC$.

\end{quote}

Finding the localized category inside the original one was the approach
taken by Bousfield in his theorem about the localization of a model
category.  See \cite[\S10]{Rav:Whatis} and \cite[Chapter 6]{HHR:ESHT}
for more discussion.

\end{remark}

Recall that a {\em relative category} $(\mcC,\mcW)$ consists of a
category $\mcC $ and a wide subcategory $\mcW $, which can be
identified with its morphism collection, which we also denote by
$\mcW$. This morphism collection, since it is that of a subcategory,
contains all identities and is closed under composition.

\begin{cor}\label{cor-qcat-relcat}
{\bf The {\qcat} for a relative category.}  Let $(\mcC,\mcW ) $ be a
relative category, let $\widetilde{\mcC} =N (\mcC )$, and let
$\widetilde{\mcW }$ be the collection of edges in $\widetilde{\mcC }$
corresponding to the collection $\mcW $ of morphisms in $\mcC $.  Then
there exists an {\qcat} $\qmcC[\mcW^{-1}]$ and a morphism of
simplicial sets $F:N (\mcC ) \to \qmcC[\mcW^{-1}]$ which exhibits
$\qmcC[\mcW^{-1}]$ as a localization of $N (\mcC )$ with respect to
$\widetilde{\mcW} $ as in \cref{def-loc-qcat}.
\end{cor}

\begin{defin}\label{def-qcat-Gspec}
The {\bf {\qcat} of orthogonal  $G$-spectra} $\qSp^{G}$ is that obtained as in
\cref{cor-qcat-relcat} where $\mcC $ the category of of orthogonal
$G$-spectra, $\Sp^{G}$ as in \cref{def-orth-spec}, and $\mcW $ is the
collection of stable {\eqiv}s as in \cref{def-pi-stable}.
\end{defin}

For trivial $G$ the {\qcat} $\qSp^{G}$, that of orthogonal spectra, is
not the same as Lurie's $\qSp$ (as in \cite[Definition
1.4.3.1]{Lurie:HA}), which might be called the {\qcat} of $\Omega
$-spectra.





\section{Some universal algebra}\label[appendix]{sec-univ-alg}
In \cite[Construction 2.2.3]{McCandless-curves} McCandless mentions
Lawvere's theory of monoids $T_{\Assoc}$, referring to
\cite{Lawvere63}. More details can be found in
\cite{Eilenberg-Wright}.  We need it in \cref{def-E1-monoid}.

\begin{defin}\label{def-theory}
\cite[\S4]{Eilenberg-Wright} An {\bf algebraic theory} is a category
$T$ whose objects are finite sets $\langle n\rangle=\left\{1,2,\dotsc
,n \right\}$ for $n\geq 0$ as in \cref{def-fin*}, which contains the
category of finite sets $\Fin$ as a wide subcategory.  We will denote
the singleton $\langle 1\rangle$ by $I$ and the empty set $\langle 0
\rangle$ by $\emptyset $.  Given morphisms $\phi_{i}:I\to \langle
p\rangle$ in $T$ for $1\leq i\leq n$, there is a unique morphism
\begin{displaymath}
\phi :\langle n\rangle \to\langle p\rangle
\end{displaymath}

\noindent such that for $1\leq i\leq p $, $\phi_{i}$ is the composition 
\begin{displaymath}
\xymatrix
@R=0mm
@C=10mm
{
{I}\ar[r]^(.4){i}
  &{\langle n\rangle }\ar[r]^(.5){\phi }
    &{\langle p\rangle }\\
{1}\ar@{|->}[r]^(.5){}
  &{i.}
}
\end{displaymath}

\noindent Here $i:I\to \langle n \rangle$ denotes the map sending the
single element of $I$ to $i\in \langle n \rangle$.  We will sometimes
denote $\phi $ by $(\phi_{1},\dotsc ,\phi_{n})$, and say that the
$\phi_{i}$ are its {\bf components}.
\end{defin}

In other words the object $\langle n\rangle$ the $n$-fold coproduct of
the object $I$.  This implies there is a unique morphism $\langle 0
\rangle\to \langle n \rangle$ in $T$, as in the subcategory $\Fin$.
However unlike $\Fin$, $T$ could have morphisms $\langle n\rangle\to
\langle 0 \rangle$ for $n>0$.

\begin{defin}\label{def-T-alg}
\cite[\S5 and \S6]{Eilenberg-Wright} A {\bf $T$-algebra $A$} is a
presheaf on $T$, that is a contravariant $\Set $-valued functor, that
converts coproducts to products.  This means $\langle n \rangle\mapsto
A_{n}:= A_{1}^{\times n}$, where $A_{1}:=A (I)$. 

We also require that for a morphism $\phi :\langle n \rangle\to
\langle p \rangle$ in $\Fin$, the induced map
\begin{displaymath}
A (\phi ):A_{1}^{\times p}\to A_{1}^{\times n} 
\qquad \mbox{is}\qquad 
(x_{1},\dotsc ,x_{p})\mapsto (x_{\phi (1)},\dotsc x_{\phi (n)}). 
\end{displaymath}

A morphism of $T$-algebras is a natural transformation of such
functors.  

A morphism $\phi :I\to \emptyset$ in $T$ defines an element
\begin{numequation}\label{eq-element}
\begin{split}
\phi_{A}:=A (\phi ) (A_{0})\in A_{1}.
\end{split}
\end{numequation}%

The category of $T$-algebras is denoted by $T^{\natural}$.  (It is denoted by $T^{\flat}$ in \cite{Eilenberg-Wright}.)

The {\bf free $T$-algebra on $k$ generators} (denoted by $A_{k}$
in \cite[\S6]{Eilenberg-Wright}) is the Yoneda functor
\begin{displaymath}
\yo^{T}_{k}
  :=T (-,\langle k\rangle).
\end{displaymath}
\end{defin}

The Yoneda lemma says that 
\begin{displaymath}
T^{\natural} (\yo^{T}_{k}, A)
 = A_{k}
 \cong  A_{1}^{\times k}
 = \Set (\langle k \rangle, A_{1}).
\end{displaymath}

\noindent Thus a $k$-tuple of elements in $A_{1}$ determines a
$T$-algebra morphism $\yo^{T}_{k}\to A$, and all such morphisms arise in
this way. This justifies the term ``free $T$-algebra on $k$
generators.''


Lawvere's definition of an algebraic theory in \cite{Lawvere63}, and
the one used in \cite[Construction 2.2.3]{McCandless-curves}, was the
opposite category of that of \cref{def-theory}, so for him a
$T$-algebra was a covariant product preserving $\Set$-valued functor.
The same goes for Borceaux's \cite[Definition 3.3.1]{Borceux2}.  We
will use the Eilenberg-Wright definition.

\bigskip

The following is originally due to Lawvere \cite{Lawvere-thesis} and
is proved as \cite[Proposition 3.2.9]{Borceux2}.

\begin{theorem}\label{thm-structure}
{\bf The categorical structure of an algebraic theory.} 
An algebraic theory $T$ as in \cref{def-theory} is {\eqt} to the
category of finitely generated free $T$-algebras as in
\cref{def-T-alg}.  The equivalence sends $\langle k \rangle$ to
$\yo^{T}_{k}$.
\end{theorem}

In \cite[\S7]{Eilenberg-Wright} the authors defined a {\bf free theory}
$T=\Fin[\Omega ]$ on a sequence of sets
\begin{numequation}\label{eq-Omega}
\begin{split}
\Omega =\left\{\Omega_{0},\Omega_{1},\dotsc  \right\}
\qquad \mbox{with}\qquad 
\Omega_{n}\subseteq T (I,\langle n \rangle). 
\end{split}
\end{numequation}%

\noindent Roughly speaking, it is the smallest algebraic theory that
contains the additional morphisms (to those of $\Fin$) of
\cref{eq-Omega}.

They assigned a {\bf degree} to each morphism in $\Fin[\Omega
]$ such that morphisms in $\Fin$ have degree zero, the degree of a
morphism out of $\langle n \rangle$ is the sum of the degrees of its
$n$ components, and precomposition with a morphism in some
$\Omega_{n}$ increases degree by 1.

They used induction on degree to prove the following.

\begin{theorem}\label{thm-free-theory}
\cite[\S7]{Eilenberg-Wright}
{\bf The free theory on the morphism set $\Omega $.}  There is a unique
algebraic theory $\Fin[\Omega ]$ with the following properties.

\begin{enumerate}[label={(\roman*)},itemindent=0em]
\item \label{thm-free-theoryi} For each morphism $\phi :I\to \langle p
\rangle$ of positive degree there is a unique $k\geq 0$ with a unique
factorization
\begin{displaymath}
\xymatrix
@R=4mm
@C=10mm
{
{I}\ar[r]^(.5){\omega }
  &{\langle k \rangle}\ar[r]^(.5){\psi }
    &{\langle p\rangle}
}\qquad \mbox{with $\omega \in \Omega_{k}$.} 
\end{displaymath}

\item \label{thm-free-theoryii} For any algebraic theory $T'$
containing the morphisms of \cref{eq-Omega}, there is a unique
morphism $\Fin[\Omega ]\to T'$ of theories.

\item \label{thm-free-theoryiii} Given a set $A_{1}$ and functions
$\overline{\omega }:A_{1}^{\times n}\to A_{1} $ for all $\omega \in
\Omega_{n}$ and $n\geq 0$, there exists a unique $\Fin[\Omega
]$-algebra structure on $A$ such that 
\begin{displaymath}
A (\omega ) (x_{1}, ..., x_{n})
= \overline{\omega } (x_{1}, ..., x_{n}).
\end{displaymath}
\end{enumerate}
\end{theorem}

In \cite[\S8]{Eilenberg-Wright} they specified how to construct the
quotient of a free algebraic theory $\Fin[\Omega ]$

\begin{defin}\label{def-congruence}
A {\bf congruence $Q$} in an algebraic theory $T$ is a family of
equivalence relations $\sim$ in the sets $T (\langle n \rangle, \langle p
\rangle)$ such that

\begin{enumerate}[label={(\roman*)},itemindent=0em]
\item \label{def-congruencei}  Given a diagram 
\begin{displaymath}
\xymatrix
@R=4mm
@C=10mm
{
{\langle m \rangle}\ar[r]^(.5){\psi }
  &{\langle n \rangle}\ar@<.5ex>[r]^(.5){\phi_{1} }
                      \ar@<-.5ex>[r]_(.5){\phi_{2} }
    &{\langle p \rangle}\ar[r]^(.5){\gamma  }
      &{\langle q \rangle}
}
\end{displaymath}

\noindent in $T$, $\phi_{1}\sim \phi_{1}$ implies $\gamma \phi_{1}\sim
\gamma \phi_{1}$ and $\phi_{1}\psi \sim \phi_{1}\psi $.  Congruence
plays nicely with composition.

\item \label{def-congruenceii}  
For
\begin{displaymath}
\xymatrix
@R=4mm
@C=10mm
{
{I}\ar[r]^(.45){i }
  &{\langle n \rangle}\ar@<.5ex>[r]^(.5){\phi_{1} }
                      \ar@<-.5ex>[r]_(.5){\phi_{2} }
    &{\langle p \rangle,}
}
\end{displaymath}

\noindent if $\phi_{1}i\sim \phi_{2}i$ for $1\leq i\leq n$, then
$\phi_{1}\sim \phi_{2}$.  Two morphisms are congruent iff each of their
components are.

\item \label{def-congruenceiii} If $\phi_{1},\phi_{2}:I\to \langle p
\rangle$ are in $\Fin$ with $\phi_{1}\sim \phi_{2}$, then
$\phi_{1}=\phi_{2}$.  Distinct morphisms in the subcategory $\Fin$ are
never congruent.
\end{enumerate}
\end{defin}

These conditions enable one to define an algebraic theory $T/Q$ whose
morphism sets are the sets of congruence classes in the morphism sets in
$T$.  A $T/Q$-algebra is a $T$-algebra $A$ in which 
\begin{displaymath}
A (\phi_{1}) (x_{1},\dotsc ,x_{p})
 = A (\phi_{2}) (x_{1},\dotsc ,x_{p})
\qquad \mbox{whenever }\phi_{1}\sim \phi_{2}. 
\end{displaymath}

\noindent  This makes $(T/Q)^{\natural}$ a subcategory of $T^{\natural}$.

\begin{ex}\label{ex-algebras}
In a $T$-algebra $A$, one has a set $A_{1}=A (I)$, and each morphism $I\to
\langle n \rangle$ in $T$ determines an {\bf $n$-ary operation}, meaning
a map $A^{\times n}_{1}\to A_{1}$.

\begin{enumerate}[label={(\roman*)},itemindent=0em]
\item \label{ex-algebrasi} 
If $T=\Fin$, the least interesting case, one has $n$ such maps.  The
resulting operations merely define the coordinates of an $n$-tuple.
The corresponding ``algebras'' are sets.

\item \label{ex-algebrasii} 
We will construct the theory $T_{\sg}$ whose algebras are semigroups,
meaning sets with an associative binary operation.  We need a morphism
$\pi :I \to \langle 2 \rangle$ outside of $\Fin$ for the binary
operation.  Thus we define $\Omega $ as in \cref{thm-free-theory} by
\begin{displaymath}
\Omega_{n}=\mycases{    
\left\{\pi  \right\}
       &\mbox{for }n=2\\
\emptyset 
       &\mbox{otherwise}
}
\end{displaymath}

\noindent Then a $\Fin[\Omega ]$-algebra is a set equipped with a
binary operation.  We need a congruence to make it associative.  The
following diagram must commute in the quotient theory
\begin{displaymath}
\xymatrix
@R=6mm
@C=20mm
{
{I}\ar[r]^(.5){\pi }\ar[d]_(.5){\pi }
  &{\langle 2 \rangle}\ar[d]^(.5){((1,2)\pi ,3)}\\
{\langle 2 \rangle}\ar[r]_(.5){(1,(2,3)\pi )}
    &{\langle 3 \rangle.}
}
\end{displaymath}

\noindent Here we are using the notation of \cref{def-theory} for the
two morphisms $\langle 2 \rangle\to\langle 3 \rangle$.

\item \label{ex-algebrasiii} We will construct the theory $T_{\Assoc}$
whose algebras are monoids, meaning semigroups with an identity
element.  We will derive to from $T_{\sg}$ by adjoining a morphism
$e:I\to \emptyset $, which defines the identity element $e_{A}$ as in
\cref{eq-element}.  Then we must pass to a suitable quotient to insure
that $e_{A}$ has the desired properties.  Let $e_{1}:I\to I$ be the
composite
\begin{displaymath}
\xymatrix
@R=4mm
@C=10mm
{
{I}\ar[r]^(.5){e}
  &{\emptyset }\ar[r]^(.5){\sigma }
    &{I.}
}
\end{displaymath}

\noindent  where $\sigma :\emptyset \to I$ is the unique such morphism.

Then the following diagram must commute in $T_{\Assoc}$.
\begin{displaymath}
\xymatrix
@R=8mm
@C=15mm
{
{I}\ar[r]^(.5){\pi }\ar[d]_(.5){\pi }\ar[dr]^(.5){1}
  &{\langle 2 \rangle}\ar[d]^(.5){(e_{1}, 1)}\\
{\langle 2 \rangle}\ar[r]_(.5){(1, e_{1})}
    &{I.}
}
\end{displaymath}
\end{enumerate}  
\end{ex}

\section{Symmetric monoidal {\qcats}}\label[appendix]{sec-SM-qcat} 

The structure of a symmetric monoidal {\qcat} $\qmcC$ is not as simple
as a functor $\qmcC \times \qmcC \to \qmcC$ with the expected
properties. It is discussed at length by Lurie leading up to  \cite[Definition
2.0.0.7]{Lurie:HA}, by Moritz Groth in \cite[3.2]{Groth}, and by
Nikolaus-Scholze in \cite[Appendix A]{NS18}.

We begin with three definitions in {\em ordinary category theory} that
will motivate Lurie's definition of a symmetric monoidal {\qcat}.

\begin{defin}\label{def-fin*}
The {\bf categories of finite sets $\Fin$ and pointed finite sets
$\Fin_{*}$} have as  objects the sets
\begin{numequation}\label{eq-langle-n}
\begin{split}
\langle n\rangle
 &:= \left\{1,\dotsc ,n \right\}\\
\aand 
\langle n\rangle_{*}
 &:= \left\{0,1,\dotsc ,n \right\}
\end{split}
\end{numequation}%

\noindent respectively, with 0 as basepoint in $\Fin_{*}$. A morphism
$f:\langle n\rangle_{*}\to \langle m\rangle_{*}$ is a map of sets sending 0
to 0.  For $1\leq i\leq n$, we define $\varrho^{i}:\langle
n\rangle_{*}\to \langle 1\rangle_{*}$ by 
\begin{numequation}\label{eq-rho-i}
\begin{split}
\varrho^{i} (j):=\mycases{    
1      &\mbox{for }j=i\\
0      &\mbox{otherwise.}
}
\end{split}
\end{numequation}%

A morphism $f$ in $\Fin_{*}$ as above is {\bf active} if it sends each
nonzero element in the domain to one in the codomain. It is {\bf
inert} if the preimage of each nonzero nonzero element in the codomain
is a singleton.  Such a map defines an injection $\langle n \rangle\to
\langle m \rangle$ sending each element to its preimage under $f$.
\end{defin}

Note that 
\begin{numequation}\label{eq-subset-S}
\begin{split}
\xymatrix
@R=4mm
@C=10mm
{
{\langle n\rangle_{*}}\ar[r]^(.5){f}
  &{\langle m\rangle_{*}}
    &{\leadsto }
      &{\langle n\rangle_{*}\supseteq f^{-1} (\langle m\rangle)=:S}
           \ar[r]^(.7){\alpha }
        &{\langle m\rangle,}
}
\end{split}
\end{numequation}%

\noindent {\ie } $f$ leads to a partially defined map $\langle
n\rangle\to\langle m\rangle$.  The subset $S$
above is all of $\langle n\rangle$ iff $f$ is active, so the
wide subcategory of active morphisms is isomorphic to $\Fin$.

We will give a short introduction to operads, of which the following is
a special case, in \cite[Appendix B]{Rav:gjmcyc-ext}.

\begin{defin}\label{def-A-otimes}\cite[Definition 4.1.1.3]{Lurie:HA}
The {\bf associative operad $\Assoc^{\otimes }$}  is a
category whose objects are those of $\Fin_{*}$. When $\langle
n\rangle_{*}$ as in \cref{eq-langle-n} is an object of $\Assoc^{\otimes}$,
we write it as $\langle n\rangle_{\Assoc}$.  A morphism
\begin{displaymath}
\tilde{f}:\langle n\rangle_{\Assoc}\to \langle m\rangle_{\Assoc}
\end{displaymath}

\noindent is a
map $f:\langle n\rangle_{*}\to \langle m\rangle_{*}$ in $\Fin_{*}$ together
with a linear ordering on each inverse image $f^{-1} (i)\subseteq
\langle n \rangle_{*}$ for $1\leq i\leq m$.  For a composite
\begin{displaymath}
\xymatrix
@R=1mm
@C=10mm
{
  &{j}\ar@{|->}[r]^(.5){}
    &{i}\\
{\langle n \rangle_{\Assoc}}\ar[r]^(.5){\tilde{f}}
  &{\langle m\rangle_{\Assoc}}\ar[r]^(.5){\tilde{g}}
    &{\langle \ell \rangle_{\Assoc},}
}
\end{displaymath}

\noindent in $\Assoc^{\otimes } $, we define a linear ordering on $(gf)^{-1}
(i)=f^{-1} (g^{-1} (i))$ for $1\leq i\leq \ell $ as follows.  For each
$j\in g^{-1} (i)$ we have a linear ordering of $f^{-1} (j)$ associated
with $\tilde{f}$, and the $j$s themselves have an ordering associated
with $\tilde{g}$.  These lead to a lexicographic ordering on
$(gf)^{-1} (i)$ and thus determine the composite $\tilde{g}\tilde{f}$
in $\Assoc^{\otimes }$.

The subcategory $\Assoc^{\otimes }_{\act } \subseteq \Assoc^{\otimes }$ is
the wide subcategory whose morphisms project to active morphisms in
$\Fin_{*}$, namely
\begin{numequation}\label{eq-Ass-act}
\begin{split}
\Assoc^{\otimes }_{\act }:=\Assoc^{\otimes }\times_{\Fin_{*}}\Fin.
\end{split}
\end{numequation}%

We denote by $V:\Assoc^{\otimes}\to \Fin_{*} $ the functor that assigns to each object in $\Assoc^{\otimes}$ its underlying pointed finite set.
\end{defin}







\begin{remark}\label{rem-assoc-alg}
{\bf Associative algebras as functors.}  A symmetric monoidal functor
$A$ from $\Assoc^{\otimes }$ to a symmetric monoidal category $(\mcC
,\otimes )$ defines an associative algebra $A\langle 1 \rangle$ in
$\mcC $, and each such algebra determines such a functor sending
$\langle n \rangle$ to $A^{\otimes n}$.  The {\qcatal} analog of this
statement is a special case of \cite[Proposition 2.2.4.9]{Lurie:HA}.
\end{remark}

\begin{defin}\label{def-seq-cat}
\cite[Construction 2.0.0.1]{Lurie:HA} Let $(\mcC ,\otimes )$ be a
symmetric monoidal category.  The {\bf sequence category
$\mcC^{\otimes }$} has as objects finite (possibly empty) sequences of
objects $C_{i}$ in $\mcC $ for $1\leq i\leq n$, which we denote by
$( C_{1},\dotsc ,C_{n})$.  A morphism
\begin{displaymath}
(C_{1},\dotsc ,C_{n})\to (C'_{1},\dotsc ,C'_{m})
\end{displaymath} 

\noindent consists of a subset $S\subseteq \langle n
\rangle_{*}$ as in \cref{eq-subset-S}, a map of finite sets
$\alpha :S\to \langle m \rangle_{*}$, and a collection of
morphisms
\begin{numequation}\label{eq-fj}
\begin{split}
f_{j}:\bigotimes_{i\in \alpha^{-1} (j)}C_{i} \to C'_{j}
\qquad \mbox{for }1\leq j\leq m, 
\end{split}
\end{numequation}%

\noindent where each of the tensor products is defined up to canonical
isomorphism by the monoidal structure on $\mcC $.  In the  composite
\begin{displaymath}
\xymatrix
@R=1mm
@C=10mm
{
{(C_{1},\dotsc ,C_{n })}\ar[r]^(.5){f}
  &{(C'_{1},\dotsc ,C'_{m })}\ar[r]^(.5){g}
    &{(C''_{1},\dotsc ,C''_{\ell  })}\\
{\langle n\rangle_{*}\supseteq S}
  \ar[r]^(.5){\alpha }
  &{\langle m\rangle_{*}}\\
  &{\langle m\rangle\supseteq T}
               \ar[r]^(.5){\beta }
     &{\langle \ell \rangle_{*}}
}
\end{displaymath}

\noindent The subset $U\subseteq \langle n \rangle$
associated with $gf$ is $\alpha^{-1} (T)=\alpha^{-1}\beta^{-1}
(\langle m \rangle)$, and the map $\gamma :U\to \langle \ell
\rangle$ is $\beta \alpha $.

We have a forgetful functor 
\begin{numequation}\label{eq-pr}
\begin{split}
\xymatrix
@R=0mm
@C=20mm
{
{\mcC^{\otimes }}\ar[r]^(.5){\pr}
  &{\Fin_{*}}\\
{(C_{1},\dotsc ,C_{n })}\ar@{|->}[r]^(.5){}
    &{\langle n\rangle_{*}.}
}
\end{split}
\end{numequation}%

\noindent
We denote by $\mcC^{\otimes }_{\langle n \rangle}$ the full
subcategory of sequences of length $n$.  The full subcategory of
active morphisms is
\begin{displaymath}
\mcC^{\otimes }_{\act}:= \pr ^{-1} (\Fin),
\end{displaymath}

\noindent where $\Fin$ is understood to be the wide subcategory of
active morphisms in $\Fin_{*}$.
\end{defin}


\begin{ex}\label{ex-assoc-operad-seq-cat}
{\bf The associative operad as a sequence category.}  When $\mcC $ is the
category with a single object  and a single morphism, and hence a
unique monoidal structure which is symmetric, $\mcC^{\otimes }$ is the
associative operad $\Assoc^{\otimes}$.  Each $\mcC^{\otimes }_{\langle n
\rangle}$ has a single object which we can identify with $\langle n
\rangle_{\Assoc}$.  The linear ordering of each $f^{-1} (i)$ is that
inherited from the notational ordering of the set $\langle n \rangle_{*}$.
\end{ex}

\begin{prop}\label{prop-forget-props}
{\bf Properties of the forgetful functor to $\Fin_{*}$.}

Let  $\pr$ be the forgetful functor of \cref{eq-pr}.

\begin{enumerate}[label={(\roman*)},itemindent=0em]
\item \label{prop-forget-propsi} $\pr $ is a {\bf Grothendieck
op-fibration of categories} (see \cite[Definition 2.8.1]{HHR:ESHT}),
meaning that for every object
\begin{displaymath}
 C =
(C_{1},\dotsc  , C_{n}) \in  \mcC^{\otimes }_{\langle n \rangle}
\end{displaymath}

\noindent
and every morphism $f:\langle n \rangle_{*}\to \langle m \rangle_{*}$ in
$\Fin_{*}$, there exists a morphism
\begin{displaymath}
\overline{f} : C \to C' = (C'_{1}, \dotsc  , C'_{m})
\end{displaymath}

\noindent which covers $f$ (meaning $\pr (\overline{f})=f$), and is
universal in the sense that composition with $\overline{f}$ induces a
bijection
\begin{numequation}\label{eq-op-fib}
\begin{split}
\xymatrix
@R=6mm
@C=10mm
{
{\Hom_{\mcC^{\otimes }}(C', C'')}\ar[d]^(.5){\cong }\\
{\Hom_{\mcC^{\otimes }}(C, C'')
   \times_{\Hom_{\Fin_{*}} (\langle n \rangle_{*},  \langle \ell  \rangle_{*})}
    \Hom_{\Fin_{*}} (\langle m \rangle_{*},  \langle \ell  \rangle_{*})}
}
\end{split}
\end{numequation}%

\noindent for every object $C''=(C''_{1},\dotsc ,C''_{\ell })$ in
$\mcC^{\otimes }_{\langle \ell \rangle}$.  (Such a morphism can be
constructed using the morphisms $f_{j}$ of \cref{eq-fj}.)

\item \label{prop-forget-propsii} 
$\mcC^{\otimes }_{\langle 1 \rangle}$ is {\eqt} to $\mcC $, and
functors associated with the morphisms $\varrho^{i}$ of
\cref{eq-rho-i} lead to an equivalence of $\mcC^{\otimes }_{\langle n
\rangle}$ with $\mcC^{\times n} $ for all $n\geq 1$.
\end{enumerate}
\end{prop}

{\bf The punchline.}  Now we come to a brilliant observation of Lurie.
{\em Suppose we forget how the sequence category $\mcC^{\otimes }$ was
constructed, and assume only that it is equipped with a
$\Fin_{*}$-valued functor $\pr $ having analogs of the two properties in
\cref{prop-forget-props}.}

Then we have full subcategories $\mcC^{\otimes }_{\langle n
\rangle}:=\pr ^{-1}\langle n \rangle_{*}$ for all $n\geq 0$, and we can
define $\mcC $ to be $\mcC^{\otimes }_{\langle 1 \rangle}$.  In
\cref{eq-op-fib}, the objects $C$, $C'$ and $C''$ lie in
$\mcC^{\otimes }_{\langle n \rangle}$, $\mcC^{\otimes }_{\langle m
\rangle}$ and $\mcC^{\otimes }_{\langle \ell \rangle}$ respectively.
For \cref{prop-forget-propsii} we can still require that functors
associated with the morphisms $\varrho^{i}$ of \cref{eq-rho-i} lead to
an equivalence of $\mcC^{\otimes }_{\langle n \rangle}$ with
$\mcC^{\times n} $ for all $n\geq 1$.

For example 
\begin{itemize}
\item [$\bullet$] $\mcC^{\otimes }_{\langle 0 \rangle}$ has one
object, and the unique map $\langle 0 \rangle_{*}\to\langle 1 \rangle_{*}$
induces a functor $\mcC^{\otimes }_{\langle 0 \rangle}\to \mcC $
identifying the unit object.

\item [$\bullet$] The unique active map $f:\langle 2 \rangle_{*}\to\langle
1 \rangle_{*}$ leads to the monoidal structure\linebreak 
$\mcC\times \mcC  \to \mcC $.

\item [$\bullet$] Let $t:\langle 2 \rangle_{*}\to\langle 2 \rangle_{*}$ be the
automorphism that interchanges 1 and 2.  The identity $f t=f $ leads
to the symmetry condition.

\item [$\bullet$] The unique active map $\langle 3 \rangle_{*}\to\langle 1
\rangle_{*}$ can be factored as the composite of order preserving active
maps $\langle 3 \rangle_{*}\to\langle 2 \rangle_{*}\to\langle 1 \rangle_{*}$ in
two different ways, which leads to the associativity condition.
\end{itemize}

Further relations in $\Fin_{*}$ lead to further structure in $\mcC$
that would be tedious to spell out explicitly.  The functor $\pr$ gives
us a painfree and obvious way to specify it.

This suggests replacing $\mcC^{\otimes }$ by an {\qcat} $\qmcCot$ and
the forgetful functor $\pr$ by a map of simplicial sets to $N(\Fin_{*})$, the
nerve of \cref{def-nerve}.

\begin{defin}\label{def-symm-mon-qcat}
\cite[Definition 2.0.0.7]{Lurie:HA} and \cite[Definition A.1]{NS18}. 
A {\bf symmetric monoidal {\qcat}} is an {\qcat} $\qmcCot$
(the {\bf total space}) with a coCartesian fibration (see
\cite[Definition 2.4.2.1]{Lurie:HTT}) of simplicial sets  
\begin{displaymath}
\pr :\qmcCot
\to N (\Fin_{*}),
\end{displaymath}

\noindent in which we define $\qmcCot_{\langle n
\rangle}:=\pr ^{-1}\langle n \rangle_{*}$, having the following property
For each $n\geq 0$, the morphisms $\varrho^{i}$ of
\cref{eq-rho-i} induce functors 
\begin{displaymath}
\varrho^{i}_{!}:\qmcCot_{\langle n \rangle}
\to \qmcCot_{\langle 1\rangle}
\qquad \mbox{ for } \leq i\leq n
\end{displaymath}

\noindent which determine an equivalence $\qmcCot_{\langle n
\rangle}\simeq (\qmcCot_{\langle 1 \rangle})^{n}$.

The {\bf underlying {\qcat} of $\qmcCot$} is $\qmcC:=\qmcCot_{\langle
1 \rangle}$.
\end{defin}

Lurie then goes on to say 

\begin{quote}
One of our main goals in this book is to show that Definition 2.0.0.7
is reasonable: that is, it provides a robust generalization of the
classical theory of symmetric monoidal categories.
\end{quote}

As in \cref{rem-assoc-alg}, the {\qcat} of associative algebras in
$\qmcC$ is
\begin{displaymath}
\Alg (\qmcC):= \Fun^{\otimes } (\Assoc^{\otimes }, \qmcC),
\end{displaymath}

\noindent the {\qcat} of symmetric monoidal functors; see
\cite[Proposition 2.2.4.9]{Lurie:HA}.

\bibliography{../../math} 
\bibliographystyle{alpha} 

\end{document}